%% file: noc2-SphericalMultiBubble.tex
\documentclass[a4paper,twosided,11pt]{article}
\usepackage{amsfonts}
\usepackage{latexsym}
\usepackage{amsmath}
\usepackage{amssymb}
\usepackage{amsthm}
\usepackage{amsmath}
\usepackage[shortlabels]{enumitem}
\usepackage{hyperref}
\usepackage{verbatim}
\usepackage{tikz}
\usepackage{tikz-cd}
\usepackage{enumitem} \usetikzlibrary{calc}

\tikzset{vertex/.style = {shape=circle,fill,inner sep=0pt,minimum size=2mm}}

\newtheorem{theorem}{Theorem}[section]
\newtheorem*{theorem*}{Theorem}

\newtheorem{corollary}[theorem]{Corollary}

\newtheorem{lemma}[theorem]{Lemma}

\newtheorem*{prop*}{Proposition}
\newtheorem{proposition}[theorem]{Proposition}

\newtheorem*{conjecture*}{Conjecture}
\newtheorem{definition}[theorem]{Definition}
\newtheorem*{definition*}{Definition}
\theoremstyle{definition}
\newtheorem{remark}[theorem]{\textbf{Remark}}
\newtheorem*{remark*}{Remark}
\newtheorem*{fact*}{Fact}

\newcommand{\vol}{\textrm{Vol}}
\newcommand{\norm}[1]{\left\Vert#1\right\Vert}
\newcommand{\snorm}[1]{\Vert#1\Vert}
\newcommand{\abs}[1]{\left\vert#1\right\vert}
\newcommand{\set}[1]{\left\{#1\right\}}
\newcommand{\brac}[1]{\left(#1\right)}
\newcommand{\scalar}[1]{\left \langle #1 \right \rangle}
\newcommand{\sscalar}[1]{\langle #1 \rangle}
\newcommand{\Y}{\mathbf{Y}}
\newcommand{\T}{\mathbf{T}}

\newcommand{\R}{\mathbb{R}}
\newcommand{\GG}{\mathbb{G}}
\newcommand{\RP}{\mathbb{RP}}
\newcommand{\GL}{\mathrm{GL}}
\newcommand{\PGL}{\mathrm{PGL}}
\newcommand{\CC}{\mathbf{C}}
\newcommand{\G}{\mathcal{G}}

\newcommand{\E}{\mathcal{E}}

\newcommand{\II}{\mathrm{I\!I}}
\newcommand{\Id}{\text{Id}}
\newcommand{\sym}{\text{sym}}
\newcommand{\M}{\mathbb{M}}
\newcommand{\F}{\mathcal{F}}

\renewcommand{\H}{\mathcal{H}}
\newcommand{\eps}{\epsilon}

\renewcommand{\SS}{\mathcal{S}}

\newcommand{\Ric}{\text{\rm Ric}}

\newcommand{\tr}{\text{\rm tr}}

\renewcommand{\S}{\mathbb{S}}
\newcommand{\B}{\mathbb{B}}
\renewcommand{\div}{\text{\rm div}}

\newcommand{\n}{\mathfrak{n}}
\renewcommand{\c}{\mathbf{c}}
\newcommand{\s}{\mathbf{s}}
\renewcommand{\k}{\mathbf{k}}
\newcommand{\ck}{\mathbf{ck}}
\newcommand{\tang}{\mathbf{t}}
\newcommand{\cyclic}{\mathcal{C}}
\newcommand{\C}{\mathcal{C}}
\newcommand{\Tr}{\mathcal{T}}
\newcommand{\Lip}{\textrm{Lip}}

\newcommand{\no}{\textrm{no}}
\newcommand{\simplex}{\Delta}
\newcommand{\sspan}{\text{span}}
\renewcommand{\L}{\text{L}}
\newcommand{\per}{P}
\newcommand{\pot}{W}

\DeclareMathOperator*{\argmin}{arg\,min}
\DeclareMathOperator{\interior}{int}
\DeclareMathOperator{\closure}{cls}
\DeclareMathOperator{\rank}{rank}

\newlength{\defbaselineskip}
\numberwithin{equation}{section}

\begin{document}

\title{The Structure of Isoperimetric Bubbles on $\R^n$ and $\S^n$} \date{}

\author{Emanuel Milman\textsuperscript{1} and Joe Neeman\textsuperscript{2}}

\footnotetext[1]{Department of Mathematics, Technion - Israel
Institute of Technology, Haifa, Israel; and Oden Institute, University of Texas at Austin. Email: emilman@tx.technion.ac.il.}
\footnotetext[2]{Department of Mathematics, University of Texas at Austin. Email: joeneeman@gmail.com.\\
The research leading to these results is part of a project that has received funding from the European Research Council (ERC) under the European Union's Horizon 2020 research and innovation programme (grant agreement No 101001677). J.N. was partially supported by the Deutsche Forschungsgemeinschaft (DFG, German Research Foundation) under Germany’s Excellence Strategy – EXC-2047/1 – 390685813, and by a fellowship from the Alfred P. Sloan Foundation.}

\begingroup    \renewcommand{\thefootnote}{}    
\endgroup

\maketitle

\begin{abstract}
The multi-bubble isoperimetric conjectures in $n$-dimensional Euclidean and spherical spaces from the 1990's assert that standard bubbles uniquely minimize total perimeter among all $q-1$ bubbles enclosing prescribed volume, for any $q \leq n+2$. The double-bubble conjecture on $\R^3$ was confirmed in 2000 by Hutchings--Morgan--Ritor\'e--Ros, and is nowadays fully resolved on $\R^n$ for all $n \geq 2$. The double-bubble and triple-bubble conjectures on $\R^2$ and $\S^2$ have also been resolved, but all other cases are in general open. We confirm the conjectures on $\R^n$ and on $\S^n$ for all $q \leq \min(5,n+1)$, namely: the double-bubble conjectures for $n \geq 2$, the triple-bubble conjectures for $n \geq 3$ and the quadruple-bubble conjectures for $n \geq 4$. In fact, we show that for all $q \leq n+1$, a minimizing cluster necessarily has spherical interfaces, and after stereographic projection to $\S^n$, its cells are obtained as the Voronoi cells of $q$ affine-functions, or equivalently, as the intersection with $\S^n$ of convex polyhedra in $\R^{n+1}$. Moreover, the cells (including the unbounded one) are necessarily connected and intersect a common hyperplane of symmetry, resolving a conjecture of Heppes. We also show for all $q \leq n+1$ that a minimizer with non-empty interfaces between all pairs of cells is necessarily a standard bubble. The proof makes crucial use of considering $\R^n$ and $\S^n$ in tandem and of M\"obius geometry and conformal Killing fields;  it does not rely on establishing a PDI for the isoperimetric profile as in the Gaussian setting, which seems out of reach in the present one. 
\end{abstract}

\section{Introduction}

A weighted Riemannian manifold $(M^n,g,\mu)$ consists of a smooth complete $n$-dimensional Riemannian manifold $(M^n,g)$ endowed with a measure $\mu$ with $C^\infty$ smooth positive density $\Psi$ with respect to the Riemannian volume measure $\vol_g$. The metric $g$ induces a geodesic distance on $M^n$, and the corresponding $k$-dimensional Hausdorff measure is denoted by $\H^k$. Let $\mu^{k} := \Psi \H^{k}$.  The $\mu$-weighted perimeter of a Borel subset $U \subset M^n$ of locally finite perimeter is defined as $\per_\mu(U) := \mu^{n-1}(\partial^* U)$, where $\partial^* U$ is the reduced boundary of $U$ (see Section \ref{sec:prelim} for definitions).

A \emph{$q$-cluster} $\Omega = (\Omega_1, \ldots, \Omega_q)$ is a $q$-tuple of Borel subsets with locally finite perimeter $\Omega_i \subset M^n$ called cells, such that $\set{\Omega_i}$ are pairwise disjoint, $\mu(M^n \setminus \bigcup_{i=1}^q \Omega_i) = 0$, $\per_\mu(\Omega_i) < \infty$ for each $i$, and $\mu(\Omega_i) < \infty$ for all $i=1,\ldots,q-1$. Note that when $\mu(M^n) = \infty$, then necessarily $\mu(\Omega_q) = \infty$. Also note that the cells are not required to be connected.
The $\mu$-weighted volume of $\Omega$ is defined as:
\[
  \mu(\Omega) := (\mu(\Omega_1), \ldots, \mu(\Omega_q)) \in \simplex^{(q-1)}_{\mu(M^n)} ,
\]
where $\simplex^{(q-1)}_T := \{v \in \R^{q}_+  \; ; \; \sum_{i=1}^q v_i = T\}$ if $T < \infty$ and 
$\simplex^{(q-1)}_\infty := \R^{q-1}_+ \times \{ \infty\}$. 
The $\mu$-weighted total perimeter of a cluster is defined as:
\[
\per_\mu(\Omega) := \frac{1}{2} \sum_{i=1}^q \per_\mu(\Omega_i) = \sum_{1 \leq i < j \leq q} \mu^{n-1}(\Sigma_{ij}) ,
\]
where $\Sigma_{ij} := \partial^* \Omega_i \cap \partial^* \Omega_j$ denotes the $(n-1)$-dimensional interface between cells $i$ and $j$.

The isoperimetric problem for $q$-clusters consists of identifying those clusters $\Omega$ of prescribed volume $\mu(\Omega) = v$ which minimize the total perimeter $\per_\mu(\Omega)$. Note that for a $2$-cluster $\Omega = (\Omega_1,\Omega_2)$, $\per_\mu(\Omega) = \per_\mu(\Omega_1) = \per_\mu(\Omega_2)$, and so the case $q=2$ corresponds to the classical isoperimetric setup of minimizing the perimeter of a single set of prescribed volume. Consequently, the case $q=2$ is referred to as the ``single-bubble" case (with the bubble being $\Omega_1$ and having complement $\Omega_2$). Accordingly, the case $q=3$ is called the ``double-bubble" problem, the case $q=4$ is the ``triple-bubble" problem, etc... The case of general $q$ is referred to as the ``multi-bubble" problem. 

\bigskip

In this work, we will focus on (arguably) the two most natural and important settings: 
\begin{itemize}
\item Euclidean space $\R^n$, namely $(\R^n,|\cdot|)$ endowed with the Lebesgue measure $V$. 
\item Spherical space $\S^n$, namely the unit-sphere $(\S^n,g)$ in its canonical embedding in $\R^{n+1}$, endowed with its Haar measure $V$ (normalized to have total mass $1$).  
\end{itemize}
The following definition and corresponding conjectures were put forth by J.~Sullivan in the 90's \cite[Problem 2]{OpenProblemsInSoapBubbles96}. Recall that the (open) Voronoi cells of distinct $\set{x_1,\ldots,x_q} \subset \R^N$ are defined as:
\[
\Omega_i = \interior \set{ x \in \R^N \;;\; i \in \argmin_{j=1,\ldots,q} \abs{x-x_j} } = \set{ x \in \R^N \;;\; \argmin_{j=1,\ldots,q} \abs{x-x_j} = \{i\} } ~,~  i=1,\ldots, q ~,
\]
where $\interior$ denotes the (relative) interior operation and $\argmin$ denotes the subset of indices on which the corresponding minimum is attained.

\begin{definition}[Standard Bubble on $\R^n$ and $\S^n$]
Given $2 \leq q \leq n+2$, the equal-volume standard $(q-1)$-bubble on $\S^n$ is the $q$-cluster $\Omega$ obtained by intersecting $\S^n$ with the Voronoi cells of $q$ equidistant points in $\S^n \subset \R^{n+1}$. \\
A standard $(q-1)$-bubble on $\R^n$ is a stereographic projection of the equal-volume standard $(q-1)$-bubble $\Omega$ on $\S^n$ with respect to some North pole in (the open) $\Omega_q$. \\
A standard $(q-1)$-bubble on $\S^n$ is a stereographic projection of a standard $(q-1)$-bubble on $\R^n$. 
\end{definition}

\begin{figure}
    \begin{center}
            \raisebox{-0.1\height}{\includegraphics[scale=0.1]{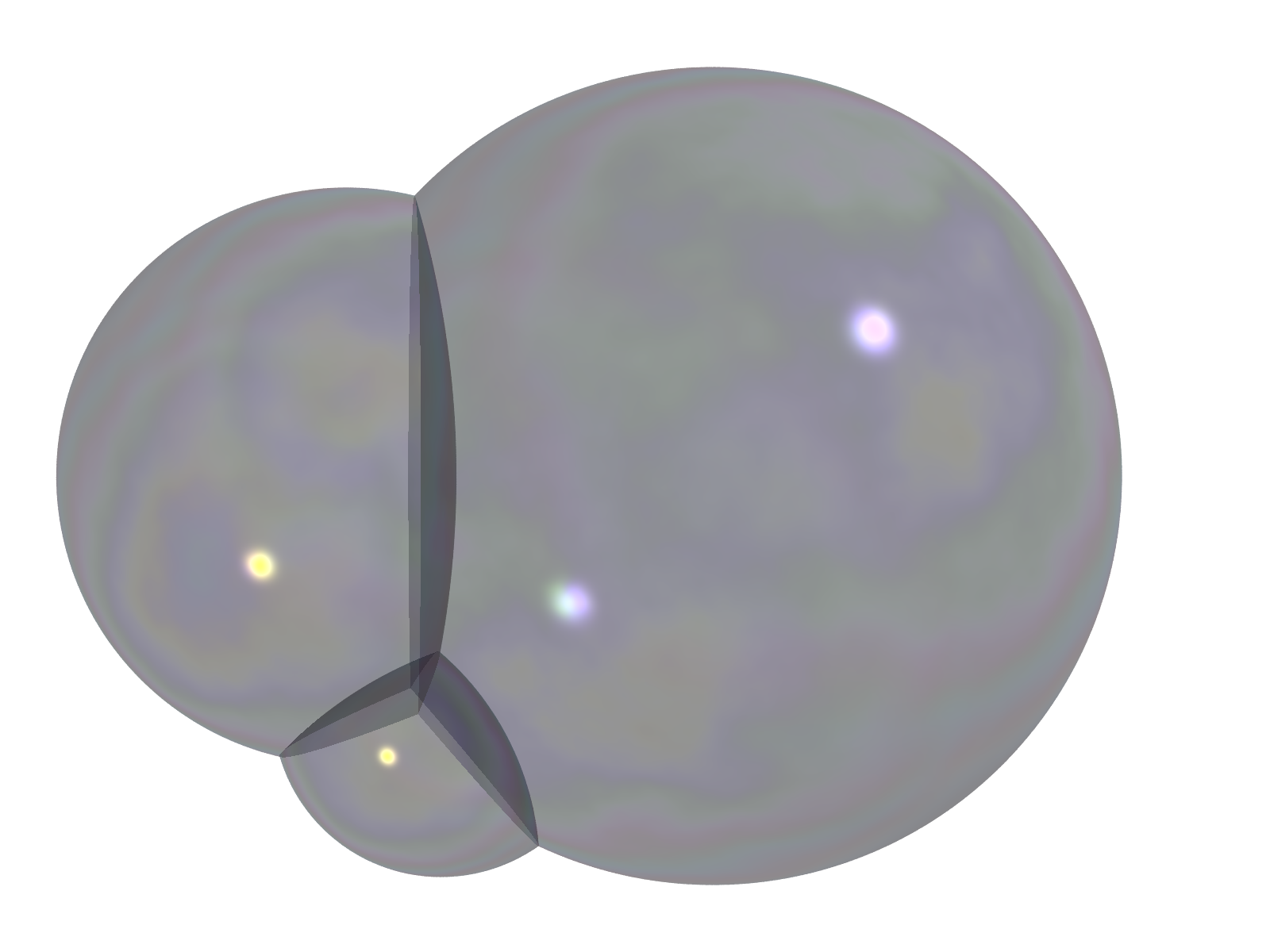}}
        \hspace{20pt}
        \begin{tikzpicture}[scale=1.2]
            \input{triple-bubble}
        \end{tikzpicture}
     \end{center}
     \caption{
         \label{fig:triple-bubble}
        Left: a standard triple-bubble in $\R^3$. Right: the $2$D cross-section through its plane of symmetry. 
     }
\end{figure}

\begin{figure}
    \begin{center}
            \includegraphics[scale=0.1]{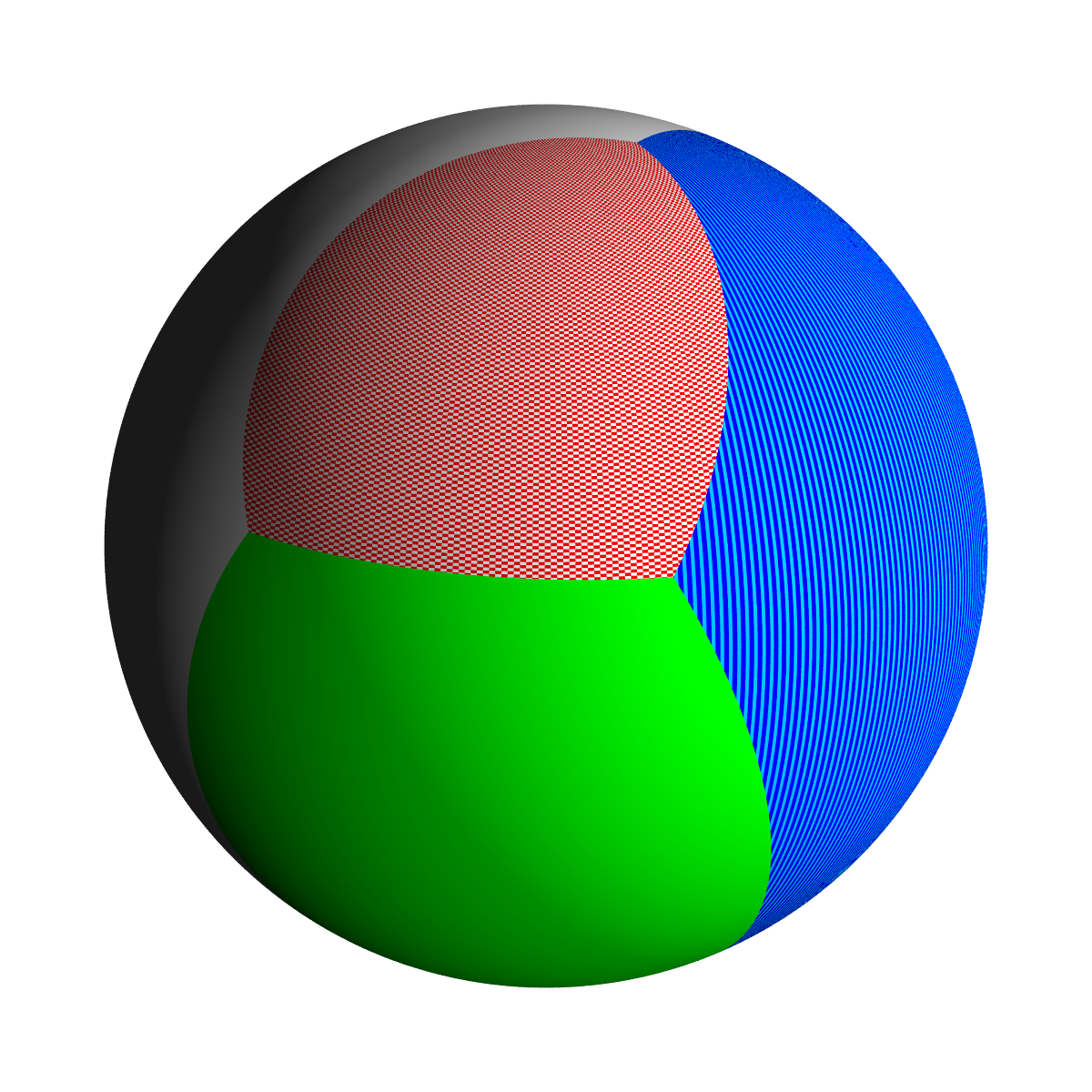}
     \end{center}
     \caption{
         \label{fig:triple-bubble-S2}
        A standard triple-bubble in $\S^2$; also, the cross-section of a standard triple-bubble in $\S^3$ through its hyperplane of symmetry. 
     }
\end{figure}

\begin{figure}
    \begin{center}
        \includegraphics[scale=0.09]{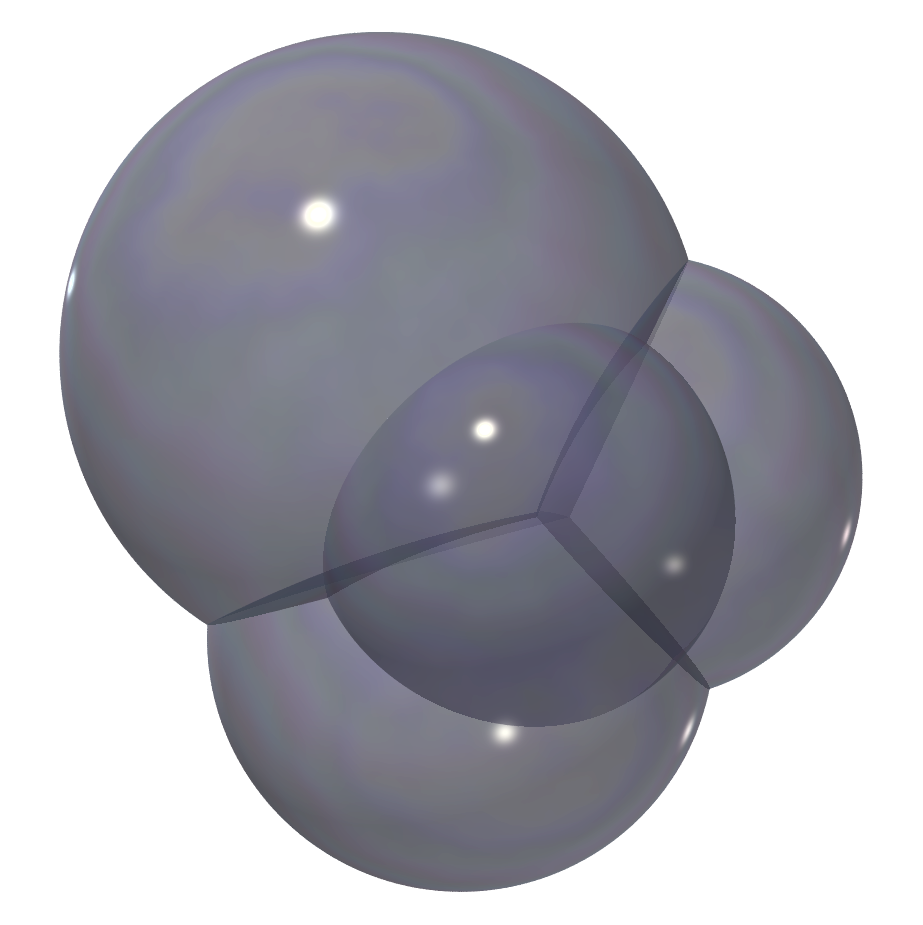}
        \includegraphics[scale=0.09]{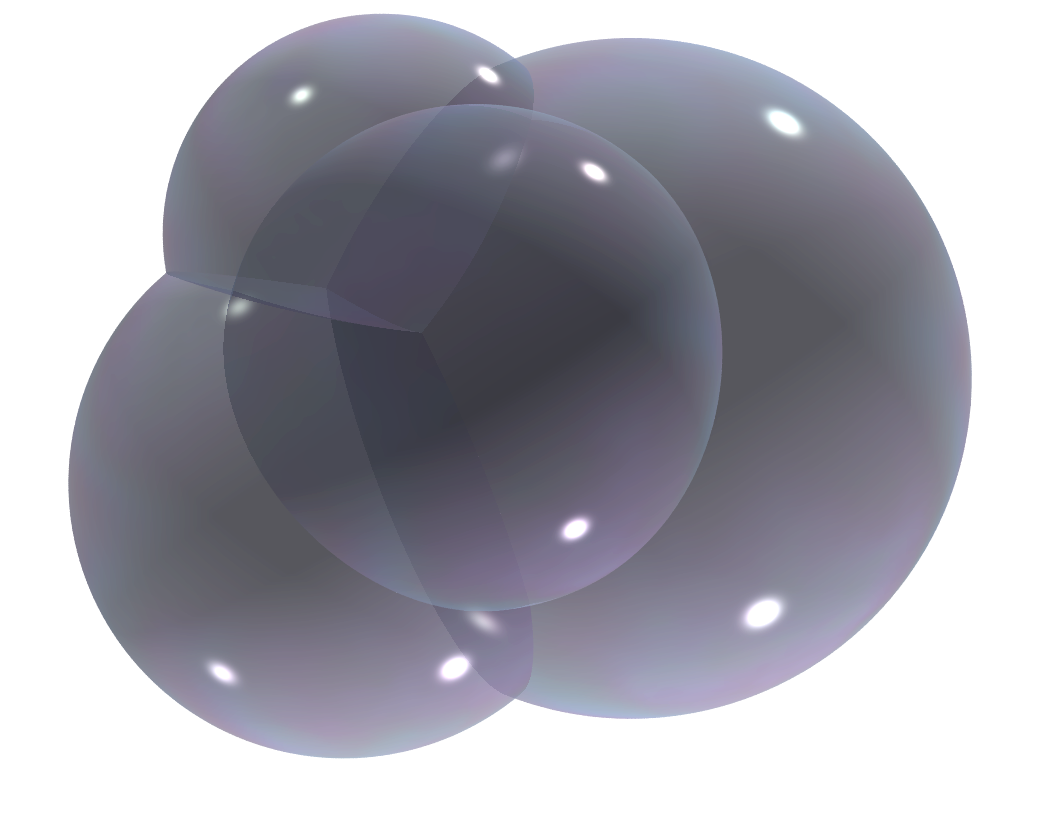}
        \includegraphics[scale=0.09]{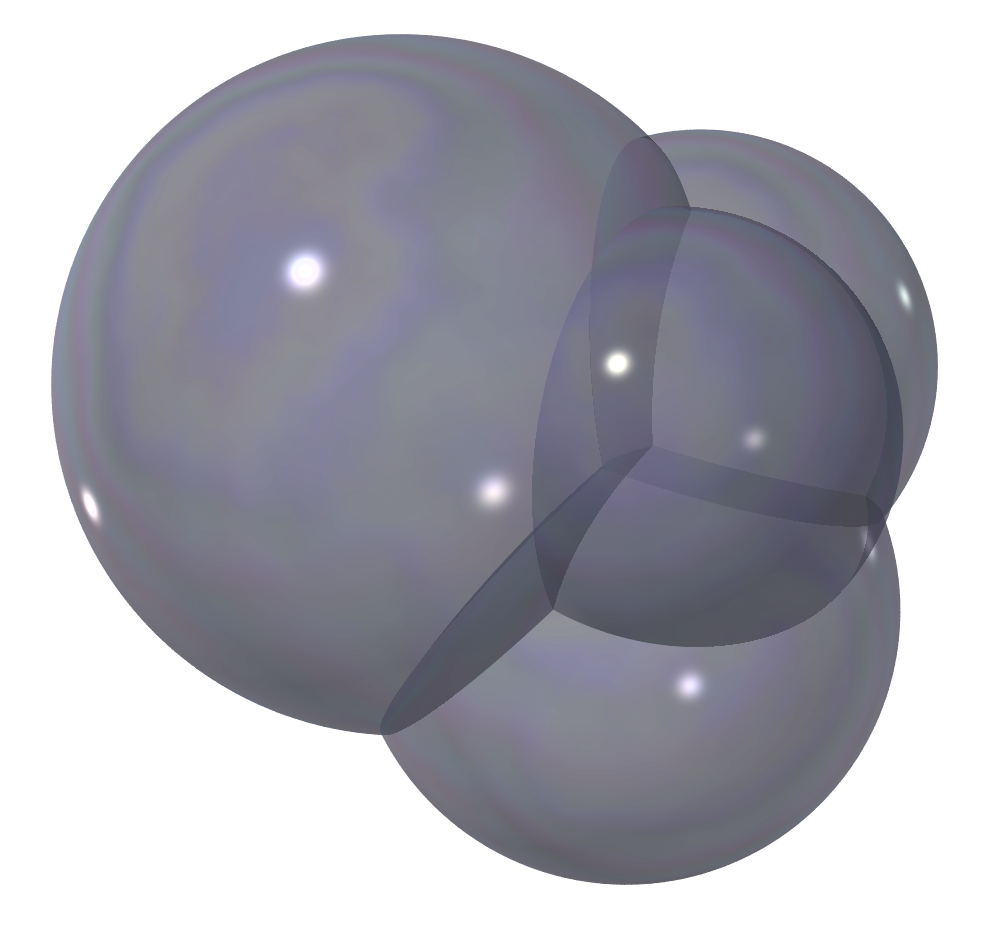}
        \includegraphics[scale=0.09]{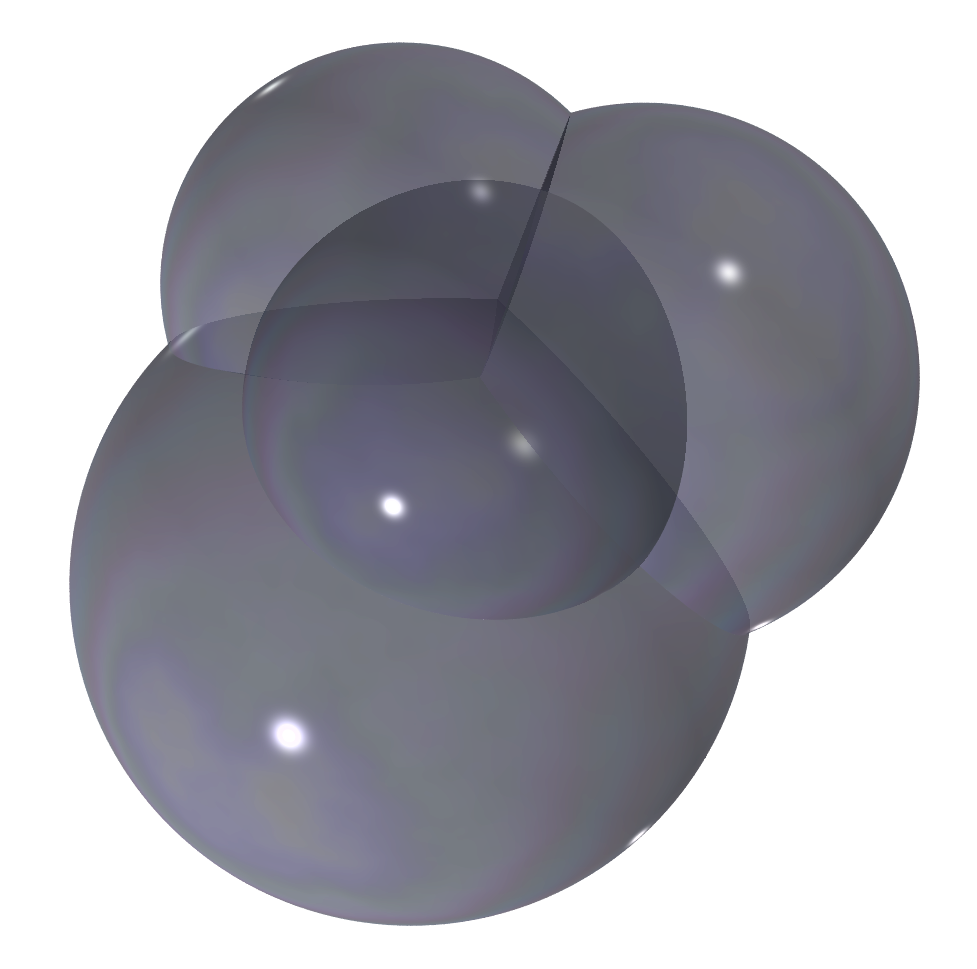}
     \end{center}
     \caption{
         \label{fig:quad-bubble-R3}
        A standard quadruple-bubble in $\R^3$ (also, the cross-section of a standard quadruple-bubble in $\R^4$ through its hyperplane of symmetry) from different angles. 
     }
\end{figure}

Note that the above construction ensures that a standard $(q-1)$-bubble $\Omega$ on $\R^n$ has a single unbounded cell $\Omega_q$, yielding a legal cluster whose boundary 
can be thought of as a soap film enclosing $q-1$ bounded (connected) air bubbles; we will henceforth omit the $q-1$ index when referring to standard bubbles. Also note that the equal-volume standard-bubble on $\S^n$ has flat interfaces lying on great spheres, and that whenever three of these interfaces meet, they do so at $120^\circ$-degree angles. As stereographic projection is conformal and preserves (generalized) spheres (of possibly vanishing curvature, i.e.~totally geodesic hypersurfaces), it follows that all standard bubbles on $\R^n$ and $\S^n$ have (generalized) spherical interfaces which meet in threes at $120^\circ$-degree angles (see Figures \ref{fig:triple-bubble}, \ref{fig:triple-bubble-S2} and \ref{fig:quad-bubble-R3}). The latter angles of incidence are a well-known necessary condition for any isoperimetric minimizing cluster candidate (see Lemma \ref{lem:boundary-normal-sum}).

\begin{conjecture*}[Multi-Bubble Isoperimetric Conjecture on $\R^n$]
For all $2 \leq q \leq n+2$, a standard bubble uniquely minimizes total perimeter among all $q$-clusters $\Omega$ on $\R^n$ of prescribed volume $V(\Omega) = v \in \interior \Delta^{(q-1)}_{\infty}$. 
\end{conjecture*}

\begin{conjecture*}[Multi-Bubble Isoperimetric Conjecture on $\S^n$]
For all $2 \leq q \leq n+2$, a standard bubble uniquely minimizes total perimeter among all $q$-clusters $\Omega$ on $\S^n$ of prescribed volume $V(\Omega) = v \in \interior \Delta^{(q-1)}_{1}$. 
\end{conjecture*}

To avoid constantly writing ``up to null sets" above and throughout this work, we will always modify an isoperimetric minimizing cluster on a null set so that its cells $\Omega_i$ are open and $\overline{\partial^* \Omega_i} = \partial \Omega_i$ (this is always possible by Theorem \ref{thm:Almgren}). Note that the single-bubble cases $q=2$ above correspond to the classical isoperimetric problems on $\R^n$ and $\S^n$, where geodesic balls are well-known to be the unique isoperimetric minimizers \cite{BuragoZalgallerBook}; this case was included in the formulation of the conjectures for completeness. That a standard bubble in $\R^n$ exists and is unique (up to isometries) for all volumes $v \in \interior \Delta^{(q-1)}_\infty$ and $2 \leq q \leq n+2$ was proved by Montesinos-Amilibia \cite{MontesinosStandardBubbleE!}. An analogous statement equally holds on $\S^n$ for all $v \in \interior \Delta^{(q-1)}_1$ (see Corollary \ref{cor:standard-volume}). See Figures~\ref{fig:S-and-B}, \ref{fig:triple-bubble-perturbation}, \ref{fig:almost}, \ref{fig:min-degree} and \ref{fig:quad-sliding} for a depiction of some non-standard bubbles.

\subsection{Previously known and related results}

Before going over the previously known results regarding the multi-bubble conjectures in $\R^n$ and $\S^n$, let us add to our discussion Gaussian space $\GG^n$, consisting of Euclidean space $(\R^n,|\cdot|)$ endowed with the standard Gaussian probability measure $\gamma^n = (2 \pi)^{-\frac{n}{2}} \exp(-|x|^2/2) dx$. The following conjecture was confirmed in our previous work \cite{EMilmanNeeman-GaussianMultiBubble}:

\begin{theorem*}[Multi-Bubble Isoperimetric Conjecture on $\GG^n$]
For all $2 \leq q \leq n+1$, the unique Gaussian-weighted isoperimetric minimizers on $\GG^n$ of prescribed Gaussian measure $v \in \interior \simplex^{(q-1)}_1$ are simplicial clusters, obtained as the Voronoi cells of $q$ equidistant points in $\R^n$ (appropriately translated). 
\end{theorem*}

\noindent When $q=2$, the cells of a simplicial cluster are precisely halfspaces, and the single-bubble conjecture on $\GG^n$ holds by the classical Gaussian isoperimetric inequality \cite{SudakovTsirelson,Borell-GaussianIsoperimetry} and its equality cases \cite{EhrhardGaussianIsopEqualityCases, CarlenKerceEqualityInGaussianIsop}; this case was included in the formulation above for completeness. Note that the interfaces of these Gaussian simplicial clusters are all flat, as opposed to the spherical interfaces of the standard bubbles on $\R^n$ and $\S^n$; this flattening when passing from $\S^N$ to $\GG^n$ is a well-known phenomenon, due to the need to rescale $\S^N$ so as to match the ``curvature" of $\GG^n$ (see below).

\medskip

Below we list some of the previously known results regarding the above three conjectures, and refer to the excellent book by F.~Morgan \cite[Chapters 13,14,18,19]{MorganBook5Ed} for additional information. 

\begin{itemize}
\item On $\R^n$ -- 
Long believed to be true, but appearing explicitly as a conjecture in an undergraduate thesis by J.~Foisy in 1991 \cite{Foisy-UGThesis}, 
the Euclidean double-bubble problem in $\R^n$ was considered in the 1990's by various authors. In the Euclidean plane $\R^2$, the conjecture was confirmed in \cite{SMALL93} (see also \cite{MorganWichiramala-StablePlanarDoubleBubble,LawlorEtAl-PlanarDoubleBubbleViaMetacalibration,CLM-StabilityInPlanarDoubleBubble}). In $\R^3$, the equal volumes case $v_1 = v_2$ was settled in \cite{HHS95,HassSchlafly-EqualDoubleBubbles}, and the structure of general double-bubbles was further studied in \cite{Hutchings-StructureOfDoubleBubbles}. This culminated in the work of Hutchings--Morgan--Ritor\'e--Ros \cite{DoubleBubbleInR3-Announcement,DoubleBubbleInR3}, who confirmed the double-bubble conjecture in $\R^3$; their method was subsequently extended by Reichardt-et-al to resolve the conjecture in $\R^4$ and $\R^n$ \cite{SMALL03,Reichardt-DoubleBubbleInRn} (see also Lawlor \cite{Lawlor-DoubleBubbleInRn} for an alternative proof using ``unification" which applies to arbitrary interface weights). The triple-bubble case $q=4$ in the Euclidean plane $\R^2$ was confirmed by Wichiramala in \cite{Wichiramala-TripleBubbleInR2} (see also \cite{Lawlor-TripleBubbleInR2AndS2}). 
\item On $\S^n$ --
The double-bubble and triple-bubble conjectures were resolved on $\S^2$ by Masters \cite{Masters-DoubleBubbleInS2} and Lawlor \cite{Lawlor-TripleBubbleInR2AndS2}, respectively, but on $\S^n$ for $n\geq 3$ only partial results are known \cite{CottonFreeman-DoubleBubbleInSandH, CorneliHoffmanEtAl-DoubleBubbleIn3D,CorneliCorwinEtAl-DoubleBubbleInSandG}.  In particular, Corneli--et-al \cite{CorneliCorwinEtAl-DoubleBubbleInSandG} confirmed the double-bubble conjecture on $\S^n$ for all $n \geq 3$ when the prescribed measure $v \in \simplex^{(2)}_1$ satisfies $\max_{i} \abs{v_i - 1/3} \leq 0.04$. Their proof employed a result of Cotton--Freeman \cite{CottonFreeman-DoubleBubbleInSandH}, stating that if the minimizing cluster's cells are known to be connected, then it must be the standard double-bubble. 
\item On $\GG^n$ -- 
In the Gaussian setting, the original proofs of the single-bubble case made use of the classical fact that the projection onto a fixed $n$-dimensional subspace of the uniform measure on a rescaled sphere $\sqrt{N} \S^N$, converges to the Gaussian measure $\gamma^n$ as $N \rightarrow \infty$. Building upon this idea, it was shown by Corneli--et-al \cite{CorneliCorwinEtAl-DoubleBubbleInSandG} that verification of the multi-bubble conjecture on $\S^N$ for a sequence of $N$'s tending to $\infty$ and a fixed $v \in \interior \simplex^{(q-1)}_1$ will verify the multi-bubble conjecture on $\GG^n$ for the same $v$ and for all $n \geq q-1$. As a consequence, they confirmed the double-bubble conjecture on $\GG^n$ for all $n \geq 2$ (without uniqueness, which is lost in the approximation procedure) when the prescribed measure $v \in \simplex^{(2)}_1$ satisfies $\max_{i} \abs{v_i - 1/3} \leq 0.04$. As already mentioned, the Gaussian double-bubble, and more generally, multi-bubble conjecture, was confirmed in its entirety for all $2 \leq q \leq n+1$ in \cite{EMilmanNeeman-GaussianMultiBubble}.
\end{itemize} 

One can say a bit more in the equal-volume cases:
\begin{itemize}
\item The equal-volume case $v_1 = \ldots = v_q = \frac{1}{q}$ of the multi-bubble conjecture on $\S^n$ for $2 \leq q \leq n+2$ follows immediately from the equal-volume case on $\GG^{n+1}$. Indeed, both measure and perimeter on $\S^n$ and $\GG^{n+1}$ coincide for \emph{centered} cones, and the unique equal-volumes minimizer on $\GG^{n+1}$ for all $2 \leq q \leq (n+1) + 1$ is the centered simplicial cluster (whose cells are centered cones). 
\item After informing each other of our respective results, we learned from Gary Lawlor (personal communication) that he has managed to confirm the equal-volume case $v_1 = v_2 = v_3$ of the triple-bubble conjecture on $\R^3$.
\item While this falls outside the scope of Sullivan's conjecture, we mention that Paolini, Tamagnini and Tortorelli \cite{PaoliniTamagnini-PlanarQuadraupleBubbleEqualAreas,PaoliniTortorelli-PlanarQuadrupleEqualAreas} have identified a unique (up to reordering cells, scaling and isometries) minimizing quadruple-bubble of equal volumes $v_1 = v_2 = v_3 = v_4$ in the plane $\R^2$.
\end{itemize}

To the best of our knowledge, with the exception of the above partial results in the equal-volume cases, no prior results were known for the triple-and-higher-bubble conjecture on $\R^n$ or $\S^n$ when $n \geq 3$, nor for the double-bubble conjecture on $\S^n$ when $n \geq 3$ with the exception of the almost-equal-volume cases mentioned above.

\subsection{Main results}

\begin{definition}[$\S^m$-symmetry]
A cluster $\Omega$ on $M^n$, $M^n \in \{ \R^n , \S^n \}$, is said to have $\S^m$-symmetry ($m \in \{ 0, \ldots,n-1\}$), if there exists a totally-geodesic $M^{n-1-m} \subset M^n$ so that each cell $\Omega_i$ is invariant under all isometries of $M^n$ which fix the points of $M^{n-1-m}$. \\
In particular, $\S^0$-symmetry means invariance under reflection about some hyperplane, referred to as the hyperplane of symmetry; the totally-geodesic hypersurface $M^{n-1}$ is referred to as the equator.
\end{definition}

The starting point of our analysis in this work is the well-known observation, based on the Borsuk--Ulam (``Ham Sandwich") theorem, that for any $q$-cluster in $\R^n$ or $\S^n$ with $q \leq n+1$, there exists a hyperplane bisecting all of its finite-volume cells. Consequently, symmetrizing a minimizing cluster about that hyperplane (reflecting the half with the smaller total surface-area), one can always find a minimizing $q$-cluster with $\S^0$-symmetry when $q \leq n+1$ \cite[Remark 2.7]{Hutchings-StructureOfDoubleBubbles} (but whether there may be additional minimizers violating $\S^0$-symmetry is a-priori unclear). In fact, extending upon this simple observation when $q \leq n$, an argument of B.~White written up by Foisy \cite{Foisy-UGThesis} and extended by Hutchings \cite{Hutchings-StructureOfDoubleBubbles} in $\R^n$ (which applies to all model spaces \cite[Proposition 2.4]{CottonFreeman-DoubleBubbleInSandH}), shows that \emph{any} minimizing $q$-cluster will \emph{necessarily} have $\S^{n+1-q}$-symmetry. We will not require the latter extension here, but rather deduce a much stronger statement in Theorem \ref{thm:intro-stable-spherical-Voronoi} below, which in fact applies to all $q \leq n+1$ (see Remark \ref{rem:intro-Sm-symmetry}).

\medskip

A cluster on $M^n \in \{\R^n,\S^n\}$ is called spherical if all connected components $\{ \Sigma^\ell_{ij} \}$ of its $(n-1)$-dimensional interfaces $\{\Sigma_{ij}\}$ are pieces of (generalized) geodesic spheres $S_{ij}^\ell$. It is called perpendicularly spherical if in addition all spheres $S^\ell_{ij}$ are perpendicular to a common hyperplane (note that the interfaces $\Sigma_{ij}$ themselves are not required to intersect this hyperplane). If the latter hyperplane is also a hyperplane of symmetry for the cluster, the cluster is called spherical perpendicularly to its hyperplane of symmetry (cf.~Definition \ref{def:spherical-S0}). Our first main theorem states the following:

\begin{theorem}[Perpendicular Sphericity] \label{thm:intro-stable-spherical}
Let $M^n \in \{ \R^n , \S^n \}$.  Then an isoperimetric minimizing $q$-cluster with $\S^0$-symmetry is necessarily spherical perpendicularly to its hyperplane of symmetry. In fact, this holds for any bounded stable regular $q$-cluster with $\S^0$-symmetry. In particular, whenever $q \leq n+1$, any minimizing $q$-cluster $\Omega$ is perpendicularly \textbf{spherical} (but, at this point, possibly without $\S^0$-symmetry -- this will only be established in Theorem \ref{thm:intro-stable-spherical-Voronoi}). 
\end{theorem}

We refer to Section \ref{sec:prelim} for the definitions of boundedness, regularity, stationarity and stability of a cluster. The proof of Theorem \ref{thm:intro-stable-spherical} is based on a second order variational argument, testing stability using a carefully chosen family of vector-fields, and \emph{combining} stability information from \emph{different} fields to obtain an expression with an appropriate sign. 

\medskip

With some additional work involving convex geometry and simplicial homology of the corresponding cell complex, we can obtain  stronger information. 
\begin{definition}[Spherical Voronoi Cluster] \label{def:intro-spherical-Voronoi}
A $q$-cluster $\Omega$ on $\S^n$ is called a spherical Voronoi cluster if there exist $\{ \c_i \}_{i=1,\ldots,q} \subset \R^{n+1}$ and $\{ \k_i \}_{i=1,\ldots,q} \subset \R$ so that for all $i=1,\ldots,q$:
\begin{enumerate}[(1)]
\item \label{it:intro-spherical-Voronoi}
For every non-empty interface $\Sigma_{ij} \neq \emptyset$, $\Sigma_{ij}$ lies on a single geodesic sphere $S_{ij}$ with quasi-center $\c_{ij} = \c_i - \c_j$ and curvature $\k_{ij} = \k_i - \k_j$. \\
The quasi-center $\c$ of a geodesic sphere $S$ is the vector $\c := \n - \k p$, where $\k$ is the curvature with respect to the unit normal $\n$ at $p \in S$. Note that $\c$  does not depend on $p \in S$ -- see Remark \ref{rem:quasi-center} and Figure \ref{fig:quasi-center}. 
\item \label{it:intro-spherical-Voronoi-2}
The following Voronoi representation holds:
\begin{equation} \label{eq:intro-Voronoi-rep}
\Omega_i  = \set{ p \in \S^n \; ; \; \argmin_{j=1,\ldots,q} \scalar{\c_j , p} + \k_j = \{i\} } = \bigcap_{j \neq i} \; \set{ p \in \S^n \; ;\; \scalar{\c_{ij},p} + \k_{ij} < 0 } .
\end{equation}
\end{enumerate}
A $q$-cluster $\Omega$ on $\R^n$ is called a spherical Voronoi cluster if there exists a stereographic projection onto $\S^n$ so that the projected cluster on $\S^n$ is spherical Voronoi. \\
As the parameters $\{\c_i\}$ and $\{\k_i\}$ are defined up to translation, we will always employ the convention that $\sum_{i=1}^q \c_i = 0$ and $\sum_{i=1}^q \k_i = 0$. 
\end{definition}
\begin{remark} \label{rem:intro-spherical-Voronoi}
While it may not be immediately apparent, all standard bubbles on $\R^n$ and $\S^n$ are spherical Voronoi clusters (Corollary \ref{cor:standard-Voronoi}); they are precisely characterized as those spherical Voronoi clusters for which the interfaces $\Sigma_{ij}$ are non-empty for all $1 \leq i < j \leq q$ -- see Proposition \ref{prop:standard-char}. Note that each cell of a spherical Voronoi cluster on $\S^n$ is the intersection of an open convex polyhedron in $\R^{n+1}$ with $\S^n$. Hence, while there isn't any apparent convexity in the multi-bubble conjectures on $\R^n$ or $\S^n$, convexity is nevertheless hidden in the structure of minimizers, as we shall verify in this work. 
\end{remark}
\begin{remark} \label{rem:intro-quasi-centers}
Note that a spherical Voronoi cluster on $\S^n$ is entirely determined by the vectors $\{ \c_i \}_{i=1,\ldots,q}$ and scalars $\{ \k_i \}_{i=1,\ldots,q}$, which we call its quasi-center and curvature parameters, respectively. Condition \ref{it:intro-spherical-Voronoi} is equivalent to the requirement that $|\c_{ij}|^2 = 1 + \k_{ij}^2$ for all $i < j$ such that $\Sigma_{ij} \neq \emptyset$ -- see Remark \ref{rem:quasi-center}. The property of being a spherical Voronoi cluster on $\R^n$ does not depend on the particular stereographic projection to $\S^n$ used, as follows from Lemma \ref{lem:Mobius-preserves-Voronoi}. See Lemma \ref{lem:Voronoi-rep-on-Rn} for a direct representation of the cells of a spherical Voronoi cluster on $\R^n$ as the intersection of Euclidean balls, their complements and halfplanes, avoiding the stereographic projection to $\S^n$. See Remark \ref{rem:ck-S-R-non-invertible} for the reason we insist on defining a spherical Voronoi cluster on $\R^n$ via stereographic projection to $\S^n$. 
\end{remark}

\begin{figure}
    \begin{center}
         \hspace*{37pt}
        \includegraphics[scale=0.108]{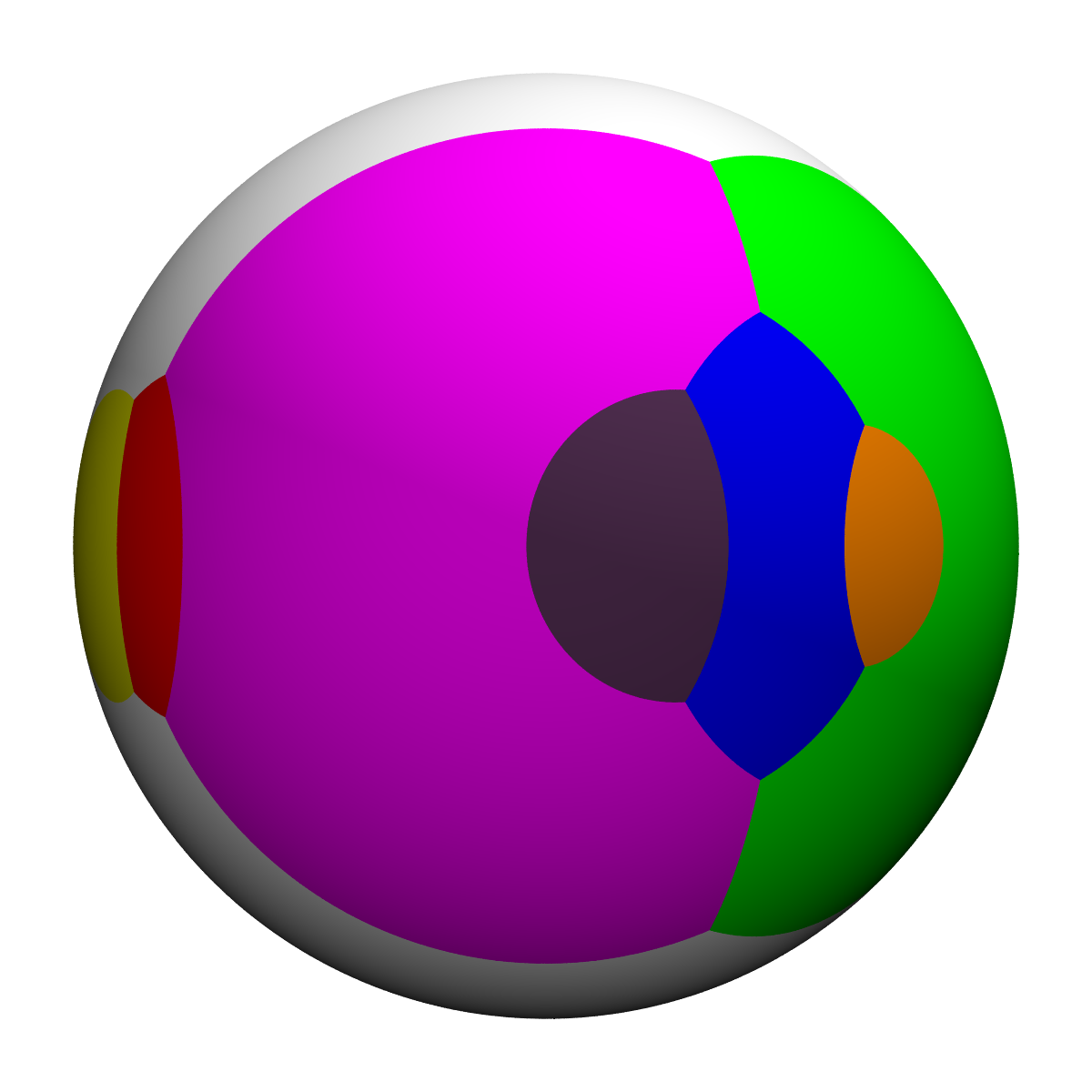}
        \hspace{30pt}
        \raisebox{0.01\height}{
        \begin{tikzpicture}[scale=1.76]
            \input{cluster-stereographic-color}
        \end{tikzpicture}
        }
        \newline
                \includegraphics[scale=0.108]{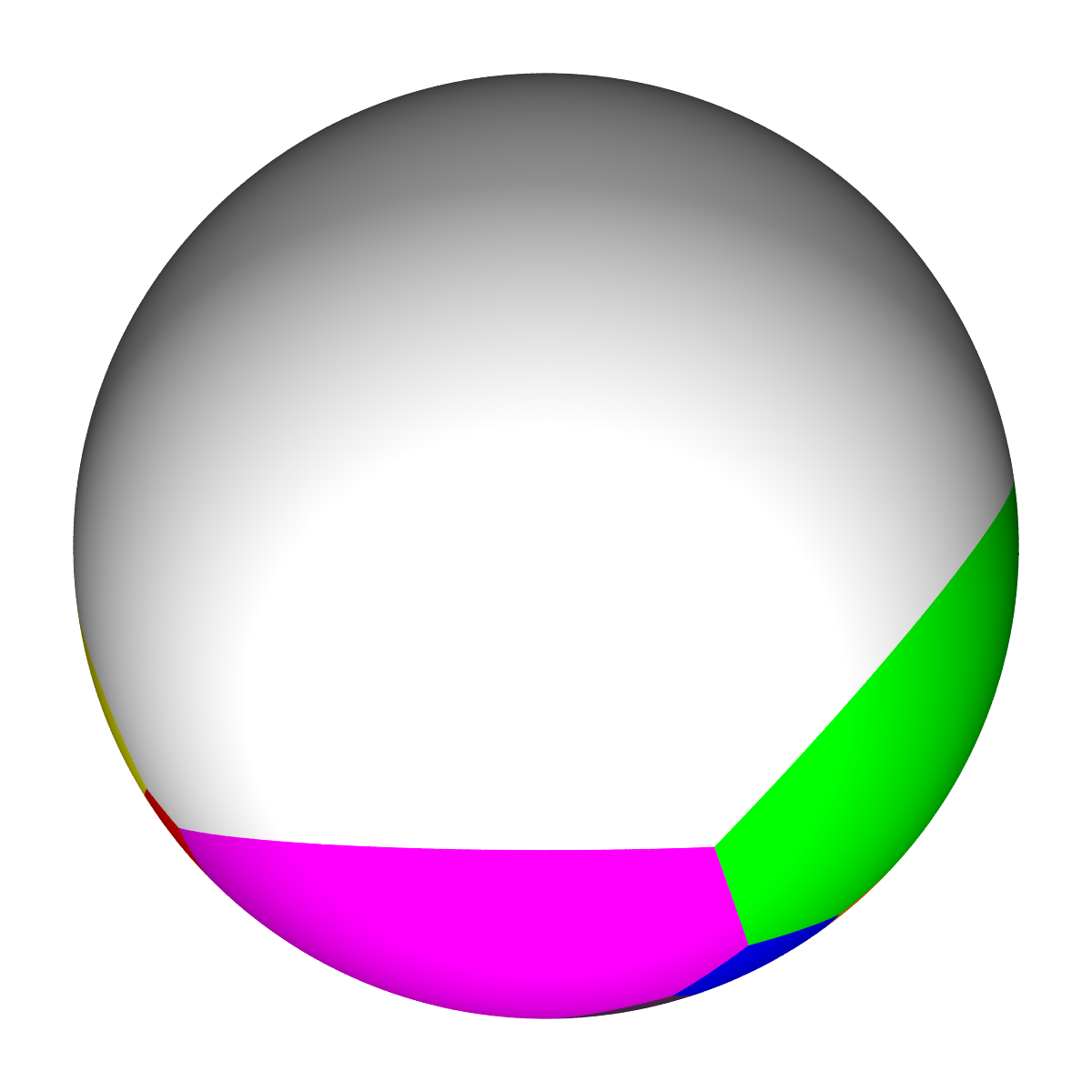}
        \hspace{60pt}
                \raisebox{0.08\height}{
        \begin{tikzpicture}[scale=1.97]
            \input{cluster-brep-color}
        \end{tikzpicture}
        }
     \end{center}
     \caption{
         \label{fig:S-and-B}
         A perpendicularly spherical Voronoi cluster $\Omega$. Top left: $\Omega^\S$ on $\S^2$ drawn with the North pole $N$ at the top (the equator of symmetry is the central parallel); Top right: $\Omega^\R$ on $\R^2$ after stereographic projection from a pole on the equator; Bottom left: $\Omega^\S$ drawn with North pole in the center (the equator coincides with the circumference); Bottom right: the orthogonal projection $\Omega^\B$ of $\Omega^\S$ onto the equatorial disc $\B^2$ consists of convex polyhedral cells (colors lightened for better contrast). 
     }
\end{figure}

\begin{definition}[Perpendicularly Spherical Voronoi Cluster] \label{def:intro-perp-spherical-Voronoi}
A $q$-cluster $\Omega$ with $\S^0$-symmetry on $\S^n$ is called perpendicularly spherical Voronoi if it is spherical Voronoi and all of its quasi-center parameters $\{ \c_i \}_{i=1,\ldots,q}$ lie on the hyperplane of symmetry. See Figure~\ref{fig:S-and-B}.\\
A $q$-cluster $\Omega$ on $\R^n$ with $\S^0$-symmetry is called perpendicularly spherical Voronoi if there exists a stereographic projection onto $\S^n$ so that the projected cluster on $\S^n$ is $\S^0$-symmetric and perpendicularly spherical Voronoi.
\end{definition}

Given an open set $A \subset M^n$ with $\S^0$-symmetry,  its connected components are either equatorial (i.e.~intersect the equator $M^{n-1}$) and hence $\S^0$-symmetric, or non-equatorial and hence come in pairs so that their union is $\S^0$-symmetric. We define $A$'s ``connected components modulo its $\S^0$-symmetry" as the collection of its connected components after declaring the above pairs as belonging to the same component (which we call ``connected component modulo its $\S^0$-symmetry", cf. Definition \ref{def:cc-S0}). 

\begin{theorem}[Connected Components modulo $\S^0$-symmetry are Perpendicularly Spherical Voronoi] \label{thm:intro-spherical-Voronoi-prelim} 
Let $M^n \in \{ \R^n , \S^n \}$, and let $\Omega$ be an isoperimetric minimizing $q$-cluster with $\S^0$-symmetry. Denote by $\underline{\Omega}$ the cluster whose cells are the connected components of $\Omega$'s (open) cells modulo their common $\S^0$-symmetry. Then $\underline{\Omega}$ is a perpendicularly spherical Voronoi cluster. In fact, this holds for any bounded stationary regular cluster with $\S^0$-symmetry $\Omega$ which is spherical perpendicularly to its hyperplane of symmetry, and whose (open) cells have finitely many connected components. 
\end{theorem}

Building upon Theorem \ref{thm:intro-spherical-Voronoi-prelim}, we can say more when $q \leq n+1$ by using stability again and elliptic regularity to establish the cells' \textbf{connectedness}. This resolves a conjecture of A.~Heppes \cite[Problem 5]{OpenProblemsInSoapBubbles96} and the open question of whether there can be empty chambers trapped by minimizing bubbles \cite[Chapter 13]{MorganBook5Ed}
in the latter range. 

\begin{theorem}[Connectedness of Spherical Voronoi Cells] \label{thm:intro-stable-spherical-Voronoi}
Let $M^n \in \{ \R^n , \S^n \}$, let $\Omega$ be an isoperimetric minimizing $q$-cluster, and assume that $q \leq n+1$. Then $\Omega$ is $\S^0$-symmetric, perpendicularly \textbf{spherical Voronoi}, and its open cells $\{\Omega_i\}_{i=1,\ldots,q}$ (including the unbounded cell $\Omega_q$ when $M^n = \R^n$) are all equatorial 
 (if non-empty) and \textbf{connected}. 
In fact, this holds for any bounded stable regular $q$-cluster with $\S^0$-symmetry so that $\Sigma = \cup_{i=1}^q \partial \Omega_i$ is connected and each $\Omega_i$ has (a-priori) finitely many connected components. 
\end{theorem}

\begin{remark}
Heppes asked whether each cell of a minimizing cluster in $\R^n$ is necessarily connected. The question of whether the unbounded cell $\Omega_q$ must always be connected, or equivalently, of whether no empty chambers can be trapped by the bubbles, was also open. 
An interesting question of Almgren \cite[Problem 1]{OpenProblemsInSoapBubbles96} is whether there is a stable cluster $\Omega$ of bubbles in $\R^3$ with some bubble $\Omega_i$ being topologically a torus. Clearly, we may always reduce to the case that $\Sigma$ is connected by removing cells if necessary. For a single bubble, a classical result of Alexandrov \cite{Alexandrov-MethodOfMovingPlanes} asserts that an embedded closed (connected) hypersurface which is stationary (i.e.~of constant mean-curvature) must be a Euclidean sphere, and Barbosa--Do Carmo \cite{BarbosaDoCarmo-StabilityInRn} have shown that this also holds for stable immersed closed hypersurfaces. When $\Omega$ is a stable double bubble in $\R^3$ with $\S^1$-symmetry, it was shown by Hutchings--Morgan--Ritor\'e--Ros \cite{DoubleBubbleInR3} that $\Omega$ must be a standard bubble. Theorem \ref{thm:intro-stable-spherical-Voronoi} allows to relax the symmetry assumption to $\S^0$-symmetry. As for stable $q$-clusters in $\R^n$ with $\S^0$-symmetry and $4 \leq q \leq n+1$, Theorem \ref{thm:intro-stable-spherical-Voronoi} asserts that its cells are connected and obtained as stereographic projections of $P_i \cap \S^n$ for some convex $\S^0$-symmetric polyhedra $P_i$ with at most $q-1$ facets, which limits the topological complexity of the cells. See also Lemma \ref{lem:bubble-ring} for related information. 
\end{remark}

\begin{remark} \label{rem:intro-Sm-symmetry}
Note that the affine-rank of $\{ \c_i \}_{i=1,\ldots,q}$ from Definition \ref{def:intro-spherical-Voronoi} is at most $q-1$, and so when $q \leq n+1$, any spherical Voronoi $q$-cluster on $\S^n$ (and thus on $\R^n$ after stereographic projection) is $\S^{n+1-q}$-symmetric and perpendicularly spherical Voronoi. Consequently, Theorem \ref{thm:intro-stable-spherical-Voronoi} implies that a minimizing $q$-cluster is actually $\S^{n+1-q}$-symmetric for all $q \leq n+1$, recovering the White--Hutchings symmetry result \cite{Hutchings-StructureOfDoubleBubbles} when $q \leq n$ and extending it to $q=n+1$ (cf. \cite[Remark 2.7]{Hutchings-StructureOfDoubleBubbles}). 
\end{remark}

In view of Remark \ref{rem:intro-spherical-Voronoi}, we immediately deduce:
\begin{corollary}[Minimizers with Full Interfaces are Standard Bubbles] \label{cor:intro-bubble}
Let $M^n \in \{ \R^n , \S^n \}$, let $\Omega$ be an isoperimetric minimizing $q$-cluster, and assume that $q \leq n+1$. If the interfaces $\Sigma_{ij} = \partial^* \Omega_i \cap \partial^* \Omega_j$ are non-empty for all $1 \leq i < j \leq q$, then $\Omega$ is necessarily a standard bubble. 
\end{corollary}

\medskip

Since it is well-known that the boundary of an isoperimetric minimizing cluster on $M^n$ is necessarily connected (see Lemma \ref{lem:Sigma-connected}), Corollary \ref{cor:intro-bubble} already gives an immediate proof of the double-bubble conjecture in $\S^n$, as well as an alternative proof of the double-bubble theorem \cite{DoubleBubbleInR3,SMALL03,Reichardt-DoubleBubbleInRn} in $\R^n$, for all $n \geq 2$ (see Section \ref{sec:global} for details). 
However, to handle more general $q \geq 4$, additional work is needed building off Theorem \ref{thm:intro-stable-spherical-Voronoi} and Corollary \ref{cor:intro-bubble}. We are able to show the following:

\begin{theorem}[Double, Triple and Quadruple Bubble Conjectures on $\R^n$ and $\S^n$] \label{thm:intro-234}
The Multi-Bubble Conjectures on $\R^n$ and $\S^n$ hold for all $2 \leq q \leq \min(5,n+1)$. Namely, the double-bubble conjectures hold for all $n \geq 2$, the triple-bubble conjectures hold for all $n \geq 3$, and the quadruple-bubble conjectures hold for all $n \geq 4$.
\end{theorem}

\subsection{Further extensions}

With some additional (considerable) work, which we leave for another occasion, we can also obtain the following results, which we only state here without proof:
\begin{theorem}[Quintuple Bubble Conjecture on $\S^n$]
The quintuple bubble conjecture (case $q=6$) holds on $\S^n$ for all $n \geq 5$. 
\end{theorem}
\begin{remark}
By scale-invariance and approximately embedding a small cluster in $\R^n$ into $\S^n$, it follows that a standard quintuple bubble in $\R^n$ for $n \geq 5$ is indeed an isoperimetric minimizer, confirming the quintuple-bubble isoperimetric \emph{inequality} on $\R^n$. However, uniqueness is lost in the approximation procedure, and so contrary to the $\S^n$ case, we cannot exclude the existence of additional quintuple-bubble minimizers on $\R^n$. 
\end{remark}

A spherical Voronoi cluster $\Omega$ on $\S^n$ is called conformally flat if there exists a conformal diffeomorphism $T : \S^n \rightarrow \S^n$ so that all of the interfaces of the cluster $T \Omega = (T \Omega_1,\ldots, T \Omega_n)$ are flat. By Liouville's classical theorem, when $n \geq 3$ then all such maps $T$ are described by M\"obius automorphisms of $\S^n$.
A spherical Voronoi cluster on $\S^n$ is called pseudo conformally flat if a certain more general condition holds (see Definition \ref{def:PCF}); in particular, a full-dimensional cluster (having affine-rank of $\{\c_i\}_{i=1,\ldots,q}$ equal to $q-1 \leq n+1$) is pseudo conformally flat. By construction, a standard bubble is a M\"obius image of the flat equal-volume standard bubble, and is therefore (pseudo) conformally flat. 

\begin{theorem}[Conditional verification assuming pseudo conformal flatness] \label{thm:intro-conditional}
Fix $n \geq 6$ and $7 \leq q \leq n+1$. Assume that for every $v \in \Delta^{(q-1)}_1$, there exists an isoperimetric minimizing $q$-cluster $\Omega$ on $\S^n$ with $V(\Omega) = v$ so that $\Omega$ is pseudo conformally flat. 
Then the multi-bubble conjecture for $p$-clusters on $\S^n$ holds for all $2 \leq p \leq q$. 
\end{theorem}

\subsection{Method of proof and comparison with previous approaches}

The resolution of the double-bubble conjecture by Hutchings--Morgan--Ritor\'e--Ros on $\R^3$ \cite{DoubleBubbleInR3} (and subsequently on $\R^n$ \cite{SMALL03,Reichardt-DoubleBubbleInRn}) relied on the following crucial ingredients: 
\begin{itemize}
\item The $\S^{n-2}$-symmetry of a minimizing cluster following an argument of White, written up by Foisy \cite{Foisy-UGThesis} and extended by Hutchings \cite{Hutchings-StructureOfDoubleBubbles}. The symmetry argument also applies on $\S^n$ \cite{CottonFreeman-DoubleBubbleInSandH}. This reduces the double-bubble problem to the study of certain constant-mean-curvature (CMC) curves in the upper half-plane of $\R^2$ (or upper hemisphere of $\S^2$). 
\item Hutchings' estimates on the number of connected components of a minimizing cluster's cells and corresponding structure theory \cite{Hutchings-StructureOfDoubleBubbles}. According to Cotton--Freeman \cite{CottonFreeman-DoubleBubbleInSandH}, extending the bounds on the connected components to $\S^n$ seems difficult, and they were only able to provide a structure theorem assuming one of the cells is connected \cite[Theorem 6.5]{CottonFreeman-DoubleBubbleInSandH}. 
\item A minimizing double-bubble with connected cells must be a standard bubble. This argument extends to $\S^n$ \cite[Proposition 7.3]{CottonFreeman-DoubleBubbleInSandH}. 
\item Ruling out clusters with multiple connected components. Due to the difficulty in bounding the number of connecting components on $\S^n$, this step is absent in the analysis of \cite{CottonFreeman-DoubleBubbleInSandH}. Furthermore, this step in \cite{DoubleBubbleInR3} relies on the classification of Delauney hypersurfaces (CMC hypersurfaces of revolution). 
\end{itemize}

With the exception of the symmetry argument, it is not clear how to extend any of these ingredients to the triple-and-higher bubble case $q \geq 4$ in dimension $n \geq 3$. In that case, Hutchings' connected component analysis becomes extremely complicated even in $\R^n$ (however, for the triple-bubble problem in $\R^2$, this was carried out in \cite{Wichiramala-TripleBubbleInR2}). Furthermore, the $\S^{n+1-q}$-symmetry reduces the problem to the study of a collection of hypersurfaces in dimension $q-1 \geq 3$, which no longer satisfy an ODE but rather a PDE, and in particular are not classified as in the Delauney case. 

\bigskip

On the other hand, our prior resolution of the multi-bubble conjecture on $\GG^n$ in \cite{EMilmanNeeman-GaussianMultiBubble} relied on the following crucial ingredients:
\begin{itemize} 
\item Utilizing the fact that $Q(T_a) \leq 0$, where $\{T_a\}$ are the constant vector-fields generating the translation group, and $Q$ denotes the stability index-form (see Subsection \ref{subsec:Q} for a precise definition). 
\item In particular, the previous property easily implies that stable clusters on $\GG^n$ have a product structure, facilitating the use of product vector-fields for testing stability. 
\item Establishing a sharp matrix-valued partial differential inequality (MPDI) satisfied by the Gaussian isoperimetric profile. 
\end{itemize}

On $\R^n$ and $\S^n$, it is not the translation group but rather the M\"obius group which generates the standard bubble conjectured minimizers; modding out isometries, the quotient is generated by an $(n+1)$-dimensional family of vector-fields we call ``M\"obius fields" $\{W_\theta\}$. Contrary to the Gaussian setting, $Q(W_\theta)$ does not have a clear sign, and in fact we suspect that it is not always negative semi-definite on a general cluster (even though, a-posteriori, we can show that $Q(W_\theta) \leq 0$ for a stable $q$-cluster when $q \leq n+1$). For this reason, contrary to the Gaussian setting, we are not able to handle the case when $q$ is maximal according to the conjectures, i.e. $q= n+2$, and restrict our analysis to $q \leq n+1$. While we could use the known $\S^{n+1-q}$-symmetry of a minimizing $q$-cluster on $\R^n$ and $\S^n$ (at least, for $q \leq n$ \cite{Hutchings-StructureOfDoubleBubbles}), it only implies a warped-product representation of the cluster, which leads to convoluted formulas (in contrast to the convenient product case on $\GG^n$) and does not seem to be more useful than just using $\S^0$-symmetry. Lastly, we don't know how to derive a reasonable PDE for the model profile on $\R^n$ (even after modding out homogeneity to reduce to a compact set); 
and while it is possible to derive a corresponding PDE for the model profile on $\S^n$, we were not able to establish a sharp PDI for the actual isoperimetric profile on $\S^n$ in full generality (only conditionally, yielding Theorem \ref{thm:intro-conditional}, whose proof will be presented in a separate work). Consequently, \textbf{we do not invoke any MPDI argument in this work}.

\medskip

Instead, our argument proceeds rather differently:
\begin{enumerate}[(1)]
\setcounter{enumi}{-1}
\item
As already mentioned, instead of utilizing the full $\S^{n+1-q}$-symmetry of a minimizing cluster as guaranteed by \cite[Theorem 2.6]{Hutchings-StructureOfDoubleBubbles} when $q \leq n$ (but which does not extend to $q=n+1$ \cite[Remark 2.7]{Hutchings-StructureOfDoubleBubbles}), our starting point is the elementary (well-known) observation that there \emph{exists} a minimizer with $\S^0$-symmetry (i.e. reflection symmetry) whenever $q \leq n+1$ by the Borsuk--Ulam theorem. This is used to compensate 
for not a-priori knowing that $Q(W_\theta) \leq 0$, and allows us to ensure that the first variation of volume is zero for any field which is odd with respect to reflection about the hyperplane of symmetry. Note that contrary to previous approaches, $\S^0$-symmetry does not reduce the effective dimensionality of the minimizer. 
\item \label{it:intro-step-1}
Our very first task in Section \ref{sec:spherical} is then to directly derive Theorem \ref{thm:intro-stable-spherical} on the \emph{sphericity} of a stable cluster with $\S^0$-symmetry, in contrast to the approach of \cite{DoubleBubbleInR3}. To that end, we need to test \emph{several} families of vector-fields and combine their individual contributions to the index-form $Q$ before an appropriate sign is reached and stability may be invoked. This is in contrast to the argument from \cite{EMilmanNeeman-GaussianMultiBubble}, where a \emph{single} family sufficed to establish flatness. A high-level reason for why sphericity is more difficult to establish than flatness is because while flatness is encoded in the vanishing of the second fundamental form $\II = 0$, sphericity is encoded in the vanishing of the \emph{traceless} form $\II_0 = \II - \frac{H}{n-1} \Id = 0$, where $H$ is the mean-curvature; since $H$ is a-priori unknown, we have more degrees of freedom to account for. Another difference with the Gaussian setting is that, while flatness is trivially the same as translation flatness, i.e.~flatness after an appropriate translation, sphericity is strictly weaker than conformal flatness, i.e.~flatness (of the entire cluster) after an appropriate conformal map -- for this reason, we are only able to verify the conditional Theorem \ref{thm:intro-conditional}. This also provides some insight as to why we need to use the conformal M\"obius field $W_N$ as one of the family members in the above argument (in the direction of the North pole $N$). 

\textbf{This first step is the critical step in the present work}, before which we could not make any progress on the problem. The importance of conformal Killing fields and conformal boundary conditions (and in particular that of the M\"obius group and its generators $W_\theta$) to the isoperimetric problem for clusters on $\R^n$ and $\S^n$ is clarified in this work, and constitutes one of its main novel ingredients.

In the case of $\R^n$, we also need to employ the following remarkable isotropicity of a minimizing cluster's boundary $\Sigma^1 = \bigcup_{i<j} \Sigma_{ij}$ (regardless of the volumes of the cells or their number!), which may be of independent interest (see Remark \ref{rem:isotropic}): 
\[
\int_{\Sigma^1} \n \otimes \n \; dp = \frac{1}{n} \int_{\Sigma^1} \Id \; dp ;
\]
here $\n$ denotes the (outward) unit-normal. See Figure~\ref{fig:triple-bubble} to test if this is obviously apparent. 

\item The second step is to derive the bulk of Theorem \ref{thm:intro-spherical-Voronoi-prelim} in Section \ref{sec:Voronoi-prelim}, asserting that the connected components of a spherical cluster actually have a spherical Voronoi structure. To this end, we first show that after appropriately projecting these components, they are actually convex (!). We then establish the vanishing of the first cohomology of an appropriate two-dimensional simplicial complex constructed from the cluster's incidence structure. These ingredients were partly also observed in passing in the Gaussian setting, but were not required for establishing the minimality of the simplicial clusters thanks to the MPDI argument employed there. In contrast, these ingredients are completely crucial on $\R^n$ and $\S^n$, and the interplay between these two spaces is (surprisingly) essential for the argument. 
\item \label{it:intro-step-3}
The third step is to derive Theorem \ref{thm:intro-stable-spherical-Voronoi} in Section \ref{sec:Voronoi-connected}, asserting that whenever $q \leq n+1$, all cells are equatorial and connected. To establish that any equatorial cell must be connected, stability is invoked once again, but we also need to couple it with additional information on the higher-order connectivity of the adjacency graph of the cluster's connected components, and a strong maximum principle for the discrete Laplacian on the latter graph. We then use the aforementioned convexity in the spherical Voronoi representation to show that whenever $q \leq n+1$, all cells must be equatorial. 
\item 
In the final step we establish Theorem \ref{thm:intro-234} in Section \ref{sec:global}. Contrary to the previous steps, in which only local arguments (such as stationarity and stability) were applied, this step requires a global argument for excluding spherical Voronoi clusters having missing interfaces; it is shown in Section \ref{sec:Mobius} that a spherical Voronoi cluster having all of its interfaces present must be a standard bubble. By inspecting the cluster's adjacency graph, for each value of $q$ there are only finitely many possible graphs on $q$ vertices, representing the cluster's cells, to consider -- see Figures \ref{fig:4-graphs} and \ref{fig:5-graphs} in Section \ref{sec:global}. Some of these graphs may be excluded by utilizing the combinatorial and topological information derived in prior steps. Other graphs are excluded by arguments from Euclidean geometry -- see Figure \ref{fig:angles}. One final argument excludes non-rigid configurations, by sliding a ``loose" bubble until it hits some other bubbles in a manner which is prohibited by the known regularity theory for minimizing clusters -- see Figures \ref{fig:triple-bubble-perturbation}, \ref{fig:almost} and \ref{fig:min-degree}. The latter is a known argument in the double-bubble and planar settings \cite{SMALL93,Hutchings-StructureOfDoubleBubbles, Wichiramala-TripleBubbleInR2}, but various complications arise in the general multi-bubble setting: the number of the possible adjacency graphs grows super-exponentially with $q$, and furthermore, the classification of minimizing cones is only available in dimensions $2$ and $3$ thanks to Taylor's work \cite{Taylor-SoapBubbleRegularityInR3}, and so we can only determine if a meeting point of several bubbles is illegal when it involves at most $5$ cells. This explains why we cannot at present extend Theorem \ref{thm:intro-234} to handle arbitrary $q \leq n+1$. 
\end{enumerate}

\medskip

On a technical level, we continue to develop and expand the technical results obtained in \cite{EMilmanNeeman-GaussianMultiBubble}, which are summarized in the preliminary Section \ref{sec:prelim} and serve as the starting point of this work. The main technical difficulty consists of approximating a ``non-physical" Lipschitz scalar-field $f_{ij} = -f_{ji}$ defined on $\Sigma_{ij}$ (satisfying $f_{ij} + f_{jk} + f_{ki} = 0$ at triple-points, where three cells meet) as the normal component of a ``physical" global vector-field $X$, in a manner so that the first variations of volume remain the same, 
and the corresponding second variations encapsulated in the index-form $Q$ are arbitrarily close. Contrary to the double-bubble setting, where every smooth scalar-field is physical (i.e.~is the normal component of a smooth global vector-field $X$), this is not the case in the triple-and-higher-bubble setting due to several factors: the presence of quadruple-points (which are only known to have $C^{1,\alpha}$ regularity) and of additional possible singularities; the potential blowing-up of curvature near these points; and the existence of non-simplicial flat minimal cones in dimensions 4 and higher \cite{Brakke-MinimalConesOnCubes}, which allow for linear dependencies between the normals at meeting points of $5$ and more cells (even if the cluster is already known to be completely regular, e.g.~spherical Voronoi) -- see the discussion in Section \ref{sec:non-physical}. Consequently, we do not know how to rigorously approximate a non-physical scalar-field by a physical vector-field as above in general, but are able to do so in two cases -- in an averaged sense, so that the contribution to $Q$ of the boundary integral at the triple-points vanishes; and without averaging, but only after establishing that the cluster has locally bounded curvature. This approximation procedure is developed in Sections \ref{sec:bounded-curvature} and \ref{sec:non-physical}. In Section \ref{sec:Jacobi}, we rigorously justify the various formulas for $Q$ we subsequently require, which involve the second-order Jacobi operator and various equivalent forms of the boundary integral. In Section \ref{sec:conformal}, it is shown how these formulas simplify when applied to conformal Killing fields and other fields satisfying conformal boundary-conditions. 

\bigskip

\noindent
\textbf{Acknowledgments.} We thank Frank ``Chip" Morgan for his comments and interest. E.M.~warmly thanks the Oden Institute and the Math Department at the University of Texas in Austin for their generous support throughout his Sabbatical stay in Austin, during which this work had culminated; in particular, it is a pleasure to thank Rachel Ward and Francesco Maggi. Finally, we thank the referee for carefully reading the paper and their very helpful comments, which have improved the overall clarity and presentation.

\tableofcontents

\section{Preliminaries} \label{sec:prelim}

In the first couple of sections of this work, we further develop the theory of isoperimetric minimizing clusters on general weighted Riemannian manifolds (as this does not really pose a greater generality over the case of the model manifolds we are primarily interested in). Our concern in this work is the scenario when the minimizing cluster's finite volume cells are all bounded, and so we will only work with compactly supported fields. However, it is possible to extend the theory to the case when the cells are unbounded by using appropriately admissible fields, such as in the Gaussian setting of \cite{EMilmanNeeman-GaussianMultiBubble}.

\begin{definition}[Weighted Riemannian Manifold]
A smooth complete $n$-dimensional Riemannian manifold $(M^n,g)$ endowed with a measure $\mu$ with $C^\infty$ smooth positive density $\exp(-W)$ with respect to the Riemannian volume measure $\vol_g$ is called a weighted Riemannian manifold $(M^n,g,\mu)$. 
\end{definition}

The Levi-Civita connection on $(M,g)$ is denoted by $\nabla$. The Riemannian metric $g$ will often be denoted by $\scalar{\cdot,\cdot}$. It induces a geodesic distance on $(M,g)$, and we denote by $B(x,r)$ an open geodesic ball of radius $r >0$ in $(M,g)$ centered at $x \in M$. Recall that $\mu^k = e^{-W} \H^k$, where $\H^k$ denotes the $k$-dimensional Hausdorff measure.

Recall that $\simplex^{(q-1)}_T = \{ v \in \R^q_+ \; ; \; \sum_{i=1}^q v_i = T \}$ when $T < \infty$ and $\simplex^{(q-1)}_\infty := \R^{q-1}_+ \times \{ \infty\}$. We denote by $\simplex^{(q-1)} = \simplex^{(q-1)}_1$ the $(q-1)$-dimensional probability simplex. Its tangent space is denoted by $E^{(q-1)} := \{ x \in \R^q \; ; \; \sum_{i=1}^q x_i = 0 \}$.

Throughout this work we will often use the convention that $a_{ij}$ denotes $a_i - a_j$ whenever the individual objects $a_i$ are defined. For instance, if $\{e_i\}_{i=1,\ldots,q}$ denote unit-vectors in $\R^q$, then $e_{ij} = e_i - e_j \in E^{(q-1)}$. Denoting by $\delta^k_i$ the delta function $\textbf{1}_{k=i}$, we also have $\delta^k_{ij} = \delta^k_i - \delta^k_j$. 

Given distinct $i,j,k \in \{1,\ldots,q\}$, we define the set of cyclically ordered pairs in $\{i,j,k\}$:
\[
\cyclic(i,j,k) := \{ (i,j) , (j,k) , (k,i) \}  .
\]

\subsection{Weighted divergence and mean-curvature}

We write $\div X$ to denote divergence of a smooth vector-field $X$, and $\div_\mu X$ to denote its weighted divergence:
\begin{equation} \label{eq:weighted-div}
\div_{\mu} X := \div (X e^{-W}) e^{+W} = \div X - \nabla_X \pot . 
\end{equation}
For a smooth hypersurface $\Sigma \subset M^n$ co-oriented by a unit-normal field $\n$, let $H_\Sigma$ denote its mean-curvature, defined as the trace of its second fundamental form $\II_{\Sigma}$. We employ the sign convention that $\II_{\Sigma}(u,v) = \scalar{\nabla_u \n,v}$ for $u,v \in T_p \Sigma$ (so a sphere co-oriented by its outward normal has positive curvature). 
The weighted mean-curvature $H_{\Sigma,\mu}$ is defined as:
\[
H_{\Sigma,\mu} := H_{\Sigma} - \nabla_\n \pot .
\]
We write $\div_\Sigma X$ for the surface divergence of a vector-field $X$ defined on $\Sigma$, i.e. $\sum_{i=1}^{n-1} \scalar{\tang_i,\nabla_{\tang_i} X}$ where $\{\tang_i\}$ is a local orthonormal frame on $\Sigma$; this coincides with $\div X - \scalar{\n,\nabla_\n X}$ for any smooth extension of $X$ to a neighborhood of $\Sigma$. 
The weighted surface divergence $\div_{\Sigma,\mu}$ is defined as:
\[
\div_{\Sigma,\mu} X = \div_{\Sigma} X - \nabla_X \pot,
\]
so that $\div_{\Sigma,\mu} X = \div_{\Sigma} (X e^{-\pot}) e^{+\pot}$ if $X$ is tangential to $\Sigma$.  
Note that $\div_{\Sigma} \n = H_{\Sigma}$ and $\div_{\Sigma,\mu} \n = H_{\Sigma,\mu}$.  We will also abbreviate $\scalar{X,\n}$ by $X^\n$, and we will write $X^\tang$ for the tangential part of $X$, i.e. $X - X^{\n} \n$. 

Note that the above definitions ensure the following weighted version of Stokes' theorem: if $\Sigma$ is a smooth $(n-1)$-dimensional manifold with $C^1$ boundary, denoted $\partial \Sigma$, (completeness of $\Sigma \cup \partial \Sigma$ is not required), and $X$ is a smooth vector-field on $\Sigma$, continuous up to $\partial \Sigma$, with compact support in $\Sigma \cup \partial \Sigma$, then since:
\[
\div_{\Sigma,\mu} X = \div_{\Sigma,\mu} (X^\n \n) + \div_{\Sigma,\mu} X^{\tang} = H_{\Sigma,\mu} X^{\n} + \div_{\Sigma,\mu} X^{\tang} ,
\]
then:
\begin{equation} \label{eq:Stokes-classical}
\int_\Sigma \div_{\Sigma,\mu} X d\mu^{n-1} = \int_{\Sigma} H_{\Sigma,\mu} X^{\n} d\mu^{n-1} + \int_{\partial \Sigma} X^{\n_{\partial}} d\mu^{n-2} ,
\end{equation}
where $\n_{\partial}$ denotes the exterior unit co-normal to $\partial \Sigma$.

Finally, we denote the surface Laplacian of a smooth function $f$ on $\Sigma$ by $\Delta_{\Sigma} f := \div_{\Sigma} \nabla^\tang f$, which coincides with $\sum_{i=1}^{n-1} \nabla^2_{\tang_i,\tang_i} f - H_{\Sigma} \nabla_{\n} f$ for any smooth extension of $f$ to a neighborhood of $\Sigma$ in $M$. The weighted surface Laplacian is defined as:
\[
 \Delta_{\Sigma,\mu} f := \div_{\Sigma,\mu} \nabla^\tang f = \Delta_{\Sigma} f - \scalar{\nabla^\tang f, \nabla^\tang \pot} . 
 \]

\subsection{Reduced boundary and perimeter}

Given a Borel set $U \subset \R^n$ with locally-finite perimeter, its reduced boundary $\partial^* U$ is defined as the subset of $\partial U$ for which there is a uniquely defined outer unit normal vector to $U$ in a measure theoretic sense (see \cite[Chapter 15]{MaggiBook} for a precise definition). The definition of reduced boundary canonically extends to the Riemannian setting by using a local chart, as it is known that $T(\partial^* U) = \partial^* T(U)$ for any smooth diffeomorphism $T$ (see \cite[Lemma A.1]{KMS-LimitOfCapillarity}). It is known that $\partial^* U$ is a Borel subset of $\partial U$, and that modifying $U$ on a null-set does not alter $\partial^* U$. If $U$ is an open set with $C^1$ smooth boundary, it holds that $\partial^* U = \partial U$  (e.g. \cite[Remark 15.1]{MaggiBook}). Recall that the $\mu$-weighted perimeter of $U$ is defined as:
\[
\per_\mu(U) = \mu^{n-1}(\partial^* U). 
\]

\subsection{Cluster interfaces}

Let $\Omega = (\Omega_1, \dots, \Omega_q)$ denote a $q$-cluster on $(M^n,g,\mu)$. 
Recall that the cells $\set{\Omega_i}_{i=1,\ldots,q}$
of a $q$-cluster $\Omega$ are assumed to be pairwise disjoint Borel subsets of $M^n$ so that $\mu(M^n \setminus \cup_{i=1}^q \Omega_i) = 0$ and $\mu(\Omega) = (\mu(\Omega_1),\ldots,\mu(\Omega_q)) \in \Delta^{(q-1)}_{\mu(M^n)}$. In addition, they are assumed to have locally-finite perimeter, and moreover, finite $\mu$-weighted perimeter $P_\mu(\Omega_i) < \infty$.

We define the interface between cells $i$ and $j$ (for $i \ne j$) as:
\[
    \Sigma_{ij} = \Sigma_{ij}(\Omega) := \partial^* \Omega_i \cap \partial^* \Omega_j . 
\]
 It is standard to show (see \cite[Exercise 29.7, (29.8)]{MaggiBook}) that for any $S \subset \set{1,\ldots,q}$:
\begin{equation} \label{eq:nothing-lost-many}
\H^{n-1}\brac{\partial^*(\cup_{i \in S} \Omega_i) \setminus \cup_{i \in S , j \notin S} \Sigma_{ij}} = 0 .
\end{equation}
In particular: \begin{equation} \label{eq:nothing-lost}
\H^{n-1} \brac{ \partial^* \Omega_i  \setminus \cup_{j \neq i} \Sigma_{ij} } = 0 \;\;\; \forall i=1,\ldots,q ,
\end{equation}
and hence:
\[
\per_\mu(\Omega_i) = \sum_{j \neq i} \mu^{n-1}(\Sigma_{ij}) , 
\]
and:
\[
\per_\mu(\Omega) = \frac{1}{2} \sum_{i=1}^q \per_\mu(\Omega_i) = \sum_{i < j} \mu^{n-1}(\Sigma_{ij})  = \mu^{n-1}(\Sigma^1). 
\]
In addition, it follows (see \cite[(3.6)]{EMilmanNeeman-GaussianMultiBubble}) that:
\begin{equation} \label{eq:top-nothing-lost}
\forall i \;\;\; \overline{\partial^* \Omega_i} = \overline{\cup_{j \neq i} \Sigma_{ij}} .
\end{equation}

\subsection{Existence and boundedness of minimizing clusters} \label{subsec:prelim-minimizing}

\begin{definition} \label{def:bounded}
A cluster $\Omega$ is called bounded if, up to null-set modification of its cells, its boundary $\Sigma = \cup_{i=1}^q \partial \Omega_i$ is a bounded set. \end{definition}

\begin{theorem}[Existence and boundedness of isoperimetric minimizing clusters]  \label{thm:existence}
\hfill
  \begin{enumerate}[(i)]
\item \label{it:existence-prob}
   If $\mu$ is a probability measure, then for any prescribed $v \in \simplex^{(q-1)}_1$, an isoperimetric $\mu$-minimizing $q$-cluster $\Omega$ satisfying $\mu(\Omega) = v$ exists. 
\item \label{it:existence-non-compact}
  If $(M,g,\mu)$ is non-compact but has an isometry group (i.e. smooth automorphisms $T$ satisfying $T_* g = g$ and $T_* \mu = \mu$) so that the resulting quotient space is compact, then for any $v \in \Delta^{(q-1)}_\infty$, a $\mu$-minimizing $q$-cluster $\Omega$ satisfying $\mu(\Omega) = v$ exists.  \\ Moreover, a $\mu$-minimizing cluster is necessarily bounded. 
\end{enumerate}
\end{theorem}
\begin{proof}
The first part is a classical argument based on the lower semi-continuity of (weighted) perimeter -- see the proof of \cite[Theorem 4.1 (i)]{EMilmanNeeman-GaussianMultiBubble}. The second part in Euclidean space $(\R^n,|\cdot|,dx)$ is due to Almgren \cite{AlmgrenMemoirs}, with the main challenge being that volume can ``escape to infinity". The extension to general (unweighted) manifolds with compact quotients by their isometry groups is due to Morgan (see the comments following \cite[Theorem 13.4]{MorganBook5Ed}). The extension to the weighted setting is straightforward (see e.g. the proof of \cite[Lemma 13.6]{MorganBook5Ed}). Boundedness of the finite-volume cells follows from the usual $n$-dimensional isoperimetric inequality (which continues to hold up to constants on any compact weighted Riemannian manifold) by a classical argument of Almgren (e.g. \cite[Lemma 13.6]{MorganBook5Ed} or \cite[Theorem 29.1]{MaggiBook}). It follows that necessarily $\Sigma$ is bounded (after applying null-set modifications to the cells as specified in Theorem \ref{thm:Almgren} \ref{it:Almgren-ii} below). 
\end{proof}

\subsection{Interface-regularity}

The following theorem is due to Almgren \cite{AlmgrenMemoirs} (see also~\cite[Chapter 13]{MorganBook5Ed} and~\cite[Chapters 29-30]{MaggiBook}). Recall that: 
\begin{equation} \label{eq:Sigma-Sigma1}
    \Sigma := \cup_{i} \partial \Omega_i ~,~  \Sigma^1 := \cup_{i < j} \Sigma_{ij} .
\end{equation}

\begin{theorem}[Almgren] \label{thm:Almgren}
    For every isoperimetric minimizing cluster $\Omega$ on $(M^n,g,\mu)$:
    \begin{enumerate}[(i)] 
\item \label{it:Almgren-ii}
  $\Omega$ may and will be modified on a $\mu$-null set (thereby not altering $\set{\partial^* \Omega_i}$) so that all of its cells are open, and so that for every $i$, 
$\overline{\partial^* \Omega_i} = \partial \Omega_i$  and $\mu^{n-1}(\partial \Omega_i \setminus \partial^* \Omega_i) = 0$. 
\item \label{it:Almgren-iii} 
    For all $i \neq j$ the interfaces $\Sigma_{ij} = \Sigma_{ij}(\Omega)$ are a locally-finite union of embedded
     $(n-1)$-dimensional $C^\infty$ manifolds, relatively open in $\Sigma$, and for every $x \in
    \Sigma_{ij}$ there exists $\epsilon > 0$ such that $B(x,\epsilon) \cap
    \Omega_k = \emptyset$ for all $k \neq i,j$ and $B(x,\epsilon) \cap \Sigma_{ij}$ is an embedded
     $(n-1)$-dimensional $C^\infty$ manifold.
\item \label{it:Almgren-density}
    For any compact set $K$ in $M$, there exist constants $\Lambda_K,r_K > 0$ so that:
    \[     \mu^{n-1}(\Sigma \cap B(x,r)) \leq \Lambda_K r^{n-1} \;\;\; \forall x \in \Sigma \cap K \;\;\; \forall r \in (0,r_K) . 
    \] \end{enumerate}
\end{theorem}
Whenever referring to the cells of a minimizing cluster or their topological boundary (and in particular $\Sigma$) in this work, we will always choose a representative such as in Theorem \ref{thm:Almgren} \ref{it:Almgren-ii}. In particular, there is no need to apply any null-set modifications in Definition \ref{def:bounded} of boundedness when using the latter representative; in addition, a cell $\Omega_i$ is non-empty if and only if $\mu(\Omega_i) > 0$. 

\begin{definition}[Interface--regular cluster] \label{def:interface-regular}
A cluster $\Omega$ satisfying parts \ref{it:Almgren-ii} and \ref{it:Almgren-iii} of Theorem \ref{thm:Almgren}
is called interface--regular. 
\end{definition}

The definition of interface-regular cluster should not be confused with the stronger definition of regular cluster, introduced below in Definition \ref{def:regular}.  Given an interface--regular cluster, let $\n_{ij}$ be the (smooth) unit normal field along $\Sigma_{ij}$ that points from $\Omega_i$ to $\Omega_j$. We use $\n_{ij}$ to co-orient $\Sigma_{ij}$, and since $\n_{ij} = -\n_{ji}$, note that $\Sigma_{ij}$ and $\Sigma_{ji}$ have opposite orientations. When $i$ and $j$ are clear from the context, we will simply write $\n$. We will typically abbreviate $H_{\Sigma_{ij}}$ and $H_{\Sigma_{ij},\mu}$ by $H_{ij}$ and $H_{ij,\mu}$, respectively. 

\medskip

It will also be useful to record the following result due to Leonardi (used in \cite[Corollary 30.3]{MaggiBook} for the proof given there of Theorem \ref{thm:Almgren} \ref{it:Almgren-iii}). 
Recall that the (lower) density of a measurable set $A \subset (M^n,g)$ at a point $p \in M^n$ is defined as:
\[
\Theta(A,p) := \liminf_{r \rightarrow 0+} \frac{\H^n(A \cap B(p,r))}{\H^n(B(p,r))} . 
\]

\begin{lemma}[Leonardi's Infiltration Lemma] \label{lem:infiltration}
Let $\Omega$ be an isoperimetric $\mu$-minimizing $q$-cluster on $(M^n,g,\mu)$, and recall our convention from Theorem \ref{thm:Almgren} \ref{it:Almgren-ii}. There exists a constant $\eps_n > 0$, depending solely on $n$, so that for any $p \in M^n$, $i=1,\ldots,q$ and connected component $\Omega_i^\ell$ of $\Omega_i$: 
\begin{equation} \label{eq:density}
\Theta(\Omega^\ell_i,p) < \eps_n \;\; \Rightarrow \;\; p \notin \overline{\Omega^\ell_i} . 
\end{equation}
\end{lemma}
\begin{proof}
This was proved by Leonardi in \cite[Theorem 3.1]{Leonardi-Infiltration} (cf. \cite[Lemma 30.2]{MaggiBook}) for the cells of a minimizing cluster on $\R^n$. In case the cell $\Omega_i$ has more than one connected component, simply define a new $(q+1)$-cluster by splitting $\Omega_i$ into $\Omega_i^\ell$ and $\Omega_i \setminus \Omega_i^\ell$, and note that this new cluster is itself minimizing. Lastly, the proof clearly extends to any weighted Riemannian manifold by passing to a local chart (and the positive smooth density is obviously immaterial). 
\end{proof}

\subsection{Linear Algebra}

\begin{lemma}[Graph connectedness]\label{lem:LA-connected}
Let $\Omega$ be a interface-regular $q$-cluster on $(M^n,g,\mu)$ with $\mu(\Omega) \in \interior \simplex^{(q-1)}_{\mu(M^n)}$. Consider the undirected graph $G$ with vertices $\{1, \dots, q\}$ and an edge between $i$ and $j$ iff $\Sigma_{ij} \neq \emptyset$. 
Then the graph $G$ is connected.
\end{lemma}
\begin{proof}
This was proved in \cite[Lemma 3.4]{EMilmanNeeman-GaussianMultiBubble} on $(\R^n,|\cdot|,\mu)$. The proof extends to any  weighted Riemannian manifold $(M^n,g,\mu)$ on which $\mu(K) \in (0,\mu(M^n))$ implies $P_{\mu}(K) > 0$, which is always the case by connectedness of $M^n$ and the strictly positive density of $\mu$. Indeed, if $S \subset \{1, \dots, q\}$ were a non-trivial connected component then $U = \bigcup_{i \in S} \Omega_i$ would satisfy $\mu(U) > 0$, and by replacing $S$ with its complement if necessary, we also have $\mu(U) < \mu(M^n)$ (recall that a cluster may have at most one cell of infinite measure). At the same time $\per_\mu(U) = \sum_{i \in S , j \notin S} \mu^{n-1}(\Sigma_{ij}) = 0$ by (\ref{eq:nothing-lost-many}), a contradiction to the single-bubble isoperimetric inequality stated above. 
\end{proof}

\begin{definition}[Quadratic form $\L_A$] \label{def:LA}
Given real-valued weights $A = \{ A^{ij} \}_{i,j =1,\ldots,q}$ which are non-oriented (i.e. satisfy $A^{ij} = A^{ji}$), define the following $q$ by $q$ symmetric matrix called the discrete (weighted) Laplacian:
\begin{equation} \label{eq:L_A}
        \L_A := \sum_{1 \leq i<j \leq q} A^{ij} e_{ij} \otimes e_{ij} .
\end{equation}
We will mostly consider $\L_A$ as a quadratic form on $E^{(q-1)}$. \\
Given an interface-regular $q$-cluster $\Omega$ on $(M^n,g,\mu)$, we set $\L_1 := \L_{\{\mu^{n-1}(\Sigma_{ij})\}}$. 
\end{definition}

Given an undirected graph $G$ on $\{1,\ldots,q\}$, we will say that the weights $A = \{ A^{ij} \}_{i,j=1,\ldots,q}$ are supported along the edges of $G$, and write $A = \{ A^{ij} \}_{i \sim j}$, if $A^{ij} = 0$ whenever there is no edge in $G$ between vertices $i$ and $j$. 

\begin{lemma}[$L_A$ is positive-definite] \label{lem:LA-positive}
 For all strictly positive non-oriented weights $A = \{ A^{ij} > 0\}_{i \sim j}$ supported along the edges of a connected graph $G$, $\L_A$ is positive-definite as a quadratic form on $E^{(q-1)}$ (in particular, it has full-rank $q-1$). 
\end{lemma}
\begin{proof}
The discrete Laplacian on a connected graph is positive definite on the subspace $E^{(q-1)}$ perpendicular to the all-ones vector -- see \cite[Lemma 3.4]{EMilmanNeeman-GaussianMultiBubble}.
\end{proof}

We call the weights $A = \{ A^{ij} \}_{i\sim j}$ oriented if $A^{ji} = -A^{ij}$ for all $i\sim j$. 

\begin{lemma} \label{lem:LA}
Fix an undirected connected graph $G$ on $\{1,\ldots,q\}$ and oriented weights $A = \{ A^{ij} \}_{i \sim j}$. Then the following are equivalent:
\begin{enumerate}[(i)]
\item \label{it:LA1} $A^{ij} = a_i - a_j$ for some $a \in E^{(q-1)}$ and all $i \sim j$. 
\item \label{it:LA2} For every directed cycle $(i_1,\ldots,i_{N+1})$ on $G$ (with $i_k \sim i_{k+1}$ and $i_{N+1} = i_1$) it holds that $\sum_{k=1}^N A^{i_k i_{k+1}} = 0$. 
\item \label{it:LA3} For all oriented weights $B = \{ B^{ij} \}_{i \sim j}$ so that $\sum_{j \; ; \; j \sim i} B^{ij} = 0$ for all $i=1,\ldots,q$, it holds that $\sum_{i < j \; ; \; i\sim j} A^{ij} B^{ij} = 0$.
\end{enumerate}
\end{lemma}
\begin{proof}
The equivalence between statements \ref{it:LA1} and \ref{it:LA2} is straightforward, thanks to the connectivity of $G$ (vanishing of the zeroth reduced homology). Statement \ref{it:LA1} implies \ref{it:LA3} since:
\[
\sum_{i<j \; ; \; i \sim j } (a_i - a_j) B^{ij} = \sum_i a_i \sum_{j \; ; \; j \sim i} B^{ij} = \sum_i a_i 0 = 0 . 
\]
Finally, assume that statement \ref{it:LA3} holds and let $(i_1,\ldots,i_{N+1})$ be a simple directed cycle (with no repeating edges) on $G$. Define the oriented weights $B^{ij}$ by $B^{i_k i_{k+1}} = 1 = -B^{i_{k+1} i_{k}}$ along the cycle and zero everywhere else. Then $B$ satisfies $\sum_{j \; ; \; j \sim i} B^{ij} = 0$ for all $i$ (this trivially holds for vertices outside the cycle, and since $1-1=0$ on the cycle itself), and hence $\sum_{k=1}^N A^{i_k i_{k+1}} = \sum_{i < j \; ; \; i \sim j} A^{ij} B^{ij} = 0$, verifying statement \ref{it:LA2} for simple directed cycles. The case of general directed cycles follows by decomposing into simple cycles. 
\end{proof}

\subsection{First variation information -- stationary clusters}

Given a $C_c^\infty$ vector-field $X$ on $M$, let $F_t$ be the associated flow along $X$, defined as the family of $C^\infty$ diffeomorphisms $\set{F_t : M^n \to M^n}_{t \in \R}$ solving the following ODE:
\[ \frac{d}{dt} F_t(x) = X \circ F_t(x) ~,~ F_0 = \Id .
\] Clearly, $F_t(\Omega) = (F_t(\Omega_1), \ldots,F_t(\Omega_q))$ remains a cluster for all $t$. 
We define the $k$-th variations of weighted volume and perimeter of $\Omega$, $k \geq 1$, as:
\begin{align*}
  \delta_X^k V(\Omega)_i = \delta_X^k V(\Omega_i) & := \left. \frac{d^k}{(dt)^k}\right|_{t=0} \mu(F_t(\Omega_i)) \;\;  i = 1,\ldots,q-1 ~,\\
  \delta_X^k A(\Omega) & := \left. \frac{d^k}{(dt)^k}\right|_{t=0} P_\mu(F_t(\Omega)) .
\end{align*}
By \cite[Lemma 3.3]{EMilmanNeeman-GaussianMultiBubble}, $t \mapsto \mu(F_t(\Omega_i))$ ($i=1,\ldots,q-1$) and $t \mapsto P_\mu(F_t(\Omega))$ are $C^\infty$ functions in some open neighborhood of $t=0$, and so in particular the above variations are well-defined and finite. To properly treat the case when $\mu(M) = \infty$ and hence $\mu(\Omega_q) = \infty$, we define:
\[
\delta_X^k V(\Omega)_q = \delta_X^k V(\Omega_q) := -  \sum_{i=1}^{q-1} \delta_X^k V(\Omega_i) .
\]
Note that when $\mu(M) < \infty$, the above definition coincides with:
\[
 \delta_X^k V(\Omega)_q = \delta_X^k V(\Omega_q) := \left. \frac{d^k}{(dt)^k}\right|_{t=0} \mu(F_t(\Omega_q)) . 
\]
Our convention ensures that regardless of whether $\mu(M) = \infty$ or $\mu(M) < \infty$, we have $\delta_X^k V(\Omega) \in E^{(q-1)}$. When $\Omega$ is clear from the context, we will simply write $\delta_X^k V$ and $\delta_X^k A$. 

\bigskip

By testing the first-variation of a minimizing cluster, one obtains the following well-known information (see \cite[Lemmas 4.3 and 4.5]{EMilmanNeeman-GaussianMultiBubble} for a detailed proof in the case that $\mu(M) < \infty$; the proof immediately extends to the case $\mu(M)=\infty$ with the above convention): 

\begin{lemma}[First-order conditions] \label{lem:first-order-conditions}
  For any isoperimetric minimizing $q$-cluster $\Omega$ on $(M^n,g,\mu)$:
  \begin{enumerate}[(i)]
    \item \label{it:first-order-constant} 
    On each $\Sigma_{ij}$, $H_{ij,\mu}$ is constant. 
    \item \label{it:first-order-cyclic}
    There exists $\lambda \in E^{(q-1)}$ such that $H_{ij,\mu} = \lambda_i - \lambda_j$ for all $i \neq j$; moreover, $\lambda \in E^{(q-1)}$ is unique whenever $\mu(\Omega) \in \interior \simplex^{(q-1)}_{\mu(M^n)}$. 
        \item \label{it:weak-angles}
    $\Sigma^1$ does not have a boundary in the distributional sense -- for every $C_c^\infty$ vector-field $X$:
      \begin{equation} \label{eq:no-boundary}
        \sum_{i<j} \int_{\Sigma_{ij}} \div_{\Sigma,\mu} X^\tang\, d\mu^{n-1} = 0.
      \end{equation}
        \end{enumerate}
\end{lemma}

\begin{remark}
The physical interpretation of $\lambda_i$ is that of air pressure inside cell $\Omega_i$; the weighted mean-curvature $H_{ij,\mu}$ is thus the pressure difference across $\Sigma_{ij}$. 
\end{remark}
 
\begin{definition}[Stationary Cluster]
An interface-regular $q$-cluster $\Omega$ satisfying the three conclusions of Lemma \ref{lem:first-order-conditions} is called stationary (with Lagrange multiplier $\lambda \in E^{(q-1)}$). \end{definition}

\noindent The following lemma provides an interpretation of $\lambda \in E^{(q-1)}$ as a Lagrange multiplier for the isoperimetric constrained minimization problem:

\begin{lemma}[Lagrange Multiplier] \label{lem:Lagrange}      
Let $\Omega$ be stationary $q$-cluster with Lagrange multiplier $\lambda \in E^{(q-1)}$.
    Then for every $C_c^\infty$ vector-field $X$: 
    \begin{align}
    \nonumber
     \delta^1_X V(\Omega)_i &= \sum_{j \neq i} \int_{\Sigma_{ij}} X^{\n_{ij}}\, d\mu^{n-1} \;\;\; \forall i = 1,\ldots,q \;\; , \\
     \label{eq:1st-var-area}
      \delta^1_X A(\Omega) &= \sum_{i<j} H_{ij,\mu} \int_{\Sigma_{ij}} X^{\n_{ij}}\, d\mu^{n-1} .
    \end{align} 
    In particular, 
    \[
      \delta^1_X A = \scalar{\lambda,\delta^1_X V} .
    \]
  \end{lemma}

\subsection{Regular clusters} \label{subsec:regularity}

Given a minimizing cluster $\Omega$, recall (\ref{eq:Sigma-Sigma1})
 and observe that $\Sigma = \overline{\Sigma^1}$ by (\ref{eq:top-nothing-lost}) and our convention from Theorem \ref{thm:Almgren} \ref{it:Almgren-ii}. We will require additional information on the higher codimensional structure of $\Sigma$. To this end, define two special cones:
\begin{align*}
    \Y &:= \{x \in E^{(2)}\; ; \; \text{ there exist $i \ne j \in \{1,2,3\}$ with $x_i = x_j = \max_{k \in \{1,2,3\}} x_k$}\} , \\
    \T &:= \{x \in E^{(3)}\; ; \; \text{ there exist $i \ne j \in \{1,2,3,4\}$ with $x_i = x_j = \max_{k \in \{1,2,3,4\}} x_k$}\}.
\end{align*}
Note that $\Y$ consists of $3$ half-lines meeting at the origin in $120^\circ$ angles, and that $\T$ consists of $6$ two-dimensional sectors meeting in threes at $120^{\circ}$ angles along $4$ half-lines, which in turn all meet at the origin in $\cos^{-1}(-1/3) \simeq 109^{\circ}$ angles. The next theorem asserts that on the codimension-$2$ and codimension-$3$ parts of a minimizing cluster, $\Sigma$ locally looks like $\Y \times \R^{n-2}$ and $\T \times \R^{n-3}$, respectively.

\begin{theorem}[Taylor, White, Colombo--Edelen--Spolaor] \label{thm:regularity}
    Let $\Omega$ be a minimizing cluster on $(M^n,g,\mu)$, with $\mu = \exp(-W) \vol_g$ and $W \in C^\infty(M)$. Then there exist $\alpha > 0$
          and sets $\Sigma^2, \Sigma^3, \Sigma^4 \subset \Sigma$ such that:
    \begin{enumerate}[(i)]
        \item \label{it:regularity-union}
        $\Sigma$ is the disjoint union of $\Sigma^1,\Sigma^2,\Sigma^3,\Sigma^4$;             
        \item \label{it:regularity-Sigma2}         
        $\Sigma^2$ is a locally-finite union of embedded $(n-2)$-dimensional $C^{\infty}$ manifolds, and for every $p \in \Sigma^2$ there is a $C^\infty$ diffeomorphism mapping a neighborhood of $p$ in $M$ to a neighborhood of the origin in $E^{(2)} \times \R^{n-2}$, so that $p$ is mapped to the origin and $\Sigma$ is locally mapped to $\Y \times \R^{n-2}$ ; 
        \item \label{it:regularity-Sigma3}
         $\Sigma^3$ is a locally-finite union of embedded $(n-3)$-dimensional $C^{1,\alpha}$ manifolds, and for every $p \in \Sigma^3$ there is a $C^{1,\alpha}$ diffeomorphism mapping a neighborhood of $p$ in $M$ to a neighborhood of the origin in $E^{(3)} \times \R^{n-3}$, so that $p$ is mapped to the origin and $\Sigma$ is locally mapped to $\T \times \R^{n-3}$ ;
         \item \label{it:regularity-dimension}
         $\Sigma^4$ is closed and $\dim_{\H}(\Sigma^4) \leq n-4$.             
    \end{enumerate}
\end{theorem}
\begin{remark}
In the classical unweighted Euclidean setting (when $\Omega$ is a minimizing cluster with respect to the Lebesgue measure in $\R^n$), the case $n=2$ was shown by F.~Morgan in \cite{MorganSoapBubblesInR2} building upon the work of Almgren \cite{AlmgrenMemoirs}, and also follows from the results of  J.~Taylor~\cite{Taylor-SoapBubbleRegularityInR3}. The case $n=3$  was established by Taylor \cite{Taylor-SoapBubbleRegularityInR3} for general $(\M,\eps,\delta)$ sets in the sense of Almgren. 
 When $n \geq 4$, Theorem \ref{thm:regularity} was announced by B.~White \cite{White-AusyAnnouncementOfClusterRegularity,White-SoapBubbleRegularityInRn} for general $(\M,\eps,\delta)$ sets. Theorem \ref{thm:regularity} with part \ref{it:regularity-Sigma3} replaced by $\dim_{\H}(\Sigma^3) \leq n-3$ follows from the work of L.~Simon \cite{Simon-Codimension2Regularity}.
A version of Theorem \ref{thm:regularity} for multiplicity-one integral varifolds in an open set $U \subset \R^n$ having associated cycle structure, no boundary in $U$, bounded mean-curvature and whose support is $(\M,\eps,\delta)$ minimizing, was established by M.~Colombo, N.~Edelen and L.~Spolaor  \cite[Theorem~1.3, Remark~1.4, Theorem~3.10]{CES-RegularityOfMinimalSurfacesNearCones}; in particular, this applies to isoperimetric minimizing clusters in $\R^n$ \cite[Theorem~3.8]{CES-RegularityOfMinimalSurfacesNearCones}. By working in smooth charts and inserting their effect as well as that of the smooth positive density into the excess function, the latter extends to the smooth weighted Riemannian setting --  see the proof of \cite[Theorem~5.1]{EMilmanNeeman-GaussianMultiBubble} for a verification. The work of Naber and Valtorta \cite{NaberValtorta-MinimizingHarmonicMaps} implies that $\Sigma^4$ is actually $\H^{n-4}$-rectifiable and has locally-finite $\H^{n-4}$ measure, but we will not require this here. 
\end{remark}
\begin{remark} \label{rem:extend-to-Riemannian}
 In the aforementioned references, \ref{it:regularity-Sigma2} is established with only $C^{1,\alpha}$ regularity, but elliptic regularity for systems of PDEs and a classical reflection argument of Kinderlehrer--Nirenberg--Spruck \cite{KNS} allows to upgrade this to the stated $C^\infty$ regularity -- see \cite[Corollary 5.6 and Appendix F]{EMilmanNeeman-GaussianMultiBubble} for a proof on $(\R^n,|\cdot|,\mu)$. The latter argument extends to the Riemannian setting by working in a smooth local chart and inserting the effect of the Riemannian metric into the system of PDEs for the constant (weighted) mean curvatures of $\Sigma_{ij}$, $\Sigma_{jk}$ and $\Sigma_{ki}$, which are already in quasi-linear form. 
\end{remark}

\begin{corollary} \label{cor:locally-finite-Sigma2}
For any compact set $K \subset \R^n$ which is disjoint from $\Sigma^4$, $\mu^{n-2}(\Sigma^2 \cap K) < \infty$. 
\end{corollary}

  \begin{definition}[Regular Cluster] \label{def:regular} A $q$-cluster $\Omega$ with $V_\mu(\Omega) \in \interior \Delta^{(q-1)}_{\mu(M)}$ satisfying parts \ref{it:Almgren-ii}, \ref{it:Almgren-iii} and \ref{it:Almgren-density} of Theorem \ref{thm:Almgren} and the conclusions of Lemma \ref{lem:infiltration} and Theorem \ref{thm:regularity} 
  is called regular.  \end{definition}

Note that to avoid pathological cases, we have included the assumption $V_\mu(\Omega) \in \interior \Delta^{(q-1)}_{\mu(M)}$, namely that all (open) cells are non-empty, in the definition of regularity. 
For a regular cluster, every point in $\Sigma^2$ (called the \emph{triple-point set}) belongs to the closure of exactly three cells, as well as to the closure of exactly three interfaces. Given distinct $i,j,k$, we will write $\Sigma_{ijk}$ for the subset of $\Sigma^2$ which belongs to the closure of $\Omega_i$, $\Omega_j$ and $\Omega_k$, or equivalently, to the closure of $\Sigma_{ij}$, $\Sigma_{jk}$ and $\Sigma_{ki}$. Similarly, we will call $\Sigma^3$ the \emph{quadruple-point set}, and given distinct $i,j,k,l$, denote by $\Sigma_{ijkl}$ the subset of $\Sigma^3$ which belongs to the closure of $\Omega_i$, $\Omega_j$, $\Omega_k$ and $\Omega_l$, or equivalently, to the closure of all six $\Sigma_{ab}$ for distinct $a,b \in \{ i , j , k, l \}$. 
We will extend the normal fields $\n_{ij}$ to $\Sigma_{ijk}$ and $\Sigma_{ijkl}$ 
by continuity (thanks to $C^1$ regularity). 

\smallskip

Let us also denote:
\[
\partial \Sigma_{ij} := \bigcup_{k \neq i,j} \Sigma_{ijk}.
\]
Note that $\Sigma_{ij} \cup \partial \Sigma_{ij}$ is a (possibly incomplete) $C^\infty$ manifold with boundary.
We define $\n_{\partial ij}$ on $\partial \Sigma_{ij}$ to be the outward-pointing unit (boundary) co-normal to $\Sigma_{ij}$. We denote by $\II^{ij}$ the second fundamental form on $\Sigma_{ij}$, which may be extended by continuity to $\partial \Sigma_{ij}$.
We will abbreviate $\II^{ij}_{\partial \partial}$ for $\II^{ij}(\n_{\partial ij}, \n_{\partial ij})$ on $\partial \Sigma_{ij}$.
When $i$ and $j$ are clear from the context, we will write $\n_\partial$ for $\n_{\partial ij}$ and $\II$ for $\II^{ij}$.

\subsection{Angles at triple-points} 

\begin{lemma}\label{lem:boundary-normal-sum}
    For any stationary regular cluster:
    \begin{equation} \label{eq:sum-n-zero}
    \sum_{(i,j) \in \cyclic(u,v,w)} \n_{ij} = 0 ~,~ \sum_{(i,j) \in \cyclic(u,v,w)} \n_{\partial ij} = 0  \;\;\;\;\; \forall p \in \Sigma_{uvw} . 
    \end{equation}
      In other words, $\Sigma_{uv}$, $\Sigma_{vw}$ and $\Sigma_{wu}$ meet at $\Sigma_{uvw}$ in $120^{\circ}$ angles. 
\end{lemma}

\noindent
See \cite[Corollary 5.5]{EMilmanNeeman-GaussianMultiBubble} for a proof. 
We will frequently use that for all $p \in \Sigma_{ijk}$:
\begin{equation} \label{eq:sqrt3}
\n_{\partial ij} = \frac{\n_{ik} + \n_{jk}}{\sqrt{3}} .
\end{equation}

As a consequence of the $120^\circ$ angles of incidence at triple-points, we have the following simple identities \cite[Lemmas 6.9 and 6.10]{EMilmanNeeman-GaussianMultiBubble} on any stationary regular cluster:

\begin{lemma}\label{lem:angles-form}
At every point $p \in \Sigma_{uvw}$ and for all $y,z \in T_p \Sigma_{uvw}$:
    \begin{enumerate}[(a)]
    \item \label{it:angles-form-a} $\II^{ij}(y,\n_{\partial ij})$ is independent of $(i,j) \in \cyclic(u,v,w)$.
    \item \label{it:angles-form-b} $\sum_{(i,j) \in \cyclic(u,v,w)} \II^{ij}(y,z) = 0$.
    \item \label{it:angles-form-c} $\sum_{(i,j) \in \cyclic(u,v,w)} \II^{ij}_{\partial \partial} = 0$. 
    \end{enumerate}
 \end{lemma}
 \begin{proof}
 Part \ref{it:angles-form-a} was shown in the proof of \cite[Lemma 6.9]{EMilmanNeeman-GaussianMultiBubble}. Part \ref{it:angles-form-b} is immediate since $\sum_{(i,j) \in \cyclic(u,v,w)} \n_{ij} = 0$ along $\Sigma_{uvw}$. Part \ref{it:angles-form-c} follows from part \ref{it:angles-form-b} since:
 \[
 \sum_{(i,j) \in \cyclic(u,v,w)} H_{\Sigma_{ij}} = \sum_{(i,j) \in \cyclic(u,v,w)} (H_{\Sigma_{ij},\mu} + \scalar{\nabla W,\n_{ij}}) = 0 .
 \]
 \end{proof}
 
 \begin{lemma}\label{lem:three-tensor-vanishes}
    At every point $p \in \Sigma_{uvw}$, the following $3$-tensor is identically zero:
    \[
    T^{\alpha \beta \gamma} = \sum_{(i,j) \in \cyclic(u,v,w)} 
    \brac{\n_{i j}^{\alpha} \n_{i j}^\beta \n_{\partial i j}^{\gamma} - \n_{\partial i j}^{\alpha} \n_{i j}^\beta \n_{i j}^{\gamma}} .
    \]
 \end{lemma}

\subsection{Local integrability of curvature away from $\Sigma^4$}

We denote the Hilbert-Schmidt norm of $\II^{ij}$ by $\norm{\II^{ij}}$.

\begin{proposition} \label{prop:curvature-integrability-Sigma4} 
Let $\Omega$ be a stationary regular cluster on $(M,g,\mu)$. 
For any compact set $K \subset M$ which is disjoint from $\Sigma^4$:
\begin{enumerate}[(1)]
\item \label{it:curvature-integrability-on-Sigma1} $\displaystyle \sum_{i,j} \int_{\Sigma_{ij} \cap K} \norm{\II^{ij}}^2 d\mu^{n-1} < \infty$. 
\item \label{it:curvature-integrability-on-Sigma2} $\displaystyle \sum_{i,j,k} \int_{\Sigma_{ijk} \cap K} \norm{\II^{ij}} d\mu^{n-2} < \infty$. 
\end{enumerate}
\end{proposition}
\begin{proof}
In the Euclidean setting, this is the content of \cite[Proposition 5.7]{EMilmanNeeman-GaussianMultiBubble}. As explained after its formulation there, by compactness, this is a local integrability statement, and the extension to the Riemannian setting is straightforward as explained in Remark \ref{rem:extend-to-Riemannian}. 
\end{proof}

\subsection{Cutoffs and Stokes' Theorem}

A function on $(M,g)$ which is $C^\infty$ smooth and takes values in $[0, 1]$ will be called a cutoff function. The following useful lemmas were proved in \cite[Section 6]{EMilmanNeeman-GaussianMultiBubble} for stationary regular clusters on $(\R^n,|\cdot|,\mu)$, but the proofs immediately extend to the weighted Riemannian setting: 

\begin{lemma}\label{lem:cutoff-Sigma4}
Let $\Omega$ be a stationary regular cluster on $(M,g,\mu)$. 
    For every $\epsilon > 0$, there is a compactly supported cutoff function $\eta$
    such that $\eta \equiv 0$ on a neighborhood of $\Sigma^{4}$, and:
    \[
        \int_{\Sigma^1} |\nabla \eta|^2 \, d\mu^{n-1} \le \epsilon ~,~ 
        \mu^{n-1} \{p \in \Sigma^1: \eta(p) < 1\} \le \epsilon.
    \]
\end{lemma}

\begin{lemma}[Stokes' theorem on $\Sigma_{ij} \cup \partial \Sigma_{ij}$] \label{lem:Stokes}
    Let $\Omega$ be a stationary regular cluster on $(M^n,g,\mu)$.  Suppose that $Z_{ij}$ is a vector-field which is $C^\infty$ on $\Sigma_{ij}$ and continuous
    up to $\partial \Sigma_{ij}$.
    Suppose, moreover, that
    \begin{equation} \label{eq:Stokes-assumptions}
        \int_{\Sigma_{ij}} |Z_{ij}|^2 d\mu^{n-1}, \quad
        \int_{\Sigma_{ij}} |\div_{\Sigma,\mu} Z_{ij}| d\mu^{n-1}, \quad
        \int_{\partial \Sigma_{ij}} |Z_{ij}^{\n_\partial}| d\mu^{n-2} < \infty . 
    \end{equation}
        Then:
    \[
        \int_{\Sigma_{ij}} \div_{\Sigma,\mu} Z_{ij} \, d\mu^{n-1}
        = \int_{\Sigma_{ij}} H_{ij,\mu} Z_{ij}^\n \, d\mu^{n-1} + \int_{\partial \Sigma_{ij}} Z_{ij}^{\n_{\partial}} \, d\mu^{n-2}.
    \]
\end{lemma}

The latter version of Stokes' theorem extends the classical one (\ref{eq:Stokes-classical}), when the vector-field is assumed to be compactly supported in $\Sigma_{ij} \cup \partial \Sigma_{ij}$, which for applications is typically not the case whenever $\Sigma_{ij} \cup \partial \Sigma_{ij}$ is incomplete (as in our setting).

\subsection{Second variation information -- stable clusters and index-form} \label{subsec:Q}

\begin{definition}[Stability Index-Form $Q$] 
The stability index-form $Q$ associated to a stationary $q$-cluster $\Omega$ on $(M,g,\mu)$ with Lagrange multiplier $\lambda \in E^{(q-1)}$ is defined as the following quadratic form on $C_c^\infty$ vector-fields $X$ on $(M,g)$:
\[
Q(X) := \delta^2_X A - \scalar{\lambda,\delta^2_X V} .
\]
\end{definition}

\begin{lemma}[Stability]\label{lem:unstable}
For any isoperimetric minimizing cluster $\Omega$ and $C_c^\infty$ vector-field $X$:
    \begin{equation}  \label{eq:Q}
      \delta^1_X V = 0 \;\;\; \Rightarrow \;\;\; Q(X) \ge 0 .
     \end{equation}
\end{lemma}

\begin{definition}[Stable Cluster]
A stationary cluster satisfying the conclusion of Lemma \ref{lem:unstable} is called stable. \end{definition}

The following formula for $Q(X)$ may be derived under favorable conditions. 

\begin{definition}[Weighted Ricci Curvature]
The weighted Ricci curvature on $(M,g,\mu = \exp(-W) d\vol_g)$ is defined as the following symmetric $2$-tensor:
\[
\Ric_{g,\mu} := \Ric_g + \nabla^2 W . 
\]
\end{definition}

\begin{definition}[Vector-field Index-Form $Q^1$]
For any $C_c^\infty$ vector-field $X$ for which all integrals below are finite, the (vector-field) index-form $Q^1(X)$ is defined as:
\begin{align} 
\label{eq:Q1-def}
 Q^1(X) := \sum_{i < j} \Big[ & \int_{\Sigma_{ij}} \brac{ |\nabla^\tang X^\n|^2 -(\Ric_{g,\mu}(\n,\n) + \|\II\|_2^2)  (X^\n)^2 } d\mu^{n-1}  \\
 \nonumber & - \int_{\partial \Sigma_{ij}} X^\n X^{\n_\partial} \II_{\partial\partial} \, d\mu^{n-2}\Big] .
\end{align}
\end{definition}

\begin{theorem} \label{thm:Q-Sigma4}
Let $\Omega$ be a stationary regular cluster on $(M^n,g,\mu)$. Then for any $C_c^\infty$ vector-field $X$ whose support is disjoint from $\Sigma^4$, one has $Q(X) = Q^1(X)$. In particular, all terms appearing in the definition of $Q^1(X)$ above are integrable.
\end{theorem}
\begin{proof}
In unweighted Euclidean space $(M^n,g,\mu) = (\R^n,|\cdot|,\vol)$, for $q=3$ and when $\Sigma = \Sigma^1 \cup \Sigma^2$ does not contain the quadruple set $\Sigma^3$ nor any singularities $\Sigma^4$, the formula $Q = Q^1$ was derived (in an equivalent form) in \cite[Proposition 3.3]{DoubleBubbleInR3}. For general $q$ when $\Sigma = \Sigma^1 \cup \Sigma^2 \cup \Sigma^3 \cup \Sigma^4$ may contain singularities, for weighted Euclidean space $(M^n,g,\mu) = (\R^n,|\cdot|,\mu)$ when $\Ric_{g,\mu} = \nabla^2 W$ only accounts for the curvature of the measure, the above formula for $Q$ was proved in \cite[Theorem 6.2]{EMilmanNeeman-GaussianMultiBubble}, relying on the integrability of curvature away from $\Sigma^4$ derived in Proposition \ref{prop:curvature-integrability-Sigma4}. The extra $\Ric_g(\n,\n)$ term which appears in the Riemannian setting is well-known and expected -- we verify this in Appendix \ref{app:Ricci}.
\end{proof}

\subsection{Outward Fields}

\begin{proposition}[Existence of Approximate Outward Fields] \label{prop:inward-fields}
Let $\Omega$ be a stationary regular $q$-cluster on $(M,g,\mu)$. 
For every $\epsilon_1 > 0$, there is a subset $K \subset M$, such that for every $\epsilon_2 > 0$, there is a family of vector-fields $X_1, \dots, X_q$ with the following properties:
\begin{enumerate}
        \item \label{it:inward-fields-smooth}
        $K$ is compact, disjoint from $\Sigma^4$, satisfies $\mu^{n-1}(\Sigma \setminus K) \le \epsilon_1$, and each $X_k$ is $C_c^\infty$ and supported inside $K$;
            
        \item \label{it:inward-fields-gradient}
         for every $k$,   $\displaystyle
                \int_{\Sigma^1} |\nabla^\tang X_k^\n|^2 \, d\mu^{n-1} \le \epsilon_1;
            $
        \item \label{it:inward-fields-on-Sigma1}            
        $\sum_{i < j} \mu^{n-1}\{p \in \Sigma_{ij} \; ; \;\exists k \in \{ 1,\ldots,q \} \;\;\; X_k^{\n_{ij}}(p) \ne \delta^k_{ij} \} \le \epsilon_1 + \epsilon_2;$ 
       
        \item \label{it:inward-fields-pointwise}
          for every $i < j$, for every $p \in \Sigma_{ij}$, there is
            some $\alpha \in [0, 1]$ such that for every $k$, 
            $|X_k^{\n_{ij}}(p) - \alpha \delta^k_{ij}| \le \alpha \epsilon_2;$
            
        \item \label{it:inward-fields-bounded}
        for every $k$, and at every point in $M$, $|X_k| \le \sqrt {3/2}$.             
    \end{enumerate}
\end{proposition}

\begin{definition}[Approximate Outward Fields]
A family $X_1,\ldots,X_q$ of vector-fields satisfying properties (\ref{it:inward-fields-smooth})-(\ref{it:inward-fields-bounded}) above is called a family of $(\eps_1,\eps_2)$-approximate outward fields. 
\end{definition}

\begin{proof}[Proof of Proposition \ref{prop:inward-fields}]
Proposition \ref{prop:inward-fields} was proved in \cite[Section 7]{EMilmanNeeman-GaussianMultiBubble} for stationary regular clusters in Euclidean space $(\R^n,|\cdot|,\mu)$. The sign convention employed there was opposite to the one we use in this work, and so these fields were called ``inward fields" in \cite{EMilmanNeeman-GaussianMultiBubble}. 
Since the construction of the approximate inward fields in \cite{EMilmanNeeman-GaussianMultiBubble} is entirely local, by using a partition of unity and working in charts on the smooth manifold $M$, the exact same proof carried over to the Riemannian setting. 
\end{proof}

As useful corollaries, one easily obtains the following (see \cite[Lemmas 7.8, 7.9]{EMilmanNeeman-GaussianMultiBubble}). Recall the notation $L_1$ from Definition \ref{def:LA}.

\begin{lemma} \label{lem:inward-combination} 
Let $\Omega$ be a stationary regular cluster on $(M,g,\mu)$, and let $(X_1,\ldots,X_q)$ be a collection of $(\eps_1,\eps_2)$-approximate outward fields with $\eps_2 < 1$. Given $a \in \R^q$, set $X_a := \sum_{k=1}^q a_k X_k$. Then: 
\begin{enumerate}[(i)]
\item \label{it:inward-nabla} $\int_{\Sigma^1}  |\nabla^{\tang} X^{\n}_a|^2 d\mu^{n-1} \leq q |a|^2 \eps_1$ . 
\item \label{it:inward-L2} $\int_{\Sigma^1} (X^{\n}_a)^2 d\mu^{n-1} \geq a^T L_1 a -  2 |a|^2  (\eps_1 + \eps_2)$. 
\item \label{it:inward-delta-V} For all $i$, $\abs{(\delta^1_{X_a} V(\Omega) - L_1 a)_i} \leq \sqrt{q} |a| (\eps_1+\eps_2) $. 
\end{enumerate}
In addition, we have the following simple estimates:
\begin{enumerate}[(i)]
\setcounter{enumi}{3}
\item\label{it:inward-max1} $\sup_{p \in \Sigma_{ij}} | X_a^{\n_{ij}} - a_{ij} |\leq \sum_{k=1}^q |a_k| \leq \sqrt{q} |a|$ for all $i,j$. 
\item\label{it:inward-max2} $|X_a| \leq \sqrt{3/2}  \sum_{k=1}^q |a_k| \leq \sqrt{3q/2} |a|$. 
\end{enumerate}
\end{lemma}

\begin{lemma} \label{lem:inward-trace} 
Under the same assumptions and notation as above, the following estimate holds at $p$ for all distinct $i,j,k$ and $p \in \Sigma_{ijk}$: 
\[
\abs{\tr(a \mapsto X^{\n_{ij}}_a X^{\n_{\partial  ij}}_a)} \leq C_q \eps_2  ,
\]
where $C_q > 0$ is a constant depending solely on $q$. 
\end{lemma}

\section{Clusters with locally bounded curvature} \label{sec:bounded-curvature}

While Proposition \ref{prop:curvature-integrability-Sigma4} ensures the integrability of curvature away from the singular set $\Sigma^4$, and thus is of vital importance for ``kickstarting" any rigorous stability testing, we will also want to test fields whose compact support may not be disjoint from $\Sigma^4$. While we do not know how to justify such an attempt on a general stable regular cluster, we will establish later on in Section \ref{sec:spherical} that our minimizing cluster is spherical, and in particular, of bounded curvature. This will alleviate any concerns regarding potential curvature blow-up or non-integrability on compact sets, and thus enable a much more refined analysis. In anticipation of this, we point out in this section the additional crucial information provided once it is known that the curvature is (locally) bounded. 

\begin{definition}[Cluster with locally bounded curvature]
We say that a regular cluster $\Omega$ has locally bounded curvature if $\II_{\Sigma^1}$ is bounded on every compact set. 
\end{definition}

\begin{lemma} \label{lem:Sigma2}
Let $\Omega$ be a regular cluster on $(M,g,\mu)$ with locally bounded curvature and satisfying (\ref{eq:sum-n-zero}). Then for any compact subset $K \subset M$:
\begin{enumerate}
\item There exist $\Lambda_{K},r_{K} \in (0,\infty)$ so that:
\begin{equation} \label{eq:Sigma2-regularity}
\mu^{n-2}(\Sigma^2 \cap B(x,r)) \leq \Lambda_{K} r^{n-2} \;\;\; \forall x \in \Sigma \cap K \;\;\; \forall r \in (0,r_{K}) . 
\end{equation}
\item In particular:
\begin{equation} \label{eq:Sigma2-finite}
 \mu^{n-2}(\Sigma^2 \cap K) < \infty .
 \end{equation}
\item As a consequence, curvature is locally integrable:
\begin{equation} \label{eq:curvature-integrability}
\int_{\Sigma^1 \cap K} \norm{\II}^2 d\mu^{n-1} < \infty \text{ and } \int_{\Sigma^2 \cap K} \norm{\II}  d\mu^{n-2} < \infty .
\end{equation}
\end{enumerate}
\end{lemma}
The analogous estimate to (\ref{eq:Sigma2-regularity}) for $\Sigma^1$ holds on any regular cluster by definition, and hence on any isoperimetric minimizing cluster by Theorem \ref{thm:Almgren}; however, to establish (\ref{eq:Sigma2-regularity}) on $\Sigma^2$, it seems that additional information on the curvature is required, such as local boundedness.
\begin{proof}
Let $K_1$ denote the distance-one closed neighborhood of $K$. 
 By compactness, we may cover $K_1$ with a finite number of smooth local charts $\{ (U_j , \varphi_j)\}$ for $M$. Compactness and regularity imply by property \ref{it:regularity-Sigma2} of Theorem \ref{thm:regularity} 
   that 
  $\Sigma^2 \cap K_1$ is a union of $N_{K_1}$ embedded $(n-2)$-dimensional $C^\infty$ manifolds $\{M_i\}_{i=1,\ldots,N_{K_1}}$. As the curvature of each of these manifolds is bounded above\footnote{A normal to $\Sigma_{ijk}$ can be expressed as an affine combination of the normals $\n_{ij}$, $\n_{jk}$ and $\n_{ki}$ with bounded coefficients thanks to (\ref{eq:sum-n-zero}), and so the curvature is a bounded affine combination of the curvatures of $\Sigma_{ij}$, $\Sigma_{jk}$ and $\Sigma_{ki}$, and hence bounded on a compact set.}, it follows that if we set $r_{K} \in (0,1)$ to be small enough, then for all $x \in \Sigma \cap K$, $G_{i,j,x} := \varphi_j(M_i \cap B(x,r_{K}) \cap U_j)$ is a graph over any of its tangent planes $T_y G_{i,j,x}$ with Lipschitz constant uniformly bounded in $x \in \Sigma \cap K$, $i$, $j$ and $y \in G_{i,j,x}$  (see e.g. \cite[p. 166]{Tinaglia-CurvatureEstimatesForCMCSurfaces}).
Consequently:
\[
\H^{n-2}(M_i \cap B(x,r)) \leq \Lambda'_{K_1} r^{n-2} \;\;\; \forall x \in \Sigma \cap K \;\;\; \forall r \in (0,r_{K_0}) ,
\]
for some $\Lambda'_{K_1} < \infty$. Setting $\Lambda_{K} = N_{K_1} \norm{\exp(-W)}_{L^\infty(K_1)} \Lambda'_{K_1}$, (\ref{eq:Sigma2-regularity}) follows. 

Clearly, (\ref{eq:Sigma2-finite}) then follows by compactness, which thanks to the  boundedness of the curvature on $\Sigma^1 \cap K$ (and hence, by continuity, on $\Sigma^2 \cap K$), immediately implies (\ref{eq:curvature-integrability}). 
\end{proof}

This lemma will be crucially used throughout this work. For starters, we have:

\begin{lemma} \label{lem:equivalent-stationary}
Let $\Omega$ be a regular $q$-cluster with locally bounded curvature on $(M^n,g,\mu)$. Then $\Omega$ is stationary if and only if the following two properties hold:
\begin{enumerate}[(1)]
\item \label{it:stationarity-k}
For all $i < j$, $H_{\Sigma_{ij},\mu}$ is constant and equal to $\lambda_i - \lambda_j$, for some $\lambda \in E^{(q-1)}$.
\item \label{it:stationarity-n} 
$\sum_{(i,j) \in \cyclic(u,v,w)} \n_{ij} = 0$ for all $p \in \Sigma_{uvw}$ and $u < v < w$. 
\end{enumerate}
\end{lemma}
\begin{proof}
The claim boils down to justifying the equivalence between property \ref{it:stationarity-n} and (\ref{eq:no-boundary}) for an arbitrary $C_c^\infty$ vector-field $X$. This would be immediate by Stokes' theorem (Lemma \ref{lem:Stokes}) if one could justify the integrability assumptions (\ref{eq:Stokes-assumptions}) for the vector-field $Z_{ij} = X^\tang$ on $\Sigma_{ij}$. Note that $\div_{\Sigma,\mu} X^\tang = \div_{\mu} X - \scalar{\n, \nabla_{\n} X} - X^\n H_{\Sigma,\mu}$, and so thanks to property \ref{it:stationarity-k}, the latter expression is bounded and thus $\mu$-integrable on $\Sigma_{ij}$. It remains to verify the last requirement from (\ref{eq:Stokes-assumptions}), namely that $\int_{\partial \Sigma_{ij}} |X^{\n_\partial}| d\mu^{n-2} < \infty$, which is a consequence of the locally bounded curvature and (\ref{eq:Sigma2-finite}). 
\end{proof}

More importantly, we have:
\begin{proposition} \label{prop:Q-bounded-curvature}
Let $\Omega$ be a stationary regular cluster on $(M,g,\mu)$ with locally bounded curvature. Then for any $C_c^\infty$ vector-field $X$, the formula $Q(X) = Q^1(X)$ from Theorem \ref{thm:Q-Sigma4} remains valid, and in particular, all terms appearing in the definition of $Q^1(X)$ are integrable. 
\end{proposition}
\begin{proof}
Inspecting the proof of Theorem \ref{thm:Q-Sigma4} (which builds upon \cite[Theorem 6.2]{EMilmanNeeman-GaussianMultiBubble}), the assumption that the support $K$ of $X$ is disjoint from $\Sigma^4$ was only used to ensure by Corollary \ref{cor:locally-finite-Sigma2} and Proposition \ref{prop:curvature-integrability-Sigma4} that (\ref{eq:Sigma2-finite}) and (\ref{eq:curvature-integrability}) hold. Thanks to Lemma \ref{lem:Sigma2}, we know that this in fact holds whenever the curvature is locally bounded, without assuming that $K$ is disjoint from $\Sigma^4$, and so the proof and conclusions of \cite[Theorem 6.2]{EMilmanNeeman-GaussianMultiBubble} and Theorem \ref{thm:Q-Sigma4} apply. 
\end{proof}

In addition, we will require:

\begin{proposition} \label{prop:outward-fields-Sigma2}
Let $\Omega$ be a stationary regular cluster on $(M,g,\mu)$ with locally bounded curvature. Then for any compact set $K \subset M$, the $(\eps_1,\eps_2)$ outward fields of Proposition \ref{prop:inward-fields} may be constructed so that in addition to properties (\ref{it:inward-fields-smooth})-(\ref{it:inward-fields-bounded}), the following property holds:
\begin{enumerate}
\setcounter{enumi}{5}
 \item \label{it:inward-fields-on-Sigma2}            
        $\sum_{u< v<w} \mu^{n-2}\{p \in K \cap \Sigma_{uvw} \; ; \; \exists k \in \{ 1,\ldots,q \} \;\; \exists \{i,j\} \subset \{u,v,w\} \;\; X_k^{\n_{ij}}(p) \ne \delta^k_{ij} \} \le \epsilon_1 + \epsilon_2$. 
\end{enumerate}
\end{proposition}
\begin{proof}
Property (\ref{it:inward-fields-on-Sigma2}) follows in exactly the same manner as property (\ref{it:inward-fields-on-Sigma1}) by inspecting the proof of \cite[Proposition 7.1]{EMilmanNeeman-GaussianMultiBubble}. Indeed, employing the notation used there (with our outward sign convention), recall that $\eta$ and $\xi$ denote two compactly-supported cutoff functions which truncate away neighborhoods of $\Sigma^4$ and $\Sigma^3$, respectively.  Since $X_k^{\n_{ij}} = Z_k^{\n_{ij}} = \delta^k_{ij}$ on $\Sigma_{ij}$ and $\Sigma_{ij\ell}$ for all $k = 1,\ldots,q$ whenever $\eta = \xi = 1$, it follows that:
\[
    \bigcup_{u<v<w} \{p \in \Sigma_{uvw} \; ; \; \exists k \;\; \exists \{i,j\} \subset \{u,v,w\} \;\; X_k^{\n_{ij}} \ne \delta^k_{ij} \} \subset \{p \in \Sigma^2 \; ; \; \eta < 1 \text{ or } \xi < 1  \}.
\]
Hence, by the union-bound, it is enough to make sure that $\eta$ and $\xi$ can be constructed to additionally ensure that $\mu^{n-2}(K  \cap \Sigma^2 \cap \{\eta < 1\}) \leq \eps_1$ and $\mu^{n-2}(K \cap \Sigma^2 \cap \{\xi < 1\}) \leq \eps_2$ for any given compact $K$. Inspecting the corresponding constructions in the proofs of \cite[Lemmas 6.5 and 6.6]{EMilmanNeeman-GaussianMultiBubble}, what is required is precisely the estimate (\ref{eq:Sigma2-regularity}) of Lemma \ref{lem:Sigma2}. This concludes the proof. 
\end{proof}

Additional applications of Lemma \ref{lem:Sigma2} will appear in subsequent sections.

\section{Non-physical scalar-fields vs. physical vector-fields} \label{sec:non-physical}

\begin{definition}[Scalar-Field]
Given a stationary regular cluster $\Omega$, a collection $f = \{ f_{ij} \}$ of uniformly continuous bounded functions defined on the (non-empty) interfaces $\Sigma_{ij}$ (and hence their closure) is called a scalar-field on $\Sigma$ if:
\begin{enumerate}
\item $f$ is oriented, i.e. $f_{ji} = -f_{ij}$; and
\item $f$ satisfies Dirichlet-Kirchoff boundary conditions: at every triple-point in $\Sigma_{ijk}$, we have $f_{ij} + f_{jk} + f_{ki} = 0$. 
\end{enumerate} 
We will say that $f$ has a certain property on $\Sigma$ if each $f_{ij}$ has that property on the closure of $\Sigma_{ij}$. For example, we will write $f \in C^\infty_c(\Sigma)$ or $f \in \Lip_c(\Sigma)$ if each $f_{ij}$ is $C^\infty$ smooth or Lipschitz on $\Sigma_{ij}$ and compactly supported on its closure. The support of $f$ is defined as the union of the supports of $f_{ij}$ in $\Sigma$. 
\end{definition}

\begin{definition}[Physical Scalar-Field]
A scalar-field $f = \{f_{ij} \} \in C^\infty_c(\Sigma)$ is called physical if there exists a $C^\infty_c$ vector-field $X$ on $(M,g)$ so that $f_{ij} = X^{\n_{ij}}$ on $\Sigma_{ij}$ (we will say that $f$ is derived from the physical vector-field $X$). Otherwise, it is called non-physical. 
\end{definition}

Working with scalar-fields is in practice much more convenient than with vector-fields. However, given a $C_c^\infty(\Sigma)$ scalar-field $f$ on a stationary regular cluster with $\Sigma \setminus \Sigma^1 \neq \emptyset$,  it is seldom the case that $f$ will be physical, and in general, the only way we know how to ensure this, is to start with a $C_c^\infty$ vector-field $X$ and set $f_{ij} := X^{\n_{ij}}$. This point deserves a brief discussion:
\begin{itemize}
\item Even in the single-bubble setting, one needs to approximate smooth scalar-fields by physical ones. For example,  consider the constant function $f_{12} = -f_{21} \equiv 1$ (making it compactly supported) on Simons' minimal cone $\Sigma = \{ x \in \R^8 \; ; \;  \sum_{i=1}^4 x_i^2 = \sum_{i=5}^8 x_i^2\}$, which has a singleton singularity $\Sigma \setminus \Sigma^1 = \{0\}$; it is not possible to find a smooth vector-field $X$ with $X^{\n_{ij}} = f_{ij}$. Fortunately, it is known that the singular set in the single-bubble setting on $(M^n,g,\mu)$ is of Hausdorff dimension at most $n-8$ \cite{MorganBook5Ed}, and hence it is easy to cut away from the singular set without influencing the first and second (in fact, up to sixth) variations. 
\item In the multi-bubble setting on $(M^n,g,\mu)$, Theorem \ref{thm:regularity} ensures that the singular set $\Sigma^4$ is of Hausdorff dimension at most $n-4$, and so again we may cut away from it without influencing the first and second variations. However, Theorem \ref{thm:regularity} only guarantees that $\Sigma^1$ is $C^{1,\alpha}$ smooth around quadruple points $p \in \Sigma^3$, meaning that $\n_{ij}$ is only $C^{0,\alpha}$ around those points, and so one cannot write $f_{ij} = X^{\n_{ij}}$ in the smooth or even Lipschitz class. Furthermore, this also means that curvature could be blowing up near $\Sigma^3$ which results in further complications. Unfortunately, one cannot simply cut away from $\Sigma^3$, since such a truncation would be felt by the second variation of area. 
\item It should be pointed out that on an $n$-dimensional \emph{model-space}, there is no need for any approximations at all in the single and double-bubble settings: there are enough symmetries to ensure that a minimizer has rotation symmetry, and so one can ensure that $\Sigma = \Sigma^1$ (single-bubble) and $\Sigma = \Sigma^1 \cup \Sigma^2$ (double-bubble) and there are no singularities \cite{Hutchings-StructureOfDoubleBubbles}. In that case, every smooth scalar-field is physical \cite{DoubleBubbleInR3}. 
\item However, this is no longer the case even in the triple-bubble case on a model-space when $n \geq 3$, since the rotation symmetry will only ensure that $\Sigma = \Sigma^1 \cup \Sigma^2 \cup \Sigma^3$, and the challenges expounded above of dealing with the fact that $\Sigma^1$ is only known to be $C^{1,\alpha}$ around points $p \in \Sigma^3$ apply. For general $q \geq 5$ on a model-space of dimension $n \geq 4$ (in fact, for $q \geq 3$ on a general weighted Riemannian manifold of that dimension), we have the decomposition $\Sigma = \Sigma^1 \cup \Sigma^2 \cup \Sigma^3 \cup \Sigma^4$, and nothing can be said regarding the behavior of $\Sigma$ around the singular set $\Sigma^4$. In all of these cases, an approximation by physical fields is necessary. 
\item Finally, an unexpected complication occurs when $q \geq 6$, even when the cluster is known to be completely regular, for instance when it is a spherical Voronoi cluster, and its cells meet in threes and fours as standard $\Y$ and $\T$ clusters, respectively. The reason is that Taylor's classification of minimizing cones \cite{Taylor-SoapBubbleRegularityInR3} is only available in dimensions two and three, and does not extend to dimension four and higher, where additional non-simplicial minimizing cones are known to exist \cite{Brakke-MinimalConesOnCubes}. Consequently, $m \geq 6$ cells of a minimizing spherical cluster  could potentially meet in a strange cone of affine dimension strictly smaller than $m-1$ (but necessarily strictly larger than $3$). This would incur various linear dependencies between the normals $\n_{ij}$ at the meeting point, and prevent writing $f_{ij} = X^{\n_{ij}}$ for a well-defined vector-field $X$. So an approximation by physical fields is necessary in this case as well (!). 
\end{itemize}

\medskip

The main technical results of this section are two approximation procedures of very general scalar-fields by physical ones (obtained as the normal components of a genuine $C_c^\infty$ vector-field), for the purpose of computing their index-form $Q$ and testing the cluster's stability:
\begin{itemize}
\item
On a general stable regular cluster, we can only do this in an averaged sense, after tracing out the boundary contribution -- see Subsection \ref{subsec:scalar-fields-trace}.
\item
On a stationary regular cluster with \emph{locally bounded curvature}, we obtain a very useful general approximation result in Subsection \ref{subsec:scalar-fields-bounded-curvature}, which should also prove useful for subsequent investigations.
\end{itemize}

We introduce the following definitions on a stationary regular cluster $\Omega$. 

\begin{definition}[Scalar-field first variation of volume $\delta^1_f V$]
Given a scalar-field $f = \{ f_{ij} \}$ on $\Sigma$, the first variation of volume $\delta^1_f V(\Omega) \in E^{(q-1)}$ is defined as:
\[
\delta^1_f V(\Omega)_i := \delta^1_f V(\Omega_i) = \sum_{j \neq i} \int_{\Sigma_{ij}} f_{ij} d\mu^{n-1} . 
\]
\end{definition}

\begin{definition}[Scalar-field Index-Form $Q^0$] \label{def:Q0}
For any $f = \{ f_{ij} \}$ locally Lipschitz scalar-field on $\Sigma$ with compact support for which all integrals appearing below are finite, the (scalar-field) index-form $Q^0(f)$ is defined as: 
\begin{align} 
\label{eq:Q0-def} 
Q^0(f) := \sum_{i<j} \big [ & \int_{\Sigma_{ij}} \brac{|\nabla^\tang f_{ij}|^2 - (\Ric_{g,\mu}(\n,\n) + \|\II\|_2^2) f^2_{ij}} d\mu^{n-1} \\
\nonumber & - \sum_{k \neq i,j} \int_{\Sigma_{ijk}} f_{ij} \frac{f_{ik} + f_{jk}}{\sqrt{3}} \II^{ij}_{\partial\partial} \, d\mu^{n-2} \big ] . 
\end{align}
\end{definition}

\begin{remark}
By Proposition \ref{prop:curvature-integrability-Sigma4} and Lemma \ref{lem:Sigma2}, $Q^0(f)$ is well-defined whenever the compact support of $f$ is disjoint from $\Sigma^4$, or whenever the cluster has locally bounded curvature.
\end{remark}

\begin{remark}
By Lemma \ref{lem:Lagrange}, whenever $f$ is a physical scalar-field derived from a vector-field $X$, we have:
\[
\delta^1_X V(\Omega) = \delta^1_f V(\Omega) .
\]
By Theorem \ref{thm:Q-Sigma4}, Proposition \ref{prop:Q-bounded-curvature} and (\ref{eq:sqrt3}), whenever $f$ is a physical scalar-field derived from a $C_c^\infty$ vector-field $X$ so that either the support of $X$ is disjoint from $\Sigma^4$, or the cluster has locally bounded curvature, we have:
\[
Q(X) = Q^1(X) = Q^0(f) . 
\]
\end{remark}

\subsection{Correcting variation of volume}

We will need to correct for small volume changes in our approximation procedure. 

\begin{proposition} \label{prop:volume-offset}
Let $\Omega$ denote a regular stationary cluster on $(M,g,\mu)$. 
There exist $B>0$ and a compact set $K$ in $M$ disjoint from $\Sigma^4$, so that the following holds: \\
For any $M_{1,2},M_\infty > 0$, there exists $A > 0$, so that for any vector-field $Y \in C_c^\infty(M \setminus \Sigma^4,TM)$ with $\int_{\Sigma^1 \cap K} |\nabla^{\tang} Y^{\n_{ij}}|^2 d\mu^{n-1} \leq M^2_{1,2}$ and $\norm{Y^{\n_{ij}}}_{L^\infty(\Sigma^1 \cap K)} \leq M_\infty$, and for every $v \in E^{(q-1)}$, there exists another vector-field $Z\in C_c^\infty(M \setminus \Sigma^4,TM)$  with:
\begin{itemize}
\item $\delta^1_Z V = \delta^1_Y V + v$, and
\item $Q^1(Z) \leq Q^1(Y) + A|v| + B |v|^2$. 
\end{itemize}
\end{proposition}
\begin{proof}
Let $X_1,\ldots,X_q$ be $(\eps,\eps)$ approximate outward-fields for an appropriately small $\eps > 0$. They are  supported in a compact set $K_\eps$ disjoint from $\Sigma^4$. Given $a \in \R^q$, set $X_a = \sum_{i=1}^q a_k X_k$. By Theorem \ref{thm:Q-Sigma4}, Lemma \ref{lem:inward-combination} \ref{it:inward-nabla} and \ref{it:inward-max2}, Proposition \ref{prop:curvature-integrability-Sigma4} and (\ref{eq:sqrt3}):
\begin{align*}
Q^1(X_a) &= \sum_{i < j} \Big[\int_{\Sigma_{ij}} \brac{ |\nabla^\tang X^\n_a|^2 -(\Ric_{g,\mu}(\n,\n) + \|\II\|_2^2)  (X^\n_a)^2 } d\mu^{n-1}  - \int_{\partial \Sigma_{ij}} X^\n_a X^{\n_\partial}_a \II_{\partial\partial} \, d\mu^{n-2}\Big] \\
& \leq \sum_{i < j} \Big[\int_{\Sigma_{ij}} |\nabla^\tang X^\n_a|^2 d\mu^{n-1} - \int_{\Sigma_{ij} \cap K_\eps} \Ric_{g,\mu}(\n,\n) (X^\n_a)^2  d\mu^{n-1}  + \int_{\partial \Sigma_{ij} \cap K_{\eps}} |X^\n_a| |X^{\n_\partial}_a| \norm{\II} \, d\mu^{n-2}\Big] \\
& \leq B_\eps |a|^2 ,
\end{align*}
for some $B_\eps < \infty$ depending on $\eps > 0$, as the boundary curvature and weighted Ricci curvature are integrable on $K_{\eps}$. 

Denote the linear map $E^{(q-1)} \ni a \mapsto \delta^1_{X_a} V \in E^{(q-1)}$ by $\tilde L_1 a$. 
By Lemma \ref{lem:inward-combination} \ref{it:inward-delta-V} we have:
\[
|\tilde L_1 a - L_1 a| \leq \sum_{i=1}^q |(\tilde L_1 a - L_1 a)_i| \leq 2 \sqrt{q} |a| \eps . 
\]
Since $L_1$ is positive-definite and in particular invertible on $E^{(q-1)}$ by Lemma \ref{lem:LA-positive}, it follows that for small enough $\eps > 0$, $\tilde L_1$ is also invertible. We proceed by fixing such an $\eps > 0$. 

Consequently, given $v \in E^{(q-1)}$, we define $a \in E^{(q-1)}$ to be such that $\delta^1_{X_a} V = \tilde L_1 a = v$. Given a smooth vector-field $Y$ compactly supported away from $\Sigma^4$, set $Z = Y + X_a$. Clearly $\delta^1_Z V = \delta^1_Y V + v$. Let us write:
\[
Q^1(Z) = Q^1(Y) + 2 Q^1(Y,X_a) + Q^1(X_a) ,
\]
where $Q^1(Y,X_a)$ is the polarized symmetric form (a more systematic study of which will be pursued in Section \ref{sec:Jacobi}): 
\begin{align*}
Q^1(Y,X_a) := \sum_{i < j} \Big[\int_{\Sigma_{ij}} \brac{ \scalar{\nabla^\tang Y^\n ,\nabla^\tang X^\n_a} -(\Ric_{g,\mu}(\n,\n) + \|\II\|_2^2)  Y^\n X^\n_a } d\mu^{n-1} \\
 -  \frac{1}{2} \int_{\partial \Sigma_{ij}} \brac{Y^\n X^{\n_\partial}_a + X^\n_a Y^{\n_\partial}} \II_{\partial\partial} \, d\mu^{n-2}\Big] .
\end{align*}
It remains to verify that $Q^1(Y,X_a) \leq A |a|$ for some $A$ depending on $M_{1,2},M_\infty,K_{\eps}$. Indeed, applying Cauchy-Schwarz and invoking Lemma \ref{lem:inward-combination} and Proposition \ref{prop:curvature-integrability-Sigma4}, since $X_a$ is supported on $K_\eps$ where curvature is integrable, we deduce:
\[
Q^1(Y,X_a) \leq M_{1,2} \sqrt{q} |a| \sqrt{\eps} + M_\infty \sqrt{3q /2} |a| A^{(1)}_\eps + M_\infty \sqrt{q} |a| A^{(2)}_\eps =: \frac{1}{2} A_\eps |a| . 
\]
Setting $B = B_\eps \snorm{(\tilde L_1)^{-1}}^2$ and $A = A_\eps \snorm{(\tilde L_1)^{-1}}$, the proof is complete. 
\end{proof}

\subsection{Traced Index-Form} \label{subsec:scalar-fields-trace}

It will be convenient to introduce the following boundary-less variant of the index-form for compactly supported functions $\Psi$; we shall call these functions ``non-oriented" to distinguish them from the (oriented) scalar-fields. Formally, this boundary-less variant is obtained by tracing:
\[
Q_{\tr}^0(\Psi) = \frac{1}{2} \tr\brac{a \mapsto Q^0(f^a)} ~,~ f^a = \{f^a_{ij}\} ~,~ f^a_{ij}  := a_{ij} \Psi ~,~ a_{ij} := a_i - a_j , 
\]
since it is immediate to check that the trace of the boundary's integrand vanishes pointwise: 
\begin{equation} \label{eq:trace-vanishes}
\frac{1}{2} \tr \brac{ a \mapsto a_{ij} \Psi \frac{a_{ik} \Psi + a_{jk} \Psi}{\sqrt{3}} \II^{ij}_{\partial\partial} } = 0 .
\end{equation}
Note that this is only formal, since unless the compact support of $\Psi$ is disjoint from $\Sigma^4$, or the cluster is already known to have locally bounded curvature, there is no guarantee that $\Psi^2 \II^{ij}_{\partial \partial}$ is integrable on $\partial \Sigma_{ij}$. 

\begin{definition}[Traced Index-Form] \label{def:Q0-tr} 
Let $\Omega$ be a stationary regular cluster on $(M,g,\mu)$. Given $\Psi \in C^\infty_c(M)$ so that $\int_{\Sigma^1} \norm{\II}^2 \Psi^2 d\mu^{n-1} < \infty$, define the traced index-form:
\[
Q_{\tr}^0(\Psi) := \int_{\Sigma^1} \brac{|\nabla^\tang \Psi|^2 - (\Ric_{g,\mu}(\n,\n) + \norm{\II}^2) \Psi^2} d\mu^{n-1}. 
\]
\end{definition}

We will show in Corollary \ref{cor:curvature-integrable-on-Sigma4} below that whenever the cluster is in addition stable, the assumption that $\int_{\Sigma^1} \norm{\II}^2 \Psi^2 d\mu^{n-1} < \infty$ is in fact always satisfied. 

\medskip

We can now formulate the following proposition, which allows us to rigorously deduce a boundary-less stability estimate for non-oriented test functions $\Psi$ on $\Sigma^1$. 

\begin{theorem}[Stability-in-trace] \label{thm:stability-in-trace}
Let $\Omega$ be a stable 
regular $q$-cluster on $(M,g,\mu)$. 
Let $\Psi \in C^\infty_c(M)$ with $\int_{\Sigma_{ij}} \Psi d\mu^{n-1} = 0$ for all $1 \leq i < j \leq q$. Then:
\begin{enumerate}
\item $\int_{\Sigma^1} \norm{\II}^2 \Psi^2 d\mu^{n-1} < \infty$.
\item $Q_{\tr}^0(\Psi) \geq 0$. 
\end{enumerate}
\end{theorem}

The idea is to approximate the non-physical scalar-field $a_{ij} \Psi$ by $Y^{\n_{ij}}_a$, where $Y_a := X_a \Psi$ is a physical vector-field with $X_a = \sum_{k=1}^q a_k X_k$  constructed from our approximate outward fields, and then trace over $a$ using Lemma \ref{lem:inward-trace} to get rid of the boundary's contribution. Along the approximation, we will need to modify $Y_a$ using Proposition \ref{prop:volume-offset} to make sure that $\delta^1_{Y_a} V = 0$, so that $Q(Y_a) \geq 0$ is ensured by stability. A simpler instance of this idea was employed in \cite[Lemma 8.3]{EMilmanNeeman-GaussianMultiBubble}, where there was no need to make sure that $\delta^1_{Y_a} V = 0$, and using $\Psi \equiv 1$ sufficed. Actually, in our application in this work, the cluster $\Omega$ will be symmetric and the function $\Psi$ will be odd with respect to reflection about the hyperplane of symmetry, and so it is trivial to preserve the latter oddness along the approximation by $Y_a$, thus ensuring $\delta^1_{Y_a} V = 0$ without appealing to Proposition \ref{prop:volume-offset}; however, since we will need Proposition \ref{prop:volume-offset} later on anyway, we have formulated Theorem \ref{thm:stability-in-trace} more generally. 

\begin{proof}[Proof of Theorem \ref{thm:stability-in-trace}] 
Let $X^{\eps_1,\eps_2} = X = (X_1,\ldots,X_q)$ be a collection of $(\eps_1,\eps_2)$-approximate outward fields with $\eps_i \in (0,1)$ to be determined. They are supported inside a compact set $K_{\eps_1}$ disjoint from $\Sigma^4$ with $\mu^{n-1}(\Sigma^1 \setminus K_{\eps_1}) \leq \eps_1$. 
 Given $a \in \R^q$ with $|a| \leq 1$, denote as usual the vector-field $X_a = X_a^{\eps_1,\eps_2} := \sum_{k=1}^q a_k X_k$, and set $Y_a = Y_a^{\eps_1,\eps_2} := \Psi X_a$. 
 
Denote:
\[
M^2_2 := 3 q \norm{\nabla^\tang \Psi}_{L^2(\Sigma^1,\mu^{n-1})}^2 + 2 q \norm{\Psi}^2_{L^\infty(\Sigma^1)} ~,~ M_\infty := \sqrt{3q/2} \norm{\Psi}_{L^\infty(\Sigma^1)} .
\]
Observe by Lemma \ref{lem:inward-combination} \ref{it:inward-nabla} and \ref{it:inward-max2}  that $|Y_a| \leq M_\infty$ and that:
\[
\int_{\Sigma^1} |\nabla^\tang Y_a|^2 d\mu^{n-1} \leq 2 \int_{\Sigma^1} (X_a^2 |\nabla^\tang \Psi|^2 + \Psi^2 |\nabla^\tang X_a|^2) d\mu^{n-1} \leq M_{1,2}^2 . 
\]
In particular, note that $M_{1,2},M_\infty < \infty$ are independent of $\eps_i \in (0,1)$ and $|a| \leq 1$. 

Denote $v = v^{\eps_1,\eps_2,a} = \delta_{Y_a^{\eps_1,\eps_2}} V(\Omega) \in E^{(q-1)}$, and observe that:
 \[
 v_i = \sum_{j \neq i} \int_{\Sigma_{ij}} X_a^{\n_{ij}} \Psi d\mu^{n-1} = \sum_{j \neq i} \brac{\int_{\Sigma_{ij}} (X_a^{\n_{ij}} - a_{ij}) \Psi d\mu^{n-1} + a_{ij} \int_{\Sigma_{ij}} \Psi d\mu^{n-1}} .
 \]
Since $\int_{\Sigma_{ij}} \Psi d\mu^{n-1} = 0$, it follows that:
\[
\sum_i |v_i | \leq 2 \sum_{i < j} \int_{\Sigma_{ij}} |X_a^{\n_{ij}} - a_{ij}| |\Psi| d\mu^{n-1} . 
\]
Applying Cauchy--Schwarz and recalling property (\ref{it:inward-fields-on-Sigma1}) of outward-fields and Lemma \ref{lem:inward-combination} \ref{it:inward-max1}, it follows that:
\[
|v^{\eps_1,\eps_2,a}| \leq \sum_i |v_i | \leq 2  \sqrt{\eps_1 + \eps_2} \sqrt{q} |a| \norm{\Psi}_{L^2(\Sigma^1,\mu^{n-1})} . 
\]

Applying Proposition \ref{prop:volume-offset}, it follows that there exists a $B >0$ and an $A > 0$ depending on $\Psi$ via $M_{1,2},M_\infty$, so that for all  $\eps_1,\eps_2 \in (0,1)$ and $|a| \leq 1$, there exists a vector-field $Z_a = Z_a^{\eps_1,\eps_2} \in C^\infty_c(M \setminus \Sigma^4,TM)$ so that:
\[
\delta^1_{Z_a} V = 0 ~,~ Q^1(Z_a) \leq Q^1(Y_a) + A |v| + B |v|^2 .
\]
Since $Q^1(Z_a) = Q(Z_a) \geq 0$ by stability, it follows that:
\begin{equation} \label{eq:R0}
-R^{(0)}_a(\eps_1,\eps_2) \leq Q^1(Y^{\eps_1,\eps_2}_a) ,
\end{equation}
with $R^{(0)}_a(\eps_1,\eps_2) \rightarrow 0$ as $\eps_1,\eps_2 \rightarrow 0$, uniformly in $|a| \leq 1$. 

\medskip

Let us now relate $\frac{1}{2} \tr(a \mapsto Q^1(Y^{\eps_1,\eps_2}_a))$ to $Q^0_{\tr}(\Psi)$. Denote by $G_{\eps_1,\eps_2}$ the union (over $i<j$) of the good subsets of $\Sigma_{ij}$ where $X_k^{\n_{ij}} = \delta^k_{ij}$  for all $k=1,\ldots,q$ (and hence $X^{\n_{ij}}_a = a_{ij}$), and recall that $\mu^{n-1}(\Sigma^1 \setminus G_{\eps_1,\eps_2}) \leq \eps_1 + \eps_2$. 
By Theorem \ref{thm:Q-Sigma4}:
\[
Q^1(Y_a) = \sum_{i < j} \Big[\int_{\Sigma_{ij}} \brac{ |\nabla^\tang (X^{\n}_a \Psi)|^2 -(\Ric_{g,\mu}(\n,\n) + \|\II\|_2^2)  (X^\n_a \Psi)^2 } d\mu^{n-1}  - \int_{\partial \Sigma_{ij}} X^\n_a X^{\n_\partial}_a \Psi^2 \II_{\partial\partial} \, d\mu^{n-2}\Big] .
\]

Evaluating the first term using Lemma \ref{lem:inward-combination} and Cauchy-Schwarz:
\begin{align*}
&\sum_{i<j} \int_{\Sigma_{ij}}|\nabla^\tang (X^{\n}_a \Psi)|^2 d\mu^{n-1} \\
& = \sum_{i<j} \int_{\Sigma_{ij}} \left ( |\nabla^\tang X^{\n}_a|^2 \Psi^2 + 2 X^{\n}_a \Psi \scalar{\nabla^{\tang} X^{\n}_a,\nabla^{\tang} \Psi} \right .\\
& \left . \;\;\;\;\;\;\;\;\;\;\;\;\;\;\; + ((X^{\n}_a)^2 - a_{ij}^2) |\nabla^{\tang} \Psi|^2 + a_{ij}^2 |\nabla^{\tang} \Psi|^2 \right) d\mu^{n-1} \\
& \leq  q |a|^2 \eps_1 \norm{\Psi}_\infty^2 + 2 \brac{\int_{\Sigma^1} |X^{\n}_a|^2 |\nabla^{\tang} X^{\n}_a|^2 d\mu^{n-1} \int_{\Sigma^1} \Psi^2 |\nabla^\tang \Psi|^2 d\mu^{n-1} }^{\frac{1}{2}} \\
& + q |a|^2 \int_{\Sigma^1 \setminus G_{\eps_1,\eps_2}} |\nabla^\tang \Psi|^2 d\mu^{n-1} + \sum_{i<j} \int_{\Sigma_{ij}} a_{ij}^2 |\nabla^\tang \Psi|^2 d\mu^{n-1} \\
& \leq  q |a|^2 \eps_1 \norm{\Psi}_\infty^2 + \sqrt{6} q |a|^2 \sqrt{\eps_1} \norm{\Psi}_{L^\infty(\Sigma^1)} \norm{\nabla^\tang \Psi}_{L^2(\Sigma^1,\mu^{n-1})} \\
& + q |a|^2 \int_{\Sigma^1 \setminus G_{\eps_1,\eps_2}} |\nabla^\tang \Psi|^2 d\mu^{n-1} + \sum_{i<j} \int_{\Sigma_{ij}} a_{ij}^2 |\nabla^\tang \Psi|^2 d\mu^{n-1} \\
& =: R^{(1)}_a(\eps_1,\eps_2) + \sum_{i<j} \int_{\Sigma_{ij}} a_{ij}^2 |\nabla^\tang \Psi|^2 d\mu^{n-1} ,
\end{align*}
with $R^{(1)}_a(\eps_1,\eps_2) \rightarrow 0$ as $\eps_1,\eps_2 \rightarrow 0$ uniformly in $|a| \leq 1$ (as $|\nabla^{\tang} \Psi| \in L^2(\Sigma^1, \mu^{n-1})$).

As for the boundary term in the formula for $Q^1(Y_a)$, Lemma \ref{lem:inward-trace} implies that its trace satisfies:
\[
\abs{ \frac{1}{2} \tr \brac{a \mapsto  \sum_{i<j}  \int_{\partial \Sigma_{ij}} X^\n_a X^{\n_\partial}_a \Psi^2 \II_{\partial\partial} \, d\mu^{n-2}}} \leq \frac{1}{2} C_q \eps_2 \int_{\Sigma^2 \cap K_{\eps_1}} \Psi^2 3 \norm{\II} d\mu^{n-2} =: R^{(2)}(\eps_1,\eps_2) . 
\]
Note that $R^{(2)}(\eps_1,\eps_2) < \infty$ for any fixed $\eps_1 > 0$ by Proposition \ref{prop:curvature-integrability-Sigma4}. It follows that for any fixed $\eps_1 > 0$, $\lim_{\eps_2 \rightarrow 0} R^{(2)}(\eps_1,\eps_2) = 0$. 
Finally, we bound the remaining terms as follows:
\[
\sum_{i<j} \int_{\Sigma_{ij}} (\Ric_{g,\mu}(\n,\n) + \|\II\|_2^2)  (X^\n_a \Psi)^2  d\mu^{n-1} \geq \int_{\Sigma^1 \cap G_{\eps_1,\eps_2}}  (\Ric_{g,\mu}(\n,\n) + \|\II\|_2^2) a_{ij}^2 \Psi^2 d\mu^{n-1} . 
\]

Tracing the above quadratic forms in $a$ and combining all of our estimates together with (\ref{eq:R0}), we deduce that:
\[
0 \leq \int_{\Sigma^1} |\nabla^\tang \Psi|^2 d\mu^{n-1} - \int_{\Sigma^1 \cap G_{\eps_1,\eps_2}}  (\Ric_{g,\mu}(\n,\n) + \|\II\|_2^2) \Psi^2 d\mu^{n-1} + R^{(3)}(\eps_1,\eps_2) ,
\]
where:
\[
R^{(3)}(\eps_1,\eps_2) = \frac{1}{2} \tr(a \mapsto R^{(0)}_a(\eps_1,\eps_2) + R^{(1)}_a(\eps_1,\eps_2)) + R^{(2)}(\eps_1,\eps_2) 
\]
tends to zero if we first take the limit as $\eps_2 \rightarrow 0$ and only then tend $\eps_1 \rightarrow 0$. As $\Psi \in C_c^\infty(M)$ and $\mu^{n-1}(\Sigma^1) < \infty$, both integrals $\int_{\Sigma^1} |\nabla^\tang \Psi|^2 d\mu^{n-1}$ and $\int_{\Sigma^1} \Ric_{g,\mu}(\n,\n) \Psi^2 d\mu^{n-1}$ are finite, and since $\mu^{n-1}(\Sigma^1 \setminus G_{\eps_1,\eps_2}) \rightarrow 0$ as $\eps_1,\eps_2\rightarrow 0$, we deduce that $\int_{\Sigma^1} \norm{\II}^2 \Psi^2 d\mu^{n-1}$ must be finite as well (e.g. by Fatou's lemma). It follows that $Q^0_\tr(\Psi) \geq 0$, thereby concluding the proof. 
\end{proof}

As a simple corollary of Theorem \ref{thm:stability-in-trace}, we obtain: 

\begin{corollary}[Integrability of $\norm{\II}^2$ for stable clusters] \label{cor:curvature-integrable-on-Sigma4}
Let $\Omega$ be a stable regular cluster on $(M^n,g,\mu)$. Then for any compact subset $K \subset M$:
\[
\int_{\Sigma^1 \cap K} \norm{\II}^2  d\mu^{n-1} < \infty . 
\]
\end{corollary}

\noindent
Note that by Proposition \ref{prop:curvature-integrability-Sigma4}, we already know that this holds on any stationary cluster for any compact $K$ disjoint from $\Sigma^4$,  so the added information is that the latter assumption can be removed whenever the cluster is stable. Also note that we \emph{do not} claim an analogous integrability $\int_{\Sigma^2 \cap K} \norm{\II}  d\mu^{n-2} < \infty$ on $\Sigma^2$. 

\begin{proof}
Define $\Psi$ to be identically equal to $1$ on an appropriate bounded open neighborhood $N$ of $K \cap \Sigma^4$. Now extend $\Psi$ to a $C_c^\infty(M)$ function so that $\int_{\Sigma_{ij}} \Psi d\mu^{n-1} = 0$ for all $1 \leq i < j \leq q$; as these are only a finite number of conditions, this is always possible. Theorem \ref{thm:stability-in-trace} ensures that:
\[
\int_{\Sigma^1 \cap N} \norm{\II}^2 d\mu^{n-1} \leq \int_{\Sigma^1} \norm{\II}^2 \Psi^2 d\mu^{n-1} \leq \int_{\Sigma^1} \brac{|\nabla^\tang \Psi|^2 - \Ric_{g,\mu}(\n,\n) \Psi^2} d\mu^{n-1} < \infty .
\]
On the other hand, we also have $\int_{\Sigma^1 \cap (K \setminus N)} \norm{\II}^2 d\mu^{n-1} < \infty$ by Proposition \ref{prop:curvature-integrability-Sigma4}. This concludes the proof. 
\end{proof}

\subsection{On clusters with locally bounded curvature} \label{subsec:scalar-fields-bounded-curvature}

Perhaps the most important consequence of the local boundedness of the curvature, is that it makes it possible to conveniently relate between $Q^0(f)$ for a (possibly non-physical) Lipschitz scalar-field $f$ of compact support and $Q^1(X)$ for an appropriate approximating physical vector-field $X$.

\begin{definition}[Non-oriented Lipschitz functions]
Given a closed set $K \subset M$, we denote by $\Lip_{\no}(K)$ (respectively, $\Lip_{\no,c}(K)$) the collection of all (non-oriented) functions  obtained as the restriction to $K$ of Lipschitz functions (respectively, having compact support) on $(M,g)$ . 
\end{definition}

\begin{definition}[Delta Lipschitz scalar-fields] \label{def:Delta-Lip-Fields}
We denote by $\Lip_{\Delta}(\Sigma)$ the collection of all scalar-fields $f = \{f_{ij}\}$ on $\Sigma$ generated by a family of non-oriented Lipschitz functions $\Psi_k \in \Lip_{\no,c}(\partial \Omega_k)$, $k = 1,\ldots,q$, in the sense that $f_{ij} = \Psi_i|_{\Sigma_{ij}} - \Psi_j|_{\Sigma_{ij}}$. We will write: 
\[
f_{ij} = \sum_{k=1}^q \delta^k_{ij} \Psi_k  .
\]
\end{definition}

\begin{remark}
Under mild assumptions on the cluster $\Omega$, it is not too hard to show that any scalar-field $f = \{f_{ij} \} \in \Lip_c(\Sigma)$ can be written in the above form; the corresponding $\{\Psi_k\}$ are then uniquely defined, up to replacing $\Psi_k$ by $\Psi_k + \Psi|_{\partial \Omega_k}$ for all $k$ and some global $\Psi \in \Lip_{\no,c}(\Sigma)$. 
\end{remark}

\begin{theorem}[Approximating non-physical Lipschitz scalar-fields by smooth vector-fields on clusters of locally bounded curvature] \label{thm:scalar-to-vector-Q}
Let $\Omega$ be a stationary regular cluster on $(M,g,\mu)$ with locally bounded curvature $\II_{\Sigma^1}$. For any (possibly non-physical) scalar-field $f = \{f_{ij}\} \in \Lip_{\Delta}(\Sigma)$ and $\eps > 0$, there exists a vector-field $Y_{\eps} \in C_c^\infty(M \setminus \Sigma^4 , TM)$ so that:
\begin{enumerate}
\item  $\delta^1_{Y_{\eps}} V(\Omega) = \delta^1_{f} V(\Omega)$; and
\item $Q^1(Y_{\eps}) \leq Q^0(f) + \eps$. 
\end{enumerate}
In particular, if $\Omega$ is in addition assumed to be stable, then for any $f$ as above:
\[
\delta^1_f V(\Omega) = 0 \;\; \Rightarrow \;\; 0 \leq Q^0(f) . 
\]
\end{theorem}
\begin{proof}
Let $f_{ij} = \sum_{k=1}^q \delta^k_{ij} \Psi_k$, and set $M_\infty := \max_k \norm{\Psi_k}_{L^\infty(\partial \Omega_k)}$, $M_{1,\infty} := \max_k \norm{\nabla^\tang \Psi_k}_{L^\infty(\partial \Omega_k)}$ and $M_{1,2} := \max_k \norm{\nabla^\tang \Psi_k}_{L^2(\partial \Omega_k, \mu^{n-1})}$. Let $K$ denote the union of the compact supports of $\Psi_k$ in $\Sigma$.

As a first approximation step, it is completely straightforward to reduce from the case that $\Psi_k \in \Lip_{\no,c}(\partial \Omega_k)$ to the case that $\Psi_k \in C_c^\infty(\partial \Omega_k)$ (by which we mean that $\Psi_k$ is the restriction of a function $\Phi_k \in C_c^\infty(M)$ onto $\partial \Omega_k$). Indeed, denoting by $\tilde \Phi_k \in C_c^\infty(M)$ a standard mollification of $\Phi_k$ on $(M,g)$ and setting $\tilde \Psi_k$ to be its restriction onto $\partial \Omega_k$, we may approximate $\Psi_k$ by $\tilde \Psi_k$ arbitrarily well in the norm:
\[
\norm{F}^2_k  := \sum_{j \neq k} \big [ \int_{\Sigma_{kj}} (|F|^2 + |\nabla^\tang F|^2) d\mu^{n-1} + \int_{\partial \Sigma_{kj}} |F|^2 d\mu^{n-2} \big ] ; 
\]
note that for this we appeal to Lemma \ref{lem:Sigma2}, which ensures that $\mu^{n-2}(\Sigma^2 \cap K) < \infty$ for compact sets $K$ (and of course $\mu^{n-1}(\Sigma^1) < \infty$). Denoting $\tilde f_{ij} = \sum_{k=1}^q \delta^k_{ij} \tilde \Psi_k$, we obtain an approximating scalar-field $\tilde f \in C_c^\infty(\Sigma)$ to our original $f \in \Lip_\Delta(\Sigma)$. Invoking the local boundedness of curvature again, it follows that $Q^0(\tilde f)$ and $\delta^1_{\tilde f} V$ approximate $Q^0(f)$ and $\delta^1_f V$ (respectively) arbitrarily well. We will make sure to correct for the discrepancy $|\delta^1_{\tilde f} V - \delta^1_{f} V|$ in the first variation of volume in our main approximation step below. Clearly, this approximation procedure can be performed using functions $\{\tilde \Psi_k\}$ having uniform bounds $M_\infty$, $M_{1,\infty}$ and $M_{1,2}$, as well as a common compact support $K$. Consequently, we have reduced to the case that $\Psi_k \in C_c^\infty(\partial \Omega_k)$, and so we proceed under this assumption. 

\bigskip

Let $X_1,\ldots,X_q$ be $(\eps,\eps)$ approximate outward-fields, which by Proposition \ref{prop:outward-fields-Sigma2} in addition satisfy property (\ref{it:inward-fields-on-Sigma2}). Set $Y = Y_\eps := \sum_{k=1}^q X_k \Psi_k$. We have:
\[
Q^1(Y) = Q_1(Y) - Q_2(Y) - Q_3(Y) , 
\]
where:
\begin{align*}
Q_1(Y) & := \sum_{i<j} \int_{\Sigma_{ij}} |\nabla^\tang Y^\n|^2 d\mu^{n-1} , \\
Q_2(Y) & := \sum_{i<j} \int_{\Sigma_{ij}} (\Ric_{g,\mu}(\n,\n) + \|\II\|_2^2)  (Y^{\n})^2 d\mu^{n-1} , \\
Q_3(Y) & := \sum_{i<j} \int_{\partial \Sigma_{ij}} Y^\n Y^{\n_\partial} \II_{\partial\partial} \, d\mu^{n-2}.
\end{align*}

Let $G^1_{\eps} \subset \Sigma^1$ denote the good subset from Proposition \ref{prop:inward-fields} (\ref{it:inward-fields-on-Sigma1}) where $X_k^{\n_{ij}} = \delta^k_{ij}$ on $\Sigma_{ij}$ for all $k$. Note that by Lemma \ref{lem:inward-combination} \ref{it:inward-max1} and \ref{it:inward-max2}, we have on $\Sigma_{ij}$:
\[
Y^{\n_{ij}}|_{\Sigma_{ij} \cap G^1_\eps} = f_{ij}|_{\Sigma_{ij} \cap G^1_\eps} ,
\]
and:
\begin{equation} \label{eq:Y-on-Sigma1}
 |f_{ij}| \leq 2 M_\infty ~,~ |Y^{\n_{ij}}|,|Y^{\n_{\partial ij}}| \leq \sqrt{3/2} q M_\infty ~,~ |Y^{\n_{ij}} - f_{ij}| \leq q M_\infty . \end{equation}
Hence:
\[
\abs{Q_2(Y) -\sum_{i<j} \int_{\Sigma_{ij}} (\Ric_{g,\mu}(\n,\n) + \|\II\|_2^2) f_{ij}^2 d\mu^{n-1}} \leq C_q M_\infty^2 \int_{(K \cap \Sigma^1) \setminus G^1_\eps} (\Ric_{g,\mu}(\n,\n) + \|\II\|_2^2) d\mu^{n-1} .
\]
As $\Ric_{g,\mu}(\n,\n)$ and $\norm{\II}$ are bounded on $K$ and $\mu^{n-1}(\Sigma^1 \setminus G^1_\eps) \leq 2 \eps$ by Proposition \ref{prop:inward-fields}, we conclude that the right-hand side is at most $C_2 \eps$ for some $C_2 = C_2(\Sigma,f) < \infty$. 

Similarly, let $G^2_\eps \subset \Sigma^2$ denote the good subset from Proposition \ref{prop:outward-fields-Sigma2} (\ref{it:inward-fields-on-Sigma2}) where $X_\ell^{\n_{ij}} = \delta^\ell_{ij}$ on $\Sigma_{uvw}$ for all $\ell$ and $\{i,j\} \subset \{u,v,w\}$. Then on $\Sigma_{ijk}$: 
\[
Y^{\n_{ij}}|_{\Sigma_{ijk} \cap G^2_\eps} = f_{ij}|_{\Sigma_{ijk} \cap G^2_\eps} ~,~ Y^{\n_{\partial ij}}|_{\Sigma_{ijk} \cap G^2_\eps} = \left . \frac{f_{ik} + f_{jk}}{\sqrt{3}} \right |_{\Sigma_{ijk} \cap G^2_\eps} ,
\]
and, recalling (\ref{eq:Y-on-Sigma1}):
\[
\abs{Y^{\n_{ij}} Y^{\n_{\partial ij}} - f_{ij} \frac{f_{ik} + f_{jk}}{\sqrt{3}} } \leq (3 q^2 /2 + 8/\sqrt{3}) M_\infty^2 =: C'_q M_\infty^2 . 
\]
Hence:
\[
\abs{Q_3(Y) - \sum_{i<j} \sum_{k \neq i,j} \int_{\Sigma_{ijk}} f_{ij} \frac{f_{ik}+f_{jk}}{\sqrt{3}} \II^{ij}_{\partial\partial} \, d\mu^{n-2}} \leq 3 C_q' M_\infty^2 \int_{(K \cap \Sigma^2) \setminus G^2_\eps} \norm{\II} d\mu^{n-2} .
\]
As $\norm{\II}$ is bounded on $K$ and $\mu^{n-2}((K \cap \Sigma^2) \setminus G^2_\eps) \leq 2 \eps$ by Proposition \ref{prop:outward-fields-Sigma2}, we conclude that the right-hand side is at most $C_3 \eps$ for some $C_3 = C_3(\Sigma,f) < \infty$. 

As for $Q_1(Y)$, write:
\[
\nabla^\tang Y^\n = W + U ~,~ W = \sum_k X_k^{\n} \nabla^\tang  \Psi_k ~,~ U = \sum_k \Psi_k \nabla^\tang X_k^{\n} .
\]
By Cauchy-Schwarz and Proposition \ref{prop:inward-fields} (\ref{it:inward-fields-gradient}):
\[
\int_{\Sigma^1} |U|^2 d\mu^{n-1} \leq \int \sum_k \Psi_k^2 \sum_k |\nabla^\tang X_k^{\n}|^2 d\mu^{n-1} \leq q^2 M_\infty^2 \eps .
\]
In addition, on $\Sigma_{ij}$ we have:
\begin{align*}
W|_{\Sigma_{ij} \cap G^1_\eps} & = (\nabla^\tang \Psi_i - \nabla^\tang \Psi_j) |_{\Sigma_{ij} \cap G^1_\eps} = \nabla^\tang f_{ij}|_{\Sigma_{ij} \cap G^1_\eps} ~,~\\
 \abs{|W|^2 - |\nabla^\tang f_{ij}|^2} & \leq \sum_k (X_k^{\n_{ij}})^2 \sum_k |\nabla^\tang \Psi_k|^2 + 2 (|\nabla^\tang \Psi_i|^2 + |\nabla^\tang \Psi_j|^2) \leq (3 q^2/2 + 4) M^2_{1,\infty} .
\end{align*}
Since $\mu^{n-1}(\Sigma^1 \setminus G^1_\eps) \leq 2 \eps$ by Proposition \ref{prop:inward-fields}, we deduce:
\[
\abs{\int_{\Sigma^1} |W|^2 d\mu^{n-1}  - \sum_{i<j} \int_{\Sigma_{ij}} |\nabla^\tang f_{ij}|^2 d\mu^{n-1}} \leq \int_{\Sigma^1 \setminus G^1_\eps} \abs{|W|^2 - |\nabla^\tang f_{ij}|^2}  d\mu^{n-1} \leq (3 q^2+ 8) M^2_{1,\infty} \eps .
\]
It follows that:
\begin{align}
\nonumber Q_1(Y) & = \int_{\Sigma^1} |W+U|^2 d\mu^{n-1} \leq \norm{W}_{L^2}^2 + \norm{U}_{L^2}^2 + 2 \norm{W}_{L^2} \norm{U}_{L^2} \\
\nonumber & \leq \sum_{i<j} \int_{\Sigma_{ij}} |\nabla^\tang f_{ij}|^2 d\mu^{n-1} + C''_q M_{1,\infty}^2 \eps + q^2 M_\infty^2 \eps + 2 q M_\infty \sqrt{\eps} \sqrt{M_{1,2}^2 + C''_q M_{1,\infty}^2 \eps} \\
\label{eq:nabla-Y-on-Sigma1} & \leq \sum_{i<j} \int_{\Sigma_{ij}} |\nabla^\tang f_{ij}|^2 d\mu^{n-1}  + C_1 \sqrt{\eps} ,
\end{align}
for some $C_1 = C_1(\Sigma,f) < \infty$ and small enough $\eps \in (0,\eps_0(f))$. Combining the contributions from $Q_1(Y)$, $Q_2(Y)$ and $Q_3(Y)$, we deduce that:
\[
Q^1(Y_{\eps}) \leq Q^0(f) + C_0(\Sigma,f) \sqrt{\eps} \;\;\;\;\; \forall \eps \in (0,\eps_0(f)) . 
\]

It remains to compensate for the discrepancy in the first variation of volume. By (\ref{eq:Y-on-Sigma1}), we have for all $\eps > 0$: 
\[
\norm{Y^{\n}}_{L^\infty(\Sigma^1)} \leq (q+2) M_\infty .
\]
In addition, by (\ref{eq:nabla-Y-on-Sigma1}), we have:
\[
\norm{\nabla^\tang Y^{\n}}^2_{L^2(\Sigma^1,\mu^{n-1})} = Q_1(Y) \leq M_{1,2}^2 + C_1(\Sigma,f) \sqrt{\eps} \leq 4 M_{1,2}^2 ,
\]
for all $\eps \in (0,\eps_1(\Sigma,f))$. It follows by Proposition \ref{prop:volume-offset} that there exist $A,B > 0$ independent of $\eps$ in the latter range so that for every $v \in E^{(q-1)}$, there exists a vector-field $Z_\eps \in C_c^\infty(M \setminus \Sigma^4 , TM)$ so that $\delta^1_{Z_\eps} V = \delta^1_{Y_\eps} V + v$ and $Q^1(Z_\eps) \leq Q^1(Y_\eps) + A |v| + B |v|^2$. We apply this to $v = \delta^1_f V - \delta^1_Y V$, which by (\ref{eq:Y-on-Sigma1}) satisfies:
\[
|v| \leq \sum_{i=1}^q |v_i| \leq \sum_{i=1}^q \int_{\partial \Omega_i} |Y^{\n_{ij}} - f_{ij}| d\mu^{n-1} \leq 2 q M_\infty \mu^{n-1}(\Sigma^1 \setminus G^1_\eps) \leq 4 q M_\infty \eps . 
\]
At this point, we may also add the discrepancy in the first variation of volume we incurred in our initial approximation step, which can be made smaller than $\eps$. 
It follows that for all $\eps \in (0,\eps_1(\Sigma,f))$:
\[
\delta^1_{Z_\eps} V(\Omega)= \delta^1_f V(\Omega) ~,~ Q^1(Z_\eps) \leq Q^0(f) + C_0(\Sigma,f) \sqrt{\eps} + C_A \eps + C_B \eps^2 . 
\]
Readjusting constants and relabeling $Z_\eps$ as $Y_\eps$, the assertion of the proposition is established. The ``in particular" part follows by Lemma \ref{lem:unstable} and Theorem \ref{thm:Q-Sigma4}. 
\end{proof}

\section{Jacobi Operator and Symmetric Bilinear Index-Form} \label{sec:Jacobi}

Throughout this section, we assume that $\Omega$ is a stationary regular cluster on $(M^n,g,\mu)$. 

\subsection{Symmetric bilinear index-forms}

We will need to use polarized versions of the index-forms $Q^1$ and $Q^0$. It will be convenient at this point to rewrite the boundary integrand in an equivalent form, which highlights the symmetry of the polarized forms. To this end, we denote at a point $p \in \Sigma_{ijk}$:
\begin{equation} \label{eq:def-II-partial}
\bar \II^{\partial ij} := \frac{\II^{ik}(\n_{\partial ik}, \n_{\partial ik}) + \II^{jk}(\n_{\partial jk}, \n_{\partial jk})}{\sqrt{3}} = \frac{ \II^{ik}_{\partial \partial} + \II^{jk}_{\partial \partial}}{\sqrt{3}} . 
\end{equation}

\begin{lemma} \label{lem:cyclic-XYII}
For all vectors $X,Y$ at a point $p \in \Sigma_{uvw}$, we have:
\begin{equation} \label{eq:cyclic-XYII}
\sum_{(i,j) \in \cyclic(u,v,w)} X^{\n_{ij}} Y^{\n_{\partial ij}} \II^{ij}_{\partial\partial} = 
\sum_{(i,j) \in \cyclic(u,v,w)} X^{\n_{ij}} Y^{\n_{ij}} \bar \II^{\partial ij} . 
\end{equation} 
\end{lemma}

Here and below we apply Einstein's summation convention of summing over repeated upper and lower indices.

\begin{proof}By Lemma \ref{lem:angles-form} \ref{it:angles-form-c}, at a point $p \in \Sigma_{uvw}$, we can write $\II^{ij}_{\partial \partial} = \scalar{\Lambda,\n_{ij}}$ for some vector $\Lambda$ in the two-dimensional subspace $\sspan(\n_{uv},\n_{vw},\n_{wu}) = (T_p \Sigma_{uvw})^{\perp}$ of $T_p M$. Hence:
\[
\sum_{(i,j) \in \cyclic(u,v,w)} X^{\n_{ij}} Y^{\n_{\partial ij}} \II^{ij}_{\partial\partial} = 
\sum_{(i,j) \in \cyclic(u,v,w)} X_\alpha Y_\beta \Lambda_\gamma \n_{ij}^\alpha \n_{\partial ij}^\beta \n_{ij}^\gamma . 
\]
Applying Lemma \ref{lem:three-tensor-vanishes}, we can exchange $\n_{\partial ij}^\beta$ with $\n^\gamma_{ij}$ in the above cyclic sum. It remains to note by (\ref{eq:sqrt3}) that:
\[
 \Lambda^{\n_{\partial ij}} = \frac{\Lambda^{\n_{ik}} + \Lambda^{\n_{jk}}}{\sqrt{3}} = \frac{\II^{ik}_{\partial \partial} + \II^{jk}_{\partial \partial}}{\sqrt{3}} = \bar \II^{\partial ij} ,
\]
and (\ref{eq:cyclic-XYII}) follows. 
\end{proof}

\begin{corollary}[Polarization]
The symmetric bilinear polarization of $Q^1$ defined in (\ref{eq:Q1-def}), for $C_c^\infty$ vector-fields $X,Y$ for which all integrals appearing below are finite, is given by:
\begin{align}
\nonumber Q^1(X,Y) := \sum_{i<j} \Big [ &  \int_{\Sigma_{ij}} \brac{ \scalar{\nabla^\tang X^\n,\nabla^\tang Y^\n} -(\Ric_{g,\mu}(\n,\n) + \|\II\|_2^2)  X^\n Y^\n} d\mu^{n-1}  \\
\label{eq:Q1} & - \int_{\partial \Sigma_{ij}} X^\n Y^\n \bar \II^{\partial ij} \, d\mu^{n-2}\Big] .
\end{align}
Similarly, the symmetric bilinear polarization of $Q^0$ defined in (\ref{eq:Q0-def}), for locally Lipschitz scalar-fields $f = \{ f_{ij} \}, h = \{ h_{ij} \}$ on $\Sigma$ for which all integrals appearing below are finite, is given by:
\begin{align} 
\nonumber 
Q^0(f,h) := \sum_{i<j} \Big [ & \int_{\Sigma_{ij}} \brac{\scalar{\nabla^\tang f_{ij} , \nabla^\tang h_{ij}} - (\Ric_{g,\mu}(\n,\n) + \|\II\|_2^2) f_{ij} h_{ij}} d\mu^{n-1} \\
\label{eq:Q0} & - \int_{\partial \Sigma_{ij}} f_{ij} h_{ij} \bar \II^{\partial ij} \, d\mu^{n-2} \Big ] . 
\end{align}
In other words, we have $Q^1(X) = Q^1(X,X)$ and $Q^0(f) = Q^0(f,f)$. 
\end{corollary}

\begin{remark}
We will use $Q^1(X)$ given by (\ref{eq:Q1-def}) and $Q^1(X,X)$ given by (\ref{eq:Q1}) interchangeably. Similarly for $Q^0(f)$ given by (\ref{eq:Q0-def}) and $Q^0(f,f)$ given by (\ref{eq:Q0}). 
\end{remark}

\begin{remark}
The above symmetric bilinear forms were already introduced in \cite{DoubleBubbleInR3} in the case that $q=3$, $\Sigma = \Sigma^1 \cup \Sigma^2$, and there is no quadruple set $\Sigma^3$ nor singularities $\Sigma^4$. Note that $\nabla^\tang X^\n = \nabla_\alpha X^\beta \n_\beta + X^\beta \II_{\beta \alpha}$, and hence by Proposition \ref{prop:curvature-integrability-Sigma4} and Lemma \ref{lem:Sigma2}, $Q^1(X,Y)$ is well-defined whenever the compact supports of $X,Y$ are disjoint from $\Sigma^4$, or whenever the cluster is of locally bounded curvature. 
Similarly, $Q^0(f,h)$ is well-defined whenever the compact supports of $f,h$ are disjoint from $\Sigma^4$, or whenever the cluster is of locally bounded curvature. 
\end{remark}

\begin{remark}
Whenever $f,h$ are physical scalar-fields derived from vector-fields $X,Y$, respectively, we clearly have:
\[
Q^0(f,h) = Q^1(X,Y) . 
\]
\end{remark}

\subsection{Jacobi operator $L_{Jac}$}

\begin{definition}[Jacobi operator $L_{Jac}$]
Let $\Sigma$ be a smooth hypersurface on a weighted Riemannian manifold $(M,g,\mu)$ with unit-normal $\n$. The associated Jacobi operator $L_{Jac}$ acting on smooth functions $f \in C^\infty(\Sigma)$ is defined as:
\[
L_{Jac} f := \Delta_{\Sigma,\mu} f + (\Ric_{g,\mu}(\n,\n) + \|\II\|_2^2) f . 
\]
\end{definition}

\begin{theorem} \label{thm:Q1-LJac-Sigma4}
Let $X,Y$ be $C_c^\infty$ vector-fields on $M$ whose compact supports are disjoint from $\Sigma^4$. Then:
\begin{equation} \label{eq:Q1-LJac}
Q^1(X,Y) = \sum_{i<j} \Big[ - \int_{\Sigma_{ij}} X^{\n} L_{Jac} Y^{\n} d\mu^{n-1} + \int_{\partial \Sigma_{ij}} X^{\n} \brac{\nabla_{\n_{\partial ij}} Y^{\n} -  \bar \II^{\partial ij} Y^{\n}  } d\mu^{n-2} \Big] .
\end{equation}
In particular:
\begin{equation} \label{eq:Q-LJac}
Q^1(X) = \sum_{i<j} \Big[ - \int_{\Sigma_{ij}} X^{\n} L_{Jac} X^{\n} d\mu^{n-1} + \int_{\partial \Sigma_{ij}} X^{\n} \brac{\nabla_{\n_{\partial ij}} X^{\n} - \bar \II^{\partial ij}  X^{\n}  } d\mu^{n-2} \Big] . 
\end{equation}
\end{theorem}

\begin{remark}
It is well-known and classical in the unweighted setting that the first variation of mean-curvature of a hypersurface $\Sigma$ in the normal direction is given by the unweighted Jacobi operator, and this easily extends to the weighted setting:
\[
\delta^1_{\varphi \n} H_{\Sigma,\mu} = - L_{Jac} \; \varphi . 
\]
Consequently, on a hypersurface $\Sigma$ with constant weighted mean-curvature $H_{\Sigma,\mu}$, we have:
\begin{equation} \label{eq:LJac-deltaH}
\delta^1_X H_{\Sigma_{ij},\mu} = -L_{Jac} X^\n + \delta^1_{X^\tang} H_{\Sigma_{ij},\mu} =  -L_{Jac} X^\n . 
\end{equation}
Using this, it is easy to obtain a heuristic derivation of (\ref{eq:Q-LJac}) by differentiating inside the integral in the formula for the first variation of area (\ref{eq:1st-var-area}). However, as explained in \cite[Section 1]{EMilmanNeeman-GaussianMultiBubble}, in order to justify exchanging limit and integration, one would need (at the very least) knowing the a-priori integrability of the curvature term $\norm{\II}^2$ appearing in the definition of $L_{Jac}$ on $\Sigma_{ij}$. Contrary to the single-bubble and double bubble settings, this is a genuine issue for $q \geq 4$, as the curvature may be blowing up near the quadruple set $\Sigma^3$, around which $\Sigma_{ij}$ is only known to be  $C^{1,\alpha}$ smooth according to Theorem \ref{thm:regularity}. Consequently, Proposition \ref{prop:curvature-integrability-Sigma4} is indispensable in any attempt at a rigorous argument. 
\end{remark}

For the proof of Theorem \ref{thm:Q1-LJac-Sigma4}, we will require a useful technical lemma:

\begin{lemma} \label{lem:Stokes-integrability}
Let $\Omega$ be a stationary regular cluster on $(M,g,\mu)$. Let $X,Y$ be  $C^\infty$ vector-fields on $(M,g)$, and define the vector-field $Z_{ij} := X^{\n_{ij}} \nabla^\tang Y^{\n_{ij}}$ on $\Sigma_{ij}$. Then for any compact $K \subset M$ which is disjoint from $\Sigma^4$, we have:
\[
\int_{\Sigma_{ij} \cap K} |Z_{ij}|^2 d\mu^{n-1} , \quad
        \int_{\Sigma_{ij} \cap K} |\div_{\Sigma,\mu} Z_{ij}| \, d\mu^{n-1} , \quad
        \int_{\partial \Sigma_{ij} \cap K} |Z_{ij}^{\n_{\partial ij}}| \, d\mu^{n-2} < \infty . 
\]
\end{lemma}
\begin{proof}
We will use $\alpha,\beta$ to denote tensor indices on $T \Sigma_{ij}$ and $\gamma$ on $T M$.  We freely raise and lower indices using the metric as needed; in particular, $(\II Y^{\tang})^{\alpha} = \II^{\alpha}_\beta Y^\beta$ denotes the tangential vector-field obtained by applying the Weingarten map to $Y^{\tang}$. As the indices $i,j$ are fixed, we omit them here and below. 

Write:
\begin{equation} \label{eq:tang-base}
\nabla_\alpha Y^\n = \nabla_\alpha Y^\gamma \n_\gamma + \II_{\alpha \beta} Y^\beta ~,~ Z = X^\n \nabla^\tang Y^\n = X^\n (\nabla^{\tang} Y^\gamma \n_\gamma + \II Y^\tang) .  
\end{equation}
Clearly:
\begin{equation} \label{eq:tang-base2}
|\nabla^\tang Y^{\n}| \leq \norm{\nabla Y} +  \norm{\II} |Y| ,
\end{equation}
and hence:
\[
|Z|^2 \leq 2 (\norm{\nabla Y}^2 + \norm{\II}^2 |Y|^2) |X^{\n}|^2 ,
\]
establishing the first integrability requirement by Proposition \ref{prop:curvature-integrability-Sigma4} (and $\mu^{n-1}(\Sigma^1) < \infty$). In addition:
\[
Z^{\n_{\partial}} =  X^{\n} (\scalar{\nabla_{\n_{\partial}} Y , \n} + \II(\n_{\partial} , Y^\tang)) ,
\]
and so the third integrability requirement follows by Proposition \ref{prop:curvature-integrability-Sigma4} and Corollary \ref{cor:locally-finite-Sigma2}. 

The second and final integrability requirement is delicate. Recalling (\ref{eq:tang-base}), we have:
\begin{equation} \label{eq:Z}
\div_{\Sigma,\mu} Z = \div_{\Sigma,\mu}(X^\n \nabla^{\tang} Y^\gamma \n_\gamma) + \div_{\Sigma,\mu}(X^\n \II Y^\tang ). 
\end{equation}
Let us first show that $\int_{\Sigma_{ij} \cap K} |\div_{\Sigma,\mu} (X^\n \nabla^{\tang} Y^\gamma \n_\gamma)| d\mu^{n-1} < \infty$. Indeed: 
\[
\div_{\Sigma,\mu} (X^\n \nabla^{\tang} Y^\gamma \n_\gamma) = \nabla^\alpha X^\n \nabla_\alpha Y^\gamma \n_\gamma + X^\n \brac{ \Delta_{\Sigma,\mu} Y^\gamma \n_\gamma + \nabla^\alpha Y^\beta \II_{\alpha \beta} } ,
\]
where $\Delta_{\Sigma,\mu} Y$ denotes the natural extension of the weighted surface-Laplacian $\Delta_{\Sigma,\mu}$ from scalar functions to vector-fields $Y$ using covariant differentiation. 
Consequently, by (\ref{eq:tang-base2}) (applied to $X$ instead of $Y$):
\[
|\div_{\Sigma,\mu} (X^\n \nabla^{\tang} Y^\gamma \n_\gamma)| \leq 
(\norm{\nabla X} + \norm{\II} |X|) \norm{\nabla Y} + |X^{\n}| ( |\Delta_{\Sigma,\mu} Y| + \norm{\nabla Y} \norm{\II}) ,
\]
which by Proposition \ref{prop:curvature-integrability-Sigma4} (and $\mu^{n-1}(\Sigma^1) < \infty$) is in $L^1(\Sigma_{ij} \cap K , \mu^{n-1})$. 

Recalling (\ref{eq:Z}), it remains to establish that:
\begin{equation} \label{eq:tedious}
\int_{\Sigma_{ij} \cap K} |\div_{\Sigma,\mu} (X^{\n} \II Y^\tang)| d\mu^{n-1} < \infty .
\end{equation}
To this end, a calculation performed in Lemma \ref{lem:tedious} in Appendix \ref{app:tedious} verifies that:
\begin{align*}
\div_{\Sigma,\mu}(X^\n \II Y^\tang) = & \II^2_{\alpha \beta} X^{\alpha} Y^{\beta} + \II_{\alpha\beta} (\nabla^{\alpha} X^\gamma) \n_{\gamma} Y^{\beta} + X^\n \II_{\alpha\beta} \nabla^\alpha Y^\beta - 
\norm{\II}^2 X^\n Y^\n \\
& + \Ric_{\Sigma,\mu}(Y^{\tang}, X^\n \n) + X^\n \nabla_{Y^\tang} H_{\Sigma,\mu} . 
\end{align*}
Here $\II^2$ denotes the square of the Weingarten map as an operator on $T \Sigma_{ij}$, and $\Ric_{\Sigma,\mu}$ is the weighted tangential Ricci curvature defined in (\ref{eq:Ric-Sigma-weighted}), which is smooth on the entire manifold $(M,g)$. 
Since $\Sigma_{ij}$ is of constant weighted mean-curvature $H_{\Sigma,\mu}$, the last term above vanishes, and so (\ref{eq:tedious}) follows by Proposition \ref{prop:curvature-integrability-Sigma4} (and $\mu^{n-1}(\Sigma^1) < \infty$). This concludes the proof.

\end{proof}

\begin{proof}[Proof of Theorem \ref{thm:Q1-LJac-Sigma4}]
Apply Stokes' theorem (Lemma \ref{lem:Stokes}) to the tangential field $Z_{ij} = X^{\n_{ij}} \nabla^{\tang} Y^{\n_{ij}}$ on each $\Sigma_{ij} \cup \partial \Sigma_{ij}$. As $X,Y$ are supported away from $\Sigma^4$, the integrability conditions (\ref{eq:Stokes-assumptions}) hold by Lemma \ref{lem:Stokes-integrability}. 
Using that $\div_{\Sigma,\mu} Z_{ij} = \scalar{\nabla^\tang X^{\n_{ij}}, \nabla^\tang Y^{\n_{ij}}} + X^{\n_{ij}} \Delta_{\Sigma,\mu} Y^{\n_{ij}}$ and summing over $i < j$, we obtain:
\[
\sum_{i < j} \int_{\Sigma_{ij}} \brac{ \scalar{\nabla^\tang X^\n, \nabla^\tang Y^\n} + X^\n \Delta_{\Sigma,\mu} Y^{\n}} d\mu^{n-1} = \sum_{i < j} \int_{\partial \Sigma_{ij}} X^{\n_{ij}} \nabla_{\n_{\partial ij}} Y^{\n_{ij}} d\mu^{n-2} .
\]
Plugging this into the formula (\ref{eq:Q1}) for $Q^1(X,Y)$ and recalling the definition of $L_{Jac}$, the assertion follows.
\end{proof}

We will also need the following more general version of Theorem \ref{thm:Q1-LJac-Sigma4}:
\begin{theorem} \label{thm:Q1-cutoff}
Let $\Omega$ be a stationary regular cluster on $(M,g,\mu)$. Then for any $C^\infty$ vector-fields $X,Y$ and $C_c^\infty$ cutoff function $\eta : M \rightarrow [0,1]$ whose compact support is disjoint from $\Sigma^4$, we have:
 \begin{align*}
 Q^1(\eta X,\eta Y) = \sum_{i<j} \Big [ &  \int_{\Sigma_{ij}} \brac{|\nabla^\tang \eta|^2 X^\n Y^\n + \eta \scalar{\nabla^\tang \eta , Y^\n \nabla^\tang X^\n - X^\n \nabla^\tang Y^\n} - \eta^2 X^\n L_{Jac} Y^\n} d\mu^{n-1} \\
 & + \int_{\partial \Sigma_{ij}} \eta^2 X^{\n} (\nabla_{\n_{\partial ij}} Y^{\n} - \bar \II^{\partial ij} Y^{\n}) d\mu^{n-2} \Big ] . 
 \end{align*}
 In particular:
 \[
Q^1(\eta X) = \sum_{i<j} \Big [ \int_{\Sigma_{ij}} \brac{|\nabla^\tang \eta|^2 (X^\n)^2  - \eta^2 X^\n L_{Jac} X^\n} d\mu^{n-1} + \int_{\partial \Sigma_{ij}} \eta^2 X^{\n} (\nabla_{\n_{\partial ij}} X^{\n} - \bar \II^{\partial ij} X^{\n}) d\mu^{n-2} \Big ] . 
\]
\end{theorem}
\begin{proof}
As $\eta X,\eta Y$ are smooth vector-field with compact support disjoint from $\Sigma^4$, we have:
\begin{align} 
\nonumber Q^1(\eta X,\eta Y) = \sum_{i < j} \Big[ & \int_{\Sigma_{ij}} \left ( \eta^2 \scalar{\nabla^\tang X^\n , \nabla^\tang Y^\n} + \eta \scalar{\nabla^\tang \eta, X^\n \nabla^\tang Y^\n + Y^\n \nabla^\tang X^\n} + |\nabla^\tang \eta|^2 X^\n Y^\n \right . \\
\label{eq:conforma-Q-proof3} & \left .  -(\Ric_{g,\mu}(\n,\n) + \|\II\|_2^2)  \eta^2 X^\n Y^\n \right ) d\mu^{n-1}  
  - \int_{\partial \Sigma_{ij}} \eta^2 X^\n Y^{\n} \bar \II^{\partial ij} \, d\mu^{n-2}\Big]  .
\end{align}
Now consider the $C^\infty$ vector-field $Z_{ij} := \eta^2 X^{\n_{ij}} \nabla^\tang Y^{\n_{ij}}$ defined on $\Sigma_{ij}$. Let us verify the integrability conditions (\ref{eq:Stokes-assumptions}) required to apply Lemma \ref{lem:Stokes} (Stokes' theorem) to $Z_{ij}$. Since: 
\begin{equation} \label{eq:conformal-Q-proof2}
 \div_{\Sigma,\mu}(Z_{ij}) =  \eta^2 \div_{\Sigma,\mu} (X^{\n_{ij}} \nabla^\tang Y^{\n_{ij}}) + \scalar{\nabla^\tang \eta^2 , X^{\n_{ij}} \nabla^\tang Y^{\n_{ij}}}  ,
 \end{equation}
 and $\eta$ is compactly supported on $K$ away from $\Sigma^4$, Lemma \ref{lem:Stokes-integrability} reduces the verification of (\ref{eq:Stokes-assumptions}) to verifying the integrability of the last term above. This is indeed the case since by Cauchy-Schwarz, using $\mu^{n-1}(\Sigma^1) < \infty$, and another application of Lemma \ref{lem:Stokes-integrability}, we have:
 \[
 \brac{\int_{\Sigma_{ij}} \abs{\scalar{\nabla^\tang \eta^2 , X^{\n_{ij}} \nabla^\tang Y^{\n_{ij}}} } d\mu^{n-1}}^2 \leq \int_{\Sigma_{ij}} 4 \eta^2 |\nabla \eta|^2 d\mu^{n-1} \int_{\Sigma_{ij} \cap K} |X^{\n_{ij}} \nabla^\tang Y^{\n_{ij}}|^2 d\mu^{n-1} < \infty . 
 \]
 
 Recalling (\ref{eq:conformal-Q-proof2}) and using 
 \[
 \div_{\Sigma,\mu} (X^{\n} \nabla^\tang Y^{\n}) = X^{\n} \Delta_{\Sigma,\mu} Y^\n + \scalar{\nabla^\tang X^\n, \nabla^\tang Y^\n},
 \]
 Lemma \ref{lem:Stokes} yields:
 \[
 \int_{\Sigma_{ij}} \brac{\eta^2 X^\n \Delta_{\Sigma,\mu} Y^\n  + \eta^2 \scalar{\nabla^\tang X^\n,\nabla^\tang Y^\n}+ 2 \eta  \scalar{\nabla^\tang \eta, X^\n \nabla^\tang Y^\n} } d\mu^{n-1}  =  \int_{\partial \Sigma_{ij}} \eta^2 X^\n \nabla_{\n_{\partial ij}} Y^\n d\mu^{n-2} . 
 \]
 Summing over $i<j$ and plugging the result into (\ref{eq:conforma-Q-proof3}), the assertion follows. 
\end{proof}

\subsection{Additional formulas involving $L_{Jac}$}

\begin{proposition} \label{prop:Q1-LJac-bounded-curvature}
Let $\Omega$ be a stationary regular cluster on $(M,g,\mu)$ with locally bounded curvature. Then for any $C_c^\infty$ vector-fields $X,Y$, the formulas for $Q^1(X,Y)$ (and $Q^1(X)$) from Theorem \ref{thm:Q1-LJac-Sigma4} remain valid. 
\end{proposition}
\begin{proof}
Inspecting the proof of Theorem \ref{thm:Q1-LJac-Sigma4}, which indirectly builds upon the proof of \cite[Theorem 6.2]{EMilmanNeeman-GaussianMultiBubble} via Lemma \ref{lem:Stokes-integrability}, the assumption that $X$ and $Y$ are supported in $K$ which is disjoint from $\Sigma^4$ was only used to ensure by Corollary \ref{cor:locally-finite-Sigma2} and Proposition \ref{prop:curvature-integrability-Sigma4} that (\ref{eq:Sigma2-finite}) and (\ref{eq:curvature-integrability}) hold. Thanks to Lemma \ref{lem:Sigma2}, we know that this in fact holds whenever the curvature is locally bounded, without assuming that $K$ is disjoint from $\Sigma^4$, and so the proof of Theorem \ref{thm:Q1-LJac-Sigma4} applies. Surprisingly, one still needs to appeal to the technical Lemma \ref{lem:Stokes-integrability}. 
\end{proof}

\begin{proposition} \label{prop:Q0-LJac}
Let $\Omega$ be a stationary regular cluster on $(M,g,\mu)$, and let $f = \{f_{ij} \} , h = \{h_{ij}\} \in C_c^\infty(\Sigma)$ be two scalar-fields. Assume that either the supports of $f$ and $h$ are disjoint from $\Sigma^4$, or that $\Omega$ has locally bounded curvature. Then:
\begin{equation} \label{eq:Q0-LJac}
Q^0(f,h) = \sum_{i<j} \Big [ - \int_{\Sigma_{ij}} f_{ij} L_{Jac} h_{ij}  \, d\mu^{n-1} + \int_{\partial \Sigma_{ij}} f_{ij}  \brac{\nabla_{\n_{\partial ij}} h_{ij} -  \bar \II^{\partial ij} h_{ij}} d\mu^{n-2} \Big ] .
\end{equation}
\end{proposition}
\begin{proof}
The equivalent formula for $Q^0(f,h)$ above follows from Lemma \ref{lem:Stokes} (Stokes' theorem) applied to the vector-field $Z_{ij} = f_{ij} \nabla^{\tang} h_{ij}$ on $\Sigma_{ij}$. Let us verify the integrability assumptions (\ref{eq:Stokes-assumptions}) required for invoking Lemma \ref{lem:Stokes}: while $\int_{\Sigma_{ij}} |Z_{ij}|^2 d\mu^{n-1}$ and $\int_{\Sigma_{ij}} |\div_{\Sigma,\mu} Z_{ij}| d\mu^{n-1}$ are trivially finite (recall that $\mu^{n-1}(\Sigma^1) < \infty$), the finiteness of $\int_{\partial \Sigma_{ij}} |Z_{ij}^{\n_{\partial ij}}| d\mu^{n-2}$ is due to Proposition \ref{prop:curvature-integrability-Sigma4} and Lemma \ref{lem:Sigma2}, which ensure under both sets of assumptions that $\mu^{n-2}(\Sigma^2 \cap K) < \infty$, where $K$ is the union of the compact supports of $f,h$. 
\end{proof}

Recall from (\ref{eq:trace-vanishes}) that the formal trace over $a$ of the boundary integral in the original definition of $Q^0(a_{ij} \Psi)$ vanishes, and note that the same holds for the alternative representation of $Q^0(a_{ij} \Psi)$ given by Proposition \ref{prop:Q0-LJac}, since in addition:
\[
\frac{1}{2} \tr \brac{ a \mapsto  \sum_{(i,j) \in \cyclic(u,v,w)} a_{ij} \Psi \brac{ a_{ij} \nabla_{\n_{\partial ij}} \Psi } } = \Psi \sum_{(i,j) \in \cyclic(u,v,w)} \nabla_{\n_{\partial ij}} \Psi = 0 . 
\]
Accordingly, we have:
\begin{lemma} \label{lem:Q0-tr-LJac}
Let $\Omega$ be a stationary regular cluster on $(M,g,\mu)$. Given $\Psi \in C^\infty_c(M)$ so that $\int_{\Sigma^1} \norm{\II}^2 \Psi^2 d\mu^{n-1} < \infty$, recall Definition \ref{def:Q0-tr} of the traced index-form. Then:
\[
Q_{\tr}^0(\Psi) = - \int_{\Sigma^1} \Psi L_{Jac}(\Psi) d\mu^{n-1} .
\]
\end{lemma}
\begin{proof}
Define the $C^\infty_c$ vector-field $X := \Psi \nabla \Psi$ on $M$, and note that:
\[
\div_{\Sigma,\mu} X^{\tang} = |\nabla^\tang \Psi|^2 + \Psi \Delta_{\Sigma,\mu} \Psi . 
\]
Recalling (\ref{eq:no-boundary}), if follows that:
\[
\int_{\Sigma^1} |\nabla^\tang \Psi|^2  d\mu^{n-1} = - \int_{\Sigma^1} \Psi \Delta_{\Sigma,\mu} \Psi d\mu^{n-1} .
\]
Inspecting Definition \ref{def:Q0-tr}, the claim is established. 
\end{proof}

\subsection{Additional representations of boundary integrand}

It is often convenient to rewrite the boundary integrand in the formulas for $Q$ and $Q^1$ in Theorems \ref{thm:Q1-LJac-Sigma4} and \ref{thm:Q1-cutoff}. Given a (smooth) vector-field $Y$ on $(M,g)$, denote by $(\nabla Y)^{\sym}$ the symmetrized Levi-Civita covariant derivative $(\nabla Y)^{\sym}(a,b) := \frac{1}{2} \brac{\scalar{\nabla_a Y,b} + \scalar{\nabla_b Y,a}}$.

\begin{lemma} \label{lem:boundary}
Let $\Omega$ be a stationary regular cluster on $(M,g,\mu)$, and let $X,Y$ be $C^1$ vector-fields on $(M,g)$. Then at any point $p \in \Sigma_{uvw}$, we have:
\begin{align*}
& \sum_{(i,j) \in \cyclic(u,v,w)} X^{\n_{ij}} (\nabla_{\n_{\partial ij}} Y^{\n_{ij}} - \bar \II^{\partial ij} Y^{\n})  \\
= & \sum_{(i,j) \in \cyclic(u,v,w)} X^{\n_{ij}} (\nabla_{\n_{\partial ij}} Y^{\n_{ij}} - \II^{ij}_{\partial \partial} Y^{\n_{\partial ij}})  \\
= & \sum_{(i,j) \in \cyclic(u,v,w)} X^{\n_{ij}} \scalar{\nabla_{\n_{\partial ij}} Y , \n_{ij}} \\
= & \sum_{(i,j) \in \cyclic(u,v,w)} X^{\n_{ij}} \scalar{(\nabla Y)^{\sym} \n_{\partial ij} , \n_{ij}} \\
= & \sum_{(i,j) \in \cyclic(u,v,w)}  X^{\n_{\partial ij}} \scalar{\nabla_{\n_{ij}} Y , \n_{ij}} .
\end{align*}
\end{lemma}
\begin{proof}
The first identity was already noted in Lemma \ref{lem:cyclic-XYII}. 
Now note that:
\[
\nabla_{\n_{\partial}} Y^{\n} = \nabla_{\n_\partial} \scalar{Y , \n} = \scalar{\nabla_{\n_\partial} Y , \n} + \II(Y^{\tang},\n_\partial) .
\]
Denoting by $P_{T_p \Sigma_{uvw}} Y$ the orthogonal projection of $Y$ at $p$ onto $T_p \Sigma_{uvw} \subset T_p M$, it follows that:
\begin{align*}
& \sum_{(i,j) \in \cyclic(u,v,w)} X^{\n_{ij}} (\nabla_{\n_{\partial ij}} Y^{\n_{ij}} -  \II^{ij}_{\partial \partial} Y^{\n_{\partial ij}} ) \\
& = \sum_{(i,j) \in \cyclic(u,v,w)} X^{\n_{ij}} \brac{\scalar{\nabla_{\n_{\partial ij}} Y , \n_{ij}} +  \II^{ij}(P_{T_p \Sigma_{uvw}} Y,\n_{\partial ij}) } \\
& = \sum_{(i,j) \in \cyclic(u,v,w)} X^{\n_{ij}} \scalar{\nabla_{\n_{\partial ij}} Y , \n_{ij}}  . 
\end{align*}
where the last transition is due to (\ref{eq:sum-n-zero}) and the fact that $\II^{ij}(P_{T_p \Sigma_{uvw}} Y,\n_{\partial ij})$ is independent of $(i,j) \in \cyclic(u,v,w)$ by Lemma \ref{lem:angles-form} \ref{it:angles-form-a}. This yields the second asserted identity. 
The two subsequent identities follow from Lemma \ref{lem:three-tensor-vanishes}, since when tracing the three-tensor $X^\alpha \nabla^\beta Y^\gamma$ against $\n_{ij}^\alpha \n^\beta_{\partial ij} \n^\gamma_{ij}$ and cyclically summing over $(i,j) \in \cyclic(u,v,w)$, we can exchange $\n_{\partial ij}$ with any of the $\n_{ij}$'s. 
\end{proof}

\subsection{Elliptic regularity}

We will require the following useful lemma proved in \cite[Lemma 3.8]{DoubleBubbleInR3} for physical scalar-fields when $\Sigma = \Sigma^1 \cup \Sigma^2$, $q=3$ and there are no singularities; for completeness, we sketch its proof in our more general setting. 
\begin{lemma}[Elliptic regularity] \label{lem:elliptic-regularity}
Let $\Omega$ be a stable regular cluster on $(M,g,\mu)$ with locally bounded curvature, and let $f$ be a (possibly non-physical) scalar-field $f = \{f_{ij}\} \in \Lip_{\Delta}(\Sigma)$ (recall Definition \ref{def:Delta-Lip-Fields}). Assume that:
\[
\delta^1_f V(\Omega) = 0 \text{ and } Q^0(f) = 0 . 
\]
Then there exists $\lambda \in E^{(q-1)}$ so that for every $i<j$, $L_{Jac} f_{ij} \equiv \lambda_{ij}$ on $\Sigma_{ij}$. In particular, $f_{ij}$ is $C^\infty$-smooth on (the relatively open) $\Sigma_{ij}$. 
\end{lemma}
\begin{proof}[Proof Sketch]
The local boundedness of curvature ensures that we may invoke Theorem \ref{thm:scalar-to-vector-Q} on the approximation of (possibly non-physical) scalar-fields by $C^\infty_c$ vector-fields whose compact support is disjoint from $\Sigma^4$. Consequently, stability and Theorem \ref{thm:scalar-to-vector-Q} imply that for all scalar-fields $h = \{h_{ij}\} \in \Lip_{\Delta}(\Sigma)$ with $\delta^1_h(\Omega) = 0$ we must have $Q^0(f + th) \geq 0$ for all $t \in \R$. As we have equality at $t=0$ by assumption, it follows by bilinearity and symmetry of $Q^0$ that $Q^0(h,f) = 0$ for all such $h$. By using the representation (\ref{eq:Q0-LJac}) for $Q^0$ and testing zero-mean bump functions $h$ supported inside (the relatively open) $\Sigma_{ij}$, it follows that necessarily $L_{Jac} f_{ij} \equiv A^{ij}$ is constant in the distributional sense on each $\Sigma_{ij}$. As the operator $L_{Jac}$ is clearly elliptic, and as the manifold $M^n$, the hypersurface $\Sigma_{ij}$ and the positive density of $\mu$ are $C^\infty$-smooth, uniformly on compact subsets of $\Sigma_{ij}$, this already implies by elliptic interior regularity that $f_{ij}$ is  $C^\infty$-smooth on $\Sigma_{ij}$. It remains to note that necessarily $A^{ij} = \lambda_i - \lambda_j = \lambda_{ij}$ for some $\lambda \in E^{(q-1)}$ and all $\Sigma_{ij} \neq \emptyset$, as follows from Lemmas \ref{lem:LA-positive} and \ref{lem:LA} applied with $B^{ij} = \int_{\Sigma_{ij}} h_{ij} d\mu^{n-1}$, since $\sum_{i \sim j} A^{ij} B^{ij} = Q^0(h,f) = 0$ for all  $h = \{h_{ij}\} \in \Lip_{\Delta}(\Sigma)$ with $\sum_{j ; j \neq i} B^{ij} = \delta^1_h V(\Omega_i) = 0$ for all $i$. 
\end{proof}
\begin{remark}
It also follows that $f$ must satisfy the conformal boundary conditions defined in the next section, but we will not require this here. 
\end{remark}

\section{Conformal Killing Fields and Boundary Conditions} \label{sec:conformal}

 Recall that $X$ is called a Killing field if $X$ generates a one-parameter family of isometries, or equivalently, if $(\nabla X)^{\sym} = 0$. More generally, $X$ is called a conformal Killing field if $X$ generates a one-parameter family of conformal mappings $F_t$, or equivalently, if $(\nabla X)^{\sym}= f_X \Id$ for some function $f_X$ called the conformal factor of $X$; in that case, the pull-back metric via $F_t$ satisfies $F_t^*(g) = (\exp(2 t f_X) + o(t)) g$. If these properties hold on a subset $\Omega \subset M$, $X$ is called a Killing (conformal Killing) field on $\Omega$, and $X$ is said to act isometrically (conformally) on $\Omega$. 

\begin{lemma} \label{lem:conformal-boundary}
Let $\Omega$ be a stationary regular cluster on $(M,g,\mu)$. Then for any vector-fields $X,Y$ on $(M,g)$ so that $Y$ acts conformally at a point $p \in \Sigma_{uvw}$, we have at that point that:
\[
\sum_{(i,j) \in \cyclic(u,v,w)} X^{\n_{ij}} (\nabla_{\n_{\partial ij}} Y^{\n_{ij}} - \bar \II^{\partial ij} Y^{\n_{ij}})  = 0 . 
\]
In other words, $\nabla_{\n_{\partial ij}} Y^{\n_{ij}} - \bar \II^{\partial ij} Y^{\n_{ij}}$ is independent of $(i,j) \in \cyclic(u,v,w)$. 
\end{lemma}
\begin{proof}
Immediate from Lemma \ref{lem:boundary}, since $(\nabla Y)^{\sym} = f_Y \Id$ and hence $(\nabla Y)^{\sym}(\n , \n_{\partial}) = 0$ as $\n$ and $\n_{\partial}$ are perpendicular. 
\end{proof}

In view of Lemma \ref{lem:conformal-boundary},we introduce the following definition:

\begin{definition}[Conformal boundary conditions (BCs)] \label{def:conformal-BCs}
We say that a (possibly non-physical) scalar-field $f = \{f_{ij}\} \in Lip(\Sigma)$, so that $f_{ij}$ is differentiable $\mu^{n-2}$-a.e. on $\partial \Sigma_{ij}$, satisfies conformal boundary conditions (BCs) on $\Sigma^2$, if for $\mu^{n-2}$-a.e. $p \in \Sigma^2$, $\nabla_{\n_{\partial ij}} f_{ij} - \bar \II^{\partial ij} f_{ij}$ is independent of $(i,j) \in \cyclic(u,v,w)$. 
\end{definition}

It follows from Lemma \ref{lem:conformal-boundary} that the physical scalar-field obtained as the normal component of a conformal Killing field always satisfies conformal BCs. It will also be useful to introduce:

\begin{definition}[Non-oriented conformal BCs]
We say that a Lipschitz function $\Psi \in \Lip_{\no}(M^n)$, so that $\Psi$ is differentiable $\mu^{n-2}$-a.e. on $\Sigma^2$, satisfies the non-oriented conformal BCs on $\Sigma^2$, if $\nabla_{\n_{ij}} \Psi = \II^{ij}_{\partial\partial} \Psi$ for $\mu^{n-2}$-a.e. $p \in \partial \Sigma_{ij}$ and every $i,j$. 
\end{definition}

\begin{lemma} \label{lem:non-oriented-conformal-BCs}
If $\Psi \in \Lip_{\no}(M^n)$ satisfies the non-oriented conformal BCs on $\Sigma^2$, then the scalar-field $f^a = \{ f^a_{ij} \}$ given by $f^a_{ij} = a_{ij} \Psi$ satisfies conformal BCs on $\Sigma^2$ for all $a \in E^{(q-1)}$. 
\end{lemma}
\begin{proof}
By (\ref{eq:sqrt3}) and (\ref{eq:def-II-partial}), we are given that $\nabla_{\n_{\partial ij}} \Psi = \bar \II^{\partial ij} \Psi$ for $\mu^{n-2}$-a.e. $p \in \partial \Sigma_{ij}$ and every $i,j$. The assertion now follows by linearity.  
\end{proof}

\subsection{Formula for Index-Form $Q$}

\begin{theorem} \label{thm:conformal-Q}
Let $\Omega$ be a stationary regular cluster on $(M,g,\mu)$. Then for any $C_c^\infty$ vector-field $X$ on $M$ which acts conformally on $\Sigma^2$, and for any scalar-field $f = \{ f_{ij} \} \in C_c^\infty(\Sigma)$ which satisfies conformal BCs on $\Sigma^2$:
\begin{enumerate}[(i)] \item \label{it:conformal-Q-general}  
For any $C_c^\infty$ cutoff function $\eta : M \rightarrow [0,1]$ whose compact support is disjoint from $\Sigma^4$, we have:
\[
Q^1(\eta X) = \int_{\Sigma^1} \brac{|\nabla^\tang \eta|^2 (X^\n)^2  - \eta^2 X^\n L_{Jac} X^\n} d\mu^{n-1} . 
\]
\item \label{it:conformal-Q-bounded} 
If $\Omega$ has locally bounded curvature then in fact:
\[
Q^1(X) = - \int_{\Sigma^1} X^{\n} L_{Jac} X^{\n} d\mu^{n-1} ,
\]
and
\[
Q^0(f) =  - \int_{\Sigma^1} f  L_{Jac} f d\mu^{n-1} . 
\]
\end{enumerate}
In particular, all integrals above are finite. 
\end{theorem}

\begin{proof}
Part \ref{it:conformal-Q-general} follows by plugging in the pointwise vanishing of the boundary integrand asserted in Lemma \ref{lem:conformal-boundary} into the statement of Theorem \ref{thm:Q1-cutoff}. 
When the curvature of the cluster is assumed to be locally bounded, there is no need to use any cutoffs, and part \ref{it:conformal-Q-bounded} for the vector-field $X$ follows immediately from Proposition \ref{prop:Q1-LJac-bounded-curvature}, formula (\ref{eq:Q-LJac}) and Lemma \ref{lem:conformal-boundary}. For the scalar-field $f$, part \ref{it:conformal-Q-bounded} follows from Proposition \ref{prop:Q0-LJac}, Definition \ref{def:conformal-BCs}, and Dirichlet-Kirchoff requirement $f_{ij} + f_{jk} + f_{ki} = 0$ on $\Sigma_{ijk}$. 
\end{proof}

\subsection{Formula for Jacobi Operator $L_{Jac}$}

Throughout this subsection, let $\Sigma$ denote a smooth hypersurface with unit-normal $\n$ and constant mean-curvature $H_{\Sigma}$ on an \emph{unweighted} Riemannian manifold $(M^n,g,\vol_g)$. Let $L_{Jac}$ denote the associated Jacobi operator on $\Sigma$. Under these assumptions, recall  by (\ref{eq:LJac-deltaH}) that for any vector-field $X$:
\begin{equation} \label{eq:LJac-deltaH-local}
L_{Jac} X^{\n} = -\delta^1_X H_{\Sigma} . 
\end{equation}

The following is completely classical:
\begin{lemma}
For any Killing field $W$ in a neighborhood of $p \in \Sigma$ in $M^n$, we have at $p$:
\[
L_{Jac} W^\n = 0 . 
\]
\end{lemma}
\begin{proof}
Since $W$ generates a one-parameter family of isometries (which in particular preserve the mean-curvature $H_{\Sigma}$), the assertion follows immediately from (\ref{eq:LJac-deltaH-local}). 
\end{proof}

Let us extend this classical fact to conformal Killing fields:

\begin{lemma} \label{lem:LJac-conformal}
For any conformal Killing field $W$ in a neighborhood of $p \in \Sigma$ in $M^n$ with conformal factor $f_W$, we have at $p$:
\[
L_{Jac} W^\n = f_W H_{\Sigma} - (n-1) \scalar{\nabla f_W , \n} .
\]
\end{lemma}
\begin{proof}
Recall by (\ref{eq:LJac-deltaH}) that under our assumptions, for any field $X$:
\begin{equation} \label{eq:LJac-deltaH-loc}
L_{Jac} X^{\n} = -\delta^1_X H_{\Sigma} . 
\end{equation}
It is also well-known (e.g. \cite[(12)]{MondinoNguyen-ConformalInvariants}) that under a conformal change of metric $\tilde g = \exp(2\Psi) g$, the second fundamental form changes as:
\[
\II^{\tilde g}_{ij} = e^\Psi (\II^g_{ij} + \scalar{\nabla^g \Psi,\n^g} g^{\Sigma}_{ij}) ,
\]
where $g^{\Sigma}$ is the induced metric on $T \Sigma$. Hence, contracting with respect to $\tilde g^\Sigma$, the corresponding mean-curvatures are related by:
\begin{equation} \label{eq:conformal-change-of-H}
H^{\tilde g}_{\Sigma} = e^{-\Psi} ( H^g_{\Sigma} + \scalar{\nabla^g \Psi,\n^g} (n-1)) .
\end{equation}
Applying this to $\Psi = t f_W$, the conformal factor change induced by the flow along $W$ for an infinitesimal time $t$, and taking derivative at $t=0$, we obtain:
\[
\delta^1_{W} H_{\Sigma} =  -f_W H_{\Sigma} + \scalar{\nabla f_W , \n} (n-1) . 
\]
The assertion now follows by (\ref{eq:LJac-deltaH-loc}). 
\end{proof}

\subsection{M\"obius Fields and Quasi-Centers on $\R^n$ and $\S^n$} \label{subsec:dilation}

We shall treat $\R^n$ and $\S^n$ simultaneously by using the standard embedding $\S^n \subset \R^{n+1}$, and using $\scalar{\cdot,\cdot}$ to denote the standard scalar product on $\R^n$ or $\R^{n+1}$, respectively. 

\medskip

By Liouville's classical theorem \cite{Udo-MobiusDifferentialGeometry,Blair-InversionTheory}, all conformal automorphisms of $\bar \R^n$, the one-point-at-infinity compactification of $\R^n$, when $n \geq 3$, are given by M\"obius transformations, obtained by composing isometries (orthogonal linear transformations and translations) with scaling ($p \mapsto e^\lambda p$) and spherical inversion ($p \mapsto p / |p|^2$). 
Using stereographic projection, this also classifies all conformal automorphisms of $\S^n$, as well as all global conformal Killing fields on both model spaces. The conformal Killing fields which are non-Killing (i.e.~do not generate isometries) constitute an $(n+1)$-dimensional linear space -- on $\S^n$, these are precisely given by the $(n+1)$-dimensional family of ``M\"obius fields", introduced below, whereas on $\R^n$ they are spanned by the $n$-dimensional family of ``M\"obius fields" as well as by the scaling field $p \mapsto \lambda p$ (which can be thought of as a M\"obius field in the direction of infinity). 

\begin{definition}[M\"obius Field $W_\theta$]  \label{def:dilation-field}
\hfill \\
On $\R^n$, the M\"obius field $W_\theta$ in the direction of $\theta \in \R^n$ is the section of $T \R^n$ given by:
\[
W_\theta(p) := \frac{|p|^2}{2} \theta - \scalar{\theta,p} p  . 
\]
On $\S^n$, the M\"obius field $W_\theta$ in the direction of $\theta \in \R^{n+1}$ is the section of $T \S^n$ given by:
\[
W_\theta(p) := \theta - \scalar{\theta,p} p . 
\]
\end{definition}

\begin{remark} \label{rem:dilation-conformal} 
Note that:
\[
\nabla W_\theta(p) = \begin{cases} p \otimes \theta - \theta \otimes p - \scalar{p,\theta} \Id & M = \R^n \\
-\scalar{p,\theta} \Id & M = \S^n \end{cases} ,
\]
confirming that these are conformal Killing fields with conformal factor $f_{W_\theta} = -\scalar{p,\theta}$. It is easy to check the geometric significance of the field $W_\theta$: on $\S^n$, it is obtained by pulling back the scaling field $p \mapsto |\theta| p$ on $\R^n$ via the stereographic projection from $\S^n$ using $\theta / |\theta| \in \S^n$ as the North pole; on $\R^n$, it is obtained by pulling back the rotation field in the $\{e_{n+1},\theta\}$ plane on $\S^n \subset \R^n \times \R$ via the standard stereographic projection from $\R^n$. 
\end{remark}

\begin{definition}[Quasi-center $\c$]
Given an oriented smooth hypersurface $\Sigma$ on the model space $(M,g)$, $M \in \{ \R^n , \S^n \}$, we define the quasi-center vector $\c \in \{ \R^n , \R^{n+1} \}$ (respectively) at $p \in \Sigma$, as:
\[
\c := \n - \k p ,
\]
where $\k = \frac{H_{\Sigma}}{n-1}$ is the normalized mean-curvature with respect to the unit-normal $\n$ at $p$. 
\end{definition}

\begin{remark} \label{rem:quasi-center}
The geometric meaning (and hence our choice of nomenclature) of the quasi-center $\c$ on $\R^n$ is easy to see -- whenever $\Sigma$ is a subset of a sphere $S$ (of curvature $\k \neq 0$ with respect to $\n$), $-\c / \k = p - \n/\k$ is the sphere's center; if $S$ is a hyperplane (generalized sphere of curvature $\k=0$), then $\c = \n$ is the unit-normal to $S$. On $\S^n$, whenever $\Sigma$ is a subset of a geodesic sphere $S$ (of curvature $\k$ with respect to $\n$), we have $S = \S^n \cap E(\c,\k)$ where $E(\c,\k) := \{ x \in \R^{n+1} \; ; \; \scalar{\c,x} + \k = 0\}$; note that this representation of the affine hyperplane $E(\c,\k)$ is uniquely determined (up to orientation) by the property that $|\c|^2 = |\n - \k p|^2 = 1 + \k^2$, and hence the signed distance of $E(\c,\k)$ from the origin is $\k / \sqrt{1 + \k^2}$. In all of these cases, $\c$ remains constant on the (generalized) sphere $S$, and we call $\c$ the quasi-center of $S$. See Figure \ref{fig:quasi-center}. 
\end{remark}

\begin{figure} 
    \begin{center}
                \includegraphics[scale=0.3]{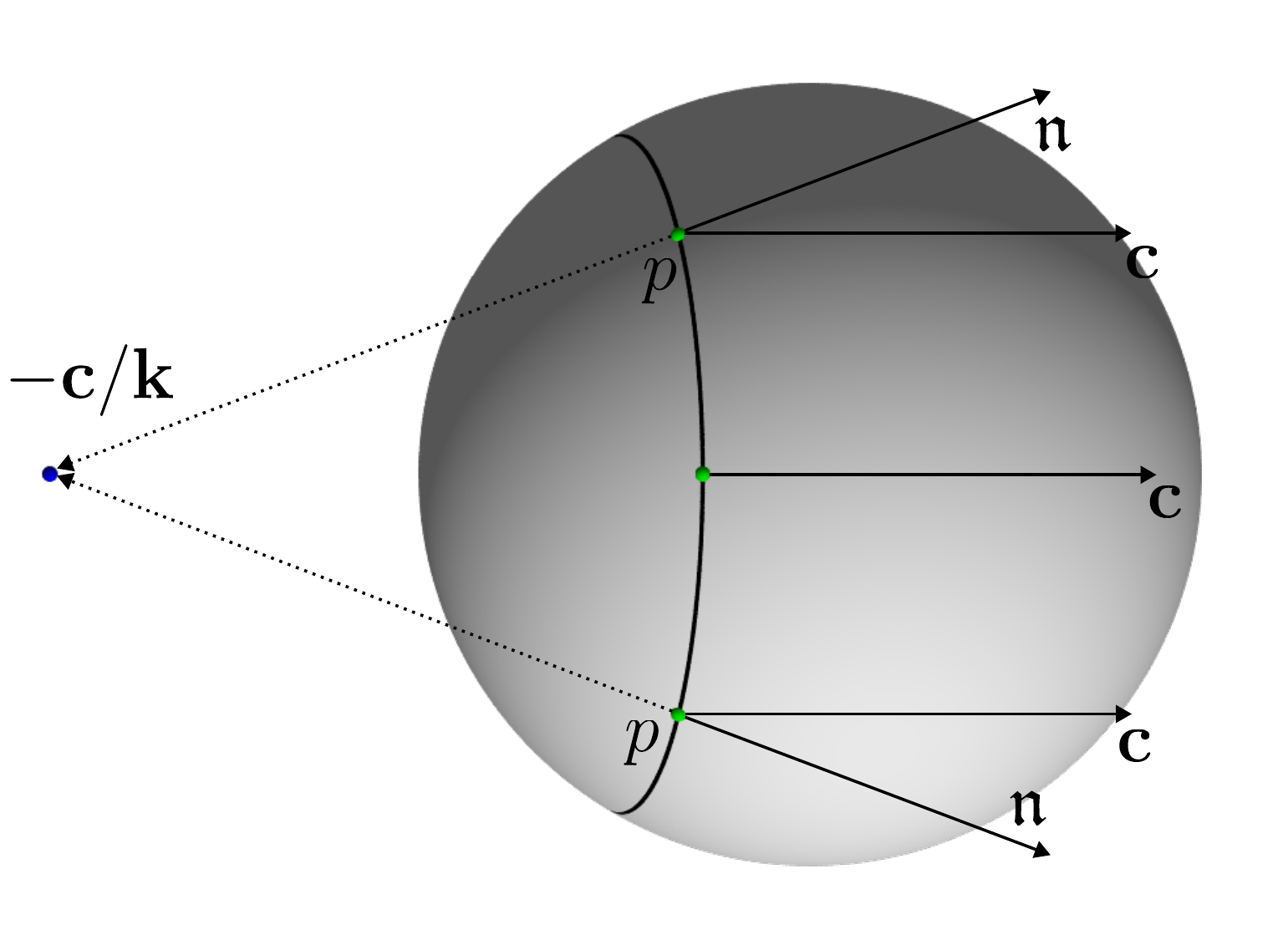}
             \end{center}
     \caption{
         \label{fig:quasi-center}
         The (constant) quasi-center of a spherical cap on $\S^n$ is $\c = \n - \k p$; the usual center is at $-\c/\k = p - \n/\k$. 
     }
\end{figure}

\begin{lemma} \label{lem:LJac-dilation}
Let $\Sigma$ denote a smooth hypersurface of constant mean-curvature on the model space $(M^n,g)$, $M^n \in \{ \R^n , \S^n \}$. Then:
\[
L_{Jac} W_\theta^\n = (n-1) \scalar{\theta,\c} . 
\]
\end{lemma}
\begin{proof}
Recall from Remark \ref{rem:dilation-conformal} that the conformal factor of $W_\theta$ is $f_{W_\theta} = -\scalar{p,\theta}$. Consequently, by Lemma \ref{lem:LJac-conformal}, we have:
\[
L_{Jac} W_\theta^\n = - H_{\Sigma} \scalar{p,\theta} + (n-1) \scalar{\theta, \n} = (n-1) \scalar{\theta , \n - \k p} . 
\]
\end{proof}

\section{Stable clusters with $\S^0$-symmetry are spherical} \label{sec:spherical}

We are finally ready to give a proof of Theorem \ref{thm:intro-stable-spherical}. By removing empty cells, we may reduce to the case that all cells are non-empty. Note that a minimizing cluster on a weighted Riemannian manifold, all of whose cells are non-empty, is regular, stationary and stable by the results of Section \ref{sec:prelim}. Moreover, on all model space-forms, any minimizing cluster is bounded by Theorem \ref{thm:existence}. 
Theorem \ref{thm:intro-stable-spherical} is therefore a consequence of the following theorem, which we state separately for convenience:

\begin{theorem} \label{thm:stable-spherical}
Let $\Omega$ be a bounded stable (and in particular, stationary) regular $q$-cluster on the model space $M^n$, $M^n \in \{ \R^n , \S^n \}$, with $\S^0$-symmetry. Then $\Omega$ is necessarily spherical perpendicularly to its hyperplane of symmetry. 
\end{theorem}

Let us make the definition of ``sphericity perpendicularly to the hyperplane of symmetry", already presented in the Introduction, more precise. 
Let $N \in M^n$ with $|N|=1$ representing a North Pole, and denote by $M^{n-1} := M^n \cap N^{\perp}$ the equator. We denote by $M^n_{\pm}$ the two open connected components of $M^n \setminus N^{\perp}$ (called hemispheres), with $M^n_+$ the Northern hemisphere containing $N$. 

\begin{definition}[Connectedness modulo $\S^0$-symmetry] \label{def:connected-S0}
A set $A \subset M^n$ which is symmetric with respect to $M^{n-1}$, is said to be ``connected modulo its $\S^0$-symmetry", if $A \cap \overline{M^n_+}$ is connected (note that $A$ itself will have two connected components if $A \cap M^{n-1} = \emptyset$). 
\end{definition}

\begin{definition}[Connected components modulo $\S^0$-symmetry] \label{def:cc-S0}
Given a set $A \subset M^n$ with $\S^0$-symmetry with respect to $M^{n-1}$, its connected components modulo its $\S^0$-symmetry, or connected (modulo $\S^0$-symmetry) components, are its maximal subsets which are symmetric with respect to $M^{n-1}$ and connected modulo $\S^0$-symmetry.  
\end{definition}

In other words, if $\{\underline A_\ell\}_{\ell}$ are the usual connected components of $A$, then its connected components modulo its $\S^0$-symmetry are $\{ \underline A_\ell \; ; \; \underline A_\ell \cap M^{n-1} \neq \emptyset \} \cup \{ \underline A_\ell \cup -\underline A_\ell \; ; \; \underline A_\ell \subset M^n_+ \}$, where $-\underline A_\ell$ denotes the reflection of $\underline A_\ell$ about $M^{n-1}$. 

\begin{definition}[Spherical cluster perpendicularly to its hyperplane of symmetry] \label{def:spherical-S0}
A regular cluster on the model space $M^n$, $M^n \in \{ \R^n , \S^n \}$, with reflection symmetry with respect to $N^{\perp}$,
is called ``spherical perpendicularly to its hyperplane of symmetry", if for all $i,j$:
\begin{enumerate}
\item $\II^{ij}_0 \equiv 0$ on all of $\Sigma_{ij}$.  
\item $\c_{ij}\equiv \c_{ij}^\ell$ and $\k_{ij} \equiv \k_{ij}^\ell$ are constant on every connected (modulo $\S^0$-symmetry) component $\Sigma^\ell_{ij}$ of $\Sigma_{ij}$, and $\sscalar{\c^\ell_{ij} , N} = 0$. 
\item Every connected (modulo $\S^0$-symmetry) component $\Sigma^\ell_{ij}$ is a subset of a (generalized) geodesic sphere with quasi-center $\c^\ell_{ij}$ and curvature $\k^\ell_{ij}$. 
\end{enumerate}
Here $\II^{ij}_0 = \II^{ij} - \k_{ij} \Id$ denotes the traceless second fundamental form of $\Sigma_{ij}$, $\k_{ij} = \frac{1}{n-1} H_{\Sigma_{ij}}$ denotes the normalized mean-curvature of $\Sigma_{ij}$ and $\c_{ij} = \n_{ij} - \k_{ij} p$ is the quasi-center vector-field. 
\end{definition}

\begin{remark} \label{rem:finite-ccs}
Note that we deliberately do not assume stationarity in the above definition. When the cluster is in addition stationary (with Lagrange multiplier $\lambda \in E^{(q-1)}$) as in Theorem \ref{thm:stable-spherical}, then necessarily $\k^\ell_{ij} =\k_{ij} = \k_i - \k_j$ for all $\ell$, with $\k := \frac{1}{n-1} \lambda$. Also note that for any bounded regular cluster, $\Sigma_{ij}$ always consists of a finite number of connected components $\{\Sigma^\ell_{ij}\}_{\ell \in \Lambda_{ij}}$ by compactness of $\Sigma$ and the local-finiteness in Theorem \ref{thm:Almgren} \ref{it:Almgren-iii} which is part of the definition of regularity. 
\end{remark}

\begin{proof}[Proof of Theorem \ref{thm:intro-stable-spherical} given Theorem \ref{thm:stable-spherical}]
It remains to establish the ``in particular" part of Theorem \ref{thm:intro-stable-spherical}. 
As explained in the Introduction, if $\Omega$ is a minimizing $q$-cluster on $M^n$ with $q \leq n+1$ (no $\S^0$-symmetry assumed), 
one can always find a hyperplane $N^{\perp}$ bisecting all finite-volume cells of $\Omega$ (after a possible translation of $\Omega$ when $M^n = \R^n$). Reflecting each of the two halves $\Omega_{\pm}$ of $\Omega$ about this hyperplane, we obtain two clusters $\bar \Omega_{\pm}$ with $\S^0$-symmetry and $V(\bar \Omega_{\pm}) = V(\Omega)$, each of which must be minimizing (otherwise a contradiction to the minimality of $\Omega$ would follow). It follows by Theorem \ref{thm:stable-spherical} that both $\bar \Omega_{\pm}$ are spherical perpendicularly to $N^{\perp}$, and so both halves $\Omega_{\pm}$ of the original cluster $\Omega$ also have this property in their respective hemispheres $M^n_{\pm}$. It remains to note that a connected component $\Sigma^\ell_{ij}$ of the interface $\Sigma_{ij}$ of the original cluster $\Omega$ is a smooth connected (relatively open) manifold by regularity, and so since each $\Sigma^\ell_{ij} \cap M^n_{\pm}$ is either empty or a subset of a (generalized) sphere $(S^\ell_{ij})_{\pm}$ which intersects $N^{\perp}$ perpendicularly (with quasi-center $(\c^{\ell}_{ij})_{\pm}$ and curvature $(\k^\ell_{ij})_{\pm}$), we must have $(S^\ell_{ij})_+ = (S^\ell_{ij})_-$ in case both parts are non-empty. Consequently, $\Sigma^\ell_{ij}$ is a subset of the (generalized) sphere $(S^\ell_{ij})_+ = (S^\ell_{ij})_-$, and so $\Omega$ itself is spherical perpendicularly to $N^{\perp}$ (yet, potentially does not have $\S^0$-symmetry -- we will see that this is not the case after establishing Theorem \ref{thm:intro-stable-spherical-Voronoi}). 
\end{proof}

This section is dedicated to proving Theorem \ref{thm:stable-spherical} via a second variation argument. 

\subsection{Stability}

Consider a smooth function $\Psi$ on $M^n$ which is odd with respect to reflection about $N^{\perp}$. On $\R^n$, since $\Omega$ is assumed bounded, we may always smoothly truncate $\Psi$ outside a large enough ball without altering its values on $\Sigma$, so as to ensure that $\Psi \in C_c^\infty(M^n)$. The cluster's symmetry and oddness of $\Psi$ with respect to reflection about $N^{\perp}$ obviously imply that $\int_{\Sigma_{ij}} \Psi dp = 0$ for all interfaces $\Sigma_{ij}$. We may therefore invoke Theorem \ref{thm:stability-in-trace}, which asserts that stability of $\Omega$ implies that $\int_{\Sigma^1} \Psi^2 \norm{\II}^2 dp < \infty$ and that the traced-index-form is non-negative. Recalling Lemma \ref{lem:Q0-tr-LJac}, this means:
\begin{equation} \label{eq:Q0-non-negative}
0 \leq Q^0_{\tr}(\Psi) = -\int_{\Sigma^1} \Psi L_{Jac} \Psi \; dp   . 
\end{equation}

\begin{remark}
It is worthwhile repeating what is encoded in Theorem \ref{thm:stability-in-trace} by means of the assertion that $Q^0_{\tr}(\Psi) \geq 0$: for all coefficients $a \in \R^q$, $f^a_{ij} := a_{ij} \Psi$ is a scalar-field on $\Sigma$ which can be approximated by a smooth vector-field $X_a \in C_c^\infty(M \setminus \Sigma^4 , TM)$ so that $X^{\n_{ij}}_a \simeq f^a_{ij}$ on each $\Sigma_{ij}$, $\delta^1_{X_a} V(\Omega) = 0$, and $\tr(a \mapsto Q^1(X_a))$ is non-negative and has vanishing boundary term in (\ref{eq:Q-LJac}). 
\end{remark}

Our goal will be to find an odd function $\Psi$ so that $Q^0_{\tr}(\Psi) \leq 0$, which together with stability and Theorem \ref{thm:stability-in-trace} would imply that $Q^0_{\tr}(\Psi) = 0$, and hence (hopefully) that the cluster is spherical. As we shall see, the sign of $Q^0_{\tr}(\Psi)$ is not easily recognizable, and moreover, will require superimposing with another field. We invest a few paragraphs to provide motivation for our choice of test fields.

\subsection{Skew-Fields}

Recall from Section \ref{sec:conformal} that one trivial way to obtain $Q(X) = 0$ is to use a Killing field $X$ (generating a one-parameter family of isometries). Of course a genuine Killing field carries no information, so our goal will be to find a function $\Psi$ so that \emph{on a model cluster with $\S^0$-symmetry}, $a_{ij} \Psi$ is the normal component of a Killing field -- specifically, a rotation field in the plane spanned by $N$ and $\theta \perp N$:
\[
R_{\theta,N} = \scalar{N,p} \theta - \scalar{\theta,p} N .
\]
Recalling the definition of the quasi-centers $\c_{ij}$, we have:
\[
R^{\n_{ij}}_{\theta,N} = \scalar{N,p} \scalar{\theta,\n_{ij}} - \scalar{\theta,p} \scalar{N,\n_{ij}} = 
\scalar{N,p} \scalar{\theta,\c_{ij}} - \scalar{\theta,p} \scalar{N,\c_{ij}} . 
\]
On a standard bubble with reflection symmetry with respect to $N^{\perp}$, the quasi-center vector $\c_{ij}$ is constant on the entire interface $\Sigma_{ij}$ and perpendicular to $N$, and so on such a standard bubble we have:
\[
R^{\n_{ij}}_{\theta,N} = a_{ij} \scalar{N,p} ~,~ a_{ij} = \scalar{\theta,\c_{ij}} . 
\]
Consequently, we recognize $\Psi := \scalar{N,\cdot}$ as a potentially useful function for stability testing. 
\begin{definition}[Skew scalar-fields]
Given a stationary regular cluster with $\S^0$-symmetry with respect to $N^{\perp}$ and $a \in E^{q-1}$, the scalar-field $f^a = \{ f^a_{ij}\}$ given by $f^a_{ij} := a_{ij} \scalar{N,p}$ is called a skew-field. 
\end{definition}

\begin{lemma} \label{lem:LJac-skew}
Let $\Sigma$ denote a smooth hypersurface on the model space $(M^n,g)$, $M^n \in \{ \R^n , \S^n \}$. Then for any direction $\theta$:
\[
L_{Jac} \scalar{\theta,p} = \norm{\II}^2 \scalar{\theta,p} - (n-1) \k \scalar{\theta,\n} = \norm{\II_0}^2 \scalar{\theta,p} - (n-1) \k \scalar{\theta,\c} .
\]
\end{lemma}
\begin{proof}
Let us perform the calculation on $\frac{1}{s} \S^n$, an $n$-sphere of curvature $s$. As:
\[
\nabla^\tang \scalar{\theta,p} = \theta - \scalar{\theta,\n} \n - s^2 \scalar{\theta,p} p ,
\]
we deduce that:
\[
\Delta_{\Sigma} \scalar{\theta,p} = -\scalar{\theta,\n} H_{\Sigma} - s^2 (n-1) \scalar{\theta,p} . 
\]
Recalling that $L_{Jac} = \Delta_{\Sigma} + \Ric(\n,\n) + \norm{\II}^2$ and that $\k = \frac{1}{n-1} H_{\Sigma}$, we deduce:
\[
L_{Jac} \scalar{\theta,p} = - (n-1)  \scalar{\theta, \n \k + s^2 p} + ((n-1) s^2 + \norm{\II}^2) \scalar{\theta,p} = \norm{\II}^2 \scalar{\theta,p} - (n-1) \k \scalar{\theta,\n} ,
\]
and we obtain the assertion on $\frac{1}{s} \S^n$ with no dependence on $s$. Applying this to $s=1$ and $s \rightarrow 0$, the cases of $\S^n$ and $\R^n$ follow.  
\end{proof}

Applying (\ref{eq:Q0-non-negative}) to the odd function $\Psi = \scalar{N,\cdot}$ and using Lemma \ref{lem:LJac-skew}, we deduce:
\begin{equation} \label{eq:skew-Q0}
0 \leq - \sum_{i<j} \int_{\Sigma_{ij}} \brac{\scalar{N,p}^2 \snorm{\II^{ij}_0}^2 - (n-1) \k_{ij} \scalar{N,p} \scalar{N,\c_{ij}}} dp . 
\end{equation}
Note that indeed $\int_{\Sigma^1} \scalar{N,p}^2 \snorm{\II_0}^2 dp < \infty$ by Theorem \ref{thm:stability-in-trace}.

\smallskip
Unfortunately, the second term in the above expression does not have a clear sign, which would have enabled us to conclude the sphericity of the cluster $\II_0 \equiv 0$. To conclude the argument, we need to use an additional vector-field -- the conformal Killing field $W_N$, which applies an infinitesimal M\"obius transformation in the direction of $N$ (recall Definition \ref{def:dilation-field}). 

\subsection{Concluding on $\S^n$}

Note that $W_N^\n = \scalar{N,\n}$ and that $L_{Jac} W_N^\n = (n-1) \scalar{N,\c}$ by Lemma \ref{lem:LJac-dilation}. As $\scalar{N,\n_{ij}}$ is clearly odd on $\Sigma_{ij}$ with respect to reflection about $N^{\perp}$, then so is $\eta W_N^\n$ for any function $\eta : \S^n \rightarrow \R$ which is even with respect to said reflection, and therefore $\delta^1_{\eta W_N} V(\Omega) = 0$. 

Consequently, stability, Theorem \ref{thm:Q-Sigma4} and Theorem \ref{thm:conformal-Q} \ref{it:conformal-Q-general} imply that:
\[
0 \leq Q^1(\eta W_N) = \int_{\Sigma^1} \brac{|\nabla^\tang \eta|^2 \scalar{N,\n}^2  - \eta^2 (n-1) \scalar{N,\n} \scalar{N,\c}} dp 
\]
whenever $\eta : \S^n \rightarrow [0,1]$ is a cutoff function whose compact support is disjoint from $\Sigma^4$ and which is in addition even with respect to $N^{\perp}$. By Lemma \ref{lem:cutoff-Sigma4}, given $\eps > 0$, we may always select such a cutoff function so that $\int_{\Sigma^1} |\nabla^\tang \eta|^2 dp \leq \eps$ and $\int_{\Sigma^1} 1_{\{ \eta < 1 \}} dp \leq \eps$. Note that $\eta$ may be selected to respect any finite number of symmetries of $\Sigma$, either by construction or (simpler) by symmetrizing $\eta$ while incurring only a finite constant factor in the second error estimate by the union-bound. Consequently, $\eta$ may be selected to be even with respect to $N^{\perp}$, and we deduce:
\[
-\eps (1 + (n-1) \sup_{p \in \Sigma^1} |\c|)  \leq -(n-1) \int_{\Sigma^1} \scalar{N,\n} \scalar{N,\c} dp \;\;\; \forall \eps > 0 . 
\]
But since:
\[ |\c_{ij}|^2 = |\n_{ij} - \k_{ij} p|^2 = 1 + \k^2_{ij} = 1 + \frac{H_{\Sigma_{ij}}^2}{(n-1)^2} ,
\]and $H_{\Sigma_{ij}} = \lambda_i - \lambda_j$ is constant on $\Sigma_{ij}$ by stationarity and hence uniformly bounded on $\Sigma^1$, the left-hand-side above is finite, and we may take the limit as $\eps \rightarrow 0$ to deduce:
\[
0 \leq -(n-1) \sum_{i<j} \int_{\Sigma_{ij}} \scalar{N,\n_{ij}} \scalar{N,\c_{ij}} dp  . 
\]

Combining this with (\ref{eq:skew-Q0}) and recalling that $\n - \k p = \c$, we conclude that:
\[
0 \leq - \sum_{i<j} \int_{\Sigma_{ij}} \brac{\scalar{N,p}^2 \snorm{\II^{ij}_0}^2 + (n-1) \scalar{N,\c_{ij}}^2} dp \leq 0 . 
\]
Note that if the smooth (relatively open) $\Sigma_{ij}$ intersects the equator $M^{n-1} = M^n \cap N^{\perp}$, it must do so perpendicularly by $\S^0$-symmetry. It follows that $\II^{ij}_0 = 0$ and $\scalar{N,\c_{ij}} = 0$ almost-everywhere on $\Sigma_{ij}$, and hence everywhere as $\n_{ij}$ is $C^\infty$ smooth on $\Sigma_{ij}$. 

\subsection{Concluding on $\R^n$}

It turns out that on $\R^n$, there is no need to obtain a second inequality beyond (\ref{eq:skew-Q0}) by testing stability on another vector-field, but it is not easy to realize this without first understanding the case of $\S^n$ and the role played by the conformal Killing field $W_N$. It was thus in this order that we were able to obtain a proof of Theorem \ref{thm:stable-spherical}, and so it is in this order that we have chosen to present our argument. 

\begin{lemma} \label{lem:nc-identity} Let $\Omega$ be a bounded stationary regular cluster on $\R^n$, which in addition is assumed to either be stable, or of bounded curvature. Then:
\[
\int_{\Sigma^1} \n \otimes \c \; dp = 0 . 
\]
\end{lemma}
\begin{remark} \label{rem:isotropic}
Using that $0 = \int_{\Sigma^1} \frac{1}{2} \Delta_{\Sigma} \scalar{\theta,p}^2 dp = \int_{\Sigma^1} (-(n-1) \k \scalar{\theta,\n} \scalar{\theta,p} + |\theta^\tang|^2) dp $ by stationarity, it is easy to check that the assertion is equivalent to following one of isotropicity:
\[
\int_{\Sigma^1} \n \otimes \n \; dp = \frac{1}{n} \int_{\Sigma^1} \Id \; dp . 
\]
\end{remark}
\begin{proof}[Proof of Lemma \ref{lem:nc-identity}]
Let $T_{\theta_1} \equiv \theta_1$ denote the Killing field of translation by $\theta_1 \in \R^n$, and recall that $W_{\theta_2}$ is the conformal Killing field generating a M\"obius transformation in the direction of $\theta_2 \in \R^n$. Heuristically, the argument is simply to invoke the conformality of these fields (which results in the boundary integrand in the expression (\ref{eq:Q1-LJac}) for $Q^1$ to vanish by Lemma \ref{lem:conformal-boundary}), the formulas for  $L_{Jac} T_{\theta_1}^\n = 0$ and $L_{Jac} W_{\theta_2}^\n = (n-1) \scalar{\theta_2 , \c}$, and the symmetry of $Q^1$, as follows:
\begin{align*}
& (n-1) \int_{\Sigma^1} \scalar{\theta_1,\n} \scalar{\theta_2,\c} dp = \int_{\Sigma^1} T_{\theta_1}^\n L_{Jac} W_{\theta_2}^\n dp = -Q^1(T_{\theta_1},W_{\theta_2}) \\
& = -Q^1(W_{\theta_2},T_{\theta_1}) =\int_{\Sigma^1} W_{\theta_2}^\n L_{Jac} T_{\theta_1}^\n dp = 0 . 
\end{align*}

By Proposition \ref{prop:Q1-LJac-bounded-curvature}, this is a perfectly rigorous argument whenever the curvature is known to be (locally) bounded, but we would like to invoke Lemma \ref{lem:nc-identity} before deducing this from sphericity. 
Consequently, given $\eps > 0$, we invoke Lemma \ref{lem:cutoff-Sigma4} to select a $C_c^\infty$ cutoff function $\eta : \R^n \rightarrow [0,1]$ whose compact support is disjoint from $\Sigma^4$ so that $\int_{\Sigma^1} |\nabla^\tang \eta|^2 dp \leq \eps$ and $\int_{\Sigma^1} 1_{\{ \eta < 1 \}} dp \leq \eps$. Using conformality, we have by Theorem \ref{thm:Q1-cutoff} that:
 \begin{align*}
 &Q^1(\eta T_{\theta_1},\eta W_{\theta_2}) \\
 & = \sum_{i<j} \int_{\Sigma_{ij}} \brac{|\nabla^\tang \eta|^2 T_{\theta_1}^\n W_{\theta_2}^\n + \eta \scalar{\nabla^\tang \eta , W_{\theta_2}^\n \nabla^\tang T_{\theta_1}^\n - T_{\theta_1}^\n \nabla^\tang W_{\theta_2}^\n} - \eta^2 T_{\theta_1}^\n L_{Jac} W_{\theta_2}^\n} d\mu^{n-1} .
 \end{align*}
 Note that $R := \sup_{p \in \Sigma^1} |p| < \infty$ by boundedness of the cluster, and hence $|W_{\theta_2}^\n| \leq \frac{3 R^2}{2} |\theta_2|$ and $|L_{Jac} W_{\theta_2}^\n| = (n-1) |\scalar{\theta_2,\c}| \leq (n-1) |\theta_2| (1 + R \max|\k_{ij}|)$. Of course $|T_{\theta_1}^\n| \leq |\theta_1|$.  It follows by Cauchy--Schwarz that for some constant $C<\infty$:
 \[
\abs{Q^1(\eta T_{\theta_1},\eta W_{\theta_2}) + \int_{\Sigma^1} T_{\theta_1}^\n L_{Jac} W_{\theta_2}^\n dp} \leq C |\theta_1| |\theta_2| \eps + \sqrt{ \eps \int_{\Sigma^1} |W_{\theta_2}^\n \nabla^\tang T_{\theta_1}^\n - T_{\theta_1}^\n \nabla^\tang W_{\theta_2}^\n|^2 dp} . 
\]
Compute $\nabla^\tang T_{\theta_1}^\n = \II \theta_1^\tang$ and $\nabla^\tang W_{\theta_2}^\n = p^\tang \theta_2^\n + \frac{|p|^2}{2} \II \theta_2^\tang - \theta_2^\tang p^\n - \theta^p_2 \II p^\tang$. Now observe that $\int_{\Sigma^1} \norm{\II}^2 dp < \infty$ by invoking stability via Corollary \ref{cor:curvature-integrable-on-Sigma4} (note that Proposition \ref{prop:curvature-integrability-Sigma4} would \emph{not} be enough here, as cutting away from $\Sigma^4$ would yield a problematic dependence on $\eps$). Using this and that $R < \infty$, we can finally conclude that:
 \[
\abs{Q^1(\eta T_{\theta_1},\eta W_{\theta_2}) + \int_{\Sigma^1} T_{\theta_1}^\n L_{Jac} W_{\theta_2}^\n dp} \leq C' |\theta_1| |\theta_2| \sqrt{\eps} .
\]
An analogous estimate holds with $T_{\theta_1}$ and $W_{\theta_2}$ exchanged. Taking the limit as $\eps \rightarrow 0$, we obtain a rigorous verification of the initial heuristic argument. 
\end{proof}

Applying Lemma \ref{lem:nc-identity} in the directions $\theta_1 = \theta_2 = N$, we deduce:
\[
0  = -(n-1) \sum_{i<j} \int_{\Sigma_{ij}} \scalar{N,\n_{ij}} \scalar{N,\c_{ij}} dp  . 
\]
Combining this with (\ref{eq:skew-Q0}), the rest of the argument is identical to that on $\S^n$. 

\subsection{Final remarks}

Since $\k_{ij} = \frac{1}{n-1} H_{\Sigma_{ij}}$ is constant on $\Sigma_{ij}$ by stationarity, we have on both $\S^n$ and $\R^n$ that:
\[
\nabla^\tang \c_{ij} = \II^{ij} - \k_{ij} \Id = \II^{ij}_0 = 0 ,
\]
and we deduce that $\c_{ij}(p) \equiv \c_{ij}^\ell$ is constant on every connected component $\Sigma_{ij}^\ell$ of $\Sigma_{ij}$ (in fact, on every connected component modulo its $\S^0$-symmetry, since $\Sigma_{ij}^\ell$ is invariant under reflection). It follows by Remark \ref{rem:quasi-center} that $\Sigma^\ell_{ij}$ is a subset of a (generalized) geodesic sphere with quasi-center $\c^\ell_{ij}$ and curvature $\frac{1}{n-1} H_{\Sigma_{ij}} = \k_{ij}$.

\medskip

\begin{remark} \label{rem:equator-perp} Note that if $\Omega$ is a cluster with $\S^0$-symmetry which is spherical perpendicularly to its hyperplane of symmetry, such as in the conclusion of Theorem \ref{thm:stable-spherical}, then 
wherever $\Sigma = \overline{\Sigma^1}$ transverses the equator, it does so perpendicularly (in the sense that if $p \in \overline{\Sigma_{ij}^\ell} \cap M^{n-1}$ then $\overline{\Sigma_{ij}^\ell}$ is perpendicular to $M^{n-1}$ at $p$). Consequently, the same holds for $\overline{\Sigma^2}$ and $\overline{\Sigma^3}$. In particular, $\H^{n-k}(M^{n-1} \cap \Sigma^k) = 0$ for $k=1,2,3$, and hence $\H^{n-1}(M^{n-1} \cap \Sigma) = 0$. 
\end{remark}

The sphericity of the cluster already implies that it has bounded curvature, which enables us to proceed with the subsequent analysis. 
Before concluding this section, we also deduce:
\begin{proposition} \label{prop:skew-fields-properties}
The function $\Psi = \scalar{p,N}$ satisfies non-oriented conformal BCs on $\Sigma^2$ and $L_{Jac} \Psi = 0$ on $\Sigma^1$. Consequently, all skew scalar-fields $f^a = \{ f^a_{ij} = a_{ij} \Psi \}$ satisfy conformal BCs on $\Sigma^2$, $L_{Jac} f^a = 0$ and $Q^0(f^a) = 0$, for all $a \in E^{(q-1)}$. 
\end{proposition}
\begin{proof}
Note that by Theorem \ref{thm:stable-spherical}, for any $p \in \Sigma^\ell_{ij}$, 
\[
\scalar{\nabla \Psi , \n_{ij}} = \scalar{N,\n_{ij}} = \sscalar{N , \c^\ell_{ij} + \k_{ij} p} = \k_{ij} \Psi = \II^{ij}_{\partial\partial} \Psi , 
\]
and so by regularity, the same holds for all $p \in \partial \Sigma_{ij}$, confirming that $\Psi$ satisfies non-oriented conformal BCs on $\Sigma^2$. It follows by Lemma \ref{lem:non-oriented-conformal-BCs} that all skew-fields satisfy conformal BCs on $\Sigma^2$. 

Finally, Lemma \ref{lem:LJac-skew} verifies that on $\Sigma^\ell_{ij}$:
\[
L_{Jac} \Psi = \norm{\II_0}^2 \sscalar{N,p} - (n-1) \k_{ij} \sscalar{N,\c^\ell_{ij}} = 0,
\]
since both terms vanish by Theorem \ref{thm:stable-spherical}. It follows that $L_{Jac} f^a = 0$ for all skew-fields $f^a$ by linearity. It follows that $Q^0(f^a) = 0$ by Theorem \ref{thm:conformal-Q} \ref{it:conformal-Q-bounded} (as the curvature is already known to be  bounded). 
\end{proof}

\section{Spherical Voronoi clusters} \label{sec:Voronoi-prelim}

Let $M^n \in \{ \R^n , \S^n \}$. Recall that $N \in M^n$ represents a North Pole, that $M^{n-1} := M^n \cap N^{\perp}$ is the equator, and that $M^n_{\pm}$ are the two corresponding open hemispheres.

\medskip

This section is dedicated to proving the following version of Theorem \ref{thm:intro-spherical-Voronoi-prelim}.

\begin{theorem} \label{thm:spherical-Voronoi-prelim}
Let $\Omega$ be a bounded stationary regular cluster with $\S^0$-symmetry on $M^n \in \{ \R^n , \S^n \}$, which is spherical perpendicularly to its hyperplane of symmetry (recall Definition \ref{def:spherical-S0}). Assume in addition that all of its cells $\{\Omega_i\}$ are connected modulo their common $\S^0$-symmetry (recall Definition \ref{def:connected-S0}). Then $\Omega$ is a perpendicularly spherical Voronoi cluster (recall Definition \ref{def:intro-perp-spherical-Voronoi}).
\end{theorem}

In the next section, we will apply Theorem \ref{thm:spherical-Voronoi-prelim} to the cluster obtained from the connected components (modulo $\S^0$-symmetry) of our minimizing $q$-cluster, and a-posteriori deduce when $q \leq n+1$ that all of its cells had to already be connected. We will also verify there that Theorem \ref{thm:spherical-Voronoi-prelim} implies Theorem \ref{thm:intro-spherical-Voronoi-prelim}.

\subsection{Stereographic and orthogonal projections}

For the proof of Theorem \ref{thm:spherical-Voronoi-prelim} , it will be essential to consider several equivalent representations of a cluster on $M^n \in \{ \R^n , \S^n \}$. 

The first step is to pass from $\R^n$ to $\S^n$  by means of an appropriate stereographic projection which preserves the $\S^0$-symmetry. Fix an isometric embedding $\S^n \subset \R^{n+1} = \R^n \times \sspan(P)$ for some $P \in \S^n$. 
The standard stereographic projection from $\R^n$ to $\S^n$ is then defined by mapping $x \in \R^n$ to the unique point $T(x) \in \S^n \setminus \{P\}$ so that $x$, $T(x)$ and the (North) Projection Pole $P \in \S^n$ are colinear. Recall that $\bar \R^n = \R^n \cup \{\infty\}$ denotes the one-point-at-infinity compactification of $\R^n$, and set $T(\infty) = P$. It is well-known that $T$ is a conformal diffeomorphism between $\bar \R^n$ and $\S^n$, which maps (generalized) spheres onto spheres. 

Given a diffeomorphism $T : M \rightarrow M'$ between two smooth differentiable manifolds and a cluster $\Omega = (\Omega_1,\ldots,\Omega_q)$ on $M$, we write $T \Omega$ for the cluster $(T \Omega_1,\ldots,T \Omega_q)$ on $M'$. If $A \subset M$ is a measurable subset, we similarly write $\Omega \cap A$ for the cluster $(\Omega_1 \cap A,\ldots, \Omega_q \cap A)$. Clearly, a diffeomorphism between two weighted Riemannian manifolds $(M_1,g_1,\mu_1)$ and $(M_2,g_2,\mu_2)$ preserves the regularity of a cluster, and this is still the case if $T$ is only a diffeomorphism between $M_1 \setminus A_1$ and $M_2 \setminus A_2$ for some closed null-sets $A_1$ and $A_2$ which lie in the interior of the cells of $\Omega$ and $T \Omega$, respectively.

\begin{lemma} \label{lem:stereo}
Let $\Omega = \Omega^\R$ be a bounded regular $q$-cluster on $\R^n$ with $\S^0$-symmetry about $N^{\perp}$ which is spherical perpendicularly to its hyperplane of symmetry. 
Then for any stereographic projection $T$ obtained by first translating $\Omega^\R$ by an element of $N^{\perp}$ and then applying the standard stereographic projection, the resulting cluster $\Omega^{\S} = T \Omega^\R$ is a regular cluster on $\S^n$ with $\S^0$-symmetry about $N^{\perp}$ which is spherical perpendicularly to its hyperplane of symmetry. The Projection Pole $P \in \S^n$ is an interior point of $\Omega^\S_q$, where $\Omega^\R_q$ is the unique unbounded cell of $\Omega^\R$. \\
Conversely, let $\Omega = \Omega^\S$ be a regular cluster on $\S^n$ with $\S^0$-symmetry about $N^{\perp}$ which is spherical perpendicularly to its hyperplane of symmetry. Then for any stereographic projection $T$ obtained by choosing a Projection Pole $P \in \S^n \cap N^{\perp}$ which lies in the interior of some cell (a non-vacuous condition), and applying the corresponding stereographic projection onto $\R^n$, the resulting cluster $\Omega^\R = T \Omega^\S$ is a bounded regular cluster on $\R^n$ with $\S^0$-symmetry about $N^{\perp}$ which is spherical perpendicularly to its hyperplane of symmetry.  
\end{lemma}
\begin{proof}
Sphericity is obviously preserved under stereographic projections. As $\Omega^\R_q$ contains a neighborhood of infinity, $\Omega^\S_q$ contains the Projection Pole $P$ in its interior. Conversely, a Projection Pole $P \in \S^n \cap N^{\perp}$ which lies in the interior of some cell always exists, since otherwise the entire equator $\S^{n-1}$ would be contained in $\Sigma^\S$; the latter is impossible since this would mean by reflection symmetry that a cell has an interface with itself, contradicting regularity. 
In either case, the cluster's regularity follows by the remarks preceding the lemma, and reflection symmetry about $N^{\perp}$ is obvious by construction.
\end{proof}
\begin{remark} \label{rem:no-canonical}
Note that stationarity was not assumed in the above lemma, and it is not obvious that it will be preserved -- we will deduce this later on in Corollary \ref{cor:ciki}. For this reason, the assumption that the curvatures of different connected components of $\Sigma_{ij}$ are the same was deliberately omitted from Definition \ref{def:spherical-S0}. 
Also note that there is no canonical way to pass between $\R^n$ and $\S^n$ while preserving the $\S^0$-symmetry, as we have the freedom to select which point in $N^{\perp}$ will be projected onto the South Projection Pole $-P$. 
\end{remark}

We will assume throughout this section that $\Omega$ is as in Lemma \ref{lem:stereo}. If $\Omega = \Omega^\R$ is such a cluster on $\R^n$, we denote by $\Omega^\S$ the corresponding cluster on $\S^n$ obtained by applying any stereographic projection as in Lemma \ref{lem:stereo}; and vice versa, if $\Omega = \Omega^\S$ is such a cluster on $\S^n$, we denote by $\Omega^\R$ the corresponding cluster on $\R^n$.
For both $\M \in \{\R,\S\}$, every connected component $\Sigma^{\M,\ell}_{ij}$ of $\Sigma^\M_{ij}$ lies in a (generalized) geodesic sphere with quasi-center $\c^{\M,\ell}_{ij} \in N^{\perp}$ and curvature $\k^{\M,\ell}_{ij}$. Recall that there are finitely many such connected components by Remark \ref{rem:finite-ccs}. 
Note that in view of Remark \ref{rem:no-canonical},
 there is no canonical relation between $\c^{\R,\ell}_{ij}$ and $\c^{\S,\ell}_{ij}$ nor $\k^{\R,\ell}_{ij}$ and $\k^{\S,\ell}_{ij}$. However, we can state:

\begin{lemma} \label{lem:sum-c-k}
Given $\{\M , \bar \M \} = \{\R , \S \}$ and $p \in \Sigma^\M_{uvw}$, let $\bar p = T p$ denote the corresponding point on $\Sigma^{\bar \M}_{uvw}$ given by the stereographic projection $T : \M^n \rightarrow \bar \M^n$ from Lemma \ref{lem:stereo}. Then the following are equivalent:
\begin{enumerate}[(i)]
\item \label{it:sums1}
$\sum_{(i,j) \in \cyclic(u,v,w)} \c^{\M,\ell}_{ij} = 0$ and $\sum_{(i,j) \in \cyclic(u,v,w)} \k^{\M,\ell}_{ij} = 0$ at $p$. 
\item \label{it:sums1b}
$\sum_{(i,j) \in \cyclic(u,v,w)} \n^{\M}_{ij} = 0$ and $\sum_{(i,j) \in \cyclic(u,v,w)} \k^{\M,\ell}_{ij} = 0$ at $p$. 
\item \label{it:sums2}
 $\sum_{(i,j) \in \cyclic(u,v,w)} \c^{\bar \M,\ell}_{ij} = 0$ and $\sum_{(i,j) \in \cyclic(u,v,w)} \k^{\bar \M,\ell}_{ij} = 0$ at $\bar p$. 
 \item \label{it:sums2b}
 $\sum_{(i,j) \in \cyclic(u,v,w)} \n^{\bar \M}_{ij} = 0$ and $\sum_{(i,j) \in \cyclic(u,v,w)} \k^{\bar \M,\ell}_{ij} = 0$ at $\bar p$. 
\end{enumerate}
\end{lemma}
\begin{proof}
Recall by definition of the quasi-center that $\c = \n  - \k p$, and so assertions \ref{it:sums1} and \ref{it:sums1b} are trivially equivalent, and so are assertions \ref{it:sums2} and \ref{it:sums2b}. The three unit normals $\{ \n^\M_{ij} \}_{(i,j) \in \cyclic(u,v,w)}$ lie in the two-dimensional subspace perpendicular to $T_p \Sigma^\M_{uvw}$; consequently, the statement that they sum to zero is equivalent to the statement that any pair forms $120^\circ$ degree angles. As a stereographic projection is a conformal map, it preserves angles, and therefore $\sum_{(i,j) \in \cyclic(u,v,w)} \n^{\bar \M}_{ij} = 0$ at $\bar p$. Furthermore, recall from (\ref{eq:conformal-change-of-H}) that under a conformal change of metric $\tilde g = e^{2 \Psi} g$, the normalized mean-curvature $\k$ of a hypersurface $\Sigma$ transforms as follows:
\[
\k^{\tilde g} = e^{-\Psi} ( \k^g + \scalar{\nabla^g \Psi,\n^g} ) .
\]
Summing this at $\bar p$ over $(i,j) \in \cyclic(u,v,w)$, and using that $\sum_{(i,j) \in \cyclic(u,v,w)} \k^{\M,\ell}_{ij} = 0$ and again that $\sum_{(i,j) \in \cyclic(u,v,w)} \n^{\M}_{ij} = 0$ at $p$, it follows that $\sum_{(i,j) \in \cyclic(u,v,w)} \k^{\bar \M,\ell}_{ij} = 0$ at $\bar p$. This establishes the equivalence of assertions \ref{it:sums1b} and \ref{it:sums2b}, thereby concluding the proof. \end{proof}

The next step is to orthogonally project from $\S^n$ to $\B^n$. \begin{definition}[$\B^n$ and $\Pi$]
We denote by $\B^n$ the \textbf{open} unit-ball in $N^{\perp}$, and by $\Pi: \S^n \to N^\perp$ the orthogonal projection from $\S^n$ to $N^{\perp}$. 
\end{definition}
Clearly, $\Pi$ is a diffeomorphism between $\S^n_+$ and $\B^n$, and so $\Omega^\B := \Pi (\Omega^\S \cap \S^n_+)$ is a regular cluster. By assumption, all of its cells are connected. Note that thanks to reflection symmetry about $N^{\perp}$, the openness of the cells, and the fact that a regular cluster cannot have an interface between a cell and itself, the cluster $\Omega^\B$ uniquely determines $\Omega^\S$ on the entire $\S^n$ (not just outside the equator). 
Also note that $\Pi$ is certainly not conformal, and so $\Sigma^\B_{ij}$, $\Sigma^\B_{jk}$ and $\Sigma^\B_{ki}$ will no longer meet  in $120^\circ$ angles at $\Sigma^\B_{ijk}$ (in general). See Figure~\ref{fig:S-and-B}.

\subsection{$\Omega^\B$ has flat convex cells}

\begin{proposition}\label{prop:convex}
    Let $\Omega$ be as in Theorem \ref{thm:spherical-Voronoi-prelim}. Then:
       \begin{enumerate}[(i)]
      \item For all $i< j$ so that $\Sigma^\B_{ij} \neq \emptyset$, $\Sigma^\B_{ij}$ is convex and flat, i.e. $\II^{ij} \equiv 0$. 
      \item For all $i<j$ so that $\Sigma^\B_{ij} \neq \emptyset$, $\Sigma^\S_{ij}$ and $\Sigma^\R_{ij}$ are connected modulo their $\S^0$-symmetry; consequently, there are $\c^\S_{ij} , \c^\R_{ij} \in N^{\perp}$ and $\k^\S_{ij} , \k^\R_{ij} \in \R$ so that  $\Sigma^\B_{ij}$ is contained in the affine hyperplane $\{ p \in N^{\perp} \; ; \; \sscalar{\c^\S_{ij},p} + \k^\S_{ij} = 0\}$, $\Sigma^\S_{ij}$ is contained in the geodesic sphere on $\S^n$ with quasi-center $\c^\S_{ij}$ and curvature $\k^\S_{ij}$, and $\Sigma^\R_{ij}$ is contained in a (generalized) sphere in $\R^n$ with quasi-center $\c^\R_{ij}$ and curvature $\k^\R_{ij}$. 
      \item  The (non-empty) cells of $\Omega^\B$ in $\B^n$ are convex polyhedra given by:
    \begin{equation} \label{eq:polyhedron-formula}
    \Omega^\B_i = \bigcap_{j \neq i \; ; \; \Sigma^\B_{ij} \neq \emptyset } \set{ p \in \B^n \; ;\;  \sscalar{\c^\S_{ij},p} + \k^\S_{ij} < 0 } .
    \end{equation}
        \end{enumerate}
\end{proposition}

Before presenting a proof of Proposition \ref{prop:convex} at the end of this subsection, we first collect some preliminary observations. Recall that at $p \in \Sigma_{ij}^{\S,\ell}$, $\c_{ij}^{\S,\ell}$ is related to
the normal $\n_{ij}$ (pointing from $\Omega_i^\S$ to $\Omega_j^\S$) by $\n_{ij} = \c_{ij}^{\S,\ell} + \k_{ij}^{\S,\ell} p$.
Since $\Sigma_{ij}^{\S,\ell}$ coincides with the sphere $\{x \in \S^n \; ;\; \sscalar{\c_{ij}^{\S,\ell}, x} + \k_{ij}^{\S,\ell} = 0\}$
in a neighborhood of $p$, it follows by regularity of the cluster (described in Theorem \ref{thm:Almgren} \ref{it:Almgren-iii})
that for every $p \in \Sigma_{ij}^{\S,\ell}$ there is a neighborhood $U_p$ in $\S^n$ such that
\begin{equation}\label{eq:local-convexity-Sigma1}
\Omega_i^\S \cap U_p = \{x \in U_p\; ;\; \sscalar{\c_{ij}^{\S,\ell}, x} + \k^{\S,\ell}_{ij} < 0\}
\end{equation}
(the direction of the inequality may be seen from the orientation of the normal $\n_{ij}$).
We can extend this observation to triple-points: around every $p \in \Sigma_{ijk}^{\S}$ there is a neighborhood $U_p$ on which
$\Sigma^\S$ is diffeomorphic to the standard tripod $\Y$. On the neighborhood $U_p$,
$\Sigma_{ij}^\S$ and $\Sigma_{ik}^\S$ have one connected component each, so let
$\c^{\S,\ell_j}_{ij}$ and $\c^{\S,\ell_k}_{ik} \in N^\perp$ be the quasi-centers associated with these components.
Since again we know the orientation of the interfaces on $U_p$, it follows that 
\begin{equation}\label{eq:local-convexity-Sigma2}
    \Omega_i^\S \cap U_p = \{x \in U_p\; ;\; \sscalar{\c^{\ell_j}_{ij}, x} + \k^{\ell_j}_{ij} < 0 \text{ and } \sscalar{\c^{\ell_k}_{ik}, x} + \k^{\ell_k}_{ik} < 0\}.
\end{equation}
After projecting to $\B^n$,
we can summarize~\eqref{eq:local-convexity-Sigma1}
and~\eqref{eq:local-convexity-Sigma2} as follows. We denote $\Sigma^{\S, \leq 2} := \Sigma^{\S,1} \cup \Sigma^{\S,2}$ and $\Sigma^{\S, \geq 3} := \Sigma^{\S,3} \cup \Sigma^{\S,4}$. 

\begin{lemma}\label{lem:local-convexity}
    For every $b \in \Pi (\Sigma^{\S,\leq 2})$ and every $i$, there is a
    neighborhood $W_b$ of $b$ in $N^{\perp}$ and an open, convex set $A \subset N^\perp$ such
    that $\Omega_i^\B \cap W_b = A \cap \B^n \cap W_b$. Moreover, $A$ can be
    written as the intersection of sets of the form
    $\{ x \in N^{\perp} \; ; \; \sscalar{\c^{\S,\ell}_{ij},x} + \k^{\S,\ell}_{ij} < 0 \}$
    for $\c^{\S,\ell}_{ij}$ and $\k^{\S,\ell}_{ij}$ the quasi-center and curvature of a connected component 
     $\Sigma^{\S,\ell}_{ij}$ of a non-empty interface $\Sigma^{\S}_{ij} \neq \emptyset$. \end{lemma}

\begin{proof}
    If $b = \Pi p$ with $p \in \Sigma_{ij}^\S \cap \S^n_+$, we
    apply~\eqref{eq:local-convexity-Sigma1} while taking $U_p$ small enough so
    that $U_p \subset \S^n_+$. Then $W_b = \Pi U_p \subset \B^n$ is a neighborhood of $b$, and
    because $\sscalar{\c_{ij}^{\S,\ell},\Pi x} = \sscalar{\c_{ij}^{\S,\ell},
    x}$, $\Omega_i^\B \cap W_b = \Pi(\Omega_i^\S \cap U_p) = \{x \in W_b \; ; \; \sscalar{\c_{ij}^{\S,\ell}, x} +\k_{ij}^{\S,\ell} < 0\}$.

    If $b = p \in \Sigma_{ij}^\S \cap N^\perp$ then again we apply~\eqref{eq:local-convexity-Sigma1}. This
    time, $\Pi U_p$ is no longer a neighborhood of $b$ in $N^{\perp}$ because $\Pi$ is not a diffeomorphism at the equator.
    However, we may find an open ball $W_b \subset N^\perp$ containing $b$ such that $W_b \cap \overline{\B^n} \subset \Pi U_p$.
    Then $(\Pi \Omega_i^\S) \cap W_b = \{x \in W_b \cap \overline{\B^n}\; ; \; \sscalar{\c_{ij}^{\S,\ell}, x} + \k^{\S,\ell}_{ij} < 0\}$.
    Since $\Omega_i^\B = \Pi \Omega_i^\S \setminus \partial \B^n$, it follows that
    $\Omega_i^\B \cap W_b = \B^n \cap \{x \in W_b \; ; \; \sscalar{\c_{ij}^{\S,\ell}, x} + \k^{\S,\ell}_{ij} < 0\}$. 
        
    This completes the case of $\Sigma^{\S,1}$. The argument for $\Sigma^{\S,2}$ is identical except that one
    uses~\eqref{eq:local-convexity-Sigma2} instead of~\eqref{eq:local-convexity-Sigma1}.
\end{proof}

We next recall the following local-to-global convexity result from \cite[Proposition 8.7]{EMilmanNeeman-GaussianMultiBubble}:
\begin{proposition}\label{prop:local-to-global-convexity}
    Let $K$ be an open, connected subset of $\R^n$, and let $B$ be a Borel set with $\H^{n-2}(B) = 0$.
    Assume that for every $p \in \partial K \setminus B$ there exists an open neighborhood $U_p$ of $p$ so that
    $K \cap U_p$ is convex. Then $K$ is convex.
\end{proposition}

We can now establish:
\begin{lemma} \label{lem:cells-convex}
The cells of $\Omega^\B$ are all convex. 
\end{lemma}
\begin{proof}
Lemma~\ref{lem:local-convexity} implies that the hypothesis of Proposition~\ref{prop:local-to-global-convexity} holds
for each (open) cell $\Omega_i^\B$, with $B = \Pi \Sigma^{\S,\ge 3}$. Indeed, if $b \in \partial \Omega_i^\B \setminus B$ then
either $b \in \Pi \partial \Omega_i^\S \setminus B$ (in which case $b \in \Pi \Sigma^{\S,\le 2}$ and so we may apply Lemma~\ref{lem:local-convexity} to find a convex neighborhood)
or else $b \in \partial \B^n \cap \Pi \Omega_i^\S = \partial \B^n \cap \Omega_i^\S$, in which case $W_b \cap \Omega_i^\B = W_b \cap \B^n$ for a small
enough neighborhood $W_b$ of $b$ in $N^\perp$ (by openness of $\Omega_i^\S$). Also, $\H^{n-2}(\Sigma^{\S,\ge 3}) = 0$ by regularity and $\Pi$ is a contraction so $\H^{n-2}(B) = 0$. Finally, $\Omega_i^\B$ is clearly connected since $\Omega_i^\S$ was assumed to be connected (modulo its $\S^0$-symmetry). It follows from Proposition~\ref{prop:local-to-global-convexity} that $\Omega_i^\B$ is convex for every $i$.
\end{proof}

\begin{proof}[Proof of Proposition~\ref{prop:convex}]
Note that $\Sigma_{ij}^\B \subset \overline{\Omega_i^\B} \cap \overline{\Omega_j^\B}$ and $\overline{\Sigma_{ij}^\B} = \overline{\Omega_i^\B} \cap \overline{\Omega_j^\B} = \partial \Omega_i^\B \cap \partial \Omega_j^\B$. The latter set  
is convex as the intersection of two convex sets (by Lemma \ref{lem:cells-convex}) and has Hausdorff dimension at most $n - 1$ as the intersection of their boundaries. Therefore, $\overline{\Sigma_{ij}^\B}$ is convex and lies in an affine hyperplane $E_{ij}$ in $N^{\perp}$. As $\Sigma_{ij}^\B$ is an open $(n-1)$-dimensional manifold in $E_{ij}$, it must also be convex and flat, establishing the first claim. 

In particular $\Sigma^\B_{ij}$ is connected, and therefore $\Sigma^\S_{ij}$ and thus $\Sigma^\R_{ij}$ must also be connected modulo their $\S^0$-symmetry. 
Consequently, the sphericity perpendicularly to $N^{\perp}$ of $\Omega^\S$ (and of $\Omega^\R$ by Lemma \ref{lem:stereo}) imply the second claim.

Finally, we check the third claim. Recall that $\Omega_i^\B$ is an open convex set by Lemma \ref{lem:cells-convex} and non-empty (by regularity). As such, it is the intersection of its open supporting halfspaces over all boundary points.
By continuity (because the boundary of a convex set is a Lipschitz manifold), this is the same as the interior of the intersection of its open supporting halfspaces at a dense subset of the boundary.
For a dense subset of
$\partial \Omega_i^\B$ (specifically, $\partial \Omega_i^\B \cap (\Sigma^{\B,1} \cup \partial \B^n)$ by regularity),
those supporting halfspaces are either supporting halfspaces of $\B^n$, or else take the form 
$\{x \in N^\perp\; ; \; \sscalar{\c_{ij}^{\S}, x} + \k_{ij}^{\S} < 0 \}$ for some $j \ne i$ so that $\Sigma^\B_{ij} \neq \emptyset$ (by the second claim). Moreover, every such halfspace supports $\Omega_i^\B$ on $\Sigma_{ij}^\B \neq \emptyset$. The representation~\eqref{eq:polyhedron-formula} therefore follows. 
\end{proof}

\subsection{First Simplicial Homology Vanishes}  \label{subsec:homology}

Given a regular cluster $\Omega$, we associate to it the following two-dimensional abstract simplicial complex $\SS = \SS(\Omega)$: 
its vertices are given by $\SS_0 := \{ \{i\} \; ; \; \Omega_i \neq \emptyset \}$, its edges are given by $\SS_1 := \{ \{i,j\} \; ; \;\Sigma_{ij} \neq \emptyset \}$, and its triangles are given by $\SS_2 := \{ \{i,j,k\} \; ; \; \Sigma_{ijk} \neq \emptyset \}$. This is indeed an abstract simplicial complex since the face of any simplex in $\SS$ is clearly also in $\SS$ by regularity. Note that it may be that $\partial \Omega_i \cap \partial \Omega_j \neq \emptyset$ and yet $\{i,j\} \notin \SS_1$, or that $\partial \Omega_i \cap \partial \Omega_j \cap \partial \Omega_k \neq \emptyset$ and yet $\{ i,j,k\} \notin \SS_2$ -- the complex $\SS$ only records the incidence of pair of cells along their codimension-one mutual boundary in $\Sigma^1$, and of triple cells along their codimension-two mutual boundary in $\Sigma^2$. We write $\SS_1 \simeq \SS_2$ if these simplicial complexes are isomorphic. In our setting, all cells are always assumed to be non-empty. 

\begin{lemma} \label{lem:SS-isomorphic}
For any bounded regular cluster $\Omega^\R$ on $\R^n$ with $\S^0$-symmetry which is spherical perpendicularly to its hyperplane of symmetry, $\SS(\Omega^\R) \simeq \SS(\Omega^\S)$. For any regular cluster $\Omega^\S$ on $\S^n$ with $\S^0$-symmetry which is spherical perpendicularly to its hyperplane of symmetry, $\SS(\Omega^\S) \simeq \SS(\Omega^\B)$. 
\end{lemma}
\begin{proof}
The first assertion is obvious because the stereographic projection $T$ described in Lemma \ref{lem:stereo} is a diffeomorphism between $\R^n$ and $\S^n \setminus \{P\}$ with $P$ lying in the interior of a cell; in particular, $T$ is a diffeomorphism between an $\eps$-neighborhood of $\Sigma^\S$ for some $\eps > 0$ and an open set containing $\Sigma^\R$, and therefore a bijection between $\Sigma^\R_{ij}$ and $\Sigma^\S_{ij}$ as well as between $\Sigma^\R_{ijk}$ and $\Sigma^\S_{ijk}$. \\
As $\Pi$ is a diffeomorphism between $\S^n_+$ and $\B^n$, it trivially follows that $\SS(\Omega^\B) \simeq \SS(\Omega^\S \cap \S^n_+)$, so it remains to show that $\SS(\Omega^\S \cap \S^n_+) \simeq \SS(\Omega^\S)$. By reflection symmetry, it is obvious that $\SS(\Omega^\S \cap \S^n_+) \simeq \SS(\Omega^\S \cap (\S^n \setminus N^{\perp}))$, so the remaining task is to show that absence of the equator does not affect the incidence structures of $\Sigma^1$ and $\Sigma^2$. And indeed, as $\Sigma^k$ ($k=1,2$) are smooth $(n-k)$-dimensional manifolds by regularity, we have $\Sigma^2_{ijk} \neq \emptyset$ iff $\H^{n-2}(\Sigma^2_{ijk}) > 0$ iff $\H^{n-2}(\Sigma^2_{ijk} \setminus N^{\perp}) > 0$ by Remark \ref{rem:equator-perp}, and similarly for $\Sigma^1_{ij}$. 
\end{proof}

Consequently, we can study the properties of $\SS := \SS(\Omega^\R) \simeq \SS(\Omega^\S) \simeq \SS(\Omega^\B)$ in any of our spaces $\R^n$, $\S^n$ or $\B^n$. By Lemma \ref{lem:LA-connected}, we know that the $1$-skeleton of $\SS$, i.e.~the graph obtained by considering only its edges and vertices, is necessarily connected. In other words, the zeroth (reduced, simplicial) homology of $\SS$ is trivial. To this we add: 
\begin{proposition} \label{prop:homology}
Let $\Omega$ be as in Theorem \ref{thm:spherical-Voronoi-prelim}.
Then the first (simplicial) homology and cohomology groups of $\SS = \SS(\Omega)$ are trivial (over any field $\F$). 
\end{proposition}

Recall that for an oriented 1-cochain, namely an oriented function $f : \SS^{\pm}_1 \rightarrow \F$ on the oriented edges $\SS^{\pm}_1 := \{ (i,j) \; ; \; \{i,j\} \in \SS_1\}$ with  $f(j,i) = -f(i,j)$, its coboundary $\delta f : \SS^{\pm}_2 \rightarrow \F$ is the oriented 2-cochain on the oriented triangles $\SS^{\pm}_2 := \{ (i,j,k) \; ; \; \{i,j,k\} \in \SS_2 \}$ given by $\delta f(i,j,k) = f(i,j) + f(j,k) + f(k,i)$. Triviality of the cohomology means that if $\delta f = 0$ then necessarily $f = \delta g$ for some $g : \SS_0 \rightarrow \F$, namely $f(i,j) = g(i) - g(j)$. While our treatment will be self-contained, we refer to \cite{Munkres-AlgebraicTopology} for standard background on simplicial homology theory.

For the proof, it will be enough to show that for any directed cycle $\C = (i_1,\ldots,i_N,i_{N+1} = i_1)$ in the $1$-skeleton of $\SS$, the corresponding oriented $1$-chain $C^{1} := \sum_{\ell=1}^N (i_\ell,i_{\ell+1})$ can be written as the boundary of an oriented $2$-chain  $C^{2} := \sum_{h=1}^T \Tr_h$, where  $\Tr_h$ is an oriented triangle $(i_h,j_h,k_h) \in \SS^{\pm}_2$ (whose boundary is $\partial \Tr_h =  (i_h,j_h) + (j_h,k_h) + (k_h,i_h)$). It is easy to check that Proposition \ref{prop:homology} is false for a general finite tessellation of $\R^n$ by convex cells for example. The crucial property used in the proof is the combinatorial incidence structure of cells along triple-points $\Sigma^2$, and the fact that $\H^{n-2}(\Sigma^3 \cup \Sigma^4) = 0$ (by regularity). 

\begin{proof}[Proof of Proposition \ref{prop:homology}]
By Lemma \ref{lem:SS-isomorphic}, it is enough to establish the claim for $\Omega = \Omega^\B$. 
When $\SS$ is constructed from a stable regular $q$-cluster on Euclidean space $(\R^n,|\cdot|)$ endowed with the Gaussian measure, a proof of Proposition \ref{prop:homology} was already sketched in \cite[Theorem 12.3]{EMilmanNeeman-GaussianMultiBubble}. However, the proof only utilized the finiteness of $\SS$, the convexity of the cells $\Omega_i$ (and of the ambient space) and the regularity of the cluster, and did not require the stationarity nor stability of the cluster. Consequently, the proof applies to the regular cluster $\Omega^\B$ on the convex $\B^n$, whose cells are convex by Proposition \ref{prop:convex}.
Let us inspect the argument, and in doing so provide some additional details which were skimmed over in \cite{EMilmanNeeman-GaussianMultiBubble}. 

Let $\C$ denote a directed cycle in the $1$-skeleton of $\SS$. 
Consider a closed piecewise linear directed path $P$ in $\B^n$ which emulates $\C$, meaning that it crosses from $\Omega_i$ to $\Omega_j$ transversely through $\Sigma_{ij}$ in the order specified by $\C$ (and without intersecting $\Sigma^3 \cup \Sigma^4$). Note that it is always possible to construct such a path thanks to the convexity (and in particular, connectedness) of the cells. Assume that $P = \cup_{\ell=1}^{N} [y_\ell,y_{\ell+1}]$ with $y_{N+1} = y_1$, and fix a point $o$ in any of the (non-empty open) cells $\{\Omega_i\}_{i=1,\ldots,q}$, such that $o$ is not co-linear with any of these segments. Consider the convex interpolation $P_t =  (1-t) P + t \cdot o$ contracting $P$ onto $\{ o\}$ as $t$ ranges from $0$ to $1$, i.e. $P_t = \cup_{\ell=1}^{N} [y^t_\ell,y^t_{\ell+1}]$ for $y^t_\ell = (1-t) y_\ell + t \cdot o$. We denote $\mathcal{V}(P_t) := \{y^t_\ell\}_{\ell=1,\ldots,N}$. 
It was shown in the proof of \cite[Theorem 12.3]{EMilmanNeeman-GaussianMultiBubble} that there exists a perturbation of the points $\{y_\ell\}_{\ell=1,\ldots,N} \cup \{o \}$ so that:
\begin{enumerate}[(i)]
\item $o$ remains in its original cell,  $P$ still emulates $\C$, and $o$ is not co-linear with any of the segments $[y_\ell,y_{\ell+1}]$. \label{it:homology-1}
\item $P_t$ does not intersect $\Sigma^3 \cup \Sigma^4$ for all times $t \in [0,1]$. \label{it:homology-2}
\item $P_t$ does not intersect $\Sigma^2$ except for a finite set of times $t \in (0,1)$. \label{it:homology-3}
\item For all $\ell=1,\ldots,N$, $y_{\ell+1}-y_\ell$ and $y_\ell - o$ are not perpendicular to the finite set $\{\c^\S_{ij} \}_{i<j}$; consequently, for all times $t \in [0,1]$, $P_t$ and $\{[o,v]\}_{v \in \mathcal{V}(P_t)}$ are transversal to $\Sigma^1$.\label{it:homology-4}
\end{enumerate}
Now for every $t \in [0,1)$, consider the directed cycle $\C_t$ in the $1$-skeleton of $\SS$ which the path $P_t$ traverses; it is well-defined since $P_t$ crosses from one cell to the next transversely through $\Sigma^1$, and is of finite length (the length is bounded e.g. by $N (q-1)$ since every segment can cross at most $q-1$ cells by convexity).
By openness, $\C_t$ remains constant between consecutive times when $P_t$ intersects $\Sigma^2$ (``$\Sigma^2$ crossing") or when $\mathcal{V}(P_t)$ intersects $\Sigma^1$ (``$\Sigma^1$ crossing"). A $\Sigma^2$ crossing can only happen at a finite set of times $t \in (0,1)$ by \ref{it:homology-3}, and a $\Sigma^1$ crossing can only happen at a finite set of times $t \in (0,1)$ since each $[o, y_\ell]$ is transversal to $\Sigma^1$ and intersects it at most $q-1$ times by convexity of the cells. 

Let $C^{1}_t$ denote the oriented $1$-chain corresponding to $\C_t$, and set $C^{2}_0 = 0$ to be the zero $2$-chain at time $t=0$. 
In a $\Sigma^2$ crossing at time $t_0$, there are two possibilities, depending on how $P_t$ traverses $\Sigma_{ijk}$ as $t$ crosses $t_0$: either a directed edge $(i,j)$ in $\C_t$ will transform into two consecutive directed edges $(i,k),(k,j)$, in which case we will add $(i,j,k)$ to $C^{2}_t$, or vice versa, two consecutive directed edges $(i,k),(k,j)$ will collapse into a single directed edge $(i,j)$, in which case we will add $(i,k,j)$ to $C^{2}_t$. We thereby ensure by induction that $\C^{1}_0 = \C^{1}_t + \partial \C^{2}_t$ for $t$ on either side of the crossing time $t_0$. In a $\Sigma^1$ crossing, two consecutive directed edges $(i,j) , (j,i)$ in $\C_t$ will vanish, and since $(i,j) + (j,i) = 0$, $\C^{1}_t$ does not change through the crossing and $\C^{1}_0 = \C^{1}_t + \partial \C^{2}_t$ remains valid. 

For times $T$ close enough to $1$, $\C_{T}$ must be a singleton since the path $P_T$ is already close enough to $o$ so as to remain inside a single cell, and therefore $\C^{1}_{T} = 0$. It follows that $\C^{1}_0 = \partial \C^{2}_{T}$, and we have managed to express our initial oriented $1$-chain at time $t=0$ as the boundary of the resulting oriented $2$-chain at time $t = T$. 

We now verify that the first simplicial cohomology vanishes. Let $f : \SS^{\pm}_1 \rightarrow \F$ be an oriented function with $\delta f = 0$. We claim that for any directed cycle $\C$ in the $1$-skeleton of $\SS$, $\sum_{(i,j) \in \C} f(i,j) = 0$. 
Indeed, we have shown above that the $1$-chain $C^1$ corresponding to $\C$ is the boundary of a $2$-chain $C^2$, and therefore:
\begin{equation} \label{eq:sum-on-cycle}
\sum_{(i,j) \in \C} f(i,j) = \scalar{f , C^1} = \scalar{f , \partial C^2} = \scalar{\delta f , C^2} = 0 . 
\end{equation}
Defining $g : \SS_0 \rightarrow \F$ by $g(k) = \sum_{(i,j) \in \C_{1k}} f(i,j)$ for any directed path $\C_{1k}$ in $\SS$ connecting vertex $1$ and $k$ (such a path always exists by connectivity of the $1$-skeleton, but we could avoid using this by using any fixed selection of vertices in every connected component), (\ref{eq:sum-on-cycle}) implies that this does not depend on the particular path $\C_{1k}$ and is thus well-defined. It follows that $f(i,j) = g(j) - g(i)$, i.e. $f = \delta g$, as required. 

Finally, a well-known consequence of the Universal Coefficient Theorem for cohomology \cite[pp. 324--326]{Munkres-AlgebraicTopology} is that the simplicial homology vector-space is isomorphic to the dual of the cohomology vector-space over any field. 
Consequently, the first simplicial homology vanishes as well. 
\end{proof}

\begin{corollary} \label{cor:ciki} 
Let $\Omega$ be as in Theorem \ref{thm:spherical-Voronoi-prelim}. 
Then for both $\M \in \{\R , \S \}$:
\begin{enumerate}[(i)]
\item \label{it:ciki-c}
There exist vectors $\{\c^\M_i\}_{i=1,\ldots,q} \subset N^{\perp}$ so that the quasi-center $\c^\M_{ij}$ of any non-empty $\Sigma^{\M}_{ij} \neq \emptyset$ is equal to $\c^\M_{i} - \c^\M_{j}$. 
\item \label{it:ciki-k}
 There exist scalars $\{ \k^\M_i \}_{i=1,\ldots,q} \subset \R$ so that the curvature $\k^\M_{ij}$ of any non-empty $\Sigma^\M_{ij} \neq \emptyset$ is equal to $\k^\M_{i} - \k^\M_j$. 
\end{enumerate}
In particular, both $\Omega^\R$ and $\Omega^\S$ are stationary. 
\end{corollary}

\begin{remark}
As usual, to make the choice of $\{ \c^\M_i \}$ and $\{ \k^\M_i \}$ canonical, recall our convention that $\sum_{i=1}^q \c^\M_i = 0$ and that $\sum_{i=1}^q \k^\M_i = 0$, i.e. that $\k^\M \in E^{(q-1)}$. It follows that the Lagrange multiplier of the stationary $\Omega^\M$ is given by $\lambda^\M = (n-1) \k^\M$. We subsequently define $\c_{ij}^\M := \c^\M_i - \c^\M_j$ and $\k_{ij}^\M := \k^\M_i - \k^\M_j$ for all $i,j \in \{1,\ldots,q\}$, extending the definition when $\Sigma^\M_{ij} \neq \emptyset$. 
\end{remark}

\begin{proof}
Recall the equivalent characterization of stationarity for a regular cluster of locally bounded curvature on $(M^n,g,\mu)$ given by Lemma \ref{lem:equivalent-stationary}, and note that $H^\M_{\Sigma_{ij},\mu} = (n-1) \k^\M_{ij}$ in our setting. 

Now let $\Omega$ be a cluster on $\M^n$, $\{ \M , \bar \M \} = \{ \R , \S \}$, satisfying the assumptions of Theorem \ref{thm:spherical-Voronoi-prelim}. In particular, $\Omega^\M$ is stationary and of bounded curvature. Properties \ref{it:stationarity-k} and \ref{it:stationarity-n} of Lemma \ref{lem:equivalent-stationary} verify the assertion \ref{it:sums1b} of Lemma \ref{lem:sum-c-k} for $\Omega^\M$, and therefore also assertion \ref{it:sums1} as well as assertions \ref{it:sums2} and \ref{it:sums2b} for $\Omega^{\bar \M}$. It follows that property \ref{it:stationarity-n} of Lemma \ref{lem:equivalent-stationary} holds on $\Omega^{\bar \M}$, and it remains to establish assertion \ref{it:ciki-k} of Corollary \ref{cor:ciki} to establish property \ref{it:stationarity-k} and conclude that $\Omega^{\bar \M}$ is stationary. 

Indeed, by assertions \ref{it:sums1} and \ref{it:sums2} of Lemma \ref{lem:sum-c-k}, we know for both $\M \in \{\R,\S\}$ that for all $\Sigma_{uvw} \neq \emptyset$:
\[
\delta \c^\M(u,v,w) = \sum_{(i,j) \in \cyclic(u,v,w)} \c^\M_{ij} = 0  \; ~,~\;  \delta \k^\M(u,v,w) = \sum_{(i,j) \in \cyclic(u,v,w)} \k^\M_{ij} = 0 . 
\]
Using that the first simplicial cohomology of $\SS \simeq \SS(\Omega^\R) \simeq \SS(\Omega^\S)$ vanishes by Proposition \ref{prop:homology}, the two assertions \ref{it:ciki-c} and \ref{it:ciki-k} of Corollary \ref{cor:ciki} immediately follow (perhaps after considering the real-valued oriented function $\scalar{\c^\M,\xi}$ on $\SS^{\pm}_1$ for all $\xi \in N^{\perp}$ and using linearity). 
\end{proof}

\subsection{Voronoi representation}

\begin{proposition}
Let $\Omega$ be as in Theorem \ref{thm:spherical-Voronoi-prelim}. Then we have the following representation for the cells of $\Omega^\S$:
\begin{equation} \label{eq:Voronoi}
\Omega^\S_i  = \bigcap_{j \neq i} \; \set{ p \in \S^n \; ;\; \sscalar{\c^\S_{ij},p} + \k^\S_{ij} < 0 } = \set{ p \in \S^n \; ; \; \argmin_{j=1,\ldots,q} \; \sscalar{\c^\S_j , p} + \k^\S_j = \{i\} } .
\end{equation}
\end{proposition}
\begin{proof} 
By Proposition \ref{prop:convex}, we know that:
\[
 \Omega^\S_i  = \bigcap_{j \neq i : \Sigma^\S_{ij} \neq \emptyset} \set{ p \in \S^n \; ; \;\sscalar{\c^\S_{ij},p} + \k^\S_{ij} < 0 } .
\]
Note the subtle difference with (\ref{eq:Voronoi}), where the intersection is taken over all $j \neq i$ (no requirement that $\Sigma^\S_{ij} \neq \emptyset$). Denoting the variant in (\ref{eq:Voronoi}) by $\tilde \Omega^\S_i$, we shall show that $\Omega^\S_i = \tilde \Omega^\S_i$ for all $i$. 

To see this, let $I \subset \{1,\ldots,q\}$ denote a minimal subset of indices such that $\{(\c^\S_i,\k^\S_i)\}_{i \in I} = \{ (\c^\S_i,\k^\S_i) \}_{i=1,\ldots,q}$; in particular, $\{(\c^\S_i,\k^\S_i)\}_{i \in I}$ are distinct vectors in $\R^{n+2}$. We will first show that $I = \{1,\ldots,q\}$. To this end, denote for $i \in I$:
\[
\hat \Omega^\S_i := \bigcap_{j \in I \setminus \{i\}} \; \set{ p \in \S^n \; ;\; \sscalar{\c^\S_{ij},p} + \k^\S_{ij} < 0 } = \set{ p \in \S^n \; ; \; \argmin_{j \in I} \; \sscalar{\c^\S_j , p} + \k^\S_j = \{i\} } .
\]
Then $\hat \Omega^\S$ is a genuine cluster, as $\{\hat \Omega^\S_i\}_{i \in I}$ are clearly disjoint, and 
\[
\S^n \setminus \bigcup_{i \in I} \hat \Omega^\S_i \; \subset \bigcup_{i \neq j : i,j \in I} \set{p \in \S^n \; ; \; \sscalar{\c^\S_{ij},p} + \k^\S_{ij} = 0}
\]
 is a null-set, thanks to the fact that all $\{(\c^\S_i,\k^\S_i)\}_{i \in I}$ are distinct. 
Note that given $i \in I$, if $j \in \{1,\ldots,q\} \setminus \{i\}$ is such that $\Sigma^\S_{ij} \neq \emptyset$ then $\c_{ij}^\S \neq 0$ (being a quasi-center). Consequently, there exists $j' \in I \setminus \{i\}$ so that $(\c^\S_{j'},\k^\S_{j'}) = (\c^\S_j,\k^\S_j)$, hence $(\c^\S_{ij'} , \k^\S_{ij'}) = (\c^\S_{ij} , \k^\S_{ij})$ and therefore $\Omega^\S_i \supset \hat \Omega^\S_i$ for all $i \in I$. Since $\hat \Omega^\S$ covers $\S^n$ up to null-sets, it follows that $V(\Omega^\S_i) = 0$ for all $i \notin I$. But this is impossible unless $I = \{1,\ldots,q\}$, since all (open) cells of $\Omega^\S$ are assumed non-empty (by regularity). It follows that $\hat \Omega^\S = \tilde \Omega^\S$,  $\tilde \Omega^\S_i \subset \Omega^\S_i$, $\H^{n}(\Omega^\S_i \setminus \tilde \Omega^\S_i) = 0$ and $\tilde \Omega^\S_i \neq \emptyset$ for all $i=1,\ldots,q$.

Now observe that no hyperplane $\{ x \in \R^{n+1} \; ; \; \scalar{x,\c_{ij}} + \k_{ij} = 0 \}$ with $i \neq j$ can be tangential to $\S^n$ since this would mean that either $\tilde \Omega_i^\S$ or $\tilde \Omega_j^\S$ are empty, a contradiction. 
Assume in the contrapositive that $\exists p \in \Omega^\S_i \setminus \tilde \Omega^\S_i$. Since $\H^{n}(\Omega^\S_i \setminus \tilde \Omega^\S_i) = 0$, necessarily $p \in \partial \tilde \Omega^\S_i$ and so $\scalar{p,\c_{ij}}+\k_{ij} = 0$ for some $j \neq i$. Since $\Omega^\S_i$ is open, there is some small open neighborhood $U$ of $p$ contained in $\Omega^\S_i$. But on the other hand, $U \cap \{ x \in \R^{n+1} \; ; \; \scalar{x,\c_{ij}} + \k_{ij} > 0 \}$ has positive $\H^n$-measure (since the hyperplane is non-tangential) and lies outside of $\tilde \Omega^\S_i$. It follows that $\H^n(\Omega^\S_i \setminus \tilde \Omega^\S_i) > 0$, a contradiction.
We have shown that $\Omega^\S_i  = \tilde \Omega^\S_i$ for all $i$, concluding the proof.
\end{proof}

The contents of this section complete the proof of Theorem \ref{thm:spherical-Voronoi-prelim}. 

\subsection{Explicit representation of spherical Voronoi cluster on $\R^n$}

Before concluding this section, we collect several additional properties of spherical Voronoi clusters, mostly for completeness and clarification.

\begin{lemma} \label{lem:spherical-Voronoi-is-stationary}
A regular spherical Voronoi cluster on $\S^n$ is stationary.  \end{lemma}
\begin{proof}
Immediate from Definition \ref{def:intro-spherical-Voronoi} and Lemma \ref{lem:equivalent-stationary} (as the cluster is of bounded curvature), after recalling that $\c = \n - \k p$. 
\end{proof}

As already commented on in Remark \ref{rem:intro-quasi-centers}, it will follow from Lemma \ref{lem:Mobius-preserves-Voronoi} that the property of being a spherical Voronoi cluster on $\R^n$ does not depend on the particular stereographic projection to $\S^n$ used. Consequently, we may use the standard stereographic projection to compare the spherical Voronoi representations on $\R^n$ and $\S^n$.

\begin{lemma} \label{lem:kc-S-R}
Let $\S^n$ be canonically embedded in $\R^{n+1} = \R^n \times \R$. Let $S^\R$ be a generalized sphere on $\bar \R^n$ and let $S^\S$ be the corresponding sphere on $\S^n$, obtained from $S^\R$ via the standard stereographic projection $T$. For $\M \in \{\R,\S\}$, let $\c^\M$ and $\k^\M$ be the quasi-center and curvature of $S^\M$, respectively. Write $\c^\S = (\underline \c^\S , \c^\S_0) \in \R^n \times \R$. Then:
\begin{equation} \label{eq:ck-S-R}
\c^\R = \underline \c^\S ~,~ \k^\R = \k^\S + \c^\S_0. 
\end{equation}
Furthermore, the spherical cap $B^\S_{\c^\S,\k^\S} := \{ p \in \S^n \; ; \; \scalar{\c^\S,p} + \k^\S \leq 0\}$ on $\S^n$ is stereographically projected via $T^{-1}$ to:
\[
B^\R_{\c^\S,\k^\S} := \set{x \in \R^n \; ; \; \k^\R |x|^2 + 2 \sscalar{\c^\R,x} + 2 \k^\S - \k^\R \leq 0} ,
\]
which is a Euclidean ball of curvature $|\k^\R|$ and center $-\c^\R / \k^\R$ if $\k^\R > 0$, its complement if $\k^\R < 0$, or the halfplane $\{ x \in \R^n \; ; \; \scalar{\c^\R,x} + \k^\S \leq 0\}$ if $\k^\R = 0$. 
\end{lemma}
\begin{remark} \label{rem:ck-S-R-non-invertible}
Note that the mapping $(\c^\S,\k^\S) \mapsto (\c^\R,\k^\R)$ is not invertible when $\k^\R = 0$, as $\k^\S = -\c_0^\S$ cannot be recovered in that case. In particular, the parameters $(\c^\R,\k^\R)$ do not determine $B^\R_{\c^\S,\k^\S}$ uniquely when $\k^\R = 0$. For this reason, Definition \ref{def:intro-spherical-Voronoi} of a spherical Voronoi cluster is given on $\S^n$, with the case of $\R^n$ obtained by stereographic projection. 
\end{remark}
\begin{proof}[Proof of Lemma \ref{lem:kc-S-R}]
Recall the standard formula $p = T(x) = (\frac{2x}{|x|^2 + 1} , \frac{|x|^2 - 1}{|x|^2 + 1})$. It immediately follows that $B^\S_{\c^\S,\k^\S}$ on $\S^n$ is projected to $\set{x \in \R^n \; ; \; (\k^\S + \c^\S_0) |x|^2 + 2 \sscalar{\underline \c^\S,x} + \k^\S - \c^\S_0 \leq 0}$ on $\R^n$. Given (\ref{eq:ck-S-R}), this coincides with $B^\R_{\c^\S,\k^\S}$, and it remains to verify that the quasi-center and curvature of the generalized sphere $\partial B^\R_{\c^\S,\k^\S}$ are indeed $\c^\R$ and $\k^\R$, respectively.
\end{proof}

\begin{lemma} \label{lem:Voronoi-rep-on-Rn}
Let $\Omega^\R$ be a regular spherical Voronoi $q$-cluster on $\R^n$. 
Then there exist vectors $\{ \c^\R_i \}_{i=1,\ldots,q} \subset \R^n$ and scalars $\{ \k^\R_i \}_{i=1,\ldots,q} \subset \R$, so that every non-empty interface $\Sigma^\R_{ij}$ lies on a single (generalized) sphere $S^\R_{ij}$ with quasi-center $\c^\R_{ij} = \c^\R_i - \c^\R_j$ and curvature $\k^\R_{ij} = \k^\R_i - \k^\R_j$. In particular, $\Omega^\R$ is stationary. 
 
Moreover, if $\Omega^\S$ is the spherical Voronoi $q$-cluster on $\S^n$ obtained from $\Omega^\R$ via the standard stereographic projection $T$, having curvature parameters $\{ \k^\S_i \}_{i=1,\ldots,q}$, then the following representation for $\Omega^\R$ holds (where as usual $\k^\S_{ij} := \k^\S_i - \k^\S_j$):
\begin{align} 
\nonumber \Omega^\R_i & = \set{ x \in \R^n \; ; \; \argmin_{j=1,\ldots,q} \; (\k^\R_j |x|^2 + 2 \sscalar{\c^\R_{j},x} + 2 \k^\S_{j} - \k^\R_{j}) = \{i\} } \\
\label{eq:Voronoi-rep-Rn}
 & = \bigcap_{j \neq i} \set{x \in \R^n \; ; \; \k^\R_{ij} |x|^2 + 2 \sscalar{\c^\R_{ij},x} + 2 \k^\S_{ij} - \k^\R_{ij} < 0} ~,~ i = 1,\ldots,q . 
\end{align}
\end{lemma}
\begin{proof}
Denote in addition the quasi-center parameters $\{ \c^\S_i \}_{i=1,\ldots,q} \subset \R^{n+1}$ of the spherical Voronoi cluster $\Omega^\S$ on $\S^n$. Note that by definition, for every non-empty $\Sigma^\S_{ij}$, $\c^\S_{ij} = \c^\S_i - \c^\S_j$ and $\k^\S_{ij} = \k^\S_i - \k^\S_j$ are the quasi-center and curvature of $\Sigma^\S_{ij}$.  Denote $\c^\R_i := \underline{\c}^\S_i$ and $\k^\R_i := \k^\S_i + (\c^\S_0)_i$ for $i=1,\ldots,q$. Then by (\ref{eq:ck-S-R}) and linearity, $\c^\R_{ij} = \c^\R_i - \c^\R_j$ and $\k^\R_{ij} = \k^\R_i - \k^\R_j$ are the quasi-center and curvature (respectively) of every non-empty $\Sigma^\R_{ij}$. 
The representation (\ref{eq:Voronoi}) together with Lemma \ref{lem:kc-S-R} yield the representation (\ref{eq:Voronoi-rep-Rn}). 
Finally, stationarity follows exactly as in Lemma \ref{lem:spherical-Voronoi-is-stationary}.
\end{proof}

\section{Cells are connected when $q \leq n+1$} \label{sec:Voronoi-connected}

This section is dedicated to concluding the proof of Theorem \ref{thm:intro-stable-spherical-Voronoi}, by establishing the connectedness of the cluster's cells when $q \leq n+1$. We also show the almost-obvious implication that Theorem \ref{thm:spherical-Voronoi-prelim} implies \ref{thm:intro-spherical-Voronoi-prelim}. 

We assume throughout this section that $\Omega$ is a bounded stable regular $q$-cluster on the model space $M^n \in \{ \R^n , \S^n \}$ with $\S^0$-symmetry, i.e. reflection symmetry with respect to the equator $M^{n-1} := M^n \cap N^{\perp}$ for some North Pole $N \in M^n$ with $|N|=1$. We know by Theorem \ref{thm:stable-spherical} that $\Omega$ is spherical perpendicularly to its hyperplane of symmetry, and in particular has bounded curvature.

\subsection{Truncated Skew-Fields}

\begin{definition}[Truncated skew scalar-fields]
Given a stationary regular cluster with $\S^0$-symmetry with respect to $N^{\perp}$ and $a \in E^{(q-1)}$, the scalar-field $h^a = \{ h^a_{ij}\}$ given by $h^a_{ij} := a_{ij} \abs{\scalar{N,p}}$ is called a truncated skew-field. 
\end{definition}

We could equally have defined $h^a_{ij} := a_{ij} \scalar{N,p}_+$, using the positive-part instead of the absolute-value operation above, which would better correspond to ``truncation". The truncation operation of a scalar-field on its nodal-domain is inspired by the work of Hutchings--Morgan--Ritor\'e--Ros \cite{DoubleBubbleInR3} in the double-bubble case. 

\begin{proposition} \label{prop:truncated-skew-fields}
A truncated skew scalar-field $h^a = \{ h^a_{ij}\}$ satisfies conformal BCs on $\Sigma^2$ and $Q^0(h^a) = 0$. 
\end{proposition}
\begin{proof} 
Recall from Proposition \ref{prop:skew-fields-properties} that all skew-fields $f^a_{ij} = a_{ij} \scalar{p,N}$ satisfy conformal BCs. However, truncated skew-fields $h^a$ and skew-fields $f^a$ coincide in each of the open hemispheres $M^n_{\pm}$ up to sign, and so the truncated skew-fields $h^a$ also satisfy the conformal BCs outside the equator on $\Sigma^2 \setminus M^{n-1}$. They actually also satisfy the conformal BCs on $\Sigma^2 \cap M^{n-1}$; even though this set has zero $\H^{n-2}$ measure by Remark \ref{rem:equator-perp} and is thus immaterial for the definition of conformal BCs, let us verify this nevertheless. Indeed, since $\c_{ij}^\ell \perp N$ and we have $\n_{ij} = \c_{ij}^\ell + \k_{ij} p$ on $\Sigma^\ell_{ij}$, it follows by continuity of $\n_{ij}$ that $\n_{ij} \perp N$ on $\closure \Sigma_{ij} \cap M^{n-1}$. Consequently $\n_{\partial ij} \perp N$ on $\Sigma^2 \cap M^{n-1}$ as well, and hence $\nabla_{\n_{\partial ij}} h^a_{ij} = 0$ there, as $h^a \equiv 0$ on $M^{n-1}$ by definition; in particular, $\nabla_{\n_{\partial ij}} h^a_{ij} - \bar \II^{\partial ij} h^a_{ij} = 0$ is independent of $(i,j) \in \cyclic(u,v,w)$ on $\Sigma^2 \cap M^{n-1}$, verifying the conformal BCs there. 

A formal argument for establishing that $Q^0(h^a) = 0$ is to use that $h^a_{ij} L_{Jac} h^a_{ij} = 0$ in the distributional sense (since $h^a_{ij}$ and $f^a_{ij}$ coincide up to sign and $L_{Jac} f^a_{ij} = 0$ by Proposition \ref{prop:skew-fields-properties}), and apply Theorem \ref{thm:conformal-Q} \ref{it:conformal-Q-bounded}. However, the latter theorem for $Q^0$ was only proved for $C_c^\infty$ scalar-fields, and $L_{Jac} h^a_{ij}$ is a delta-measure at the equator. To make this rigorous, recall the formula (\ref{eq:Q0}) for $Q^0(h^a) = Q^0(h^a,h^a)$ (by Lemma \ref{lem:cyclic-XYII}): 
\[
Q^0(h^a) = \sum_{i<j} \Big [ \int_{\Sigma_{ij}} \brac{|\nabla^\tang h^a_{ij}|^2 - (\Ric_{g,\mu}(\n,\n) + \|\II\|_2^2) (h^a_{ij})^2} d\mu^{n-1}  - \int_{\partial \Sigma_{ij}} (h^a_{ij})^2 \bar \II^{\partial ij} \, d\mu^{n-2} \Big ] . 
\]
Since the integrands in the above formula for $Q^0(h^a)$ and $Q^0(f^a)$ coincide outside the equator, and since $h^a = f^a \equiv 0$ on the equator and $\H^{n-1}(\Sigma^1 \cap M^{n-1}) = 0$ by Remark \ref{rem:equator-perp}, it follows that the above integrals on the entire $\Sigma^2$ and $\Sigma^1$ are the same for $h^a$ and $f^a$, and therefore $Q^0(h^a) = Q^0(f^a) = 0$ by Proposition \ref{prop:skew-fields-properties}. 
\end{proof}

\subsection{Properties of $\Sigma$ on minimizing clusters}

We assume in addition throughout this section that $\Sigma$ is connected (and thus, by reflection symmetry, must intersect the equator), and that $M^n \setminus \Sigma$ has finitely many connected components. Let us first justify these assumptions for a minimizing cluster.

\begin{lemma} \label{lem:Sigma-connected}
For any isoperimetric minimizing cluster $\Omega$ on $M^n$, $\Sigma$ is connected. 
\end{lemma}
This is well-known, but we provide some details (cf. \cite[Lemma 6.2]{CottonFreeman-DoubleBubbleInSandH}).
\begin{proof}
Otherwise, consider a connected component $\underline \Sigma_1$ of $\Sigma$, and let $y \in \underline \Sigma_1$ and $x \in \underline \Sigma_2 := \Sigma \setminus \underline \Sigma_1$ be two points realizing the (positive) distance between $\underline \Sigma_1$ and $\underline \Sigma_2$. Now move $\underline \Sigma_1$ using an isometry of $M^n$ along the geodesic connecting $y$ to $x$ until it touches $\underline \Sigma_2$ at $x$. The volumes and perimeters of all cells remain unaltered in this process, and so the resulting cluster is still an isoperimetric minimizing cluster, and hence regular by Theorem \ref{thm:regularity}. On the other hand, by a classical lemma of Gromov from \cite{GromovGeneralizationOfLevy} (see \cite[p. 523]{Gromov}), both points $y \in \underline \Sigma_1$ and $x \in \underline \Sigma_2$ must be points whose tangent cones are hyperplanes, and thus belong to $\Sigma^1$. 
Consequently, $\Sigma_{ij}$ and the translated $\Sigma_{jk}$ meet tangentially at the point $x$ in their relative interior, and therefore must coincide in some neighborhood of the meeting point by \cite[Corollary 30.3]{MaggiBook}. But this means that total perimeter has been reduced at the moment of contact, in contradiction to minimality. 
\end{proof}

\begin{lemma} \label{lem:finite-CCs}
For any isoperimetric minimizing cluster $\Omega$ on $M^n$, $M^n \setminus \Sigma$ has finitely many connected components.
\end{lemma} 
\begin{proof} 
For $q=3$, this was verified by Hutchings \cite[Corollary 4.3]{Hutchings-StructureOfDoubleBubbles} on $\R^n$ and extended to all model spaces in \cite[Proposition 4.11]{CottonFreeman-DoubleBubbleInSandH}. For general $q$, one can see this by employing the single-bubble isoperimetric inequality $V(A) \in (0,V(M^n)/2) \Rightarrow P(A) \geq c_n V(A)^{\frac{n-1}{n}}$ on $M^n$ and adapting an argument of Morgan \cite[Theorem 2.3, Step 2]{MorganSoapBubblesInR2}. Assume in the contrapositive that there were an infinite number of bounded connected components $\{\Omega^\ell_b\}_{\ell=1,\ldots,\infty}$, 
and denote $A^\Lambda :=\cup_{\ell=\Lambda}^\infty \Omega^\ell_b$. Since the cluster's total perimeter is finite, the single-bubble isoperimetric inequality (applied to each component) implies that $V(A^{\Lambda})$ is finite and hence tends to $0$ as $\Lambda \rightarrow \infty$. 
Consequently, by the single-bubble isoperimetric inequality, for all $C > 0$ arbitrarily large, we could find $\Lambda$ so that $V(A^\Lambda)$ is arbitrarily small and $P(A^\Lambda) \geq C \cdot V(A^\Lambda)$. 

Now note that each component $\Omega^{\ell}_b$ is assigned to one of the $q$ cells, and that by reassigning the components to different cells, the total perimeter of the resulting cluster can never increase, decreasing where two neighboring components have been assigned to the same cell. Consequently, we can always find an index $i_0 \in \{1,\ldots,q\}$ so that if we reassign all of the components $\{\Omega^\ell_b\}_{\ell=\Lambda}^\infty$ to the $i_0$-th cell, the decrease in perimeter would be at least $\frac{1}{q} P(A^{\Lambda}) \geq \frac{C}{q} V(A^\Lambda)$. However, this would contradict the minimality of the cluster, since it was shown by Almgren \cite[Lemma 13.5]{MorganBook5Ed} that for any stationary regular cluster $\Omega$, there exist constants $C_0,\eps > 0$ so that arbitrary volume adjustments with $v := \sum_{i=1}^q V(\Omega'_i \Delta \Omega_i) \leq \eps$ may be accomplished inside fixed small balls (intersecting a finite number of connected components $\{\Omega_1^k\}_{k=1,\ldots,K}$ of $M^n \setminus \Sigma$) so that $P(\Omega') \leq P(\Omega) + C_0 v$. Consequently, it would be more efficient to reassign the components comprising $A^\Lambda$ for $\Lambda$ large enough to the $i_0$-th cell and compensate for the change in volumes as above. Note that $\Lambda$ can always be chosen large enough so as to ensure that $A^\Lambda$ is disjoint from $\{\Omega_1^k\}_{k=1,\ldots,K}$. 

This concludes the proof on $\S^n$ since all components are bounded. On $\R^n$, as an isoperimetric minimizer is always bounded by Theorem \ref{thm:existence}, it follows that there is only one unbounded connected component of $\R^n \setminus \Sigma$. 
\end{proof}

\subsection{Component adjacency graph is $2$-connected}  \label{subsec:adjacency-graph}

Let us denote by $\{ \underline \Omega_\ell \}_{\ell = 1 , \ldots, \Lambda}$ the totality of all connected components of all of $\Omega$'s (open) cells modulo their common $\S^0$-symmetry (recall Definition \ref{def:cc-S0}). Consider the cluster $\underline \Omega = (\underline \Omega_1,\ldots,\underline \Omega_\Lambda)$; as $\underline \Sigma = \Sigma$, it is still bounded stationary with $\S^0$-symmetry, spherical perpendicularly to its hyperplane of symmetry, and in addition, all of its cells are connected modulo its $\S^0$-symmetry. Consequently, it is in addition perpendicularly spherical Voronoi by Theorem \ref{thm:spherical-Voronoi-prelim}. In particular, $\underline \Omega^\B_\ell$ is an open convex subset of $\B^n$, and $\underline \Omega^\S_\ell$ is of the form $\S^n \cap \Pi^{-1} P_\ell$ for some open convex polyhedron $P_\ell \subset N^{\perp} \subset \R^{n+1}$. 

\medskip

We already know that $M^n \setminus \cup_{\ell \in I} \underline \Omega_\ell$ is connected for any subset $I \subset \{1,\ldots,\Lambda\}$, since otherwise it would mean 
that $\Sigma = M^n \setminus \cup_{\ell=1,\ldots,\Lambda} \underline \Omega_\ell$ is disconnected, in contradiction to our standing assumption in this section. To this we add:

\begin{lemma} \label{lem:connected}
For any connected (modulo $\S^0$-symmetry) component $\underline \Omega_\ell$, $M^n \setminus \overline{(\underline \Omega_\ell)}$ is connected. 
\end{lemma}

While the difference between the connectedness of $M^n \setminus \underline \Omega_\ell$ and the latter statement is subtle, it nevertheless requires proof. Moreover, simple examples show that variations on this statement might actually be false for higher-dimensional minimizing cones (such as those over a hypercube, see \cite{Brakke-MinimalConesOnCubes}). Since the latter is a topological property, it is enough to establish it on $\S^n$, and deduce it on $\R^n$ via stereographic projection and Lemma \ref{lem:stereo}. 

Recall by regularity that the (lower) density $\Theta$ of a component at any point in its closure is strictly positive (the conclusion of Lemma \ref{lem:infiltration}). Consequently, Lemma \ref{lem:connected} is a consequence of the representation  $\underline \Omega^\S_\ell = \S^n \cap \Pi^{-1} P_\ell$ and the following:

\begin{lemma} \label{lem:density0} 
Suppose $P \subset \R^{n+1}$ is an open convex polyhedron. If $\S^n \setminus P$ is connected while $\S^n \setminus \overline P$ is disconnected, then there exists some $p \in \overline P \cap \S^n$ such that $\Theta(P \cap \S^n,p) = 0$.
\end{lemma}

\begin{proof}
    Since $\S^n \setminus \overline P$ is disconnected and open, we may find
    open, disjoint sets $U_1, U_2 \subset \S^n$ such that $U_1 \cup U_2 = \S^n
    \setminus \overline P$.
    Since $\S^n \setminus P$ is connected, $\overline{U_1}$ and $\overline{U_2}$
    are not disjoint: if they were disjoint, then by compactness they would have positive distance,
    and then by taking open neighborhoods of $\overline{U_1}$ and $\overline{U_2}$ we would have
    open disjoint sets covering $\S^n \setminus P$. So
    choose $p \in \overline{U_1} \cap \overline{U_2}$.
    For every neighborhood $W$ of $p$ the sets $U_1 \cap W$ and $U_2 \cap W$
    are open disjoint sets covering $W \setminus \overline P$, and whose closures overlap at $p$ (and possibly elsewhere).

    Let $\n_1, \dots, \n_k$ be the outward normals supporting the facets of $P$ that meet at $p$,
    so that on a sufficiently small ball $W$ around $p$, $P$ coincides with $\{x \in \R^{n+1} \; ;\; \sscalar{\n_i,x - p} < 0 \text{ for all $i$}\}$. Clearly, $k \geq 2$. 
    Let $H_{i,+} = \{x \in \R^{n+1}\; ; \; \sscalar{\n_i,x - p} > 0\}$. Then $W \setminus \overline P = W \cap \bigcup_i H_{i,+}$.
    We may assume that no $\n_i$ is parallel to $p$: if $\n_i = p$ then the corresponding facet of $P$ supports $\S^n$,
    and so removing it from $P$ will not change $P \cap \S^n$ or $\overline P \cap \S^n$, while if $\n_i = -p$ then $\overline P \cap \S^n = \{p\}$ which does not disconnect $\S^n$. Denoting by $\Pi_p$ the orthogonal projection onto $T_{p} \S^n$, it follows that $\Pi_p \n_i \neq 0$ for all $i$. For each $i$, the set $W \cap H_{i,+}$ is connected.  
    As long as $\Pi_{p} \n_i$ and $\Pi_{p} \n_j$ are not negative multiples of one another, $H_{i,+} \cap W$
    and $H_{j,+} \cap W$ overlap, and hence $(H_{i,+} \cup H_{j,+}) \cap W$ is connected. Thus, in order for
    $W \cap \bigcup_i H_{i,+}$ to be disconnected there must be some non-zero $v \in T_{p} \S^n$ such that every $\Pi_p \n_i$ 
    is parallel to $v$. Moreover, there must be some $i$ and $j$ for which $\Pi_p \n_i$ is a positive multiple of $v$ and $\Pi_p \n_j$
    is a negative multiple of $v$. It follows that $\S^n \setminus \overline{P}$ contains $\S^n \cap (H_{i,+} \cup H_{j,+})$,
    and since $\S^n \cap H_{i,+}$ and $\S^n \cap H_{j,+}$ are two spherical caps meeting tangentially at $p$, it follows that
    $\S^n \cap (H_{i,+} \cup H_{j,+})$ has density 1 at $p$. Therefore, $P \cap \S^n$ has density $0$ at $p$.
\end{proof}

\begin{definition}[Component Adjacency Graph]
The ($\S^0$-symmetric) component adjacency graph is the undirected graph whose vertices are $\{1,\ldots,\Lambda\}$, and there is an edge between vertex $i$ and $j$ iff $\underline \Sigma_{ij} \neq \emptyset$. 
\end{definition}

As a consequence of Lemma \ref{lem:connected}, we deduce:
\begin{lemma} \label{lem:2-connected}
The component adjacency graph is $2$-connected: removing any single vertex does not disconnect it. 
\end{lemma}
\begin{proof}
If removing vertex $\ell$ and its edges disconnects the graph, this would mean that $M^n \setminus (\underline\Omega_\ell \cup (\partial \underline \Omega_\ell \cap \underline \Sigma^1))$ is path-disconnected if one only uses paths which pass through the open components and $\underline \Sigma^1$. Consequently, this is equivalent to the statement that $M^n \setminus (\underline \Omega_\ell \cup \partial \underline \Omega_\ell)$ is disconnected using paths as above. But the latter is an open set, and so any path can always be slightly perturbed to have the property that it only passes through the open components and $\underline \Sigma^1$ (by compactness). So this means that $M^n \setminus \overline{(\underline \Omega_\ell)}$ is disconnected (as pathwise and regular connectedeness are the same for (open) manifolds),  in contradiction to Lemma \ref{lem:connected}. Consequently, removing the vertex could not have disconnected the component adjacency graph. 
\end{proof}

The following is an immediate corollary. Recall that $G'$ is called a quotient graph of $G$ if it is obtained by identifying disjoint subsets of vertices from $G$.

\begin{corollary} \label{cor:doesnot-disconnect}
Let $G'$ be a quotient of the component adjacency graph $G$. Then removing a vertex of $G'$ which was a single vertex in the original graph $G$ does not disconnect $G'$.
\end{corollary}

\subsection{Equatorial cells are connected} 

\begin{definition}
A cell $\Omega_i$ is called equatorial if it intersects the equator: $\Omega_i \cap M^{n-1} \neq \emptyset$. 
\end{definition}

\begin{proposition} \label{prop:equatorial}
An equatorial cell is connected. 
\end{proposition}

For the proof, we require a type of strong maximum principle for the discrete Laplacian $L_A$ (recall Definition \ref{def:LA}). Let $G$ denote an undirected connected graph on the finite set of vertices $\{1,\ldots,V\}$, $V \geq 2$. Recall that we write $i \sim j$ if there is an edge between vertices $i$ and $j$. Let  $A = \{ A^{ij} > 0\}_{i \sim j}$ be strictly positive non-oriented weights supported along the edges of $G$. Recall from Lemma \ref{lem:LA-positive} that $L_A > 0$ is positive-definite on $E^{(V-1)}$, and in particular full-rank.

\begin{lemma} \label{lem:strong-max-principle}
Let $a \in E^{(V-1)}$ be the solution to $L_A a = e_1 - e_V$. Assume that removing vertex $1$ does not disconnect the graph $G$. Then $a_1 > a_i$ for all $i=2,\ldots,V$. 
\end{lemma}
\begin{proof}
Let $W$ be the subset of vertices $i$ for which $a_i$ is maximal. Note that by definition of the discrete Laplacian, $(L_A a)_i$ is the weighted average of $a_i - a_j$ over all $j \sim i$ (with weights $A^{ij} > 0$). Consequently, we have $(L_A a)_i \geq 0$ for every $i \in W$, and since $(L_A a)_{V} = -1$ we see that $V \notin W$. As $W$ is a proper subset of $\{1,\ldots,V\}$ and the graph is connected, we also have $\sum_{i \in W} (L_A a)_i = \sum_{i \in \partial W} (L_A a)_i > 0$, where $\partial W \subset W$ denotes those vertices in $W$ which have a neighbor outside of $W$. As the only vertex $i$ for which $(L_A a)_i > 0$ is $i=1$, it follows that necessarily $1 \in W$. Finally, since for a vertex $i  \in W$ we have $(L_A a)_i > 0$ if and only if $i \in \partial W$, we deduce that $\partial W = \{1\}$. It follows that removing the vertex $1$ disconnects $W \setminus \{1\}$ from $\{1,\ldots,V\} \setminus W \neq \emptyset$. Since we assumed this is impossible, we must have $W \setminus \{1\} = \emptyset$, i.e. $W = \{1\}$, thereby concluding the proof. 
\end{proof}

\begin{proof}[Proof of Proposition \ref{prop:equatorial}]
Fix an equatorial cell, which after relabeling, we can assume is $\Omega_1$. Assume in the contrapositive that it has more than one connected component. Let $\tilde \Omega_1$ denote an equatorial connected component of $\Omega_1$, and set $\tilde \Omega_{q+1} = \Omega_1 \setminus \tilde \Omega_1$. Note that $\tilde \Omega_1$ is also an $\S^0$-symmetric connected component of $\Omega_1$ since it is equatorial. Denoting $\tilde \Omega_i = \Omega_i$ for $i=2,\ldots,q$, the resulting $\tilde \Omega$ is a $(q+1)$-cluster with $\S^0$-symmetry. 
 Note that $\partial \Omega_1 = \partial \tilde \Omega_1 \cup \partial \tilde \Omega_{q+1}$ and that $\partial^* \Omega_1 = \partial^* \tilde \Omega_1 \cup \partial^* \tilde \Omega_{q+1}$. Consequently $\tilde \Sigma = \Sigma$, $\tilde \Omega$ is still a stationary and regular cluster, and $\tilde \Sigma^k = \Sigma^k$ for $k=1,2,3,4$. In particular, a scalar-field $\tilde h = \{ \tilde h_{ij} \}_{i,j=1,\ldots,q+1}$ on $\tilde \Sigma$ induces the obvious scalar-field $h = \{ h_{ij} \}_{i,j=1,\ldots,q}$ on $\Sigma$ by setting $h_{ij} = \tilde h_{ij}$ for $i,j \in \{2,\ldots,q\}$, $h_{1 j}(p) = \tilde h_{1 j}(p)$ for $p \in \tilde \Sigma_{1 j}$ and $h_{1 j}(p) = \tilde h_{(q+1) j}(p)$ for $p \in \tilde \Sigma_{(q+1) j}$ (and requiring $h_{ji} = -h_{ij}$). Consequently, $\delta^1_{h} V(\Omega_j) = \delta^1_{\tilde h} V(\tilde \Omega_j)$ for $j=2,\ldots,q$, $\delta^1_{h} V(\Omega_1) = \delta^1_{\tilde h} V(\tilde \Omega_1) + \delta^1_{\tilde h} V(\tilde \Omega_{q+1})$, and $Q^0_{\Sigma}(h) = Q^0_{\tilde \Sigma}(\tilde h)$. 
 
Now, consider a truncated skew scalar-field $\tilde h^a = \{ \tilde h^a_{ij} = a_{ij} \abs{\scalar{p,N}} \}$ on $\tilde \Omega$ for some appropriate $a \in E^{(q)}$ to be chosen. We will show that there exists $a \in E^{(q)}$ so that $\delta^1_{\tilde h^a} V(\tilde \Omega) = e_1 - e_{q+1}$ and so that $a_{1 j} > 0$ for all $j=2,\ldots,q$. Since $Q^0_{\tilde \Sigma}(\tilde h^a) = 0$ by Proposition \ref{prop:truncated-skew-fields}, it will follow that $Q^0_{\Sigma}(h^a) = 0$ and $\delta^1_{h^a} V(\Omega) = 0$. Since $\Omega$ is stable and of bounded curvature, it will follow by elliptic regularity (Lemma \ref{lem:elliptic-regularity}) that necessarily $h^a$ (and hence $\tilde h^a$) is $C^\infty$-smooth on the entire (relatively open) $\Sigma^1 = \tilde \Sigma^1$. But since $\tilde h^a_{1 j} = a_{1 j} \abs{\scalar{p,N}}$ on $\tilde \Sigma_{1 j}$ and $a_{1 j} \neq 0$, we see that $\tilde h^a_{1 j}$  is necessarily non-smooth wherever $\tilde \Sigma_{1 j}$ intersects the equator (as it does so perpendicularly by Remark \ref{rem:equator-perp}). To reach the desired contradiction, it remains to show that $\tilde \Sigma_{1j} \cap M^{n-1} \neq \emptyset$ for some $j=2,\ldots,q$, and that an $a \in E^{(q)}$ as above may always be found.

To establish the first point, let $\E \subset \{1 , \ldots, q+1\}$ denote the indices of the equatorial cells among the cells of $\tilde \Omega$. As  $\tilde \Omega_1$ is equatorial, we have $1 \in \E$. Note that we must have $|\E|\geq 2$, since otherwise this would mean that $\overline{\tilde \Omega_1}$ covers the entire equator $M^{n-1}$, and so by reflection symmetry, $M^n \setminus \overline{\tilde \Omega_1} \neq \emptyset$ is disconnected, in contradiction to Lemma \ref{lem:connected}. 
Consider the undirected graph $G_{\E}$ on the set of vertices $\E$, having an edge between $i,j \in \E$ iff  $\tilde \Sigma_{ij} \cap M^{n-1} \neq \emptyset$, i.e. $\tilde \Omega_i$ and $\tilde \Omega_j$ are adjacent through an interface passing \emph{through the equator}. As $M^{n-1}$ is connected and thus satisfies a single-bubble isoperimetric inequality $\vol_{M^{n-1}}(A) \in (0,\vol(M^{n-1})) \Rightarrow P_{M^{n-1}}(A) > 0$, 
a repetition of the argument used to establish Lemma \ref{lem:LA-connected} implies that $G_\E$ must be connected. Since  $|\E| \geq 2$, the vertex $1 \in \E$
must have some neighbor $j \in \E \setminus \{1\}$ in $G_\E$, implying that  $\tilde \Sigma_{1j} \cap M^{n-1} \neq \emptyset$, as required. 

As for the second remaining point, let $\tilde G$ be the adjacency graph of $\tilde \Omega$, defined on the vertices $\{1,\ldots,q+1\}$ and having an edge between $i$ and $j$ iff $\tilde \Sigma_{ij} \neq \emptyset$. Note that $\tilde G$ is connected by Lemma \ref{lem:LA-connected}. Also note that $\tilde G$ is a quotient graph of the component adjacency graph $G$ (by reflection symmetry), and that the vertex $1$ representing the single equatorial connected component $\tilde \Omega_1$ is a vertex of both $\tilde G$ and $G$. Consequently, by Corollary \ref{cor:doesnot-disconnect}, removing the vertex $1$ from $\tilde G$ does not disconnect it. Now set $\tilde A^{ij} := \int_{\tilde \Sigma_{ij}} \abs{\scalar{p,N}} dp$ for $i,j=1,\ldots,q+1$ and note that $\tilde A = \{\tilde A^{ij} > 0\}_{i \sim j}$ is supported along the edges of $\tilde G$ and strictly positive on them (as $\H^{n-1}(\tilde \Sigma_{ij}) > 0$ iff $i \sim j$ and $\H^{n-1}(\tilde \Sigma_{ij} \cap M^{n-1}) = 0$ by Remark \ref{rem:equator-perp}). 
Now observe that:
\[
\delta^1_{\tilde h^a} V(\tilde \Omega_i) = \sum_{j \neq i} \int_{\tilde \Sigma_{ij}} (a_i - a_j) \abs{\scalar{p,N}} dp = (L_{\tilde A} a)_i  . 
\]
By Lemma \ref{lem:LA-positive}, $L_{\tilde A}$ is positive-definite and in particular invertible on $E^{(q)}$.
By Lemma \ref{lem:strong-max-principle}, the solution $a \in E^{(q)}$ to the linear equation $\delta^1_{\tilde h^a} V(\tilde \Omega) = L_{\tilde A} a = e_1 - e_{q+1}$ satisfies $a_{1 j} > 0$ for all $j=2,\ldots,q+1$. This concludes the proof. 
\end{proof}

\subsection{All cells are equatorial whenever $q \leq n+1$}

\begin{proposition} \label{prop:all-cells-are-equatorial}
Assume in addition that $q \leq n+1$. Then all cells are equatorial, and therefore connected. 
\end{proposition}
\begin{proof}
    Note that it is enough to establish it on $\S^n$, and deduce it on $\R^n$ via stereographic projection and Lemma \ref{lem:stereo}. 
        Assume      in the contrapositive that there are $s \leq q-1 \leq n$ equatorial cells. We may assume that
    $\Omega^\S_1, \dots, \Omega^\S_s$ are equatorial and that $\Omega^\S_{s+1}, \dots, \Omega^\S_{q}$ are not.     
    Consider the open convex polyhedron $A \subset N^\perp$ given by
    \[
        A := \set{ p \in N^\perp\; ; \; \sscalar{\c^\S_{q i}, p} +\k^\S_{q i} < 0 \;\; \forall i=1, \dots, s }.
    \]
    As $\Omega^\S_q$ is assumed not to intersect the equator, we have $\Pi \Omega^\S_q \subset \B^n$, where recall $\Pi : \S^n \rightarrow N^{\perp}$ denotes the orthogonal projection. By the perpendicular Voronoi representation~\eqref{eq:Voronoi} we have $\Pi^{-1} (A \cap \B^n)  \supset \Omega^\S_q$, and so since $\Omega^\S_q$ is non-empty then neither is $A \cap \B^n$. In addition, $\Pi^{-1} A$ is disjoint from $\bigcup_{i=1}^s \Omega^\S_i$, because $p \in \Pi^{-1} A$ implies that $\sscalar{\c^\S_{q},p} + \k_{q} < \sscalar{\c^\S_i,p} + \k_i$ for all $i = 1, \dots, s$.
    
    Now choose a non-zero $v \in N^\perp$ such that $\sscalar{\c^\S_{qi}, v} \leq 0$ for all $i = 1, \dots, s$; for example, $v$ may be
    found by taking a unit vector in the orthogonal complement of $\c^\S_{q1}$ through $\c^\S_{q,s-1}$ 
    (which is possible because $s-1 \leq q-2 \leq n-1$), and then choosing the sign of $v$ to ensure that $\sscalar{\c^\S_{qs},v} \leq 0$. 
                            Let $p \in A \cap \B^n$, and hence $p + \alpha v \in A$ for all $\alpha > 0$.
    Since $p \in \B^n$, there is some $\alpha > 0$ with $p + \alpha v \in \S^n$. For this $\alpha$, $p + \alpha v$ belongs to the equator $\S^{n-1} = \S^n \cap N^\perp$ and to the set $A$. Therefore, $A \cap \S^{n-1}$ is a non-empty, relatively open subset of $\S^{n-1}$ which is disjoint from $B := \bigcup_{i=1}^s \Omega_i \cap \S^{n-1}$, and hence disjoint from its closure $\overline{B}$.  
    
    It remains to note that as $\H^{n-1}(\Sigma \cap \S^{n-1}) = 0$ by Remark \ref{rem:equator-perp}, the set $\Sigma \cap \S^{n-1}$ must have empty relative interior in $\S^{n-1}$, and therefore $\S^{n-1} \setminus \Sigma =  \cup_{i=1}^q \Omega_i \cap \S^{n-1}$ is dense in $\S^{n-1}$. But as $\Omega_1, \dots, \Omega_s$ were assumed to be the only equatorial cells, it follows that $\overline{B} = \S^{n-1}$, yielding the desired contradiction to the non-emptiness of $A \cap \S^{n-1}$. 
    \end{proof}

\subsection{Concluding a minimizing cluster is spherical Voronoi}

We are finally ready to conclude the proofs of Theorems \ref{thm:intro-spherical-Voronoi-prelim} and \ref{thm:intro-stable-spherical-Voronoi}. 

\begin{proof}[Proof of Theorem \ref{thm:intro-spherical-Voronoi-prelim}]
Let $\Omega$ be a bounded stationary regular cluster with $\S^0$-symmetry which is spherical perpendicularly to its hyperplane of symmetry, and whose (open) cells have finitely many connected components. Let $\underline \Omega$ be the cluster whose cells are the connected components of $\Omega$'s (open) cells modulo their common $\S^0$-symmetry. As $\underline \Sigma = \Sigma$, $\underline \Omega$ is still a bounded stationary regular cluster with $\S^0$-symmetry which is spherical perpendicularly to its hyperplane of symmetry, whose (finitely many) cells are all connected modulo their common $\S^0$-symmetry. It follows by Theorem \ref{thm:spherical-Voronoi-prelim} that $\underline \Omega$ is a perpendicularly spherical Voronoi cluster. 

In particular, this holds when $\Omega$ is a minimizing cluster with $\S^0$-symmetry, all of whose cells are non-empty, as it is bounded, stationary, stable and regular by the assertions of Section \ref{sec:prelim}, spherical perpendicularly to its hyperplane of symmetry by Theorem \ref{thm:stable-spherical}, and has finitely many connected components by Lemma \ref{lem:finite-CCs}. By removing the empty cells, we may reduce to the case that all cells are non-empty, thereby concluding the proof. 
\end{proof}

\begin{proof}[Proof of Theorem \ref{thm:intro-stable-spherical-Voronoi}]
Fix $q \leq n+1$, and let $\Omega$ be a bounded stable regular $q$-cluster with $\S^0$-symmetry so that $\Sigma$ is connected and $M^n \setminus \Sigma$ has finitely many (open) connected components. By Theorem \ref{thm:stable-spherical}, it is spherical perpendicularly to its hyperplane of symmetry. By Propositions \ref{prop:equatorial} and \ref{prop:all-cells-are-equatorial}, all cells of $\Omega$ are equatorial and connected. 
 It follows by Theorem \ref{thm:spherical-Voronoi-prelim} that $\Omega$ is a perpendicularly spherical Voronoi cluster. 

Let us now turn our attention to establishing the assertion when $\Omega$ is an isoperimetric minimizing $q$-cluster on $M^n$ with $q \leq n+1$. By removing empty cells, we reduce to the case that all cells are non-empty. As explained in the Introduction, when $q \leq n+1$, one can always find a hyperplane $N^{\perp}$ bisecting all finite-volume cells of $\Omega$ (after a possible translation of $\Omega$ when $M^n = \R^n$). Reflecting each of the two halves $\Omega_{\pm}$ of $\Omega$ about this hyperplane, we obtain two $q$-clusters $\bar \Omega_{\pm}$ with $\S^0$-symmetry and $V(\bar \Omega_{\pm}) = V(\Omega)$, each of which must be minimizing (otherwise a contradiction to the minimality of $\Omega$ would follow). In particular, each $\bar \Omega_{\pm}$ is regular and stable by the assertions of Section \ref{sec:prelim}. 
By Lemmas \ref{lem:Sigma-connected} and \ref{lem:finite-CCs}, each $\bar \Sigma_{\pm}$ is connected and $M^n \setminus \bar{\Sigma}_{\pm}$ has finitely many connected components. It follows by the first part of the proof that both $\bar \Omega_{\pm}$ are perpendicularly spherical Voronoi clusters, all of whose cells are equatorial and connected. Consequently, the cells of the original cluster $\Omega$ must also intersect the equator and be connected in each hemisphere $M^n_{\pm}$, and are therefore equatorial and connected on the entire $M^n$. In additional, both halves $\Omega_{\pm}$ of the original cluster $\Omega$ are perpendicularly spherical Voronoi in each hemisphere $M^n_{\pm}$, and in particular $\Sigma_{ij} \cap M^n_{\pm}$ lies in a (generalized) sphere with quasi-center $\c^{\pm}_{ij} = \c^{\pm}_i - \c^{\pm}_j \in N^{\perp}$ and curvature $\k_{ij} = \k_i - \k_j$ with $\k = \frac{1}{n-1} \lambda$ (where $\lambda$ is the corresponding Lagrange multiplier from $\Omega$'s stationarity). As $\Sigma_{ij}$ is a smooth (relatively open) manifold, it follows that we must have $\c_{ij}^{+} = \c_{ij}^{-}$ for every interface $\Sigma_{ij}$ which intersects the equator. Repeating the argument used in the proof of Proposition \ref{prop:equatorial}, the undirected graph $G_{\E}$ on the set of vertices $\E = \{1,\ldots,q\}$ of all (equatorial) cells, having an edge between $i,j$ iff  the cells are adjacent \emph{on the equator} (i.e. $\Sigma_{ij} \cap M^{n-1} \neq \emptyset$), must be connected. Since we know that $\c_i^+ - \c_i^- = \c_j^+ - \c_j^-$ for all $i \sim j$ in $G_\E$, it follows (recalling our convention that $\sum_{i=1}^q \c_i^+ = \sum_{i=1}^q \c_i^- = 0$) that $\c_i^+ = \c_i^-$ for all $i=1,\ldots,q$. This means that we have exactly the same interfaces for the cells in both hemispheres, and therefore $\Omega$ itself is $\S^0$-symmetric about $N^{\perp}$, and is thus perpendicularly spherical Voronoi. This concludes the proof. 
\end{proof}

\begin{remark}
It is easy to see that Theorem \ref{thm:intro-stable-spherical-Voronoi} would be false without requiring that $\Sigma$ be connected (or some other global assumption on the cluster). Indeed, consider a stable $q$-cluster (e.g. a standard model one) so that $q-1$ of its cells are contained in an open hemisphere, and reflect it about the equator, thereby obtaining a $q$-cluster which is stable and yet has disconnected cells. The a-priori assumption on having finitely many connected components is purely for convenience and can  most-likely be removed, but we do not insist on this here.
\end{remark}

\section{M\"obius Geometry and Standard Bubbles} \label{sec:Mobius}

We have already introduced M\"obius automorphisms of $\bar \R^n$ and $\S^n$ in Subsection \ref{subsec:dilation}. 
Recall that $\bar \R^n$ denotes the one-point-at-infinity compactification of $\R^n$, diffeomorphic to $\S^n$ via stereographic projection. Note the degrees of freedom when selecting the stereographic projection: once the North pole is chosen on $\S^n \subset \R^{n+1}$, one is allowed to isometrically embed $\R^n$ in $\R^{n+1}$ perpendicularly to the North-South axis arbitrarily (as long as it doesn't pass through the North pole).
Set $\bar \S^n = \S^n$ and let $\{ \M_1 , \M_2 \} = \{\R , \S\}$. 
The composition of two stereographic projections, the first from $\bar \M_1^n$ to $\bar \M^n_2$ and the second from $\bar \M^n_2$ back to $\bar \M_1^n$, is a M\"obius automorphism of $\bar \M_1^n$, and in fact any M\"obius automorphism of $\bar \M_1^n$ may be obtained in this manner.
The standard bubbles on $\R^n$ and $\S^n$ are thus generated from the equal-volume bubble by M\"obius automorphisms (in the case of $\R^n$, automorphisms which send a point in the interior of $\Omega_q$ to infinity, to ensure that $\Omega_q$ remains the unique unbounded cell).

It will therefore be useful to obtain a more explicit description of M\"obius transformations. For additional background on M\"obius geometry we refer to \cite{Udo-MobiusDifferentialGeometry,GelfandEtAl-RepresentationOfLorentzGroups,Wilker-InversiveGeometry}.

\subsection{M\"obius Geometry on $\S^n$}

Let us identify $\S^n$ with its canonical embedding in $\R^{n+1} \times \{1\} \subset \R^{n+2}$, namely $\{ \bar p = (p,1) \; ; \; \sum_{i=1}^{n+1} p_i^2 = 1 \}$. On $\R^{n+2}$ we introduce the Minkowski metric $dx_1^2 + \ldots + dx_{n+1}^2 - dx_{n+2}^2$ of signature $(n+1,1)$, and denote the corresponding scalar-product by:
\[
 \scalar{x,y}_1 := \sum_{i=1}^{n+1} x_i y_i - x_{n+2} y_{n+2} .
\]
 The Minkowski light cone is defined as $L^{n+1}_1 := \{ x \in \R^{n+2} \; ; \; \scalar{x,x}_1 = 0 \}$. Note that $L^{n+1}_1 \cap (\R^{n+1} \times \{ 1\}) = \S^n$. Consequently, if we identify points $x$ in $\R^{n+2}$ belonging to the same line from the origin $[x]$, i.e. consider $\R^{n+2}$ as the homogeneous coordinates of real projective space $\RP^{n+1} \simeq \R^{n+2} / [\cdot]$, we see that $\S^n$ is identified with the ``celestial sphere" $L^{n+1}_1 / [\cdot] \subset \RP^{n+1}$. Under this identification, the group $\PGL(n+1)$ of (non-singular) projective transformations of $\RP^{n+1}$ coincides with $\GL(n+2) / \R^*$, the group of (non-singular) linear transformations of $\R^{n+2}$ modulo (non-zero) scaling. 

Denote by $O_1(n+2)$ the Lorentz group of linear isometries of Minkowski space $(\R^{n+2},\scalar{\cdot,\cdot}_1)$, namely:
\[
O_1(n+2) = \{ U \in \GL(n+2) \; ; \; U^T J U = J \} ~,~ J := \text{diag}(1,\ldots,1,-1) \in \GL(n+2) . 
\]
 It is easy to check that $U \in O_1(n+2)$ iff $U^T \in O_1(n+2)$, and that $(a,b) \in \R^{n+1} \times \R$ is the last row of some $U \in O_1(n+2)$ if and only if:
 \begin{equation} \label{eq:ba1}
 b^2 = |a|^2 + 1 . 
 \end{equation}
 Note that any Lorentz transformation preserves the light cone $L^{n+1}_1$, and so considered as a projective map, it preserves $\S^n$. Conversely, it is known that any element of $\GL(n+2)$ which preserves $L^{n+1}_1$ is a multiple of a Lorentz transformation. It follows that the subgroup of projective maps which preserve $\S^n$ coincides with the projectivization $O_1(n+2) / \{-1,+1\}$ (as $\det U \in \{-1,+1\}$ for all $U \in O_1(n+2)$).

It will be convenient to restrict to time-orientation-preserving maps. Consequently, we introduce $\RP^{n+1}_+$ as the quotient space of $\R^{n+2}_+ = \R^{n+2} \cap \{ x_{n+2} > 0\}$ via the identification of all points $x$ belonging to the same open positive ray from the origin $[x]_+ = \{ \lambda x \; ; \; \lambda > 0\}$ (making $\RP^{n+1}_+$ diffeomorphic to $\R^{n+1}$). The sphere $\S^n$ is again identified with $L^{n+1}_{1,+} / [\cdot]_+ \subset \RP^{n+1}_+$, where $L^{n+1}_{1,+}$ denotes the future light cone $L^{n+1}_1 \cap \R^{n+2}_+$. The orthochronus Lorentz group $O_1^+(n+2)$ is the index-two normal subgroup of $O_1(n+2)$  which preserves time-orientation, i.e. maps $L^{n+1}_{1,+}$ onto itself -- it is given by:
\[
O_1^+(n+2) = \{ U \in O_1(n+2) \; ; \; U_{n+2,n+2} > 0 \} = \{ U \in O_1(n+2) \; ; \; U_{n+2,n+2} \geq 1 \} 
\]
(where the last transition is due to (\ref{eq:ba1})). 
Note that $\Phi : O_1(n+2) \ni U \mapsto \text{sign}(U_{n+2,n+2}) \in \{-1,+1\}$ is a group homomorphism and $O^+_1(n+2)$ is its kernel. 
One should not confuse between $O_1^+(n+2)$ and the proper Lorentz subgroup $SO_1(n+2) = \{ U \in O_1(n+2) \; ; \; \det U = +1\}$, which is only isomorphic to $O_1^+(n+2)$. Their intersection is the restricted Lorentz group $SO^+_1(n+2)$, which is one of the two connected components of $O_1^+(n+2)$. 

A fundamental result in the M\"obius geometry of $\S^n$ is that the group of M\"obius automorphisms on $\S^n$ is isomorphic to the group of projective transformations of $\RP^{n+1}$ which preserve the quadric $\S^n \subset \RP^{n+1}$, or equivalently, to the orthochronus Lorentz group $O^+_1(n+2)$ \cite[Section 1.3]{Udo-MobiusDifferentialGeometry}. In other words, every $U \in O^+_1(n+2)$ gives rise to a M\"obius automorphism $T = T_U$ of $\S^n \subset \RP^{n+1}_+$ via $T(x) = [U x]_+$, and conversely, any such $T$ can be written in the latter form for some unique $U = U_T \in O^+_1(n+2)$.

\subsection{Projective representation of spherical Voronoi clusters}

\begin{definition}[Homogeneous spherical Voronoi parameters]
If $\{ \c_i \}_{i=1,\ldots,q} \subset \R^{n+1}$ and $\{ \k_i \}_{i=1,\ldots,q} \subset \R$ are the quasi-center and curvature parameters associated to a spherical Voronoi $q$-cluster on $\S^n$ (recall Definition \ref{def:intro-spherical-Voronoi}), its homogeneous spherical Voronoi parameters are defined as $\{ \ck_i \}_{i=1\ldots,q} \subset \R^{n+2}$, with:
\[
\ck_i := (\c_i , -\k_i) \in \R^{n+2} . 
\]
\end{definition}

Employing the above identification between $\S^n$ and $L^{n+1}_{1,+} / [\cdot]_+ \subset \RP^{n+1}_+$, we obtain the following useful projective description of the cells of a spherical Voronoi $q$-cluster $\Omega$ on $\S^n$:
\begin{equation} \label{eq:Voronoi-projective-rep}
\Omega_i = \set{ [\bar p]_+ \in L^{n+1}_{1,+} / [\cdot]_+ \; ; \; \argmin_{j=1,\ldots,q} \scalar{[\bar p]_+ , \ck_j}_1 = \{i\} } ~,~ i = 1,\ldots, q . 
\end{equation}
Note that the above $\argmin$ is indeed well-defined in the sense that one can use any representative $\bar p$ on the open positive ray $[\bar p]_+$ to evaluate it. By using the representative $\bar p = (p,1)$ with $p \in \S^n$, we see the equivalence with the original representation (\ref{eq:intro-Voronoi-rep}), after noting that $\scalar{(p,1) , \ck_j}_1 = \scalar{p,\c_j} + \k_j$. Furthermore, requirement \ref{it:intro-spherical-Voronoi} from the definition of spherical Voronoi cluster, which in view of Remark \ref{rem:intro-quasi-centers} is equivalent to $|\c_{ij}|^2 = 1 + \k_{ij}^2$ for all non-empty interfaces $\Sigma_{ij}$, is now naturally specified as:
\begin{equation} \label{eq:Voronoi-projective-rep-2}
\scalar{\ck_{ij},\ck_{ij}}_1 = 1 \;\;\; \forall i<j \; \text{such that} \; \Sigma_{ij} \neq \emptyset . 
\end{equation}

\begin{lemma}[M\"obius Automorphisms Preserve Spherical Voronoi Clusters] \label{lem:Mobius-preserves-Voronoi}
Let $\Omega$ be a spherical Voronoi cluster on $\M^n \in \{\R^n , \S^n\}$. Then for any M\"obius automorphism $T$ of $\M^n$, $T\Omega$ is a spherical Voronoi cluster. \\
Moreover, on $\S^n$, if  $\{ \ck_i \}_{i=1,\ldots,q} \subset \R^{n+2}$ are the homogeneous spherical Voronoi parameters associated to $\Omega$, then for all $U \in O^+_1(n+2)$, $\{ U \,\ck_i \}_{i=1,\ldots,q} \subset \R^{n+2}$ are the homogeneous spherical Voronoi parameters associated to $T_U \Omega$. 
\end{lemma}
\begin{proof}
By definition, it is enough to establish the first assertion on $\S^n$. Recalling the projective representation (\ref{eq:Voronoi-projective-rep}) and using that $U \in O_1(n+2)$, we see that:
\begin{align*}
T_U \Omega_i &= \set{ [U \bar p]_+ \in L^{n+1}_{1,+} / [\cdot]_+ \; ; \; \argmin_{j=1,\ldots,q} \scalar{[\bar p]_+ , \ck_j}_1 = \{i\} } \\
& = \set{ [\bar p]_+ \in L^{n+1}_{1,+} / [\cdot]_+ \; ; \; \argmin_{j=1,\ldots,q} \scalar{[U^{-1} \bar p]_+ , \ck_j}_1 = \{i\} } \\
& = \set{ [\bar p]_+ \in L^{n+1}_{1,+} / [\cdot]_+ \; ; \; \argmin_{j=1,\ldots,q} \scalar{[\bar p]_+ , U \; \ck_j}_1 = \{i\} } .
\end{align*}
It remains to verify (\ref{eq:Voronoi-projective-rep-2}) for $T_U \Omega$, which is obvious since $\scalar{U \ck_{ij}, U \ck_{ij}}_1 = \scalar{\ck_{ij},\ck_{ij}}_1$ and as $T_U$ preserves non-empty interfaces (being a diffeomorphism). This concludes the proof. 
\end{proof}

\begin{corollary}[Standard Bubbles are Spherical Voronoi] \label{cor:standard-Voronoi}
For all $2 \leq q \leq n+2$, standard $(q-1)$-bubbles on $\R^n$ and $\S^n$ are spherical Voronoi clusters. If $2 \leq q \leq n+1$, they are in fact perpendicularly spherical Voronoi clusters with $\S^{n+1-q}$-symmetry. 
\end{corollary}
\begin{proof}
The second assertion follows from the first and Remark \ref{rem:intro-Sm-symmetry}. 
To establish the first assertion, it is enough by the previous lemma to establish that the equal-volume standard $q$-bubble on $\S^n$ is spherical Voronoi for all $2 \leq q \leq n+2$. But that is trivial, since the latter has flat interfaces, and is obtained by using $\{\k_i = 0\}_{i=1,\ldots,q}$ and  $\{ \c_i \}_{i=1,\ldots,q}$ which are $q$-equidistant points in $\R^{n+1}$, rescaled so that the Euclidean distance between every pair is $1$ (to ensure that $|\c_{ij}|^2 =  1 + |\k_{ij}|^2$, i.e.~that $|\n_{ij}|^2 = |\c_{ij} + \k_{ij} p|^2 = 1$). 
\end{proof}

\subsection{Conformal equivalence classes of spherical Voronoi clusters}

Having established that spherical Voronoi clusters $\Omega$ are preserved under M\"obius automorphisms, a natural question is whether there is a parameter associated to $\Omega$ which is preserved under any M\"obius automorphism and thus shared by all members of its conformal equivalence class (defined to be its orbit under the M\"obius group with the same ordering of the cells). As it turns out, under generic conditions, there is a parameter which in fact characterizes the entire conformal equivalence class in full. 

\begin{definition}[Minkowski Gram matrix]
Let $\Omega$ be a spherical Voronoi cluster with homogeneous parameters $\{ \ck_i \}_{i=1,\ldots,q} \subset \R^{n+2}$. 
Denote by $\ck := (\ck_i)_{i=1,\ldots,q} \in (\R^{n+2})^{q}$ the ordered vector of homogeneous parameters. The associated Minkowski Gram matrix $\G_{\ck}$ is the $q \times q$ symmetric matrix given by:
\[
\G_{\ck} := ( \scalar{\ck_i , \ck_j}_1 )_{i,j = 1,\ldots,q} .
\]
\end{definition}
\begin{remark}
Note that $\G_{\ck}$ is the difference between the Gram matrices of the $\{\c_i\}$'s and of the $\{\k_i\}$'s (in their given order). In other words:
\begin{equation} \label{eq:CC-kk}
\G_{\ck} = \CC \CC^T - \k \k^T ,
\end{equation}
where $\CC = \sum_{i=1}^q e_i \c_i^T$ is the $q \times (n+1)$ matrix whose rows consist of the $\{\c_i\}$'s (and we treat all vectors as column vectors). Also note that since our convention is that $\sum_{i=1}^q \c_i = 0$ and $\sum_{i=1}^q \k_i = 0$, the rows and columns of $\G_{\ck}$ sum to zero. Consequently, we will think of $\G_{\ck}$ as a symmetric bilinear form on $E^{(q-1)}$. 
\end{remark}
\begin{remark}
A Lorentz transformation $U \in O_1(n+2)$ acts on $\ck$ by multiplication from the left of each coordinate, namely $U \ck := (U \ck_1,\ldots, U \ck_q)$. It is immediate from the definition that $\G_{U \ck} = \G_{\ck}$ for all $U \in O_1(n+2)$. Consequently, all conformally equivalent spherical Voronoi clusters have the same Minkowski Gram matrix by Lemma \ref{lem:Mobius-preserves-Voronoi}. 
\end{remark}
 
\begin{proposition}[Transitivity] \label{prop:transitive}
Let $G$ be a full-rank symmetric bilinear form on $E^{(q-1)}$. Denote the moduli space of all (ordered) spherical Voronoi homogeneous parameters $\ck$ with Minkowski Gram matrix $G$ by:
\[
\mathbf{CK}_G := \{ \ck \in (\R^{n+2})^{q} \; ;\; \G_{\ck} = G \} .
\]
If $q \leq n+2$, then the orthochronus Lorentz group $O^+_1(n+2)$ acts transitively on $\mathbf{CK}_G$. 
In other words, for all $\ck^1 , \ck^2 \in (\R^{n+2})^{q}$ with $\G_{\ck^1} = \G_{\ck^2} = G$, there exists $U \in O^+_1(n+2)$ so that $\ck^2_i = U \ck^1_i$ for all $i=1,\ldots,q$.  
\end{proposition}

For the proof, we require the following linear algebra lemma. 
\begin{lemma} \label{lem:isometry}
Let $\scalar{\cdot,\cdot}_1$ denote a non-degenerate symmetric bilinear form on $\R^N$, i.e. $\scalar{u,v}_1 = u^T J v$ for some non-singular symmetric $J \in \GL(N)$. Let $A,B$ be two $q \times N$ matrices over $\R$. Assume that their corresponding Gram matrices coincide, i.e. $G_{ij} := \scalar{A_i,A_j}_1 = \scalar{B_i,B_j}_1$ for all $i,j=1,\ldots,q$, where $C_i$ denotes the $i$-th row of $C \in \{A ,B\}$. In addition, assume that $\ker A^T = \ker B^T$ and that $\rank(G) = \rank(A) ( = \rank(B) )$. 
Then $B = A U^T$ for some isometry preserving the bilinear form, i.e. $U \in O_{\scalar{\cdot,\cdot}_1}(N) := \{ U \in \GL(N) \; ; \; U^T J U = J\}$. Moreover, if $\rank(G) \leq N-1$ and $J_0 = \text{diag}(1,\ldots,1,-1) \in O_{\scalar{\cdot,\cdot}_1}(N)$, then $U$ may be chosen to be in $\ker \Phi$ for any group homomorphism $\Phi: O_{\scalar{\cdot,\cdot}_1}(N) \rightarrow \{-1,+1\}$ so that $\Phi(J_0) = -1$. 
\end{lemma}
\begin{remark}
Note that we make no assumptions on the signature of $J$, and it is easy to check that the claim would be false without the requirement that $\ker A^T = \ker B^T$ or $\rank(G) = \rank(A)$.  In the classical definite case when $J = \Id$, the latter two requirements hold automatically from the assumption that $A A^T = B B^T$, since $A A^T v = 0$ iff $A^T v = 0$. 
\end{remark} 
\begin{proof}[Sketch of proof of Lemma \ref{lem:isometry}]
Let us only sketch the proof, leaving the details to the reader. We first reduce to the case that $A$ is full-rank by selecting $r = \rank(A)$ linearly independent rows from $A$; the assumption that $\ker A^T  = \ker B^T$ ensures that the same rows in $B$ will also be linearly independent, and that the relation $B = A U^T$ established for those rows will extend to all the other linearly dependent rows. Denote by $E_A$ the subspace spanned by the rows of $A$, and let $E^{\perp}_A$ be the perpendicular subspace in $\R^N$ with respect to $\scalar{\cdot,\cdot}_1$. The assumption that $\rank(G) = \rank(A)$ ensures that $G$ is full-rank and therefore that $E_A$ is non-degenerate, i.e. that if $x \in E_A$ and $\scalar{x,y}_1 = 0$ for all $y \in E_A$ then $x = 0$. Consequently, $\R^N$ is the direct sum of $E_A$ and $E^{\perp}_A$, and one may construct an orthonormal basis $\{v^A_{r+1} , \ldots, v^A_N\}$ of $(E^{\perp}_A, \scalar{\cdot,\cdot}_1)$, satisfying $\sscalar{v^A_j,v^A_k}_1 = \pm \delta_{jk}$ for some choice of signs \cite[pp. 10-11]{GohbergEtAl-IndefiniteLinearAlgebra}; the difference between the number of positive and negative signs is independent of the choice of basis and called the signature $\text{sig}(E^{\perp}_A)$ of $E^{\perp}_A$. The same applies to $B$, $E_B$, $E_B^{\perp}$, $\{v^B_{r+1} , \ldots, v^B_N\}$ and $\text{sig}(E_B)$. Note that:
\[
 \text{sig}(E^{\perp}_A) = \text{sig}(\R^N) - \text{sig}(E_A) = \text{sig}(J) - \text{sig}(G) = \text{sig}(\R^N) - \text{sig}(E_B) = \text{sig}(E^{\perp}_B) ,
 \]
 where $\text{sig}(J),\text{sig}(G)$ denote the difference between the positive and negative eigenvalues of $J$ and $G$, respectively. Consequently, we may arrange the basis elements $\{v^A_j\}$ and $\{v^B_j\}$ so that $\sscalar{v^A_j,v^A_j}_1 = \sscalar{v^B_j,v^B_j}_1$ for all $j$. It remains to define $U$ by mapping $A_i$ to $B_i$, $i=1,\ldots,r$ and $v^A_j$ to $v^B_j$, $j=r+1,\ldots,N$, and verifying using the common Gram matrix $G$ that $U$ is an isometry with respect to $\scalar{\cdot,\cdot}_1$. When $r \leq N-1$, we have the freedom to flip the sign of $v^B_N$ if necessary, to ensure that $\Phi(U) = +1$. 
\end{proof}

\begin{proof}[Proof of Proposition \ref{prop:transitive}]
Set $A = \sum_{i=1}^q e_i (\ck^1_i)^T$ and $B = \sum_{i=1}^q e_i (\ck^2_i)^T$, two $q \times N$ matrices for $N = n+2$. Our assumption is that the corresponding Gram matrices coincide, i.e. $G = A J A^T = B J B^T$. Then $A^T v = 0$ implies $G v = 0$, and since we assumed that $G$ is full-rank on $E^{(q-1)}$, we have $v \in \sspan{(1,\ldots,1)}$ and hence $B^T v =0$ by our convention that $\sum_{i=1}^q \ck_i = 0$; we thus confirm that $\ker A^T = \ker B^T = \sspan{(1,\ldots,1)}$ and that $\rank(G) = \rank(A) = q-1 \leq N-1$. It follows by Lemma \ref{lem:isometry} that $B = A U^T$ for some $U \in O_1(n+2)$, and that moreover $U$ may be chosen to be in $O^+_1(n+2)$. 
\end{proof}

\subsection{Standard bubbles}

\begin{proposition} \label{prop:standard-char}
The following are equivalent for a spherical Voronoi $q$-cluster $\Omega$ on $\S^n$ with $2 \leq q \leq n+2$:
\begin{enumerate}[(i)]
\item \label{it:standard-char-i} $\Omega$ is a standard-bubble. 
\item \label{it:standard-char-ii} All interfaces are non-empty: $\Sigma_{ij} \neq \emptyset$ for all $1 \leq i < j \leq q$.
\item \label{it:standard-char-iii} $\G_{\ck} = \CC \CC^T - \k \k^T = \frac{1}{2} \Id_{E^{(q-1)}}$. 
\end{enumerate}
\end{proposition}
\begin{proof}
Clearly \ref{it:standard-char-i} implies \ref{it:standard-char-ii}, since the latter is preserved under diffeomorphisms and holds true for the equal-volume standard bubble. 

Next, for any point $p \in \Sigma_{ij}$ we have $\c_{ij} = \n_{ij} - \k_{ij} p$, and therefore:
\[
|\c_{ij}|^2 = 1 + \k_{ij}^2 \;\;\; \forall i < j \text{ so that } \Sigma_{ij} \neq \emptyset . 
\]
Recalling (\ref{eq:CC-kk}) and noting that $\frac{1}{2} \scalar{e_{ij},e_{ij}} = 1$ for all $i< j$, this is the same as:
\[
\scalar{\G_{\ck} , e_{ij} e_{ij}^T} = 1 = \sscalar{\frac{1}{2} \Id_{E^{(q-1)}} , e_{ij} e_{ij}^T} \;\;\; \forall i < j \text{ so that } \Sigma_{ij} \neq \emptyset .
\]
As $\{e_{ij} e_{ij}^T\}_{i<j}$ span the space of all symmetric bilinear forms on $E^{(q-1)}$ (by comparing dimensions), it follows that \ref{it:standard-char-ii} implies \ref{it:standard-char-iii}. 

Lastly, we know by Proposition \ref{prop:transitive} that $O^+_1(n+2)$ acts transitively on $\mathbf{CK}_{G_0}$ for $G_0 := \frac{1}{2} \Id_{E^{(q-1)}}$ (which is obviously full-rank on $E^{(q-1)}$). In conjunction with Lemma \ref{lem:Mobius-preserves-Voronoi}, this means that any two spherical Voronoi $q$-clusters with Minkowski Gram matrix $G_0$ belong to the same conformal equivalence class. Now observe that the equal-volume standard-bubble $\Omega_0$ has Minkowski Gram matrix equal to $G_0$ by the already established direction that \ref{it:standard-char-i} implies \ref{it:standard-char-iii}. Consequently, if $\Omega$ satisfies \ref{it:standard-char-iii}, there is a M\"obius automorphism transforming it into $\Omega_0$, implying that $\Omega$ is a standard bubble and establishing \ref{it:standard-char-i}. 
\end{proof}

\begin{corollary}[Standard Bubbles of Prescribed Curvature] \label{cor:standard-curvature}
For all $2 \leq q \leq n+2$ and $\k \in E^{(q-1)}$, there exists a standard bubble on $\S^n$ with $q$-cells and spherical Voronoi curvature parameters $\{\k_i\}_{i=1,\ldots,q}$, i.e. the curvature of the spherical interface $\Sigma_{ij}$ is $\k_{ij} = \k_i - \k_j$. Moreover, such a standard bubble (with the same ordering of its cells) is unique up to orthogonal transformations of $\R^{n+1}$. 
\end{corollary}
\begin{proof}
Given $\k \in E^{(q-1)}$, the quadratic form $\frac{1}{2} \Id_{E^{(q-1)}} + \k \k^T$ is positive-definite on $E^{(q-1)}$. Consequently, as long as $q -1 \leq n+1$, there exists a linear map $\CC : \R^{n+1} \rightarrow E^{(q-1)}$ so that $\CC \CC^T = \frac{1}{2} \Id_{E^{(q-1)}} + \k \k^T$, and moreover, $\CC$ is determined uniquely up to composition with linear isometries of $\R^{n+1}$, i.e. replacing $\CC$ by $\CC \circ U$ where $U \in O(n+1)$. Denote the rows of $\CC$, considered as a $q \times (n+1)$ matrix,  by $\{\c_i\}_{i=1,\ldots,q}$ and note that they sum to zero. The spherical Voronoi $q$-cluster on $\S^n$ with quasi-center and curvature parameters $\{\c_i\}$ and $\{ \k_i \}$ satisfies the third condition \ref{it:standard-char-iii} of Proposition \ref{prop:standard-char} by construction, and is therefore a standard bubble. Uniqueness up to $O(n+1)$ follows from Proposition \ref{prop:standard-char} and the uniqueness of $\CC$ as described above. 
\end{proof}

\begin{corollary}[Standard Bubbles of Prescribed Volume]  \label{cor:standard-volume}
For all $2 \leq q \leq n+2$ and $v \in \interior \Delta^{(q-1)}$, there exists a standard bubble $\Omega$ on $\S^n$ with $q$-cells and $V(\Omega) = v$. Moreover, such a standard bubble (with the same ordering of its cells) is unique up to orthogonal transformations of $\R^{n+1}$. 
\end{corollary}

\begin{remark}
Analogues of Corollaries \ref{cor:standard-curvature} and \ref{cor:standard-volume} regarding standard bubbles in Euclidean space $\R^n$ and in Gauss space $\GG^n$ were proved by Montesinos Amilibia \cite{MontesinosStandardBubbleE!} and by the authors \cite{EMilmanNeeman-GaussianMultiBubble}, respectively. For standard double-bubbles in $\S^2$ and $\S^n$ and triple-bubbles in $\S^2$, these results are due to Masters \cite{Masters-DoubleBubbleInS2}, Cotton-Freeman \cite{CottonFreeman-DoubleBubbleInSandH} and Lawlor \cite{Lawlor-TripleBubbleInR2AndS2}, respectively. 
\end{remark}

\begin{proof}[Sketch of Proof of Corollary \ref{cor:standard-volume}]
The proof of existence is completely analogous to the one in \cite[Theorem 2.1]{MontesinosStandardBubbleE!} for $\R^n$, using a topological argument attributed to K.~Brakke, so we omit it here. To show uniqueness, in view of Corollary \ref{cor:standard-curvature}, it is enough to establish that the map $F_{VK} : E^{(q-1)} \ni \k \mapsto v \in \interior \Delta^{(q-1)}$ is a well-defined bijection (with $\k$ and $v$ as in Corollaries \ref{cor:standard-curvature} and \ref{cor:standard-volume}, respectively). 
This map is well-defined by Corollary \ref{cor:standard-curvature} and it is onto by the first part of the present proof. To show that the map is injective, one first observes that it is 
differentiable with non-singular differential. 
Indeed, fix $\k \in E^{(q-1)}$ and let $\Omega$ denote a fixed standard bubble with spherical Voronoi curvature parameters $\{\k_i\}_{i=1,\ldots,q}$ and $v = V(\Omega) = F_{VK}(\k)$. Let $C^{(q-1)}$ denote the $(q-1)$-dimensional affine span of $\Omega$'s quasi-center parameters $\{\c_i\}_{i=1,\ldots,q}$. Denote by $\delta^1_{V \Theta} : C^{(q-1)} \rightarrow T_v \Delta^{(q-1)}$ the linear map $\delta^1_{V \Theta}(\theta) := \delta^1_{W_\theta} V(\Omega)$, where $W_\theta$ is the M\"obius field from Definition \ref{def:dilation-field}. Similarly, denote by $\delta^1_{K \Theta} : C^{(q-1)} \rightarrow T_{\k} E^{(q-1)}$ the linear map $\delta^1_{K \Theta}(\theta) := \delta^1_{W_\theta} \kappa(\Omega)$ given by the first variation of $\k$ along $W_\theta$. It is not too hard to show that both of these linear maps are non-singular,
and since $dF_{VK} = \delta^1_{V \Theta} \circ (\delta^1_{K \Theta})^{-1}$, it follows by the chain-rule that $F_{VK}$ is differentiable everywhere with non-singular differential. It follows by e.g. \cite{SaintRaymond-DifferentiableIFT} that $F_{VK}$ is a local homeomorphism. 
It is not hard to see that $F_{VK}$ is also proper. A proper, surjective, local homeomorphism between two manifolds is a covering map. 
Since connected coverings of a manifold $X$ are in bijection with conjugacy classes of subgroups of $\pi_1(X)$, and since $\interior \Delta^{(q-1)}$ is simply connected, it follows that $F_{VK}$ must be injective, concluding the proof. 
\end{proof}

\subsection{Conformally flat clusters}

We conclude this section with a brief discussion of when is a spherical Voronoi cluster conformally flat; this will only be required to motivate the definition of a pseudo conformally flat cluster below, which is used in the statement of Theorem \ref{thm:intro-conditional}. 

\begin{definition}[Conformally Flat]
A spherical Voronoi cluster $\Omega$ on $\S^n$ is called conformally flat if there exists a M\"obius automorphism $T$ of $\S^n$ so that the spherical Voronoi cluster $T \Omega$ is flat, i.e. all of its interfaces have zero curvature. 
\end{definition}

\begin{lemma}
An interface-regular spherical Voronoi $q$-cluster $\Omega$ on $\S^n$ with $V(\Omega) \in \interior \Delta^{(q-1)}$ is conformally flat if and only if:
\[
\exists \xi \in \R^{n+1} \;\;\; |\xi| < 1 \;\; \text{ so that } \;\; \scalar{\c_i , \xi} + \k_i = 0 \;\;\; \forall i=1,\ldots,q ,
\]
(where $\{\c_i\}_{i=1,\ldots,q}$ and $\{ \k_i \}_{i=1,\ldots,q}$ are the cluster's quasi-center and curvature parameters, respectively). 
\end{lemma}
\begin{proof}
In view of Lemmas \ref{lem:LA-connected} and \ref{lem:Mobius-preserves-Voronoi}, such a spherical Voronoi cluster is conformally flat iff there exists $U \in O^+_1(n+2)$ so that $\scalar{U \; \ck_i , e_{n+2}}_1 = 0$ for all $i=1,\ldots,q$. Writing the last row of $U$ as $(a , b) \in \R^{n+1} \times \R$, this amounts to $\scalar{a,\c_i} - b \k_i = 0$ for all $i=1,\ldots,q$. By (\ref{eq:ba1}) we have $b^2 = |a|^2+1$, and necessarily $b \geq 1$ whenever $U \in O^+_1(n+2)$. Defining $\xi = -a/b$, this is equivalent to having $|\xi| < 1$, thereby concluding the proof.  
\end{proof}

In view of the previous lemma, we generalize the property of being conformally flat by removing the restriction that $|\xi| < 1$. 
\begin{definition}[Pseudo Conformally Flat Cluster] \label{def:PCF}
A spherical Voronoi cluster $\Omega$ on $\S^n$ is called pseudo conformally flat if:
\[
\exists \xi \in \R^{n+1} \;\; \text{ so that } \;\; \scalar{\c_i , \xi} + \k_i = 0 \;\;\; \forall i=1,\ldots,q .
\]
\end{definition}

\begin{remark}
Note that this implies that $\scalar{\c_{ij} , \xi} + \k_{ij} = 0$ for all $i,j =1 ,\ldots, q$. Recalling the polyhedral cell representation of a spherical Voronoi cluster on $\S^n$ as $\Omega_i = P_i \cap \S^n$ for $P_i := \bigcap_{j \neq i} \set{ x \in \R^{n+1} \; ; \; \scalar{\c_{ij} ,x} + \k_{ij} < 0 }$, it follows that $\Omega$ is pseudo conformally flat iff $C := \bigcap_{i=1}^q \overline{P_i}$ is non-empty, i.e.~that all closed polyhedra have a common meeting point $\xi$. Similarly, an interface-regular spherical Voronoi cluster on $\S^n$ with $V(\Omega) \in \interior \Delta^{(q-1)}$ is conformally flat iff $C \cap \{ x \in \R^{n+1} \; ; \; |x| < 1\} \neq \emptyset$; for a perpendicularly spherical Voronoi cluster (with North pole at $N$), this is the same as witnessing the common intersection point after orthogonally projecting onto $\B^n \subset N^{\perp}$, yielding a clear geometric characterization. 
\end{remark}

\begin{definition}[Full-dimensional cluster]
A spherical Voronoi $q$-cluster $\Omega$ on $\S^n$ is called full-dimensional if:
\[
\text{affine-rank}(\{\c_i\}_{i=1,\ldots,q}) = \min(q-1,n+1) . 
\]
\end{definition}
\begin{remark}
Clearly, when $q-1 \leq n+1$, a full-dimensional cluster is pseudo conformally flat. 
\end{remark}

\section{Standard double, triple and quadruple bubbles are uniquely minimizing} \label{sec:global}

Theorem~\ref{thm:intro-stable-spherical-Voronoi} imposes enough conditions on the structure of minimizing clusters $\Omega$
that the minimization problem is essentially reduced to a combinatorial one. As asserted in Corollary~\ref{cor:intro-bubble},
Proposition~\ref{prop:standard-char} implies that in order to show that standard bubbles are 
uniquely minimizing when $V(\Omega) \in \interior \Delta^{(q-1)}_{V(M^n)}$, it suffices to
show that all interfaces of $\Omega$ are non-empty, meaning that the cell adjacency graph is the complete graph $K_q$. Recall that by Theorem~\ref{thm:intro-stable-spherical-Voronoi} all 
 cells (which by assumption are non-empty) are connected, and so the cell adjacency graph coincides with the component adjacency graph defined in Subsection \ref{subsec:adjacency-graph}. We know by Lemma \ref{lem:2-connected} that the adjacency graph must be $2$-connected. As the only $2$-connected graph on $3$ vertices is the complete graph $K_3$, the double-bubble conjecture on $\S^n$ immediately follows, and an alternative proof to the double-bubble theorem on $\R^n$ \cite{SMALL93,DoubleBubbleInR3,SMALL03,Reichardt-DoubleBubbleInRn} is obtained ($n \geq 2$). However, to handle $q \geq 4$, additional work is required. 

\smallskip

When the number of cells $q$ is small, we can enumerate all possible cell adjacency graphs and rule out
non-complete graphs with various geometric-measure-theoretic and other geometric arguments. Similar arguments were used in proofs
of the double and triple bubble cases in $\R^2$ (or after reduction of double-bubble with $\S^{n-2}$-symmetry from $\R^n$ to the plane) \cite{SMALL93,Hutchings-StructureOfDoubleBubbles,Wichiramala-TripleBubbleInR2}; the main obstacles to extending these arguments
to larger $q$ are the combinatorial explosion in the number of potential cell adjacency graphs, and
the fact that area-minimizing cones have only been classified in dimensions $n \le 3$ by Taylor \cite{Taylor-SoapBubbleRegularityInR3}. This latter obstacle appears to be the more difficult one, as in dimensions $ n \geq 4$ there exist non-simplicial area-minimizing cones \cite{Brakke-MinimalConesOnCubes}. Consequently, in this section we are only able to establish Theorem \ref{thm:intro-234} for double ($q=3$), triple ($q=4$) and quadruple ($q=5$) bubbles. 

\subsection{Blow-up cones}

The main idea is to take a potential minimizer and ``slide'' one of the cells -- while
preserving all volumes and areas -- until it hits another cell. The resulting cluster cannot be minimal because
it contains an illegal singularity; and since it has the same volumes and areas as the original, the original
could not have been minimal either. For an example of this kind of
global deformation, see Figure~\ref{fig:triple-bubble-perturbation} below. Note that the set at which two moving cells collide will typically be a single point $p$, even when the the dimension $n$ is large. Consequently, when $n \ge 4$, the regularity results of Theorem~\ref{thm:regularity} do not immediately rule out this kind of singularity. We therefore begin by observing that Taylor's classification of minimal cones in $\R^3$ can nevertheless rule out many of these scenarios. 

\smallskip

For a set $A \subset M^n$ and a point $p \in M^n$, we say that $\tilde A \subset T_p M^n$ is the blow-up limit of
$A$ at $p$ if for some neighborhood $N_p$ of $p$ on which $\exp_p: T_p M^n \to M^n$ is a diffeomorphism,
\[
    \frac{\exp_p^{-1} (A \cap N_p)}{r} \to \tilde A \;\;  \text{as} \;\;  r \to 0+
\]
in $L^1_{\text{loc}}$ (note that we do not use a subsequence of $r$'s in our definition). For a cluster $\tilde \Omega$ in $T_p M^n$, we say that $\tilde \Omega$ is the blow-up limit of $\Omega$ at $p$ if each $\tilde \Omega_i$ is the blow-up limit of $\Omega_i$ at $p$. Clearly, if a blow-up limit exists, it is unique up to null-sets. 
\smallskip

Recall that a minimizing $q$-cluster on $M^n \in \{\R^n, \S^n\}$ with $q \leq n+1$ has spherical interfaces by
Theorem~\ref{thm:intro-stable-spherical-Voronoi}, and so in particular has
bounded curvature. This clearly implies that the blow-up limit cluster always exists, and that this cluster's cells are centered cones with flat interfaces -- we will call this cluster the blow-up cone. 
The bounded curvature also means that if $p \in \overline{\Sigma_{ij}}$ then the normal vector-field $\n_{ij}$ on $\Sigma_{ij}$ is continuous up to the singularity $p$ -- we denote the limiting value by $\n_{ij}(p)$. 
We now observe that the blow-up cone of a spherical Voronoi cluster at $p$ is itself a conical Voronoi cluster whose flat interfaces are given by the normals $\{ \n_{ij}(p) ; \overline{\Sigma_{ij}} \ni p \}$.

\begin{lemma}\label{lem:blow-up-cones}
    Let $\Omega$ be a regular spherical Voronoi $q$-cluster on $M^n \in \{\R^n, \S^n\}$. Let $p \in M^n$
    and set $I_p := \{i \in \{1, \dots, q\} \; ; \; p \in \overline{\Omega_i}\}$. Then there exist $\{\n_i(p) \}_{i \in I_p} \subset T_p M^n$ so that:
\begin{enumerate}[(i)]
    \item $\n_{ij}(p) = \n_i(p) - \n_j(p)$ for all $i \neq j \in I_p$ such that $p \in \overline{\Sigma_{ij}}$, and the blow-up cone $\tilde \Omega$ of $\Omega$ at $p$ is given for all $i \in I_p$ (up to null-sets) by:
    \[
        \tilde \Omega_i = \set{x \in T_p M^n \; ; \; \argmin_{j\in I_p} \; \sscalar{\n_j(p), x} = \{i\} } = \bigcap_{j \in I_p \setminus \{i\}} \set{ x \in T_p M^n \; ; \; \sscalar{\n_{ij}(p) , x} < 0 }  
    \]
    (and $\tilde \Omega_i = \emptyset$ for all $i \notin I_p$).
    \item We have:
    \begin{equation} \label{eq:spans}
    \text{affine-span} \{ \n_i(p) \}_{i \in I_p} = \sspan \{ \n_{ij}(p) \; ; \; \overline{\Sigma_{ij}} \ni p \} =: \mathbf{N}_p .
    \end{equation}
    In particular, denoting $d = \dim \mathbf{N}_p$, we have $d \leq |I_p|-1 \leq q-1$, and $\tilde \Omega$ is of the form $\C \times \R^{n-d}$ for a conical Voronoi cluster $\C$ in $\R^d$. 
     \item
     Each $\tilde \Omega_i$ for $i \in I_p$ is non-empty, and if there is an $(n-1)$-dimensional interface $\tilde \Sigma_{ij}$ between
    $\tilde \Omega_i$ and $\tilde \Omega_j$ then $p \in \overline{\Sigma_{ij}}$
    (in particular, $\Sigma_{ij}$ is non-empty).
    \end{enumerate}
\end{lemma}

\begin{proof}
    It suffices to consider $M^n = \S^n$. The claim for $\R^n$ will follow from the one on $\S^n$ by applying a stereographic projection $P : \bar \R^n \rightarrow \S^n$, as the corresponding blow-up clusters are related via the differential $d_p P : T_p \R^n \rightarrow T_{P(p)} \S^n$. The latter is a linear conformal map, and therefore preserves all of the above assertions, including the Voronoi structure of the blow-up clusters (as the scalar product on the corresponding tangent spaces is preserved up to a positive constant).  
    
    We set $\n_i(p) := \c_i - \scalar{\c_i , p} p$, the orthogonal projection of $\c_i$ onto $T_p \S^n$, so that $\scalar{\c_i,v} = \scalar{\n_i(p),v}$ for all $v \in T_p \S^n$. It follows that if $v \in \tilde \Omega_i \subset T_p \S^n$ , $i \in I_p$, then 
    there exists $\epsilon > 0$ so that $\sscalar{\c_i, v} < \sscalar{\c_j, v} - \epsilon$ for all $j \in I_p \setminus \{i\}$.
    Note that  exponential map satisfies $\exp_p(v) = p + v + O(|v|^2)$, where we identify  $T_p \S^n$ with $p^\perp \subset \R^{n+1}$ in the usual way. Therefore,
    \[
        \sscalar{\c_i, \exp_p(r v) - p} < \sscalar{\c_j, \exp_p(rv) - p} - \epsilon r + O(r^2)
    \]
    for sufficiently small $r>0$. Since $\Omega$ is a spherical Voronoi cluster,
    $p \in \bigcap_{i \in I_p} \overline{\Omega_i}$ implies that $\sscalar{\c_{ij}, p} = -\k_{ij}$ for all $i,j \in I_p$, and hence
    \[
        \sscalar{\c_i, \exp_p(r v)} + \k_i < \sscalar{\c_j, \exp_p(rv)} + \k_j - \epsilon r + O(r^2).
    \]
    It follows that for all sufficiently small $r>0$, $\exp_p(rv) \in \Omega_i$ and so
    \[
        v \in \frac{\exp_p^{-1}(\Omega_i \cap N_p)}{r}
    \]
    for sufficiently small $r>0$, and hence
    \[
        \frac{\exp_p^{-1} (\Omega_i \cap N_p)}{r} \setminus \tilde \Omega_i \to \emptyset
    \]
    in $L^1_\text{loc}$ as $r \to 0+$.
    On the other hand, both $\tilde \Omega_i$ and $\exp_p^{-1}(\Omega_i \cap N_p) / r$ form a partition (up to null-sets)
    of a neighborhood of zero in $T_p \S^n$, and it follows that
    \[
        \frac{\exp_p^{-1} (\Omega_i \cap N_p)}{r} \to \tilde \Omega_i
    \]
    in $L^1_\text{loc}$ as $r \to 0+$. It remains to note that for all $i \neq j \in I_p$ for which $p \in \overline{\Sigma_{ij}}$, since $\c_{ij} = \n_{ij} - \k_{ij} p$ is constant on $\Sigma_{ij}$, this also holds at $p$ by continuity, and therefore $\n_{i}(p) - \n_{j}(p) = \c_{ij} - \scalar{\c_{ij},p} p = \n_{ij}(p)$, concluding the proof of the first assertion. 

     \smallskip    
     For the second assertion, clearly $p \in \overline{\Sigma_{ij}}$ implies that $i,j \in I_p$ and therefore the right-hand side of (\ref{eq:spans}) is a subspace of the left-hand one. To see the converse, consider the graph $G$ with vertex set $I_p$ and an edge between $i \neq j \in I_p$ if $p \in \overline{\Sigma_{ij}}$. Note that this is precisely the cell adjacency graph of $\Omega$ inside a small enough open geodesic ball $B(p,\eps) \subset \S^n$. Consequently, by Lemma \ref{lem:LA-connected} (whose proof extends to $B(p,\eps)$, since it satisfies a single-bubble isoperimetric inequality), $G$ must be connected. It follows that every $\n_i(p) - \n_j(p)$ for $i,j \in I_p$ may be written as $\sum_{k=1}^K \n_{i_{k-1} i_{k}}(p)$ by using a path $i = i_0, \ldots,i_K = j$ in $G$, concluding the proof of (\ref{eq:spans}). The other statements in the second assertion are trivial. 
     
     \smallskip
     Finally, to see the third assertion, the non-emptiness of $\tilde \Omega_i$ follows from the regularity of $\Omega$, specifically~\eqref{eq:density}, which implies that $\tilde \Omega_i$ has positive measure.
    To prove the claim about the interfaces, let $v \in \tilde \Sigma_{ij}$.
    Since $v$ is in the closure of $\tilde \Omega_i$ and of $\tilde \Omega_j$ but not of any other cell $\tilde \Omega_k$, there is some $\eps > 0$ so that
    \[
        \sscalar{\c_i, w} < \sscalar{\c_k, w} - \eps \text{ and }
        \sscalar{\c_j, w} < \sscalar{\c_k, w} - \eps
    \]
    for all $w \in B(v, \epsilon)$ and $k \not \in \{i, j\}$. Since $\exp_p(rw) = p + r w + O(r^2)$, it follows as before that for sufficiently
    small $r > 0$, $\exp_p(r B(v,\epsilon))$ intersects $\overline{\Omega_i}$ and $\overline{\Omega_j}$ but no other cells. Therefore,
    $\Sigma_{ij}$ intersects $\exp_p(r B(v,\epsilon))$ for all sufficiently small $r>0$,
    and it follows that $p \in \overline{\Sigma_{ij}}$. Note that the converse implication, namely that $p \in \overline{\Sigma_{ij}}$ implies that $\tilde \Sigma_{ij} \neq \emptyset$, is equally true, but we do not require this here. 
\end{proof}

Throughout the rest of this section, we assume that $\Omega$ is a minimizing $q$-cluster in $M^n \in \{\R^n,\S^n\}$ with $q \leq n+1$. It follows by Theorem~\ref{thm:intro-stable-spherical-Voronoi} that $\Omega$ is a regular spherical Voronoi cluster, and so Lemma~\ref{lem:blow-up-cones} applies. We also forgo our convention in $\R^n$ that $\Omega_q$ is the unbounded cell. 
Recall the definition of the $\Y$ and $\T$ cones from Subsection \ref{subsec:regularity}. 

\begin{corollary}\label{cor:low-dimensional-meeting}
     Let $p \in \Sigma$, and suppose that
    \[
        d = \dim \sspan \{\n_{ij}(p) \; ; \; \overline{\Sigma_{ij}} \ni p\} \leq 3.
    \]
    Then $2 \leq |I_p| = d+1 \leq 4$, and the interfaces of the blow-up cone at $p$ form either a hyperplane ($d=1$), a $\Y \times \R^{n-2}$ ($d=2$) or a $\T \times \R^{n-3}$ ($d=3$).
\end{corollary}

\begin{proof}
            Because the span of the normals is at most 3-dimensional,
    the blow-up cone is of the form $\C \times \R^{n-3}$ for a cone $\C \subset \R^3$ by Lemma~\ref{lem:blow-up-cones}.
    By~\cite[Theorem 21.14]{MaggiBook}, the blow-up limit is an area-minimizing cone. Therefore $\C$ is also an
    area-minimizing cone,
    and so by Taylor's classification~\cite{Taylor-SoapBubbleRegularityInR3} it must be one of the asserted cones, depending on the value of $d$ (note that $d \geq 1$ because $p \in \Sigma = \overline{\Sigma^1}$). Because every one of the $|I_p|$ cells in the blow-up cone is non-empty,  we also have $|I_p| = d + 1$.
\end{proof}

\begin{corollary}\label{cor:quad-bubble-simplicial}
Assume that $q \leq 5$. If $\overline{\Omega_i} \cap \overline{\Omega_j}$ is non-empty then $\Sigma_{ij}$ is non-empty. Moreover, if $\bigcap_{i=1}^q \overline{\Omega_i}$ is non-empty then $\Omega$ is a standard bubble.
\end{corollary}

\begin{proof}
    Take $p \in \overline{\Omega_i} \cap \overline{\Omega_j}$ and define 
       $d = \dim \sspan \{\n_{ij}(p) \; ; \; \overline{\Sigma_{ij}} \ni p\}$. Recall that $d \leq |I_p|-1 \leq q-1 \leq 4$ by Lemma~\ref{lem:blow-up-cones}.
   
    If $d\leq 3$ then by Corollary~\ref{cor:low-dimensional-meeting}
    the blow-up cone at $p$ is simplicial, i.e.~has non-empty interfaces between all pairs of its non-empty (two, three or four) cells. 
    It follows by Lemma~\ref{lem:blow-up-cones} that all cells which meet at $p$ share an interface, and in particular $\Sigma_{ij} \neq \emptyset$. 
    On the other hand, if $d = 4$ then necessarily $|I_p| = q=5$ and the $\{ \n_i(p) \}_{i=1,\ldots,5}$ from Lemma~\ref{lem:blow-up-cones} are affinely independent. The blow-up cone $\tilde \Omega$ at $p$ is the conical Voronoi cluster generated by the $\{ \n_i(p) \}_{i=1,\ldots,5}$, and it has $5$ non-empty cells by Lemma~\ref{lem:blow-up-cones}. 
    By affine-independence, it follows that this cluster is a non-degenerate linear image of the standard simplicial $5$-cluster (the Voronoi cells of $5$ equidistant points), and is therefore itself simplicial, i.e.~all pairs of cells share a non-empty interface. Applying Lemma~\ref{lem:blow-up-cones} again, it follows that $\Omega$ itself also has all interfaces present between every pair of cells, and in particular $\Sigma_{ij} \neq \emptyset$. 

    For the second assertion, if $\bigcap_{i=1}^q \overline{\Omega_i}$ is non-empty
    then all interfaces are present by the first assertion, and it follows from
    Proposition~\ref{prop:standard-char} that
    $\Omega$ is a standard bubble.
\end{proof}

\subsection{The triple-bubble cases}

\begin{figure}
    \begin{center}
        \begin{tikzpicture}
            \node[vertex] (1) at (0, 0) {};
            \node[vertex] (2) at (1, -1) {};
            \node[vertex] (3) at (0, -2) {};
            \node[vertex] (4) at (-1, -1) {};
            \draw (1) -- (2) -- (3) -- (4) -- (1) -- (3);
            \draw (2) -- (4);
        \end{tikzpicture}
        \hspace{2em}
        \begin{tikzpicture}
            \node[vertex] (1) at (0, 0) {};
            \node[vertex] (2) at (1, -1) {};
            \node[vertex] (3) at (0, -2) {};
            \node[vertex] (4) at (-1, -1) {};
            \draw (1) -- (2) -- (3) -- (4) -- (1) -- (3);
        \end{tikzpicture}
        \hspace{2em}
        \begin{tikzpicture}
            \node[vertex] (1) at (0, 0) {};
            \node[vertex] (2) at (1, -1) {};
            \node[vertex] (3) at (0, -2) {};
            \node[vertex] (4) at (-1, -1) {};
            \draw (1) -- (2) -- (3) -- (4) -- (1);
        \end{tikzpicture}
     \end{center}
     \caption{
        The 2-connected graphs on 4 vertices. From left to right: the complete graph, the 2-fan, and the 4-cycle.
        \label{fig:4-graphs}
     }
\end{figure}
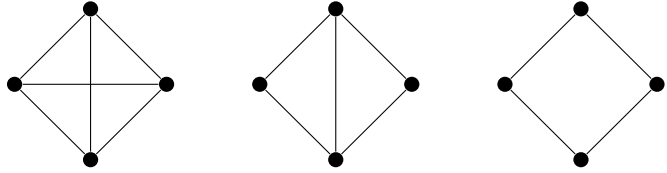

Let $\Omega$ be a minimizing triple-bubble ($q=4$) cluster in $M^n \in \{\R^n,\S^n\}$  ($n \geq q-1=3$) 
with $V(\Omega) \in \interior \Delta^{(q-1)}_{V(M^n)}$,
and let $G$ be its cell adjacency graph. If $G = K_4$ then by Proposition~\ref{prop:standard-char}
$\Omega$ is a standard bubble and we are finished. Up to isomorphism, the only other two-connected graphs on $4$ vertices are either the 4-cycle or the graph with four vertices and five edges, which we call the 2-fan (see Figure \ref{fig:4-graphs}). They are ruled by the following two lemmas.

\begin{lemma} \label{lem:edge-in-triangle}
Every edge in the adjacency graph $G$ belongs to some triangle. 
\end{lemma}
\begin{proof}
As $G$ is $2$-connected and has more than two vertices, every edge $e$ in $G$ lies in a simple cycle (otherwise, removal of $e$ would disconnect the graph, and as the graph has at least $3$ vertices, removing one of $e$'s vertices would disconnect it). On the other hand, $G$ is the $1$-skeleton of the two-dimensional simplicial complex $\SS$ from Subsection \ref{subsec:homology}, which has trivial first homology by Proposition~\ref{prop:homology}. It follows that $e$ must belong to some triangle, as asserted. 
\end{proof}

\noindent
This rules out the $4$-cycle. The $2$-fan is ruled out by the following lemma.

\begin{lemma} \label{lem:degree-3}
For $q \in \{ 4, 5 \}$, the adjacency graph $G$ has minimal degree at least $3$. 
\end{lemma}
\begin{proof}
Suppose without loss of generality that $\Omega_3$ participates in only two interfaces, $\Sigma_{31}$ and $\Sigma_{32}$; see Figures \ref{fig:triple-bubble-perturbation} (also \ref{fig:almost}) and \ref{fig:min-degree} for examples of this situation when $q=4$ and $q=5$, respectively. By Lemma \ref{lem:edge-in-triangle}, we must have an edge $\{1,2\}$ in $G$ between vertices $1$ and $2$ (otherwise the edge $\{1,3\}$ would not belong to any triangle, thus contradicting Lemma \ref{lem:edge-in-triangle}). In other words, $\Sigma_{12}$ is necessarily non-empty. From here on, we set $\Omega_5 = \emptyset$ if $q=4$, and continue the proof for $q=5$.

Note that since $\Sigma_{34} = \Sigma_{35} = \emptyset$, we know by Corollary~\ref{cor:quad-bubble-simplicial} that $\overline{\Omega_3}$ does not intersect $\overline{\Omega_4 \cup \Omega_5}$. We now claim that $\overline{\Sigma_{12}}$ must intersects $\overline{\Omega_4 \cup \Omega_5}$; if not, it would follow that $\overline{\Sigma_{12}} \cup \partial \Omega_3$ is disconnected from $\partial \Omega_4 \cup \partial \Omega_5$, in contradiction to Lemma \ref{lem:Sigma-connected}, which states that $\Sigma$ must be connected.

The sphere (or hyperplane) $S_{12}$ containing $\Sigma_{12}$ is acted upon transitively by
a subgroup of isometries of $M^n$ (either rotations or translations), and so we may use these isometries to move
$\Omega_3$ -- while preserving all volumes and surface areas -- until $\overline{\Omega_3}$ collides with
$\overline{\Omega_4 \cup \Omega_5}$. 
If $\Omega$ were minimizing, this modified cluster would also be minimizing. 
By Corollary~\ref{cor:quad-bubble-simplicial}, 
this means that either a new non-empty $\Sigma_{34}$ or $\Sigma_{35}$ interface is created at the time of collision. Consider a point $p \in \partial \Omega_4$ belonging to a newly created $\Sigma_{34}$, for example. Since $\Sigma_{34}$ is a smooth hypersurface in $\partial \Omega_4$, we see that an instant before the collision, $p$ belongs to $\Sigma_{4 i}$ for some $i \not \in \{3, 4\}$ (and in particular, $p \in \overline{\Omega_i}$ at the time of collision). This means that $\Sigma_{4i}$ and the rotated (or translated) $\Sigma_{3i}$ meet each other tangentially at $p$ in their relative interior. Consequently, the relative volume of $\Omega_i$ in a neighborhood of $p$ decreases to zero at the instant of collision, and so $p \notin \overline{\Omega_i}$ by the Infiltration Lemma \ref{lem:infiltration}, a contradiction. Alternatively (and equivalently), we can invoke \cite[Corollary 30.3]{MaggiBook}, which states that $\Sigma_{3i}$ and $\Sigma_{4i}$ which meet tangentially must coincide in some neighborhood of $p$, thereby strictly reducing the total perimeter of $\Omega$ at the time of collision and contradicting minimality (note that the proof of \cite[Corollary 30.3]{MaggiBook} is a consequence of the Infiltration Lemma). 
\end{proof}

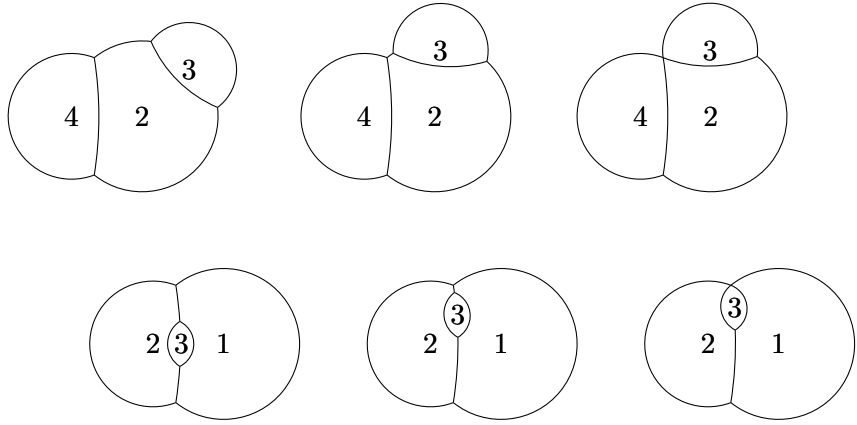
\begin{figure}
\begin{center}
    \begin{tikzpicture}
        \input{mickey-original}
    \end{tikzpicture}
    \hspace{1em}
    \begin{tikzpicture}
        \input{mickey-almost}
    \end{tikzpicture}
    \hspace{1em}
    \begin{tikzpicture}
        \input{mickey-done}
    \end{tikzpicture}
    \vspace{2em}
    \newline
    \begin{tikzpicture}
        \input{pokeball-original}
    \end{tikzpicture}
    \hspace{1em}
    \begin{tikzpicture}
        \input{pokeball-almost}
    \end{tikzpicture}
    \hspace{1em}
    \begin{tikzpicture}
        \input{pokeball-done}
    \end{tikzpicture}
\end{center}
\caption{
    Sliding cell 3 along $\Sigma_{12}$ until it hits cell 4.
    In the first row the unbounded cell is $\Omega_1$; in the second row the unbounded cell is $\Omega_4$.
    \label{fig:triple-bubble-perturbation}
}
\end{figure}

\begin{figure}
    \begin{center}
            \includegraphics[scale=0.08]{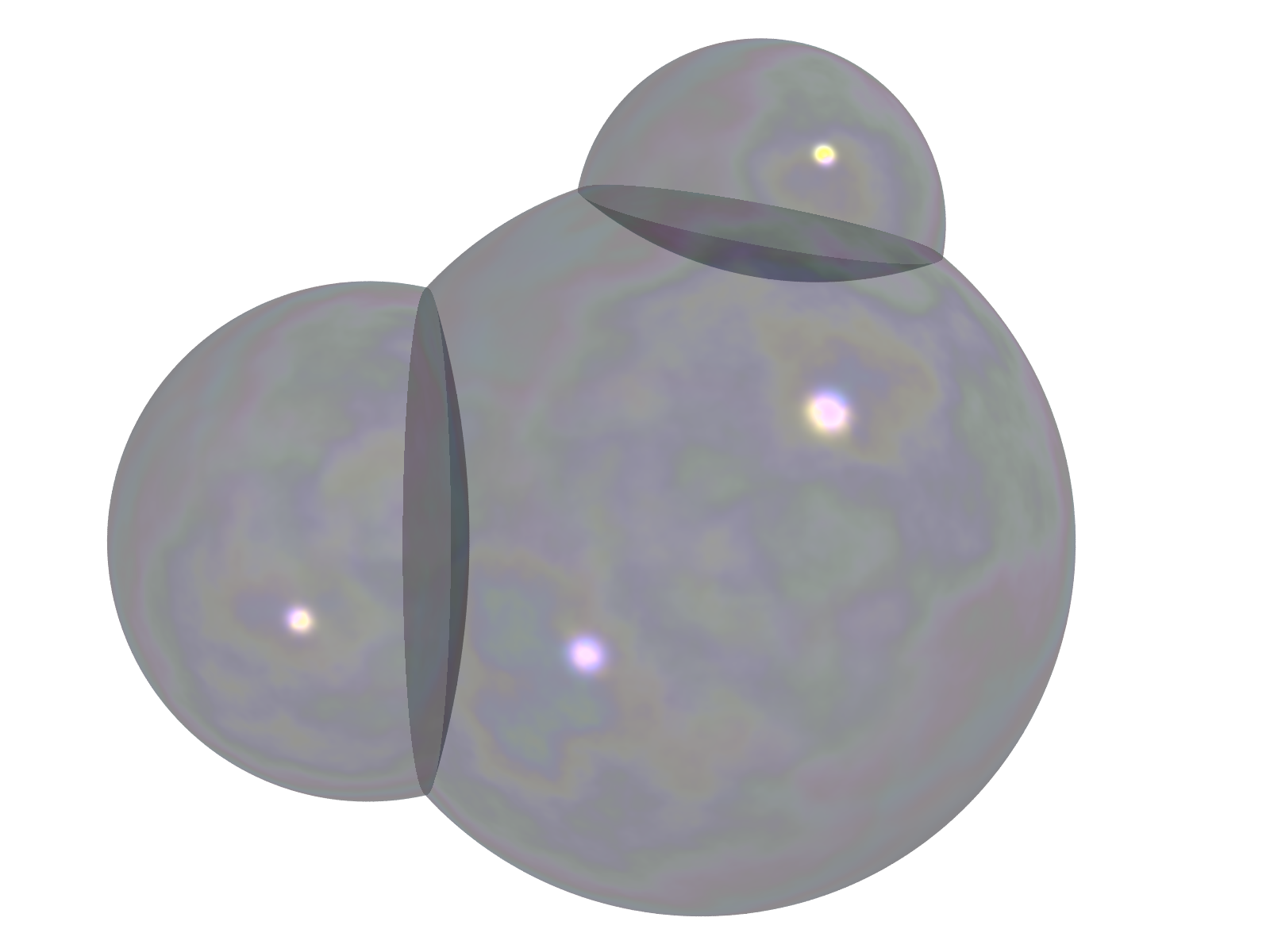}
        \includegraphics[scale=0.08]{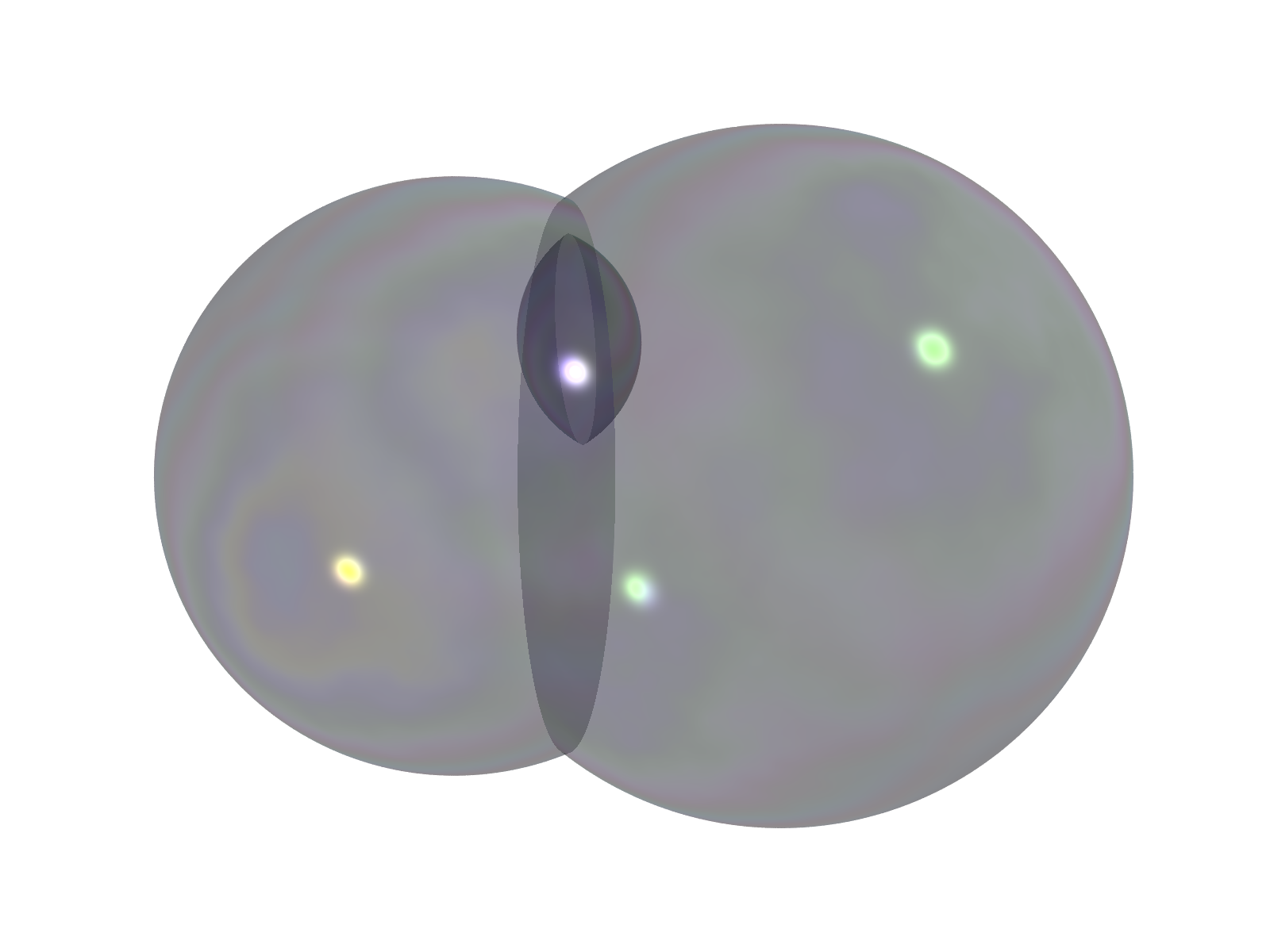}
     \end{center}
     \caption{
         \label{fig:almost}
        These non-standard triple-bubbles cannot be area minimizing, as witnessed by sliding a cell with two neighbors along the interface it sits on; the corresponding cross-sections are depicted in Figure \ref{fig:triple-bubble-perturbation}.
     }
\end{figure}

\begin{figure}
    \begin{center}
        \begin{tikzpicture}
            \input{degree-2}
        \end{tikzpicture}
     \end{center}
     \caption{
         \label{fig:min-degree}
         When the adjacency graph has a vertex of degree two (here cell 3 is adjacent to cells 1 and 2 only),
         it can be transported along $\Sigma_{12}$ to create a meeting point which violates Corollary~\ref{cor:low-dimensional-meeting} or Corollary~\ref{cor:quad-bubble-simplicial}.
     }
\end{figure}
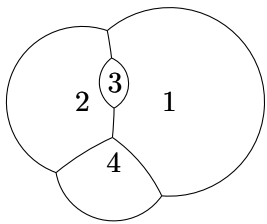

This concludes the proof of Theorem \ref{thm:intro-234} in the triple-bubble case.

\subsection{The quadruple-bubble cases}

Let $\Omega$ be a minimizing quadruple-bubble ($q=5$) cluster in $M^n \in \{\R^n,\S^n\}$ ($n \geq q-1=4$)
with $V(\Omega) \in \interior \Delta^{(q-1)}_{V(M^n)}$,
and let $G$ be its cell adjacency graph.
There are many 2-connected graphs on $5$ vertices, but we eliminate
most of them by invoking Lemma \ref{lem:degree-3}. Up to isomorphism, there are three graphs on five vertices for which every vertex has degree at least three,
depicted in Figure~\ref{fig:5-graphs}.
As usual, if $G = K_5$ then by Proposition~\ref{prop:standard-char} $\Omega$ is a standard bubble and we are finished.

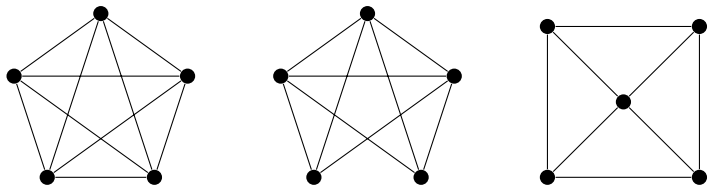
\begin{figure}
    \begin{center}
        \begin{tikzpicture}
            \node[vertex] (1) at (18:1.2) {};
            \node[vertex] (2) at (90:1.2) {};
            \node[vertex] (3) at (162:1.2) {};
            \node[vertex] (4) at (234:1.2) {};
            \node[vertex] (5) at (306:1.2) {};
            \draw (1) -- (2) -- (3) -- (4) -- (5) -- (1) -- (3) -- (5) -- (2) -- (4) -- (1);
        \end{tikzpicture}
        \hspace{2em}
        \begin{tikzpicture}
            \node[vertex] (1) at (18:1.2) {};
            \node[vertex] (2) at (90:1.2) {};
            \node[vertex] (3) at (162:1.2) {};
            \node[vertex] (4) at (234:1.2) {};
            \node[vertex] (5) at (306:1.2) {};
            \draw (1) -- (2) -- (3) -- (4);
            \draw (5) -- (1) -- (3) -- (5) -- (2) -- (4) -- (1);
        \end{tikzpicture}
        \hspace{2em}
        \begin{tikzpicture}
            \node[vertex] (1) at (-1, 1) {};
            \node[vertex] (2) at (1, 1) {};
            \node[vertex] (3) at (1, -1) {};
            \node[vertex] (4) at (-1, -1) {};
            \node[vertex] (5) at (0, 0) {};
            \draw (1) -- (2) -- (3) -- (4) -- (1);
            \draw (1) -- (5) -- (2);
            \draw (3) -- (5) -- (4);
        \end{tikzpicture}
     \end{center}
     \caption{
         \label{fig:5-graphs}
         The graphs on 5 vertices with minimum degree at least 3.
     }
\end{figure}

For the second case of Figure~\ref{fig:5-graphs}, suppose
without loss of generality that $\Sigma_{45}$ is the empty interface.
We claim that $\Sigma_{123}$ is non-empty and meets both $\Omega_4$ and $\Omega_5$;
it suffices to show this on $\S^n$ since stereographic projection preserves incidences.
Because all interfaces among cells $\{1, 2, 3, 4\}$ are present, the $4$-cluster
on $\S^n$ with data $\{\c_1, \dots, \c_4\}$ and $\{\k_1, \dots, \k_4\}$ is a spherical
Voronoi cluster having all interfaces present.
By Proposition~\ref{prop:standard-char}, this 4-cluster is a standard bubble and so
\[
    S_{1234} = \{p \in \S^n\; ; \;\sscalar{\c_1, p} + \k_1 = \cdots = \sscalar{\c_4, p} + \k_4\}
\]
is non-empty and has Hausdorff dimension $n-3$. Note that
$\overline{\Omega_5}$ is disjoint from $S_{1234}$, because otherwise $\Omega_5$ and $\Omega_4$ would
share an interface by Corollary~\ref{cor:quad-bubble-simplicial}. Therefore $S_{1234} = \Sigma_{1234}$,
which implies in particular that $\overline{\Sigma_{123}}$ is non-empty and intersects $\overline{\Omega_4}$.
By the same argument with $4$ and $5$ interchanged, it also intersects $\overline{\Omega_5}$.

\begin{figure}
    \begin{center}
       \includegraphics[scale=0.09]{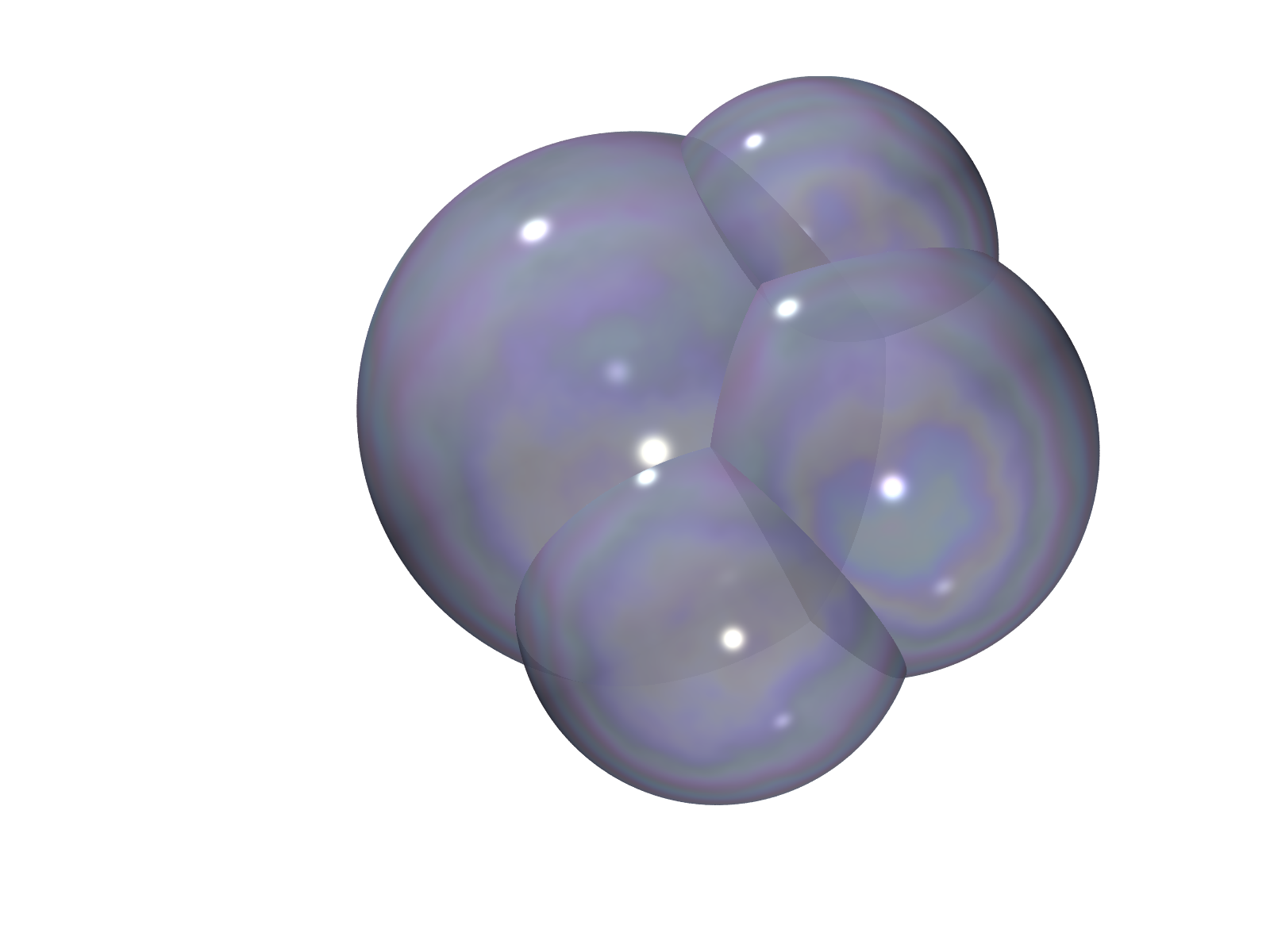}
       \includegraphics[scale=0.09]{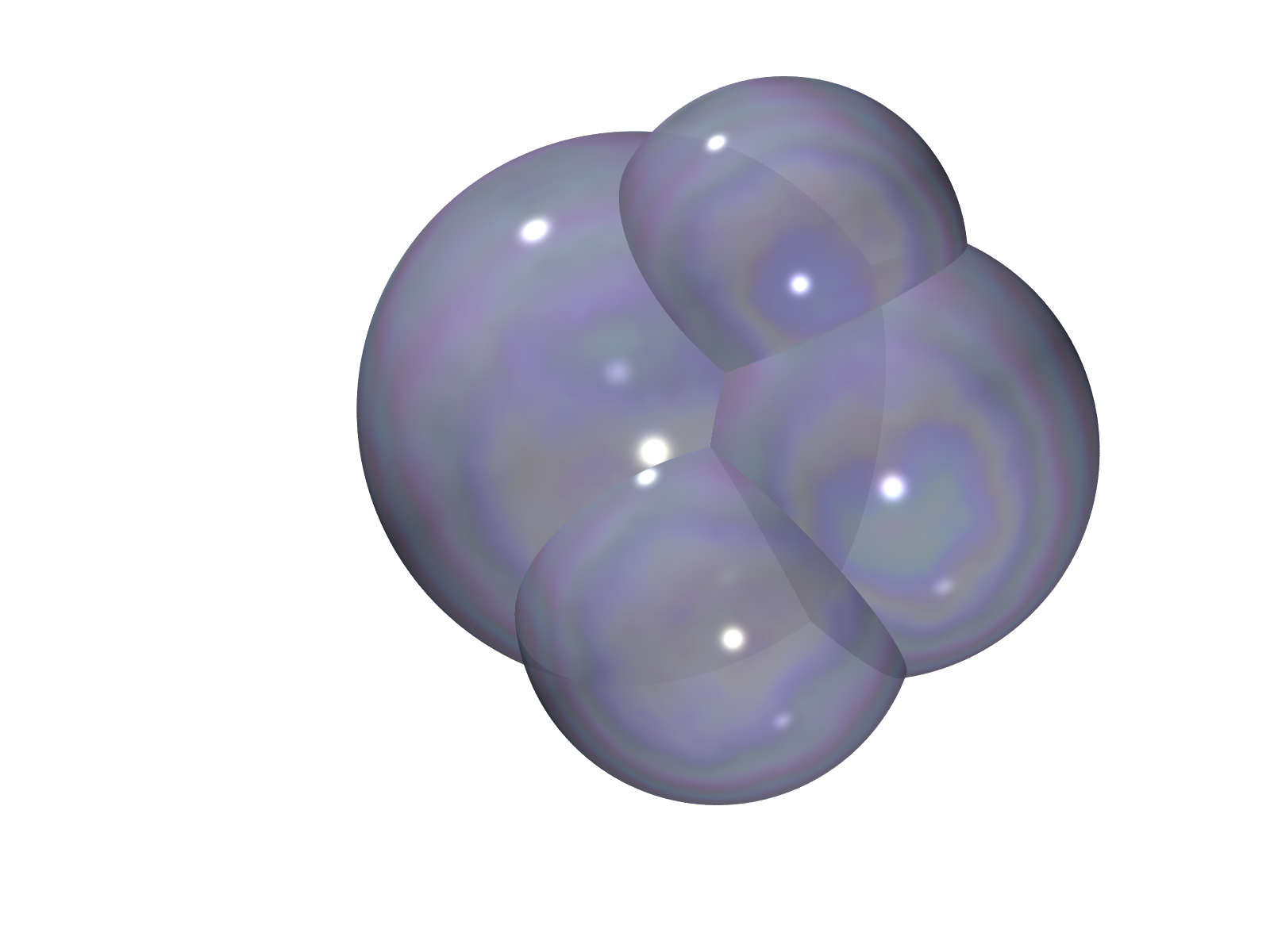}
       \includegraphics[scale=0.09]{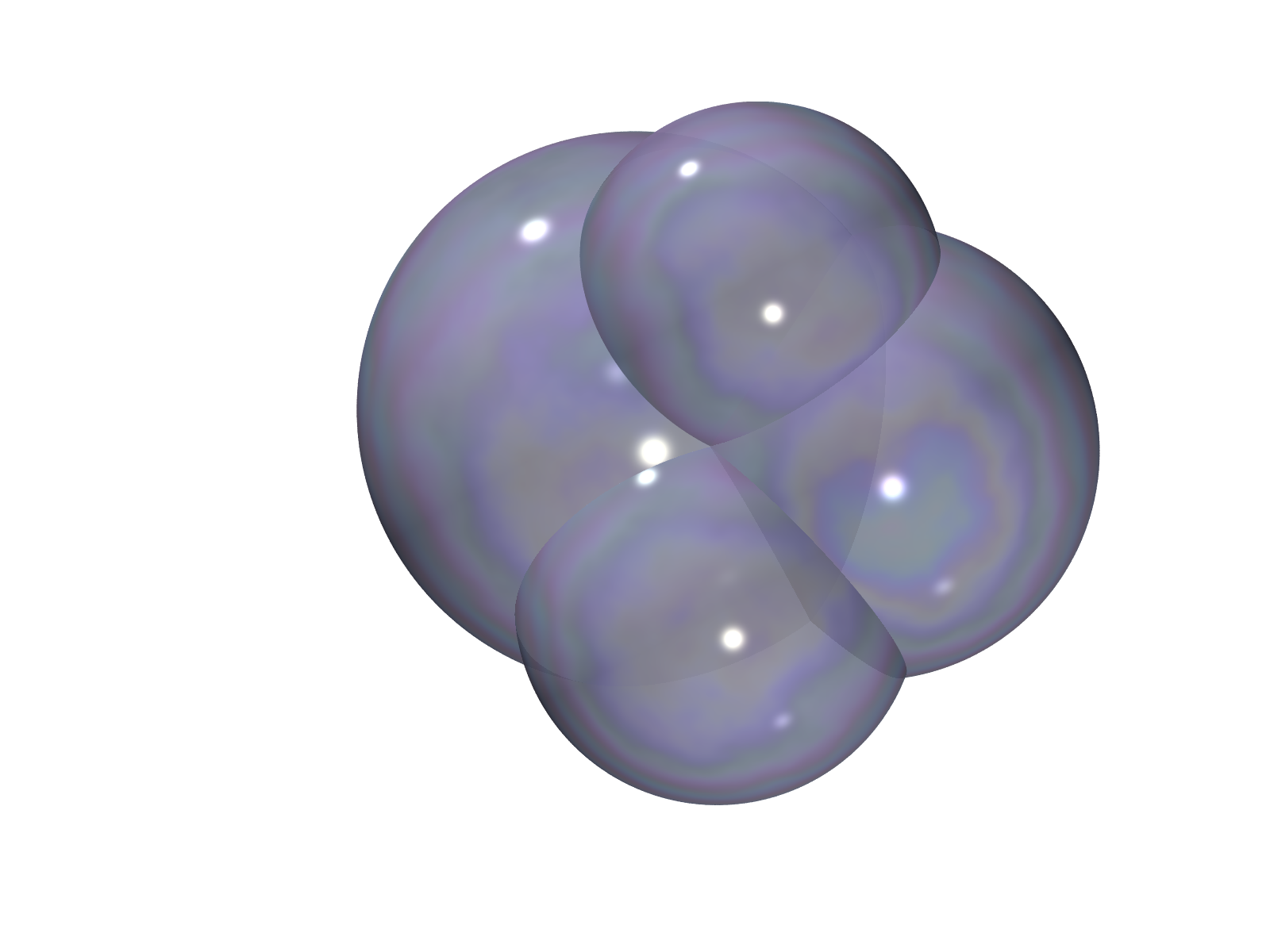}
     \end{center}
     \caption{
         \label{fig:quad-sliding}
         From left to right: sliding cell $4$ of a non-standard quadruple-bubble with adjacency graph as in the second case of Figure~\ref{fig:5-graphs}, until it hits cell $5$. 
     }
\end{figure}

We now obtain a contradiction as in the proof of Lemma \ref{lem:degree-3}. 
$S_{123}$ is a generalized $(n-2)$-sphere and so it is
acted upon transitively by
a subgroup of isometries of $M^n$, and we may use these isometries to move
$\Omega_4$ along $\Sigma_{123}$ -- while preserving all volumes and surface areas -- until it collides with $\Omega_5$; see Figure~\ref{fig:quad-sliding}. 
If $\Omega$ were minimizing, this modified cluster would also be minimizing.
By Corollary~\ref{cor:quad-bubble-simplicial}, this means that a new interface is created at the time of collision, thereby strictly reducing the total perimeter of $\Omega$, and contradicting minimality.  

It remains to consider the third case in Figure~\ref{fig:5-graphs}. We will show that it cannot happen: there
is no spherical Voronoi cluster with that component adjacency graph.

\subsection{Bubble rings}

We say that a spherical Voronoi cluster is a \emph{bubble ring} if its cell adjacency graph consists of
a vertex of degree $q-1$, with the remaining $q-1$ vertices forming a cycle as in Figure~\ref{fig:bubble-ring}. We are no longer concerned with preserving volumes or surface areas, and only care about preserving the spherical Voronoi property and incidence structure of the cluster, so we may freely apply stereographic projections (by definition) and M\"obius transformations (by Lemma \ref{lem:Mobius-preserves-Voronoi}) in our analysis. Consequently, in $\R^n$, we may apply a M\"obius inversion and assume that the cell with $q-1$ neighbors is the unbounded cell, and then a bubble ring indeed looks like a collection of bubbles arranged in a ring as in Figure~\ref{fig:bubble-ring}. We say that a bubble ring in $M^n$ is \emph{non-degenerate} if every $p \in M^n$ belongs to the closure of at most three cells.

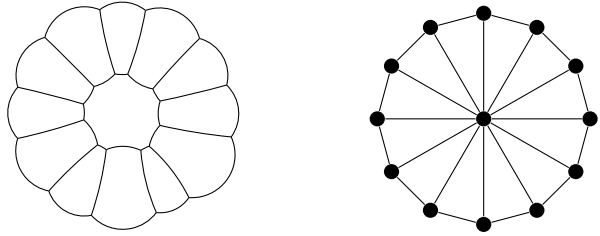
\begin{figure}
    \begin{center}
        \begin{tikzpicture}
            \input{bubble-ring}
        \end{tikzpicture}
        \hspace{40pt}
        \begin{tikzpicture}
            \node[vertex] (12) at (0, 0) {};
            \foreach \i in {0, ..., 11}
            {
                \pgfmathtruncatemacro{\theta}{\i * 360 / 12};
                \node[vertex] (\i) at (\theta:1.4) {};
                \draw (12) -- (\i);
            }
            \draw (11) -- (0);
            \foreach \i in {0, ..., 10}
            {
                \pgfmathtruncatemacro{\j}{\i + 1};
                \draw (\i) -- (\j);
            }
        \end{tikzpicture}
    \end{center}
        \caption{
                 \label{fig:bubble-ring}
         Left: a bubble ring with 12 bounded cells and one unbounded one. Right: its adjacency graph.
              }
\end{figure}

\begin{figure}
    \begin{center}
        \begin{tikzpicture}
            \input{mickey-labeled}
        \end{tikzpicture}
     \end{center}
     \caption{
         \label{fig:angles}
         The angles $\theta_{ij}$ for $(i,j) = (1, 2)$, $(2, 1)$ and $(2,3)$.
     }
\end{figure}
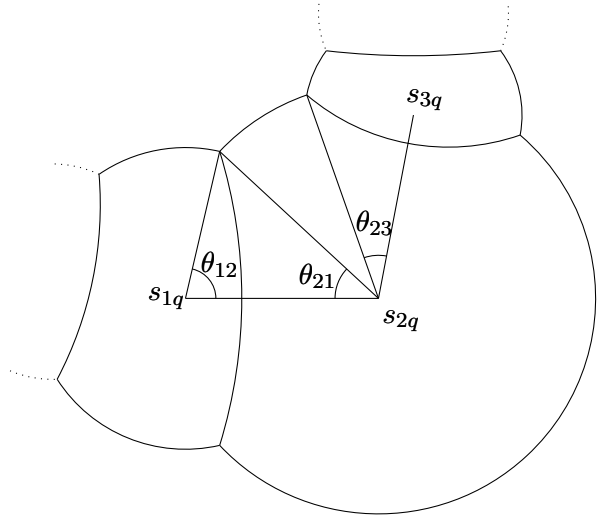

\begin{lemma} \label{lem:bubble-ring}
A non-degenerate bubble ring does not exist in $\R^n$ or $\S^n$ when $q \leq 7$. 
\end{lemma}

\begin{proof}
As already explained, it suffices to prove this in $\R^n$,
so suppose that $\Omega$ is a non-degenerate bubble ring in $\R^n$. We may relabel the cells so that $\Omega_q$ is the unbounded cell and $\Omega_1, \dots, \Omega_{q-1}$ are in the order in which they go around the cycle.
For $1 \le i \le q-1$, the interface $\Sigma_{iq}$ is a subset of the sphere $S_{iq}$ of radius $1/\k_{iq}$ centered at $\s_{iq} :=  -\c_{iq} / \k_{iq}$. 

Clearly $\Sigma_{ijq} \subset S_{ijq} = S_{iq} \cap S_{jq}$ for all $1 \le i \ne j \le q-1$. Moreover, by non-degeneracy, we actually have $\Sigma_{ijq} = S_{ijq}$; otherwise, $p \in \partial \Sigma_{ijq}$ would be in the closure of at least 4 cells $i,j,q,k$. In particular, for $p \in \Sigma_{ijq}$ the angle between $p - \s_{iq}$ and $\s_{jq} - \s_{iq}$ is constant; call it $\theta_{ij}$ and refer to Figure~\ref{fig:angles} for a depiction.

Consequently, for $i,j,k$ consecutive cells along the cycle, a two dimensional affine slice spanned by $\s_{iq},\s_{jq},\s_{kq}$ looks as in Figure~\ref{fig:angles}, and it follows that the angle between $\s_{iq} - \s_{jq}$ and $\s_{kq} - \s_{jq}$
is strictly greater than $\theta_{ji} + \theta_{jk}$.
Therefore the sum of the interior angles along the closed polygonal line
$\s_{1q}, \dots, \s_{q-1,q}, \s_{q1}$ is strictly greater than
\[
    (\theta_{1,q-1} + \theta_{1,2}) + (\theta_{2,1} + \theta_{2,3}) + \cdots + (\theta_{q-1,q-2} + \theta_{q-1,1}).
\]
On the other hand, $\theta_{ij} + \theta_{ji} = 120^{\circ}$, because together with a point $p \in \Sigma_{ij}$,
$\s_{iq}$ and $\s_{jq}$ form a triangle with angles $\theta_{ij}$, $\theta_{ji}$, and $60^{\circ}$ (by stationarity of any spherical Voronoi cluster, as interfaces meet in threes at $120^{\circ}$ angles). 
Hence, the sum of the interior angles along the closed polygonal line is strictly greater than $(q-1) 120^{\circ}$. But
since the sum of interior angles along a closed polygonal line with $q-1$ vertices is at most $(q-3) 180^{\circ}$ (with equality in the planar case), it follows that $q-1 > 6$.
\end{proof}
\begin{remark}
Note that the latter estimate on $q$ is sharp: to see this, consider an equilateral hexagon in the plane partitioned into 6 equal parts by rays from the origin to its vertices, and make its 6 faces into spherical arcs so that they meet at $120^\circ$ angles at the vertices. Together with the exterior cell, we obtain a bubble ring with 7 cells, which is however degenerate (as the 6 bounded cells all meet at the origin). By using a heptagon, one can already obtain a non-degenerate bubble ring with 8 cells. See Figure~\ref{fig:bubble-ring-67}. 
\end{remark}

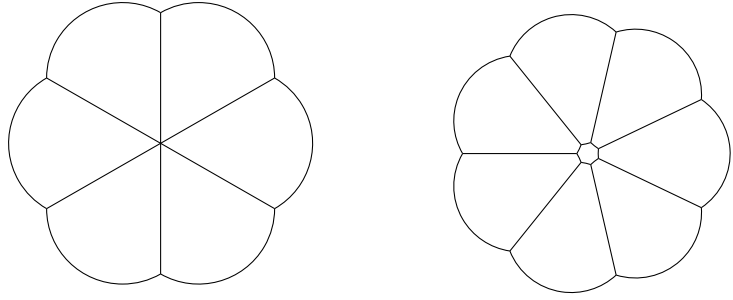
\begin{figure}
    \begin{center}
        \begin{tikzpicture}
            \input{bubble-ring-6}
        \end{tikzpicture}
        \hspace{40pt}
         \begin{tikzpicture}
            \input{bubble-ring-7}
        \end{tikzpicture}
     \end{center}
     \caption{
         \label{fig:bubble-ring-67}
         Left: a degenerate bubble ring with 6 bounded cells ($q=7$); Right: a non-degenerate bubble ring with 7 bounded cells ($q=8$). 
     }
\end{figure}

We now claim that it is not possible for a minimizing quadruple-bubble to have the third component adjacency
graph in Figure~\ref{fig:5-graphs}. Indeed, such a cluster
would have to be a non-degenerate bubble ring: non-degenerate because if the closure of $4$ or more cells were non-empty then the adjacency graph would contain a $K_4$ by Corollary~\ref{cor:quad-bubble-simplicial}, which the bubble ring graph does not. 
 But non-degenerate stationary bubble rings do not exist for $q=5$ by Lemma \ref{lem:bubble-ring}.

\medskip

This rules out all of the possible adjacency graphs besides $K_5$, thereby concluding the proof of Theorem \ref{thm:intro-234} in the quadruple-bubble case.

\appendix

\section{Contribution of curvature to Index-Form} \label{app:Ricci}

In this Appendix we extend the formula for $Q(X)$ given in Theorem \ref{thm:Q-Sigma4}, originally derived in the Euclidean setting in \cite[Appendix E]{EMilmanNeeman-GaussianMultiBubble}, to the Riemannian setting. 

Recall that the vector-field $X$ generates a flow $F_t$ via:
\[
\frac{d}{dt} F_t(x) = X \circ F_t(x)  ~,~ F_0 = \Id . 
\]
Let $\gamma(s,t) := F_t(\gamma(s))$ where $\gamma$ is a curve with $\gamma(0) = x$ and $\gamma'(0) = \xi \in T_x M$. 
Therefore:
\[
\frac{\partial \gamma}{\partial s} = d_x F_t \cdot \xi.
\]
Taking two time-derivatives using covariant differentiation (and denoting as usual $\frac{D}{\partial r} = \nabla_{\partial_r \gamma}$, $r \in \{t,s\}$), we obtain by the definition of the Riemann curvature tensor $\text{Riem}$ (and using $[\frac{D}{\partial s} , \frac{D}{\partial t}] \gamma = 0$):
\[
\brac{\frac{D}{dt}}^2 d_x F_t \cdot \xi = \frac{D}{dt} \frac{D}{dt} \frac{\partial \gamma}{\partial s} = 
\frac{D}{dt} \frac{D}{ds} \frac{\partial \gamma}{\partial t} = \frac{D}{ds}  \frac{D}{dt} \frac{\partial \gamma}{\partial t}  - \text{Riem}\brac{\frac{\partial \gamma}{\partial t},\frac{\partial \gamma}{\partial s}} \frac{\partial \gamma}{\partial t}  . 
\]
Note that we use the convention that the sectional curvature in the plane spanned by two orthonormal vectors $u,v$ is $\text{Riem}(u,v,u,v)$. Now:
\[
\frac{\partial \gamma}{\partial t} = X( \gamma(s,t)) ~,~ \frac{D}{dt} \frac{\partial \gamma}{\partial t}  = \nabla_X X (\gamma(s,t)) ~,~ \frac{\partial \gamma}{\partial s} = d_x F_t \cdot \xi  ,
\]
and hence:
\[
\left . \frac{D}{ds}  \frac{D}{dt} \frac{\partial \gamma}{\partial t} \right |_{t=0} = \nabla_{\xi} (\nabla_X X)  ~,~ \left . \frac{\partial \gamma}{\partial s} \right|_{t=0} = \xi . 
\]
As this is true for all $\xi \in T_x M$, we obtain:
\[
\brac{\frac{D}{dt}}^2 (d_x F_t)|_{t=0} = \nabla (\nabla_X X) - \text{Riem}(X,\cdot) X . 
\]
For the first derivative there is no curvature contribution and we get the usual:
\[
\frac{D}{dt} (d_x F_t)|_{t=0} = \nabla X . 
\]

Write $JF_t = \text{det}(dF_t)$ for the Jacobian of $F_t$, and observe that by the change-of-variables formula for smooth injective functions:
\[ \mu(F_t(U)) = \int_{U} J F_t  e^{-\pot \circ F_t} \, d\H^n,
\] for any Borel set $U$. Similarly, if $U$ is in addition of locally finite-perimeter, let $\Phi_t = F_t|_{\partial^* U}$ and write $J \Phi_t = \det((d_{\n_U^{\perp}} F_t)^T d_{\n_U^{\perp}} F_t)^{1/2}$ for the Jacobian of $\Phi_t$ on $\partial^* U$. Since $\partial^* U$ is locally $\H^{n-1}$-rectifiable, \cite[Proposition 17.1 and Theorem 11.6]{MaggiBook} imply:
\[ \mu^{n-1}(\partial^* F_t(U)) = \mu^{n-1}(F_t(\partial^* U)) = \int_{\partial^* U} J \Phi_t e^{-\pot \circ F_t} \, d\H^{n-1} . 
\] We assume in addition that $\partial^* U = \Sigma \cup \Xi$, where $\Sigma$ is a smooth $(n-1)$-dimensional manifold co-oriented by the outer unit-normal $\n$ and $\H^{n-1}(\Xi) = 0$.

\subsection{Second variation of volume}

We have:
\[
JF_t = \det(\Id + t A + \frac{t^2}{2} B + o(t^2)) = 1 + t \; \tr(A) + \frac{t^2}{2} \brac{\tr(B) + \tr(A)^2 - \tr(A^2) } + o(t^2) ,
\]
where, recall:
\[
A = \nabla X ~,~ B = \nabla (\nabla_X X) - \text{Riem}(X,\cdot) X .
\]
It follows that:
\[
\brac{\frac{d}{dt}}^2 (JF_t)|_{t=0} =  \tr(B) + \tr(A)^2 - \tr(A^2) = \div (\nabla_X X) - \Ric(X,X) + (\div X)^2 - \tr( (\nabla X)^2 ) . 
\]
Denoting by $\{e_i\}_{i=1,\ldots,n}$ Riemannian normal coordinates around a point $x_0 \in M$, and employing Einstein summation convention of summing over repeated indices, we compute:
\vspace{3pt}
\begin{align*}
\smash{\brac{\frac{d}{dt}}^2} (JF_t)|_{t=0} & = (\div X)^2 - \Ric(X,X) + \sscalar{\nabla_{e_i} \nabla_X X , e_i} - \sscalar{\nabla_{e_i} X , e_j} \sscalar{\nabla_{e_j} X , e_i} \\
& = (\div X)^2 - \Ric(X,X) + \sscalar{\nabla_{\nabla_{e_i} X} X + \nabla^2_{e_i,X} X , e_i} - \sscalar{\nabla_{e_i} X , e_j} \sscalar{\nabla_{e_j} X , e_i}  \\
& =  (\div X)^2 - \Ric(X,X) +  \sscalar{\nabla^2_{e_i,X} X , e_i} \\
& = (\div X)^2 + \sscalar{\nabla^2_{X,e_i} X , e_i} \\
& = (\div X)^2 + X^j \sscalar{\nabla^2_{e_j,e_i} X , e_i} \\
& = (\div X)^2 + X^j \nabla_{e_j} \sscalar{\nabla_{e_i} X , e_i} \\
& = \div(X \div X). 
\end{align*} 
It follows that the curvature of the manifold is not witnessed in the second variation of volume $\delta^2_X V$ (as expected), and that the formula:
\[
\delta^2_X V(U) = \int_{U} \div_\mu(X \div_\mu X)\, d\mu = \int_{\partial^* U} X^\n \div_\mu X\, d\mu^{n-1}
\]
obtained in the Euclidean setting equally applies in the Riemannian one.

\subsection{Second variation of area}

Similarly, for $x_0 \in \Sigma$, we identify $T_{x_0} \Sigma \subset T_{x_0} M$ with $\R^{n-1} \subset \R^n$, respectively, and write:
\begin{align*}
J \Phi_t & = \det ( (\Id_{n-1 \times n} + t  A + \frac{t^2}{2} B + o(t^2)) (\Id_{n-1 \times n} + t A + \frac{t^2}{2} B + o(t^2))^T)^{1/2} \\
& = 1 + t \; \tr(\tilde A) + \frac{t^2}{2} \brac{\tr(\tilde B) + \tr(\tilde A)^2 - \tr(\tilde A^2) + |A^{\perp}|^2} + o(t^2) ,
\end{align*}
where $\tilde C$ denotes the restriction of $C$ to the first $n-1$ columns, $C^{\perp}$ denotes the last column, and:
\[
A = \nabla^\tang X ~,~ B = \nabla^\tang (\nabla_X X) - \text{Riem}(X,\cdot|_{T \Sigma}) X .
\]
Consequently, if $\{\tau_i\}_{i=1,\ldots,n-1}$ are Riemannian normal coordinates on $T \Sigma$ around $x_0$, then:
\begin{align*}
\smash{\brac{\frac{d}{dt}}^2} (J\Phi_t)|_{t=0}  & =  \div_{\Sigma} (\nabla_X X) + (\div_{\Sigma} X)^2 - \sum_{i,j=1}^{n-1} \scalar{\nabla_{\tau_i} X,\tau_j} \scalar{\nabla_{\tau_j} X,\tau_i} + \sum_{i=1}^{n-1} \abs{\scalar{\nabla_{\tau_i} X , \n}}^2  \\
& - \sum_{i=1}^{n-1}\text{Riem}(X,\tau_i, X,\tau_i) .
\end{align*}
Denoting the tangential Ricci curvature tensor by:
\begin{equation} \label{eq:Ric-Sigma}
\Ric_{\Sigma}(X,Y) := \sum_{i=1}^{n-1} \text{Riem}(X,\tau_i,Y,\tau_i) = \Ric(X,Y) - \text{Riem}(X,\n,Y,\n), 
\end{equation}
it follows that the extra term with respect to the corresponding Euclidean expression derived in \cite[(E.3)]{EMilmanNeeman-GaussianMultiBubble} is $-\Ric_{\Sigma}(X,X)$. 

\subsection{Additional contribution of curvature}

The curvature also appears in two additional places in the Euclidean identity of \cite[Lemma E.1]{EMilmanNeeman-GaussianMultiBubble}:
\begin{itemize}
\item In the commutation:
\[
\sum_{i,j=1}^{n-1} X^j \scalar{\tau_i,\nabla_{\tau_i} \nabla_{\tau_j} X - \nabla_{\tau_j} \nabla_{\tau_i} X} = \sum_{i,j=1}^{n-1} X^j \text{Riem}(\tau_j,\tau_i,X,\tau_i)   = \Ric_{\Sigma}(X^\tang,X)  .
\]
\item
When invoking the Codazzi equation:
\begin{equation} \label{eq:Codazzi}
X^\n \sum_{i,j=1}^{n-1} X^j ( \nabla_{\tau_i} \II_{ji} - \nabla_{\tau_j} \II_{ii}) = X^\n \sum_{i,j=1}^{n-1} X^j  \text{Riem}(\tau_j,\tau_i,\n,\tau_i)  = \Ric_{\Sigma}(X^\tang , X^{\n} \n) . 
\end{equation}
\end{itemize}

Summing all of the curvature contribution together, we obtain:
\[
- \Ric_{\Sigma}(X,X) + \Ric_{\Sigma}(X^\tang,X) + \Ric_{\Sigma}(X^\tang, X^{\n} \n)   = - \Ric_{\Sigma}(X^\n \n , X^\n \n) = - \Ric(\n,\n) (X^\n)^2 ,
\]
where the last transition follows since $\text{Riem}(\n,\n,\n,\n) = 0$. 
It follows that the expressions for $\delta^2_X A$ and thus $Q(X)$ derived in \cite[Appendix E]{EMilmanNeeman-GaussianMultiBubble} contain this extra curvature term in the (weighted) Riemannian setting, thereby verifying the formula stated in Theorem \ref{thm:Q-Sigma4}.

\section{A Technical Computation} \label{app:tedious}

In this Appendix, we verify a calculation used in the proof of Lemma \ref{lem:Stokes-integrability}. Let $\Sigma$ be a smooth hypersurface in the smooth weighted Riemannian manifold $(M,g,\mu = \exp(-W) d\vol_g)$ co-oriented by the unit-normal $\n$, and let $\II$ denote its second fundamental form. Let $X,Y$ denote two smooth vector-fields on $(M,g)$. We use $\alpha,\beta$ to denote tensor indices on $T \Sigma$ and $\gamma$ on $T M$.  We freely raise and lower indices using the metric as needed; in particular, $(\II Y^{\tang})^{\alpha} = \II^{\alpha}_\beta Y^\beta$ denotes the tangential vector-field obtained by applying the Weingarten map to $Y^{\tang}$. Recall that $Y^{\tang}$ denotes the tangential projection of $Y$ onto $T \Sigma$, that $\div_{\Sigma,\mu}$ denotes the weighted tangential divergence on $\Sigma$, and that $H_{\Sigma,\mu}$ denotes its weighted mean-curvature. 

\begin{lemma} \label{lem:tedious}
\begin{align}
\label{eq:app-tedious}
\div_{\Sigma,\mu}(X^\n \II Y^\tang) = \; & \II^2_{\alpha \beta} X^{\alpha} Y^{\beta} + \II_{\alpha\beta} (\nabla^{\alpha} X^\gamma) \n_{\gamma} Y^{\beta} + X^\n \II_{\alpha\beta} \nabla^\alpha Y^\beta - 
\norm{\II}^2 X^\n Y^\n \\
\nonumber & + \Ric_{\Sigma,\mu}(Y^{\tang}, X^\n \n) + X^\n \nabla_{Y^\tang} H_{\Sigma,\mu} . 
\end{align}
Here $\II^2$ denotes the square of the Weingarten map as an operator on $T \Sigma$ and 
\begin{equation} \label{eq:Ric-Sigma-weighted}
\Ric_{\Sigma,\mu} = \Ric_{\Sigma} + \nabla^2 W
\end{equation}
denotes the weighted tangential Ricci $2$-tensor (recall the unweighted version $\Ric_{\Sigma}$ was defined in (\ref{eq:Ric-Sigma})). 
\end{lemma}
\begin{proof}
It will be more convenient to perform the computation in Riemannian normal coordinates $\{ \tau_i \}_{i=1,\ldots,n-1}$ for $T \Sigma$ around a fixed $x_0 \in \Sigma$. We denote $\II_{ij} = \II(\tau_i,\tau_j)$ and $Z^i = \scalar{Z,\tau_i}$ for a vector-field $Z$. We abbreviate $\nabla_{\tau_i}$ by $\nabla_i$, and reserve the indices $i,j,k$ to range over $\{1,\ldots,n-1\}$. 
Note that for a tangential vector-field $Z$ such as $X^\n \II \Y^\tang$, $\scalar{\nabla_{\tau_i} Z , \tau_i} = \nabla_{\tau_i} \scalar{Z, \tau_i}$. Consequently, employing Einstein summation convention of summing over repeated indices:
\begin{align*}
\div_{\Sigma}(X^\n \II Y^\tang) & = \nabla_i(X^{\n} \II^i_j Y^j) = \II^i_j \nabla_{i} (X^\n Y^j) + X^\n Y^j \nabla_i \II^i_j  \\
 & = \II^i_j Y^j ( \nabla_i X^\gamma \n_\gamma + X^k \II_{ik}) + \II^i_j X^{\n} ( (\nabla Y)_{i}^j - Y^\n \II_i^j) + X^\n Y^j \nabla_i \II^i_j . 
\end{align*}
The first four terms above correspond to the first four terms in (\ref{eq:app-tedious}). As for the last term $X^\n Y^j \nabla_i \II^i_j$, we observe using the Codazzi equation (\ref{eq:Codazzi}) (in its polarized form) that:
\begin{equation} \label{eq:app-tediousA}
X^\n Y^j \nabla_i \II^i_j = \Ric_{\Sigma}(Y^{\tang} , X^{\n} \n) + X^\n Y^j \nabla_j \II^{i}_i  = \Ric_{\Sigma}(Y^{\tang} , X^{\n} \n) + X^\n \nabla_{Y^\tang} H_{\Sigma}   . 
\end{equation}
This concludes the computation for the unweighted divergence. As for the weighted one, we have:
\begin{align} 
\nonumber \div_{\Sigma,\mu}(X^\n \II Y^\tang) & = \div_{\Sigma}(X^\n \II Y^\tang) - \scalar{\nabla W , X^\n \II Y^{\tang}} \\
\label{eq:app-tediousB}  & =  \div_{\Sigma}(X^\n \II Y^\tang) + X^\n \nabla^2_{Y^\tang,\n} W - X^\n \nabla_{Y^\tang} \nabla_\n W .
\end{align}
Recalling that $H_{\Sigma,\mu} = H_{\Sigma} - \nabla_\n W$, we see that the last two terms in (\ref{eq:app-tediousA}) and (\ref{eq:app-tediousB}) combine to yield $X^\n \nabla_{Y^\tang} H_{\Sigma,\mu}$. The corresponding two penultimate terms combine to yield $\Ric_{\Sigma,\mu}(Y^{\tang} , X^{\n} \n)$. This establishes (\ref{eq:app-tedious}) and concludes the proof. 
\end{proof}

\bibliographystyle{plain}
\bibliography{../../../ConvexBib}

\end{document}

%% file: triple-bubble.tex
\draw (0.22501758018520482, 0.9743547036924464) -- (0.17603892813789207, 0.9843832057588457);
\draw (0.17603892813789207, 0.9843832057588457) -- (0.12662027044953608, 0.9919512624677114);
\draw (0.12662027044953608, 0.9919512624677114) -- (0.07688512802761835, 0.9970399576186386);
\draw (0.07688512802761857, 0.9970399576186386) -- (0.026957812826632434, 0.9996365721238916);
\draw (0.026957812826632434, 0.9996365721238916) -- (-0.023036882867062763, 0.9997346157994976);
\draw (-0.023036882867062763, 0.9997346157994976) -- (-0.07297399835096698, 0.9973338435873281);
\draw (-0.07297399835096698, 0.9973338435873281) -- (-0.12272871684311744, 0.9924402561676153);
\draw (-0.12272871684311744, 0.9924402561676153) -- (-0.17217667745904028, 0.9850660849603775);
\draw (-0.17217667745904006, 0.9850660849603775) -- (-0.22119428604919553, 0.9752297615532386);
\draw (-0.22119428604919553, 0.9752297615532386) -- (-0.2696590241199853, 0.9629558716320583);
\draw (-0.2696590241199853, 0.9629558716320583) -- (-0.31744975506618006, 0.9482750935295212);
\draw (-0.31744975506618006, 0.9482750935295212) -- (-0.36444702694934594, 0.9312241215452823);
\draw (-0.36444702694934594, 0.9312241215452823) -- (-0.41053337106548093, 0.9118455742293275);
\draw (-0.41053337106548077, 0.9118455742293277) -- (-0.4555935955555985, 0.8901878878577948);
\draw (-0.45559359555559886, 0.8901878878577946) -- (-0.49951507332538553, 0.8663051953675072);
\draw (-0.49951507332538553, 0.8663051953675072) -- (-0.54218802355428, 0.8402571910518253);
\draw (-0.5421880235542803, 0.840257191051825) -- (-0.5835057860903533, 0.8121089813559993);
\draw (-0.5835057860903533, 0.8121089813559993) -- (-0.6233650880451393, 0.7819309221449652);
\draw (-0.6233650880451393, 0.7819309221449652) -- (-0.6616663019220747, 0.7497984428503209);
\draw (-0.661666301922075, 0.7497984428503206) -- (-0.6983136946333538, 0.7157918579360308);
\draw (-0.6983136946333538, 0.7157918579360308) -- (-0.7332156667827912, 0.6799961661540945);
\draw (-0.7332156667827915, 0.6799961661540942) -- (-0.7662849816166095, 0.64250083809193);
\draw (-0.7662849816166095, 0.64250083809193) -- (-0.7974389830698898, 0.6033995925425042);
\draw (-0.7974389830698898, 0.6033995925425042) -- (-0.8265998023636896, 0.5627901622561551);
\draw (-0.8265998023636898, 0.5627901622561547) -- (-0.8536945526364356, 0.5207740496596158);
\draw (-0.8536945526364356, 0.5207740496596158) -- (-0.8786555111231188, 0.4774562731528102);
\draw (-0.878655511123119, 0.4774562731528098) -- (-0.9014202884269387, 0.4329451046175423);
\draw (-0.9014202884269387, 0.4329451046175423) -- (-0.9219319844603022, 0.3873517987941828);
\draw (-0.9219319844603022, 0.3873517987941828) -- (-0.940139330665411, 0.340790315202755);
\draw (-0.9401393306654111, 0.34079031520275455) -- (-0.9559968181589549, 0.29337703330348514);
\draw (-0.9559968181589549, 0.29337703330348514) -- (-0.9694648114806196, 0.24523046160876233);
\draw (-0.9694648114806197, 0.24523046160876188) -- (-0.9805096476610967, 0.19647094147357275);
\draw (-0.9805096476610967, 0.19647094147357275) -- (-0.9891037203619741, 0.147220346304788);
\draw (-0.9891037203619741, 0.147220346304788) -- (-0.9952255488772064, 0.0976017769411155);
\draw (-0.9952255488772065, 0.09760177694111506) -- (-0.9988598318236921, 0.04773925396511261);
\draw (-0.9988598318236921, 0.04773925396511261) -- (-0.9999974853867629, -0.0022425922836797714);
\draw (-0.9999974853867629, -0.0022425922836802155) -- (-0.9986356660249898, -0.05221883321968288);
\draw (-0.9986356660249898, -0.05221883321968288) -- (-0.9947777775775553, -0.10206455426767898);
\draw (-0.9947777775775553, -0.10206455426767898) -- (-0.9884334627564285, -0.15165516708419818);
\draw (-0.9884334627564284, -0.15165516708419863) -- (-0.9796185790446069, -0.2008667209634916);
\draw (-0.9796185790446069, -0.2008667209634916) -- (-0.9683551590606676, -0.24957621264974186);
\draw (-0.9683551590606675, -0.24957621264974228) -- (-0.9546713554886944, -0.29766189378114716);
\draw (-0.9546713554886944, -0.29766189378114716) -- (-0.9386013707112303, -0.3450035751974167);
\draw (-0.9386013707112303, -0.3450035751974167) -- (-0.9201853713211328, -0.3914829273500814);
\draw (-0.9201853713211325, -0.3914829273500818) -- (-0.8994693877260095, -0.43698377606473854);
\draw (-0.8994693877260095, -0.43698377606473854) -- (-0.8765051990961732, -0.48139239291598473);
\draw (-0.876505199096173, -0.4813923929159851) -- (-0.8513502039436797, -0.5245977794892531);
\draw (-0.8513502039436797, -0.5245977794892531) -- (-0.8240672766559424, -0.5664919448190402);
\draw (-0.8240672766559421, -0.5664919448190405) -- (-0.7947246103425064, -0.6069701753100817);
\draw (-0.7947246103425064, -0.6069701753100817) -- (-0.7633955463877906, -0.6459312964668043);
\draw (-0.7633955463877906, -0.6459312964668043) -- (-0.7301583911358205, -0.6832779257768762);
\draw (-0.7301583911358202, -0.6832779257768764) -- (-0.6950962201651496, -0.718916716116771);
\draw (-0.6950962201651496, -0.718916716116771) -- (-0.6582966706431798, -0.7527585890709616);
\draw (-0.6582966706431794, -0.7527585890709619) -- (-0.61985172227888, -0.7847189575815702);
\draw (-0.6198517222788807, -0.7847189575815696) -- (-0.5798574674214211, -0.8147179373719567);
\draw (-0.5798574674214203, -0.8147179373719573) -- (-0.5384138708793383, -0.842680546615814);
\draw (-0.5384138708793383, -0.842680546615814) -- (-0.4956245200605816, -0.8685368933526763);
\draw (-0.4956245200605816, -0.8685368933526763) -- (-0.4515963660579337, -0.8922223501814269);
\draw (-0.4515963660579329, -0.8922223501814273) -- (-0.40643945632698114, -0.9136777157951418);
\draw (-0.406439456326982, -0.9136777157951415) -- (-0.3602666596247898, -0.9328493629535242);
\draw (-0.360266659624789, -0.9328493629535245) -- (-0.3131933838967842, -0.94968937252308);
\draw (-0.3131933838967842, -0.94968937252308) -- (-0.2653372878169987, -0.9641556532499922);
\draw (-0.2653372878169987, -0.9641556532499922) -- (-0.24100447591911336, -0.970524004127128);
\draw (0.6591181606429131, -1.2914932367352496) -- (0.7244290406039676, -1.3203976816951914);
\draw (0.7244290406039676, -1.3203976816951914) -- (0.7911029191190231, -1.3460018201007147);
\draw (0.7911029191190231, -1.3460018201007147) -- (0.8589731462148766, -1.3682416549401832);
\draw (0.8589731462148777, -1.3682416549401837) -- (0.9278700816699208, -1.3870615982087808);
\draw (0.9278700816699208, -1.3870615982087808) -- (0.9976215190263459, -1.4024146098495722);
\draw (0.9976215190263459, -1.4024146098495722) -- (1.0680531160166085, -1.4142623153291496);
\draw (1.0680531160166085, -1.4142623153291496) -- (1.1389888303282993, -1.4225751015539798);
\draw (1.1389888303282993, -1.4225751015539798) -- (1.2102513596182454, -1.4273321908877163);
\draw (1.2102513596182467, -1.4273321908877166) -- (1.2816625846760308, -1.4285216930844689);
\draw (1.2816625846760308, -1.4285216930844689) -- (1.3530440146292555, -1.4261406350082264);
\draw (1.3530440146292555, -1.4261406350082264) -- (1.424217233077768, -1.4201949680641484);
\draw (1.424217233077768, -1.4201949680641484) -- (1.4950043440417424, -1.4106995533231514);
\draw (1.4950043440417438, -1.4106995533231512) -- (1.565228416608985, -1.3976781243769718);
\draw (1.565228416608985, -1.3976781243769718) -- (1.6347139271700648, -1.3811632280165458);
\draw (1.6347139271700648, -1.3811632280165458) -- (1.7032871981359352, -1.3611961428819828);
\draw (1.7032871981359352, -1.3611961428819828) -- (1.7707768320414505, -1.3378267762874623);
\draw (1.7707768320414505, -1.3378267762874623) -- (1.8370141399497688, -1.3111135394789395);
\draw (1.83701413994977, -1.3111135394789388) -- (1.9018335630868433, -1.2811232016364504);
\draw (1.9018335630868433, -1.2811232016364504) -- (1.9650730866521324, -1.2479307229859367);
\draw (1.9650730866521324, -1.2479307229859367) -- (2.026574644771231, -1.2116190674377152);
\draw (2.026574644771231, -1.2116190674377152) -- (2.086184515578234, -1.1722789952199115);
\draw (2.086184515578234, -1.1722789952199115) -- (2.143753705440344, -1.1300088360251568);
\draw (2.143753705440345, -1.1300088360251561) -- (2.199138321364362, -1.084914243237564);
\draw (2.199138321364362, -1.084914243237564) -- (2.252199930654225, -1.0371079298542936);
\draw (2.252199930654225, -1.0371079298542936) -- (2.3028059069206623, -0.9867093867617563);
\draw (2.3028059069206623, -0.9867093867617563) -- (2.3508297615780904, -0.9338445840706295);
\draw (2.3508297615780904, -0.9338445840706295) -- (2.3961514600002065, -0.878645656256184);
\draw (2.3961514600002074, -0.8786456562561831) -- (2.438657721544039, -0.821250571890909);
\draw (2.438657721544039, -0.821250571890909) -- (2.478242302692561, -0.7618027887949377);
\draw (2.478242302692561, -0.7618027887949377) -- (2.514806262608153, -0.700450895466202);
\draw (2.514806262608153, -0.700450895466202) -- (2.548258210433168, -0.6373482396865753);
\draw (2.548258210433168, -0.6373482396865753) -- (2.5785145337194812, -0.572652545232279);
\draw (2.5785145337194817, -0.5726525452322778) -- (2.605499607416066, -0.5065255176465813);
\draw (2.605499607416066, -0.5065255176465813) -- (2.6291459828922346, -0.43913244006015817);
\draw (2.6291459828922346, -0.43913244006015817) -- (2.649394556524091, -0.37064176006933863);
\draw (2.649394556524091, -0.37064176006933863) -- (2.6661947174228096, -0.301224668704843);
\draw (2.6661947174228096, -0.301224668704843) -- (2.6795044739354976, -0.23105467254336007);
\draw (2.6795044739354976, -0.23105467254335882) -- (2.689290558602461, -0.1603071600314648);
\draw (2.689290558602461, -0.1603071600314648) -- (2.6955285113085266, -0.08915896310584723);
\draw (2.6955285113085266, -0.08915896310584723) -- (2.6982027404205917, -0.017787915205555383);
\draw (2.6982027404205917, -0.017787915205555383) -- (2.6973065617585883, 0.053627593218983194);
\draw (2.6973065617585883, 0.053627593218983194) -- (2.6928422153024503, 0.12490906058918495);
\draw (2.6928422153024503, 0.12490906058918622) -- (2.6848208595933314, 0.19587832035929625);
\draw (2.684820859593332, 0.195878320359295) -- (2.6732625438430637, 0.26635798633996405);
\draw (2.6732625438430633, 0.26635798633996527) -- (2.6581961578215676, 0.3361718960713421);
\draw (2.6581961578215676, 0.3361718960713421) -- (2.6396593596474744, 0.40514555113747974);
\draw (2.6396593596474744, 0.40514555113747974) -- (2.6176984816624427, 0.4731065533214978);
\draw (2.6176984816624422, 0.473106553321499) -- (2.5923684146244295, 0.5398850355113433);
\draw (2.5923684146244295, 0.5398850355113433) -- (2.563732470509387, 0.6053140862791004);
\draw (2.563732470509387, 0.6053140862791004) -- (2.5318622242642914, 0.6692301670726419);
\draw (2.5318622242642914, 0.6692301670726419) -- (2.4968373349070516, 0.7314735209768355);
\draw (2.4968373349070516, 0.7314735209768355) -- (2.458745346420451, 0.7918885720226324);
\draw (2.4587453464204505, 0.7918885720226335) -- (2.417681468937781, 0.85032431404597);
\draw (2.417681468937781, 0.85032431404597) -- (2.3737483407670967, 0.9066346881245335);
\draw (2.3737483407670967, 0.9066346881245335) -- (2.327055771848892, 0.9606789476490037);
\draw (2.327055771848892, 0.9606789476490037) -- (2.2777204692884436, 1.012322010116276);
\draw (2.2777204692884436, 1.012322010116276) -- (2.2258657456488233, 1.0614347947653686);
\draw (2.2258657456488224, 1.0614347947653695) -- (2.171621210733711, 1.1078945452121038);
\draw (2.171621210733711, 1.1078945452121038) -- (2.115122447630389, 1.1515851362761333);
\draw (2.115122447630389, 1.1515851362761333) -- (2.0565106738226273, 1.192397364233418);
\draw (2.0565106738226273, 1.192397364233418) -- (1.995932388220524, 1.230229219768662);
\draw (1.995932388220524, 1.230229219768662) -- (1.9335390049895227, 1.2649861429454765);
\draw (1.9335390049895216, 1.2649861429454772) -- (1.8694864750938542, 1.296581259556976);
\draw (1.8694864750938542, 1.296581259556976) -- (1.8039348965003472, 1.3249355982660491);
\draw (1.8039348965003472, 1.3249355982660491) -- (1.737048114016882, 1.349978287992578);
\draw (1.737048114016882, 1.349978287992578) -- (1.6689933097657046, 1.3716467350542265);
\draw (1.6689933097657046, 1.3716467350542265) -- (1.5999405853151858, 1.3898867796180505);
\draw (1.5999405853151845, 1.3898867796180507) -- (1.5300625365144866, 1.4046528310718716);
\draw (1.5300625365144866, 1.4046528310718716) -- (1.4595338220938276, 1.4159079819770657);
\draw (1.4595338220938276, 1.4159079819770657) -- (1.3885307271086147, 1.4236241003179397);
\draw (1.3885307271086147, 1.4236241003179397) -- (1.3172307223186133, 1.427781899817118);
\draw (1.3172307223186133, 1.427781899817118) -- (1.2458120206034742, 1.4283709881411926);
\draw (1.245812020603473, 1.4283709881411926) -- (1.174453131523343, 1.4253898928761446);
\draw (1.174453131523343, 1.4253898928761446) -- (1.1033324151379318, 1.4188460652076125);
\draw (1.1033324151379318, 1.4188460652076125) -- (1.0326276361992481, 1.4087558612968076);
\draw (1.0326276361992506, 1.408755861296808) -- (0.9625155198323054, 1.3951445013986312);
\draw (0.9625155198323054, 1.3951445013986312) -- (0.8931713098143319, 1.3780460068241682);
\draw (0.8931713098143319, 1.3780460068241682) -- (0.8247683305566085, 1.3575031149051318);
\draw (0.8247683305566085, 1.3575031149051318) -- (0.7574775538836966, 1.333567172172788);
\draw (0.7574775538836991, 1.3335671721727886) -- (0.6914671716929218, 1.3062980060183704);
\draw (0.6914671716929242, 1.3062980060183715) -- (0.6269021755622161, 1.2757637751557576);
\draw (0.6269021755622161, 1.2757637751557576) -- (0.5639439443571039, 1.2420407992601847);
\draw (0.5639439443571039, 1.2420407992601847) -- (0.5027498408675874, 1.2052133682088053);
\draw (0.5027498408675874, 1.2052133682088053) -- (0.4434728184831176, 1.1653735313998923);
\draw (0.4434728184831197, 1.1653735313998939) -- (0.38626103888878105, 1.122620867677284);
\draw (0.38626103888878316, 1.1226208676772855) -- (0.33125750173824264, 1.0770622364351283);
\draw (0.33125750173824264, 1.0770622364351283) -- (0.27859968722908823, 1.0288115105250526);
\draw (0.27859968722908823, 1.0288115105250526) -- (0.2284192124739335, 0.9779892916333413);
\draw (0.2284192124739335, 0.9779892916333413) -- (0.2250175801852059, 0.9743547036924475);
\draw (-0.24100447591911245, -0.970524004127128) -- (-0.23216768760702144, -0.9985736542374696);
\draw (-0.23216768760702144, -0.9985736542374696) -- (-0.22194004476841733, -1.0261465942503016);
\draw (-0.22194004476841722, -1.026146594250302) -- (-0.21034711118394311, -1.0531739061753023);
\draw (-0.21034711118394311, -1.0531739061753023) -- (-0.19741786315007703, -1.07958803580821);
\draw (-0.1974178631500768, -1.0795880358082104) -- (-0.1831846170534812, -1.1053229615811562);
\draw (-0.1831846170534812, -1.1053229615811562) -- (-0.16768294859686628, -1.130314359582198);
\draw (-0.16768294859686617, -1.1303143595821983) -- (-0.15095160387826057, -1.154499764331602);
\draw (-0.15095160387826057, -1.154499764331602) -- (-0.133032402545945, -1.1778187249130094);
\draw (-0.133032402545945, -1.1778187249130094) -- (-0.11397013327111055, -1.2002129560692467);
\draw (-0.11397013327111039, -1.200212956069247) -- (-0.09381244179950626, -1.2216264838851139);
\draw (-0.09381244179950626, -1.2216264838851139) -- (-0.07260971186188864, -1.2420057856930204);
\draw (-0.07260971186188842, -1.2420057856930207) -- (-0.0504149392409331, -1.2612999238517824);
\draw (-0.0504149392409331, -1.2612999238517824) -- (-0.027283599309378237, -1.2794606730641949);
\draw (-0.027283599309378237, -1.2794606730641949) -- (-0.0032735083704838464, -1.296442640915163);
\draw (-0.0032735083704838464, -1.296442640915163) -- (0.021555320852616466, -1.312203381329097);
\draw (0.02155532085261691, -1.3122033813290974) -- (0.047140829217469704, -1.326703500662997);
\draw (0.047140829217469704, -1.326703500662997) -- (0.07341906627783806, -1.3399067561700444);
\draw (0.07341906627783806, -1.3399067561700444) -- (0.10032435012651236, -1.3517801465875938);
\draw (0.10032435012651236, -1.3517801465875938) -- (0.12778943156587183, -1.362293994623144);
\draw (0.12778943156587183, -1.362293994623144) -- (0.155745662195856, -1.3714220211321122);
\draw (0.15574566219585653, -1.3714220211321124) -- (0.18412316599921388, -1.3791414108020104);
\draw (0.18412316599921388, -1.3791414108020104) -- (0.21285101399515438, -1.3854328691788438);
\draw (0.21285101399515438, -1.3854328691788438) -- (0.24185740152486215, -1.390280670893199);
\draw (0.24185740152486215, -1.390280670893199) -- (0.2710698277257475, -1.3936726989654762);
\draw (0.2710698277257475, -1.3936726989654762) -- (0.3004152767458457, -1.3956004750920288);
\draw (0.3004152767458463, -1.3956004750920288) -- (0.32982040024542153, -1.3960591808365068);
\draw (0.32982040024542153, -1.3960591808365068) -- (0.35921170072961717, -1.3950476696734384);
\draw (0.35921170072961717, -1.3950476696734384) -- (0.3885157152539161, -1.3925684698539464);
\draw (0.3885157152539161, -1.3925684698539464) -- (0.4176591990432433, -1.3886277780864376);
\draw (0.4176591990432433, -1.3886277780864376) -- (0.4465693085657584, -1.3832354440480574);
\draw (0.44656930856575894, -1.3832354440480574) -- (0.4751737836037496, -1.3764049457656289);
\draw (0.4751737836037496, -1.3764049457656289) -- (0.5034011278665436, -1.3681533559276038);
\draw (0.5034011278665436, -1.3681533559276038) -- (0.5311807876940006, -1.3585012992112317);
\draw (0.5311807876940006, -1.3585012992112317) -- (0.5584433284039191, -1.3474729007316097);
\draw (0.5584433284039191, -1.3474729007316097) -- (0.585120607842581, -1.335095725741458);
\draw (0.5851206078425815, -1.335095725741458) -- (0.6111459467046487, -1.3214007107323438);
\draw (0.6111459467046487, -1.3214007107323438) -- (0.6364542951967007, -1.3064220861095641);
\draw (0.6364542951967007, -1.3064220861095641) -- (0.6591181606429138, -1.29149323673525);
\draw (0.3135087595093453, -0.6142972030517851) -- (0.34011637622970436, -0.44978572638394126);
\draw (0.34011637622970436, -0.44978572638394126) -- (0.35846859341830095, -0.28415001964401554);
\draw (0.35846859341830095, -0.28415001964401554) -- (0.3685195400898138, -0.11780408583744938);
\draw (0.36851954008981425, -0.1178040858374464) -- (0.37024409411199555, 0.04883629683253149);
\draw (0.37024409411199555, 0.04883629683253149) -- (0.3636379449979219, 0.21535461419387983);
\draw (0.3636379449979219, 0.21535461419387983) -- (0.34871760467996094, 0.3813346571742584);
\draw (0.34871760467996094, 0.3813346571742584) -- (0.32552036623854486, 0.5463615621069517);
\draw (0.32552036623854486, 0.5463615621069517) -- (0.2941042106888898, 0.7100228476739787);
\draw (0.29410421068888937, 0.7100228476739816) -- (0.25454766205865065, 0.8719094458945745);
\draw (0.25454766205865065, 0.8719094458945745) -- (0.225017580185205, 0.974354703692445);
\draw (0.313508759509346, -0.6142972030517833) -- (0.24674807677993604, -0.6396741457603252);
\draw (0.24674807677993604, -0.6396741457603252) -- (0.18133914603499446, -0.6683560173612688);
\draw (0.18133914603499446, -0.6683560173612688) -- (0.11744545553707131, -0.700271128112842);
\draw (0.1174454555370702, -0.7002711281128424) -- (0.05522670623721948, -0.7353397068592344);
\draw (0.05522670623721948, -0.7353397068592344) -- (-0.0051615873942076895, -0.7734741004169419);
\draw (-0.0051615873942076895, -0.7734741004169419) -- (-0.06356848607274779, -0.8145789926620786);
\draw (-0.06356848607274779, -0.8145789926620786) -- (-0.11984800296942955, -0.8585516427710462);
\draw (-0.11984800296942955, -0.8585516427710462) -- (-0.173859468601817, -0.9052821420190855);
\draw (-0.173859468601818, -0.9052821420190862) -- (-0.2254678824344576, -0.9546536884948511);
\draw (-0.2254678824344576, -0.9546536884948511) -- (-0.24100447591911167, -0.9705240041271279);
\draw (0.6591181606429138, -1.29149323673525) -- (0.6537817411888638, -1.241784064166418);
\draw (0.6537817411888638, -1.241784064166418) -- (0.6459675677194314, -1.1924037249304718);
\draw (0.6459675677194314, -1.1924037249304718) -- (0.6356951715987476, -1.1434756441587177);
\draw (0.6356951715987476, -1.1434756441587177) -- (0.6229902284673535, -1.095122116571834);
\draw (0.6229902284673535, -1.095122116571834) -- (0.607884494066471, -1.0474640008067566);
\draw (0.607884494066471, -1.0474640008067566) -- (0.5904157248651882, -1.0006204173330318);
\draw (0.5904157248651882, -1.0006204173330318) -- (0.570627583688949, -0.9547084507136773);
\draw (0.570627583688949, -0.9547084507136773) -- (0.5485695305852292, -0.9098428569547512);
\draw (0.5485695305852292, -0.9098428569547512) -- (0.5242966991991761, -0.8661357766751014);
\draw (0.5242966991991761, -0.8661357766751014) -- (0.49786975896820884, -0.8236964548132197);
\draw (0.49786975896820884, -0.8236964548132197) -- (0.46935476348002025, -0.7826309675717892);
\draw (0.46935476348002025, -0.7826309675717892) -- (0.43882298537300835, -0.7430419572824211);
\draw (0.43882298537300835, -0.7430419572824211) -- (0.4063507381917978, -0.7050283758532808);
\draw (0.4063507381917979, -0.7050283758532809) -- (0.37201918564312253, -0.6686852374408508);
\draw (0.37201918564312253, -0.6686852374408508) -- (0.3359141387288269, -0.6341033809640214);
\draw (0.3359141387288269, -0.6341033809640214) -- (0.3135087595093461, -0.6142972030517836);

%% file: cluster-stereographic-color.tex
\path[draw=black,join=round,line cap=round,fill=white]
(-0.17473456106980606, 0.15261593264972131)
arc (422.17767179614356:657.8223282038565:0.17256453883459194)
arc (237.8223282038588:256.90989283616875:0.18031166690864125)
arc (196.9098928361688:272.67724808823687:0.6038018934107716)
arc (212.6772480882369:507.3227519117631:1.1171254135997817)
arc (447.32275191176313:523.0901071638312:0.6038018934107716)
arc (463.09010716383125:482.17767179614117:0.18031166690864125)
;
\path[draw=black,join=round,line cap=round,fill=green]
(0.778251208361622, -0.18901250493923663)
arc (265.66191852057625:454.33808147942375:0.1895555666872178)
arc (394.33808147942403:434.2040028896514:0.3350840027383518)
arc (374.20400288965146:387.3227519117631:1.3140298727408297)
arc (507.3227519117631:212.6772480882369:1.1171254135997817)
arc (332.6772480882369:345.7959971103484:1.3140298727408297)
arc (285.79599711034854:325.66191852057597:0.3350840027383518)
;
\path[draw=black,join=round,line cap=round,fill=magenta!90!cyan]
(0.4966377525794205, 0.18552582028004236)
arc (455.6392076859911:624.3607923140089:0.1864280592536302)
arc (204.360792314009:225.79599711034854:0.4497801185622753)
arc (345.7959971103485:332.6772480882369:1.3140298727408297)
arc (272.67724808823687:196.90989283616878:0.6038018934107716)
arc (316.90989283616886:403.09010716383114:0.2570839161572535)
arc (163.09010716383125:87.3227519117631:0.6038018934107716)
arc (27.32275191176309:14.204002889651491:1.3140298727408297)
arc (494.2040028896514:515.639207685991:0.4497801185622753)
;
\path[draw=black,join=round,line cap=round,fill=red]
(-0.1747345610698039, -0.1526159326497203)
arc (357.8223282038565:362.1776717961435:4.016378630787581)
arc (122.1776717961412:103.09010716383122:0.18031166690864125)
arc (403.09010716383114:316.90989283616886:0.2570839161572535)
arc (256.90989283616875:237.8223282038588:0.18031166690864125)
;
\path[draw=black,join=round,line cap=round,fill=yellow]
(-0.1747345610698039, -0.15261593264972081)
arc (357.8223282038565:362.1776717961435:4.016378630787581)
arc (62.177671796143564:297.82232820385644:0.17256453883459194)
;
\path[draw=black,join=round,line cap=round,fill=blue]
(0.7782512083616212, 0.18901250493923666)
arc (514.3380814794239:565.6619185205761:0.43645791675086554)
arc (325.66191852057597:285.79599711034854:0.3350840027383518)
arc (225.79599711034854:204.360792314009:0.4497801185622753)
arc (324.36079231400913:395.63920768599087:0.31840128691061304)
arc (155.639207685991:134.20400288965146:0.4497801185622753)
arc (74.20400288965143:34.338081479424005:0.3350840027383518)
;
\path[draw=black,join=round,line cap=round,fill=violet]
(0.4966377525794208, -0.18552582028004194)
arc (324.3607923140091:395.6392076859909:0.31840128691061304)
arc (95.6392076859911:264.3607923140089:0.1864280592536302)
;
\path[draw=black,join=round,line cap=round,fill=orange]
(0.7782512083616212, 0.1890125049392366)
arc (154.33808147942386:205.66191852057614:0.43645791675086554)
arc (265.66191852057625:454.33808147942375:0.1895555666872178)
;

%% file: cluster-brep-color.tex
\draw ($(0.0,0.0)+({1.0*cos(0.0)},{1.0*sin(0.0)})$) arc (0.0:360.0:1.0);
\path[draw=black,join=round,line cap=round,fill=white]
(0.9886381036979448, 0.15031533493470212)
arc (8.645201115424282:185.23981253778575:1.0)
 -- (-0.8178797441267831, -0.4971721097340872)
 -- (-0.7471205607988581, -0.5734623443633283)
 -- (0.32087243791808795, -0.5734623443633283)
 -- (0.9886381036979448, 0.15031533493470212)
;
\path[draw=black,join=round,line cap=round,fill=green!30!white]
(0.772001908553522, -0.6356202114389687)
arc (320.53400041760045:368.64520111542424:1.0)
 -- (0.32087243791808806, -0.5734623443633282)
 -- (0.40548228371089773, -0.7995305116520341)
 -- (0.6070067644003747, -0.7605930380858313)
 -- (0.7720019085535224, -0.6356202114389683)
;
\path[draw=black,join=round,line cap=round,fill=magenta!90!cyan!30!white]
(-0.6969434498307289, -0.7171260891489323)
arc (225.81771193541883:267.9398885766237:1.0)
 -- (0.2585065401130701, -0.9215648383424275)
 -- (0.4054822837108978, -0.7995305116520344)
 -- (0.32087243791808806, -0.5734623443633282)
 -- (-0.7471205607988578, -0.5734623443633283)
 -- (-0.6969434498307288, -0.7171260891489324)
;
\path[draw=black,join=round,line cap=round,fill=red!30!white]
(-0.8474378281389261, -0.5308946481545844)
arc (212.0659225477889:225.81771193541883:1.0)
 -- (-0.7471205607988578, -0.5734623443633284)
 -- (-0.8178797441267835, -0.49717210973408693)
 -- (-0.8474378281389259, -0.5308946481545846)
;
\path[draw=black,join=round,line cap=round,fill=yellow!30!white]
(-0.9958211812567914, -0.09132455836372065)
arc (185.23981253778572:212.0659225477889:1.0)
 -- (-0.8178797441267835, -0.49717210973408676)
 -- (-0.9958211812567914, -0.09132455836372072)
;
\path[draw=black,join=round,line cap=round,fill=blue!30!white]
(0.33772397519802466, -0.9412451947162567)
arc (289.7382668669709:304.2129424468633:1.0)
 -- (0.6070067644003746, -0.7605930380858315)
 -- (0.40548228371089756, -0.7995305116520343)
 -- (0.25850654011306984, -0.9215648383424281)
 -- (0.337723975198025, -0.9412451947162567)
;
\path[draw=black,join=round,line cap=round,fill=violet!30!white]
(-0.03594798045008818, -0.9993536624746817)
arc (267.9398885766237:289.7382668669709:1.0)
 -- (0.2585065401130693, -0.9215648383424281)
 -- (-0.03594798045008779, -0.9993536624746817)
;
\path[draw=black,join=round,line cap=round,fill=orange!30!white]
(0.5622701913442484, -0.8269535851096494)
arc (304.2129424468633:320.53400041760045:1.0)
 -- (0.6070067644003746, -0.7605930380858313)
 -- (0.5622701913442483, -0.8269535851096496)
;

%% file: mickey-original.tex
\draw ($(0.0,0.0)+({1.0*cos(83.21321070173818)},{1.0*sin(83.21321070173818)})$) arc (83.21321070173818:128.948275564627:1.0);
\draw ($(0.0,0.0)+({1.0*cos(231.05172443537288)},{1.0*sin(231.05172443537288)})$) arc (231.05172443537288:366.7867892982618:1.0);
\node at (0.0, 0.0) {2};
\node at (0.0, 0.0) {2};
\draw ($(-0.927960727138337,1.1364241342244148e-16)+({0.8333333333333334*cos(68.94827556462708)},{0.8333333333333334*sin(68.94827556462708)})$) arc (68.94827556462708:291.05172443537293:0.8333333333333334);
\node at (-0.927960727138337, 1.1364241342244148e-16) {4};
\node at (-0.927960727138337, 1.1364241342244148e-16) {4};
\draw ($(0.6187184335382291,0.618718433538229)+({0.625*cos(306.78678929826185)},{0.625*sin(306.78678929826185)})$) arc (306.78678929826185:503.21321070173815:0.625);
\node at (0.6187184335382291, 0.618718433538229) {3};
\node at (0.6187184335382291, 0.618718433538229) {3};
\draw ($(-5.567764362830023,6.81854480534649e-16)+({5.000000000000001*cos(351.0517244353729)},{5.000000000000001*sin(351.0517244353729)})$) arc (351.0517244353729:368.94827556462707:5.000000000000001);
\draw ($(1.6499158227686108,1.6499158227686104)+({1.6666666666666665*cos(203.21321070173818)},{1.6666666666666665*sin(203.21321070173818)})$) arc (203.21321070173818:246.7867892982618:1.6666666666666665);

%% file: mickey-almost.tex
\draw ($(0.0,0.0)+({1.0*cos(123.21869761331881)},{1.0*sin(123.21869761331881)})$) arc (123.21869761331881:128.948275564627:1.0);
\draw ($(0.0,0.0)+({1.0*cos(231.05172443537288)},{1.0*sin(231.05172443537288)})$) arc (231.05172443537288:406.79227620984244:1.0);
\node at (0.0, 0.0) {2};
\node at (0.0, 0.0) {2};
\draw ($(-0.927960727138337,1.1364241342244148e-16)+({0.8333333333333334*cos(68.94827556462708)},{0.8333333333333334*sin(68.94827556462708)})$) arc (68.94827556462708:291.05172443537293:0.8333333333333334);
\node at (-0.927960727138337, 1.1364241342244148e-16) {4};
\node at (-0.927960727138337, 1.1364241342244148e-16) {4};
\draw ($(0.07617779932770131,0.8716776599693195)+({0.625*cos(346.79227620984244)},{0.625*sin(346.79227620984244)})$) arc (346.79227620984244:543.2186976133188:0.625);
\node at (0.07617779932770131, 0.8716776599693195) {3};
\node at (0.07617779932770131, 0.8716776599693195) {3};
\draw ($(-5.567764362830023,6.81854480534649e-16)+({5.000000000000001*cos(351.0517244353729)},{5.000000000000001*sin(351.0517244353729)})$) arc (351.0517244353729:368.94827556462707:5.000000000000001);
\draw ($(0.2031407982072035,2.324473759918185)+({1.6666666666666665*cos(243.21869761331877)},{1.6666666666666665*sin(243.21869761331877)})$) arc (243.21869761331877:286.79227620984244:1.6666666666666665);

%% file: mickey-done.tex
\draw ($(0.0,0.0)+({1.0*cos(231.05172443537288)},{1.0*sin(231.05172443537288)})$) arc (231.05172443537288:412.52185416115066:1.0);
\node at (0.0, 0.0) {2};
\node at (0.0, 0.0) {2};
\draw ($(-0.927960727138337,1.1364241342244148e-16)+({0.8333333333333334*cos(68.94827556462708)},{0.8333333333333334*sin(68.94827556462708)})$) arc (68.94827556462708:291.05172443537293:0.8333333333333334);
\node at (-0.927960727138337, 1.1364241342244148e-16) {4};
\node at (-0.927960727138337, 1.1364241342244148e-16) {4};
\draw ($(-0.011225331376672936,0.8749279924287963)+({0.625*cos(352.52185416115066)},{0.625*sin(352.52185416115066)})$) arc (352.52185416115066:548.948275564627:0.625);
\node at (-0.011225331376672936, 0.8749279924287963) {3};
\node at (-0.011225331376672936, 0.8749279924287963) {3};
\draw ($(-5.567764362830023,6.81854480534649e-16)+({5.000000000000001*cos(351.0517244353729)},{5.000000000000001*sin(351.0517244353729)})$) arc (351.0517244353729:368.94827556462707:5.000000000000001);
\draw ($(-0.029934217004461156,2.3331413131434564)+({1.6666666666666665*cos(248.94827556462704)},{1.6666666666666665*sin(248.94827556462704)})$) arc (248.94827556462704:292.5218541611507:1.6666666666666665);

%% file: pokeball-original.tex
\draw ($(0.0,0.0)+({0.8333333333333334*cos(68.9482755646271)},{0.8333333333333334*sin(68.9482755646271)})$) arc (68.9482755646271:291.05172443537293:0.8333333333333334);
\node at (0.0, 0.0) {2};
\node at (0.0, 0.0) {2};
\draw ($(0.927960727138337,0.0)+({1.0*cos(231.05172443537288)},{1.0*sin(231.05172443537288)})$) arc (231.05172443537288:488.9482755646271:1.0);
\node at (0.927960727138337, 0.0) {1};
\node at (0.927960727138337, 0.0) {1};
\draw ($(-4.639803635691686,0.0)+({5.000000000000001*cos(3.417981091912093)},{5.000000000000001*sin(3.417981091912093)})$) arc (3.417981091912093:8.948275564627107:5.000000000000001);
\draw ($(-4.639803635691686,0.0)+({5.000000000000001*cos(351.0517244353729)},{5.000000000000001*sin(351.0517244353729)})$) arc (351.0517244353729:356.5820189080879:5.000000000000001);
\draw ($(0.20214271308629872,0.0)+({0.3333333333333333*cos(296.58201890808874)},{0.3333333333333333*sin(296.58201890808874)})$) arc (296.58201890808874:423.41798109191126:0.3333333333333333);
\draw ($(0.5479960237132975,0.0)+({0.35714285714285715*cos(123.41798109191234)},{0.35714285714285715*sin(123.41798109191234)})$) arc (123.41798109191234:236.5820189080877:0.35714285714285715);
\node at (0.3750693683997981, 0.0) {3};
\node at (0.3750693683997981, 0.0) {3};

%% file: pokeball-almost.tex
\draw ($(0.0,0.0)+({0.8333333333333334*cos(68.9482755646271)},{0.8333333333333334*sin(68.9482755646271)})$) arc (68.9482755646271:291.05172443537293:0.8333333333333334);
\node at (0.0, 0.0) {2};
\node at (0.0, 0.0) {2};
\draw ($(0.927960727138337,0.0)+({1.0*cos(231.05172443537288)},{1.0*sin(231.05172443537288)})$) arc (231.05172443537288:488.9482755646271:1.0);
\node at (0.927960727138337, 0.0) {1};
\node at (0.927960727138337, 0.0) {1};
\draw ($(-4.639803635691686,0.0)+({5.000000000000001*cos(7.842216670084054)},{5.000000000000001*sin(7.842216670084054)})$) arc (7.842216670084054:8.948275564627107:5.000000000000001);
\draw ($(-4.639803635691686,0.0)+({5.000000000000001*cos(351.0517244353729)},{5.000000000000001*sin(351.0517244353729)})$) arc (351.0517244353729:361.0062544862599:5.000000000000001);
\draw ($(0.1877147375255559,0.37351144653733603)+({0.3333333333333333*cos(301.00625448626073)},{0.3333333333333333*sin(301.00625448626073)})$) arc (301.00625448626073:427.84221667008325:0.3333333333333333);
\draw ($(0.5325374784696445,0.40019083557571716)+({0.35714285714285715*cos(127.84221667008394)},{0.35714285714285715*sin(127.84221667008394)})$) arc (127.84221667008394:241.00625448625996:0.35714285714285715);
\node at (0.3601261079976002, 0.3868511410565266) {3};
\node at (0.3601261079976002, 0.3868511410565266) {3};

%% file: pokeball-done.tex
\draw ($(0.0,0.0)+({0.8333333333333334*cos(68.9482755646271)},{0.8333333333333334*sin(68.9482755646271)})$) arc (68.9482755646271:291.05172443537293:0.8333333333333334);
\node at (0.0, 0.0) {2};
\node at (0.0, 0.0) {2};
\draw ($(0.927960727138337,0.0)+({1.0*cos(231.05172443537288)},{1.0*sin(231.05172443537288)})$) arc (231.05172443537288:488.9482755646271:1.0);
\node at (0.927960727138337, 0.0) {1};
\node at (0.927960727138337, 0.0) {1};
\draw ($(-4.639803635691686,0.0)+({5.000000000000001*cos(8.948275564627053)},{5.000000000000001*sin(8.948275564627053)})$) arc (8.948275564627053:8.948275564627107:5.000000000000001);
\draw ($(-4.639803635691686,0.0)+({5.000000000000001*cos(351.0517244353729)},{5.000000000000001*sin(351.0517244353729)})$) arc (351.0517244353729:362.11231338080285:5.000000000000001);
\draw ($(0.1796053020267756,0.46662826262868873)+({0.3333333333333333*cos(302.1123133808037)},{0.3333333333333333*sin(302.1123133808037)})$) arc (302.1123133808037:428.9482755646262:0.3333333333333333);
\draw ($(0.5238487975780942,0.4999588528164523)+({0.35714285714285715*cos(128.9482755646273)},{0.35714285714285715*sin(128.9482755646273)})$) arc (128.9482755646273:242.1123133808026:0.35714285714285715);
\node at (0.3517270498024349, 0.48329355772257054) {3};
\node at (0.3517270498024349, 0.48329355772257054) {3};

%% file: degree-2.tex
\draw ($(0.0,0.0)+({1.0*cos(70.89339464913091)},{1.0*sin(70.89339464913091)})$) arc (70.89339464913091:249.6970942081064:1.0);
\node at (0.0, 0.0) {2};
\node at (0.0, 0.0) {2};
\draw ($(1.14564392373896,0.0)+({1.25*cos(265.3322249251854)},{1.25*sin(265.3322249251854)})$) arc (265.3322249251854:490.8933946491309:1.25);
\node at (1.14564392373896, 0.0) {1};
\node at (1.14564392373896, 0.0) {1};
\draw ($(0.4112567931370628,-0.8083025777511853)+({0.7692307692307692*cos(189.6970942081064)},{0.7692307692307692*sin(189.6970942081064)})$) arc (189.6970942081064:325.3322249251854:0.7692307692307692);
\node at (0.4112567931370628, -0.8083025777511853) {4};
\node at (0.4112567931370628, -0.8083025777511853) {4};
\draw ($(-4.582575694955841,0.0)+({5.000000000000001*cos(6.676050025118485)},{5.000000000000001*sin(6.676050025118485)})$) arc (6.676050025118485:10.893394649130926:5.000000000000001);
\draw ($(-4.582575694955841,0.0)+({5.000000000000001*cos(354.588154703838)},{5.000000000000001*sin(354.588154703838)})$) arc (354.588154703838:359.05352792618976:5.000000000000001);
\draw ($(1.7821127702606052,-3.5026445035884697)+({-3.333333333333333*cos(294.588154703838)},{-3.333333333333333*sin(294.588154703838)})$) arc (294.588154703838:309.69709420810636:-3.333333333333333);
\draw ($(-0.7637626158259727,-2.101586702153082)+({2.0*cos(25.332224925185432)},{2.0*sin(25.332224925185432)})$) arc (25.332224925185432:54.588154703838015:2.0);
\draw ($(0.23688045158380078,0.2411738190138347)+({0.37037037037037035*cos(299.05352792618976)},{0.37037037037037035*sin(299.05352792618976)})$) arc (299.05352792618976:426.6760500251185:0.37037037037037035);
\draw ($(0.622436943306972,0.2604677245349415)+({0.4*cos(126.67605002511858)},{0.4*sin(126.67605002511858)})$) arc (126.67605002511858:239.0535279261896:0.4);
\node at (0.4296586974453864, 0.2508207717743881) {3};
\node at (0.4296586974453864, 0.2508207717743881) {3};

%% file: bubble-ring.tex
\draw ($(1.0109327423200576,0.0051230155744126675)+({0.530319892005336*cos(161.6385985245456)},{0.530319892005336*sin(161.6385985245456)})$) arc (161.6385985245456:197.17519508040473:0.530319892005336);
\draw ($(1.0109327423200576,0.0051230155744126675)+({0.530319892005336*cos(324.33585829684694)},{0.530319892005336*sin(324.33585829684694)})$) arc (324.33585829684694:405.57075749812225:0.530319892005336);
\draw ($(-2.137919642291556,13.016270570611747)+({-13.11371529252676*cos(101.63859852454563)},{-13.11371529252676*sin(101.63859852454563)})$) arc (101.63859852454563:105.57075749812222:-13.11371529252676);
\draw ($(0.8885423325942677,0.5108436673413717)+({0.5097072811425771*cos(200.28843100903615)},{0.5097072811425771*sin(200.28843100903615)})$) arc (200.28843100903615:221.63859852454559:0.5097072811425771);
\draw ($(0.8885423325942677,0.5108436673413717)+({0.5097072811425771*cos(345.5707574981222)},{0.5097072811425771*sin(345.5707574981222)})$) arc (345.5707574981222:434.51092497642674:0.5097072811425771);
\draw ($(7.335882076874603,-5.417856015370915)+({9.002585836496825*cos(134.51092497642668)},{9.002585836496825*sin(134.51092497642668)})$) arc (134.51092497642668:140.2884310090362:9.002585836496825);
\draw ($(0.5015997912355642,0.8666595872665438)+({0.540297794212609*cos(14.51092497642667)},{0.540297794212609*sin(14.51092497642667)})$) arc (14.51092497642667:110.33420629816351:0.540297794212609);
\draw ($(0.5015997912355642,0.8666595872665438)+({0.540297794212609*cos(218.7635797087002)},{0.540297794212609*sin(218.7635797087002)})$) arc (218.7635797087002:260.2884310090362:0.540297794212609);
\draw ($(-3.9725541763731247,2.1033416386616963)+({-4.34812945551475*cos(158.76357970870026)},{-4.34812945551475*sin(158.76357970870026)})$) arc (158.76357970870026:170.33420629816348:-4.34812945551475);
\draw ($(0.007089922074395545,1.0033449771428873)+({0.4805808972398061*cos(50.33420629816346)},{0.4805808972398061*sin(50.33420629816346)})$) arc (50.33420629816346:127.83387381187578:0.4805808972398061);
\draw ($(0.007089922074395545,1.0033449771428873)+({0.4805808972398061*cos(258.62441885330793)},{0.4805808972398061*sin(258.62441885330793)})$) arc (258.62441885330793:278.76357970870026:0.4805808972398061);
\draw ($(4.316002381557308,2.0163040726383032)+({4.647057314876854*cos(187.83387381187572)},{4.647057314876854*sin(187.83387381187572)})$) arc (187.83387381187572:198.62441885330796:4.647057314876854);
\draw ($(-0.4899202151101775,0.886505532445755)+({0.5360133479826334*cos(67.83387381187573)},{0.5360133479826334*sin(67.83387381187573)})$) arc (67.83387381187573:165.98726122215663:0.5360133479826334);
\draw ($(-0.4899202151101775,0.886505532445755)+({0.5360133479826334*cos(283.70385794030284)},{0.5360133479826334*sin(283.70385794030284)})$) arc (283.70385794030284:318.62441885330793:0.5360133479826334);
\draw ($(-17.007920713645284,-15.542685799508408)+({-23.024635631244305*cos(223.70385794030284)},{-23.024635631244305*sin(223.70385794030284)})$) arc (223.70385794030284:225.98726122215658:-23.024635631244305);
\draw ($(-0.8657107443525991,0.5127354440983432)+({0.5238188490336075*cos(105.9872612221566)},{0.5238188490336075*sin(105.9872612221566)})$) arc (105.9872612221566:197.4271439181208:0.5238188490336075);
\draw ($(-0.8657107443525991,0.5127354440983432)+({0.5238188490336075*cos(313.58423717955384)},{0.5238188490336075*sin(313.58423717955384)})$) arc (313.58423717955384:343.70385794030284:0.5238188490336075);
\draw ($(-4.25195878564159,-12.586281029544319)+({-13.260114045879305*cos(253.5842371795538)},{-13.260114045879305*sin(253.5842371795538)})$) arc (253.5842371795538:257.42714391812086:-13.260114045879305);
\draw ($(-0.994395392272275,0.014944879895287007)+({0.503912615544604*cos(137.42714391812086)},{0.503912615544604*sin(137.42714391812086)})$) arc (137.42714391812086:222.23415172343576:0.503912615544604);
\draw ($(-0.994395392272275,0.014944879895287007)+({0.503912615544604*cos(349.3578706519931)},{0.503912615544604*sin(349.3578706519931)})$) arc (349.3578706519931:373.58423717955384:0.503912615544604);
\draw ($(-2.9065401069843695,6.77411708314157)+({7.262823610756994*cos(282.23415172343573)},{7.262823610756994*sin(282.23415172343573)})$) arc (282.23415172343573:289.3578706519931:7.262823610756994);
\draw ($(-0.8518348754711385,-0.48898721009182006)+({0.5414819701765663*cos(8.540346963327405)},{0.5414819701765663*sin(8.540346963327405)})$) arc (8.540346963327405:49.35787065199311:0.5414819701765663);
\draw ($(-0.8518348754711385,-0.48898721009182006)+({0.5414819701765663*cos(162.2341517234358)},{0.5414819701765663*sin(162.2341517234358)})$) arc (162.2341517234358:258.66649146193583:0.5414819701765663);
\draw ($(2.812754178080218,-4.336725000483449)+({-5.022121850088589*cos(308.5403469633274)},{-5.022121850088589*sin(308.5403469633274)})$) arc (308.5403469633274:318.66649146193583:-5.022121850088589);
\draw ($(-0.495175960726577,-0.8634712148434005)+({0.4887818258998802*cos(53.175982525835714)},{0.4887818258998802*sin(53.175982525835714)})$) arc (53.175982525835714:68.54034696332742:0.4887818258998802);
\draw ($(-0.495175960726577,-0.8634712148434005)+({0.4887818258998802*cos(198.66649146193583)},{0.4887818258998802*sin(198.66649146193583)})$) arc (198.66649146193583:281.4510031401752:0.4887818258998802);
\draw ($(-4.538233942799991,0.046670833011640066)+({4.366950268336325*cos(341.4510031401752)},{4.366950268336325*sin(341.4510031401752)})$) arc (341.4510031401752:353.17598252583565:4.366950268336325);
\draw ($(0.014387585062941449,-0.9781802324072071)+({0.5503850483684556*cos(64.36215144274033)},{0.5503850483684556*sin(64.36215144274033)})$) arc (64.36215144274033:113.17598252583568:0.5503850483684556);
\draw ($(0.014387585062941449,-0.9781802324072071)+({0.5503850483684556*cos(221.45100314017517)},{0.5503850483684556*sin(221.45100314017517)})$) arc (221.45100314017517:323.423814554251:0.5503850483684556);
\draw ($(2.8088570290470156,-0.28698214791553933)+({-2.563754725988791*cos(4.362151442740347)},{-2.563754725988791*sin(4.362151442740347)})$) arc (4.362151442740347:23.423814554251013:-2.563754725988791);
\draw ($(0.5082749220270613,-0.856019661999397)+({0.45311141153239864*cos(108.63668427689494)},{0.45311141153239864*sin(108.63668427689494)})$) arc (108.63668427689494:124.36215144274031:0.45311141153239864);
\draw ($(0.5082749220270613,-0.856019661999397)+({0.45311141153239864*cos(263.423814554251)},{0.45311141153239864*sin(263.423814554251)})$) arc (263.423814554251:326.23831247492234:0.45311141153239864);
\draw ($(-1.0959694587033935,-2.0842216038271606)+({2.208498585762762*cos(26.238312474922356)},{2.208498585762762*sin(26.238312474922356)})$) arc (26.238312474922356:48.636684276894904:2.208498585762762);
\draw ($(0.9223724194150091,-0.5389885684160628)+({0.5700713359723726*cos(137.1751950804048)},{0.5700713359723726*sin(137.1751950804048)})$) arc (137.1751950804048:168.63668427689493:0.5700713359723726);
\draw ($(0.9223724194150091,-0.5389885684160628)+({0.5700713359723726*cos(266.23831247492234)},{0.5700713359723726*sin(266.23831247492234)})$) arc (266.23831247492234:384.335858296847:0.5700713359723726);
\draw ($(2.1924068270830794,7.264059236858878)+({-7.6052625806222816*cos(77.17519508040479)},{-7.6052625806222816*sin(77.17519508040479)})$) arc (77.17519508040479:84.33585829684698:-7.6052625806222816);

%% file: mickey-labeled.tex
\draw ($(0.0,0.0)+({2.857142857142857*cos(109.2759089602724)},{2.857142857142857*sin(109.2759089602724)})$) arc (109.2759089602724:136.99608805717713:2.857142857142857);
\draw ($(0.0,0.0)+({2.857142857142857*cos(223.0039119428228)},{2.857142857142857*sin(223.0039119428228)})$) arc (223.0039119428228:409.2759089602724:2.857142857142857);
\draw ($(-2.5394841192330255,0.0)+({2.0*cos(76.99608805717719)},{2.0*sin(76.99608805717719)})$) arc (76.99608805717719:124.64295784936957:2.0);
\draw ($(-2.5394841192330255,0.0)+({2.0*cos(212.43873034539752)},{2.0*sin(212.43873034539752)})$) arc (212.43873034539752:283.0039119428228:2.0);
\draw ($(0.4604279884765047,2.431142745093204)+({1.4285714285714286*cos(143.4132244463705)},{1.4285714285714286*sin(143.4132244463705)})$) arc (143.4132244463705:169.2759089602724:1.4285714285714286);
\draw ($(0.4604279884765047,2.431142745093204)+({1.4285714285714286*cos(349.27590896027243)},{1.4285714285714286*sin(349.27590896027243)})$) arc (349.27590896027243:396.5867755536295:1.4285714285714286);
\draw[dotted] ($(0.46042798847650473,3.779327465677267)+({1.25*cos(174.22256573838868)},{1.25*sin(174.22256573838868)})$) arc (174.22256573838868:203.4132244463705:1.25);
\draw[dotted] ($(0.46042798847650473,3.779327465677267)+({1.25*cos(336.5867755536295)},{1.25*sin(336.5867755536295)})$) arc (336.5867755536295:365.77743426161135:1.25);
\draw ($(-8.659997936148747,1.2406895731583087)+({5.000000000000001*cos(332.4387303453975)},{5.000000000000001*sin(332.4387303453975)})$) arc (332.4387303453975:364.6429578493695:5.000000000000001);
\draw[dotted] ($(-4.2882023526375175,0.35448273518808815)+({1.4285714285714286*cos(64.64295784936954)},{1.4285714285714286*sin(64.64295784936954)})$) arc (64.64295784936954:90.14570306104005:1.4285714285714286);
\draw[dotted] ($(-4.2882023526375175,0.35448273518808815)+({1.4285714285714286*cos(246.935985133727)},{1.4285714285714286*sin(246.935985133727)})$) arc (246.935985133727:272.4387303453975:1.4285714285714286);
\draw ($(-8.464947064110083,0.0)+({6.666666666666666*cos(343.0039119428228)},{6.666666666666666*sin(343.0039119428228)})$) arc (343.0039119428228:376.9960880571772:6.666666666666666);
\draw ($(0.9208559769530092,4.862285490186408)+({2.857142857142857*cos(229.2759089602724)},{2.857142857142857*sin(229.2759089602724)})$) arc (229.2759089602724:289.2759089602724:2.857142857142857);
\draw ($(0.4604279884765052,13.216620509765697)+({-9.999999999999991*cos(83.41322444637055)},{-9.999999999999991*sin(83.41322444637055)})$) arc (83.41322444637055:96.58677555362947:-9.999999999999991);
\node at (0.3, -0.3) {$s_{2q}$};
\node at (0.3, -0.3) {$s_{2q}$};
\node at (0.6104279884765047, 2.631142745093204) {$s_{3q}$};
\node at (0.6104279884765047, 2.631142745093204) {$s_{3q}$};
\node at (-2.7894841192330255, 0.0) {$s_{1q}$};
\node at (-2.7894841192330255, 0.0) {$s_{1q}$};
\draw (0.0, 0.0) -- (0.4604279884765047, 2.431142745093204);
\draw (0.0, 0.0) -- (-2.5394841192330255, 0.0);
\draw (0.0, 0.0) -- (-0.9431929298414623, 2.6969709681825513);
\draw (0.0, 0.0) -- (-2.0894489588626164, 1.9487094073848927);
\draw (-2.5394841192330255, 0.0) -- (-2.0894489588626164, 1.9487094073848927);
\draw ($(0.0,0.0)+({0.5714285714285715*cos(79.2759089602724)},{0.5714285714285715*sin(79.2759089602724)})$) arc (79.2759089602724:109.2759089602724:0.5714285714285715);
\draw ($(0.0,0.0)+({0.5714285714285715*cos(136.9960880571772)},{0.5714285714285715*sin(136.9960880571772)})$) arc (136.9960880571772:180.0:0.5714285714285715);
\node at (-0.05, 1.0) {$\theta_{23}$};
\node at (-0.05, 1.0) {$\theta_{23}$};
\node at (-0.8, 0.3) {$\theta_{21}$};
\node at (-0.8, 0.3) {$\theta_{21}$};
\draw ($(-2.5394841192330255,0.0)+({0.4*cos(0.0)},{0.4*sin(0.0)})$) arc (0.0:76.9960880571772:0.4);
\node at (-2.0894841192330253, 0.45) {$\theta_{12}$};
\node at (-2.0894841192330253, 0.45) {$\theta_{12}$};

%% file: bubble-ring-6.tex
\draw ($(1.0,0.0)+({0.9999999999999999*cos(300.0)},{0.9999999999999999*sin(300.0)})$) arc (300.0:420.0:0.9999999999999999);
\draw (1.5, 0.8660254037844385) -- (1.1102230246251565e-16, 1.224646799147353e-16);
\draw ($(0.5000000000000001,0.8660254037844386)+({0.9999999999999999*cos(0.0)},{0.9999999999999999*sin(0.0)})$) arc (0.0:120.00000000000007:0.9999999999999999);
\draw (0.0, 1.7320508075688772) -- (0.0, 1.1102230246251565e-16);
\draw ($(-0.4999999999999998,0.8660254037844387)+({0.9999999999999999*cos(59.99999999999999)},{0.9999999999999999*sin(59.99999999999999)})$) arc (59.99999999999999:180.0:0.9999999999999999);
\draw (-1.4999999999999996, 0.8660254037844388) -- (-1.1102230246251565e-16, 0.0);
\draw ($(-1.0,1.2246467991473532e-16)+({0.9999999999999999*cos(119.99999999999999)},{0.9999999999999999*sin(119.99999999999999)})$) arc (119.99999999999999:239.99999999999994:0.9999999999999999);
\draw (-1.5, -0.8660254037844384) -- (-1.1102230246251565e-16, -5.55111512312578e-17);
\draw ($(-0.5000000000000004,-0.8660254037844385)+({0.9999999999999999*cos(180.0)},{0.9999999999999999*sin(180.0)})$) arc (180.0:300.0:0.9999999999999999);
\draw (-2.220446049250313e-16, -1.7320508075688767) -- (-2.220446049250313e-16, -1.1102230246251565e-16);
\draw ($(0.5000000000000001,-0.8660254037844386)+({0.9999999999999999*cos(120.00000000000001)},{0.9999999999999999*sin(120.00000000000001)})$) arc (120.00000000000001:120.00000000000007:0.9999999999999999);
\draw ($(0.5000000000000001,-0.8660254037844386)+({0.9999999999999999*cos(239.99999999999997)},{0.9999999999999999*sin(239.99999999999997)})$) arc (239.99999999999997:360.0:0.9999999999999999);
\draw (1.5, -0.8660254037844386) -- (0.0, -1.1102230246251565e-16);

%% file: bubble-ring-7.tex
\draw ($(1.0,0.0)+({0.8677674782351161*cos(175.71428571428572)},{0.8677674782351161*sin(175.71428571428572)})$) arc (175.71428571428572:184.28571428571425:0.8677674782351161);
\draw ($(1.0,0.0)+({0.8677674782351161*cos(304.2857142857143)},{0.8677674782351161*sin(304.2857142857143)})$) arc (304.2857142857143:415.7142857142857:0.8677674782351161);
\draw (1.4888308262251284, 0.7169831376082642) -- (0.13465897563360507, 0.06484834485976558);
\draw ($(0.6234898018587336,0.7818314824680298)+({0.8677674782351161*cos(227.14285714285714)},{0.8677674782351161*sin(227.14285714285714)})$) arc (227.14285714285714:235.71428571428567:0.8677674782351161);
\draw ($(0.6234898018587336,0.7818314824680298)+({0.8677674782351161*cos(355.7142857142857)},{0.8677674782351161*sin(355.7142857142857)})$) arc (355.7142857142857:467.1428571428571:0.8677674782351161);
\draw (0.3677108474634314, 1.6110464864151237) -- (0.033258020438987734, 0.14571290823473004);
\draw ($(-0.22252093395631434,0.9749279121818236)+({0.8677674782351161*cos(47.14285714285714)},{0.8677674782351161*sin(47.14285714285714)})$) arc (47.14285714285714:158.57142857142853:0.8677674782351161);
\draw ($(-0.22252093395631434,0.9749279121818236)+({0.8677674782351161*cos(278.57142857142856)},{0.8677674782351161*sin(278.57142857142856)})$) arc (278.57142857142856:287.1428571428571:0.8677674782351161);
\draw (-1.030302899372565, 1.2919589715920843) -- (-0.09318690248616868, 0.1168526797072974);
\draw ($(-0.900968867902419,0.43388373911755823)+({0.8677674782351161*cos(98.57142857142857)},{0.8677674782351161*sin(98.57142857142857)})$) arc (98.57142857142857:209.99999999999997:0.8677674782351161);
\draw ($(-0.900968867902419,0.43388373911755823)+({0.8677674782351161*cos(330.0)},{0.8677674782351161*sin(330.0)})$) arc (330.0:338.5714285714286:0.8677674782351161);
\draw (-1.6524775486319896, 2.220446049250313e-16) -- (-0.1494601871728487, 0.0);
\draw ($(-0.9009688679024191,-0.433883739117558)+({0.8677674782351161*cos(21.428571428571413)},{0.8677674782351161*sin(21.428571428571413)})$) arc (21.428571428571413:29.999999999999964:0.8677674782351161);
\draw ($(-0.9009688679024191,-0.433883739117558)+({0.8677674782351161*cos(149.99999999999997)},{0.8677674782351161*sin(149.99999999999997)})$) arc (149.99999999999997:261.4285714285714:0.8677674782351161);
\draw (-1.0303028993725651, -1.291958971592084) -- (-0.09318690248616857, -0.11685267970729757);
\draw ($(-0.2225209339563146,-0.9749279121818236)+({0.8677674782351161*cos(72.85714285714285)},{0.8677674782351161*sin(72.85714285714285)})$) arc (72.85714285714285:81.42857142857139:0.8677674782351161);
\draw ($(-0.2225209339563146,-0.9749279121818236)+({0.8677674782351161*cos(201.42857142857142)},{0.8677674782351161*sin(201.42857142857142)})$) arc (201.42857142857142:312.8571428571429:0.8677674782351161);
\draw (0.3677108474634311, -1.6110464864151237) -- (0.03325802043898771, -0.14571290823472993);
\draw ($(0.6234898018587334,-0.7818314824680299)+({0.8677674782351161*cos(124.28571428571426)},{0.8677674782351161*sin(124.28571428571426)})$) arc (124.28571428571426:132.85714285714283:0.8677674782351161);
\draw ($(0.6234898018587334,-0.7818314824680299)+({0.8677674782351161*cos(252.85714285714283)},{0.8677674782351161*sin(252.85714285714283)})$) arc (252.85714285714283:364.2857142857143:0.8677674782351161);
\draw (1.4888308262251284, -0.7169831376082643) -- (0.13465897563360513, -0.06484834485976554);